\documentclass[a4paper,12pt]{amsart}
\usepackage{amsmath,amssymb,amsfonts,amsthm}
\usepackage[]{amsrefs}
\usepackage{bm}
\usepackage[top=25mm,bottom=25mm,left=25mm,right=25mm]{geometry}
\usepackage{ascmac}
\usepackage[all]{xy}
\usepackage{amsthm}
\usepackage{relsize}
\usepackage{tikz}
\usetikzlibrary{trees}
\usetikzlibrary{cd,arrows,matrix,backgrounds,shapes,positioning,calc,decorations.markings,decorations.pathmorphing,decorations.pathreplacing}
\usepackage{multirow}
\usetikzlibrary{intersections,arrows,calc,decorations.markings}
\newtheorem{theorem}{Theorem}[section]
\usepackage{mathtools,leftidx}

\newtheorem{proposition}[theorem]{Proposition}

\newtheorem{conjecture}[theorem]{Conjecture}
\newtheorem{corollary}[theorem]{Corollary}
\newtheorem{lemma}[theorem]{Lemma}

\theoremstyle{definition}
\newtheorem{remark}[theorem]{Remark}
\newtheorem{example}[theorem]{Example}
\newtheorem{definition}[theorem]{Definition}

\newtheorem{question}[theorem]{Question}

\makeatletter
 
 \@addtoreset{equation}{section}
\makeatother

\def\CC{\mathbb{C}}
\def\TT{\mathbb{T}}
\def\RR{\mathbb{R}}
\def\PP{\mathbb{P}}
\def\ZZ{\mathbb{Z}}
\def\QQ{\mathbb{Q}}

\newcommand{\slmc}[1]{{SL}(#1,\mathbb{C})}%特殊線型群
\allowdisplaybreaks[4]
\title[SL(2,Z)-matrixizations of Generalized Markov numbers]{SL(2,Z)-matrixizations of Generalized Markov numbers}
\author{Yasuaki Gyoda}
\address[Yasuaki Gyoda]{Graduate School of Mathematical Sciences, The University of Tokyo, 3-8-1 Komaba Meguro-ku Tokyo 153-8914, Japan}
\email{gyoda-yasuaki@g.ecc.u-tokyo.ac.jp}

\author{Shuhei Maruyama}
\address[Shuhei Maruyama]{School of Mathematics and Physics, College of Science and Engineering, Kanazawa University, Kakuma-machi, Kanazawa, Ishikawa, 920-1192, Japan}
\email{smaruyama@se.kanazawa-u.ac.jp}

\author{Yusuke Sato}
\address[Yusuke Sato]{Department of Mathematics, Osaka Institute of Technology, 5-16-1 Ohmiya, Asahi-ku, Osaka, 535-8585, Japan}
\email{yusuke.sato@oit.ac.jp}

\keywords{Markov number, Cohn matrix, 4-punctured sphere, snake graph, perfect matching, toric surface}
\subjclass[2020]{11D25, 11A55, 05C70, 14M25, 14B05}

\begin{document}
\begin{abstract}
 For $k\geq 0$, a $k$-generalized Markov number is an integer which appears in some positive integer solution to the $k$-generalized Markov equation $x^2 + y^2 + z^2 + k(yz + zx + xy) = (3 + 3k)xyz$. In this paper, we discuss a combinatorial structure of generalized Markov numbers. To investigate this structure in detail, we use two families of matrices: the $k$-generalized Cohn matrices and the $k$-Markov-monodromy matrices, which are elements of $SL(2, \mathbb{Z})$ whose $(1,2)$-entries are $k$-generalized Markov numbers. We show that these two families of matrices recover the tree structure of the positive integer solutions to the generalized Markov equation, and we give geometric  interpretations and a combinatorial interpretation of $k$-generalized Markov numbers. As an application, we provide a computation algorithm of classical Markov number from a one-dimensional dynamical viewpoint. Moreover, we clarify a relation between $k$-generalized Markov numbers and toric surface singularities via continued fractions. 
\end{abstract}
\maketitle
\tableofcontents

\section{Introduction and main results}
\subsection{Background}
In this paper, we will deal with the following equation involving a fixed non-negative integer $ k $,
\[
x^2 + y^2 + z^2 + k(yz + zx + xy) = (3 + 3k)xyz,
\]
and the structure of its positive integer solutions. This equation is called the \emph{$ k $-generalized Markov equation}, and its positive integer solutions are referred to as \emph{$ k $-generalized Markov triples}, and positive integers appearing in these solutions are called \emph{$ k $-generalized Markov numbers}. In this paper, we abbreviate them as the \emph{$ k $-GM equation}, \emph{$ k $-GM triples}, and \emph{$ k $-GM numbers}.

When $ k = 0 $, i.e., $ x^2 + y^2 + z^2 = 3xyz $, it is known as the \emph{Markov equation} and was discovered by Markov around 1880 from the perspective of Diophantine approximation theory \cites{Mar1,Mar2}. This equation and its positive integer solutions are studied from the aspect of hyperbolic geometry and combinatorics. For details, see \cite{aig}. In recent years, there has been a lot of attempts to solve the following conjecture regarding Markov numbers. 
\begin{conjecture}[\cite{frobenius}]\label{conj:markov}
 For any Markov number $b$, there is a unique Markov triple $(a,b,c)$ up to order such that $\max\{a,b,c\}=b$.
 \end{conjecture}
This is conjectured in 1913, but it is still open. When $b$ is a prime power, it is proved by Baragar \cite{baragar}, Button \cite{button}, Schmutz \cite{schmutz}, Zhang \cite{zhang}, Lang and Tan \cite{lang-tan} and others using various methods. One of the weak versions of Conjecture \ref{conj:markov}, the \emph{Aigner conjecture} \cite{aig}*{Conjecture 10.11}, has been proved in the last few years by McShane \cite{mcshane} and Lee, Li, Rabideau, and Schiffler \cite{llrs}.

The $k$-GM equation, where $k\neq 0$, was first considered for $ k = 1 $ by the first author \cite{gyo21}. Furthermore, it has been studied by Banaian and Sen \cite{bansen}. For general $k$, the first and second authors \cite{gyo-maru} study the symmetric case in the broader class
\begin{align}\label{eq:k1k2k3}
x^2 + y^2 + z^2 + k_1yz + k_2zx + k_3xy = (3 + k_1 + k_2 + k_3)xyz,
\end{align}
which was introduced by the first author and Matsushita \cite{gyomatsu}. The $(k_1,k_2,k_3)$-type equations \eqref{eq:k1k2k3} are defined as a generalization of the Markov equation in the context of cluster algebra theory. It is known that the positive integer solutions of the Markov equation (the case where $k_1=k_2=k_3=0$) possesses an operation, known as the \emph{Vieta jumping}, which generates another positive integer solution from one. This operation can be described as a \emph{mutation} of a certain \emph{cluster algebra} (see \cite{bbh}). Derived from this, the $(k_1, k_2, k_3)$-type equations are given in \cite{gyomatsu} as a class of equations where positive integer solutions has the Vieta jumping described by mutations of \emph{generalized cluster algebras}, which is a broader class than cluster algebras. 

Subsequently, the first and second authors found that several facts about the Markov numbers can be extended to the $k$-GM case. One of them is a partial solution to the following conjecture, which is a generalization of Conjecture \ref{conj:markov}:
\begin{conjecture}[\cite{gyo-maru}*{Conjecture 1.8}]\label{conj:gen-Markov}
 Let $k\in \mathbb Z_{\geq 0}$. For any $k$-GM number $b$, there is a unique $k$-GM triple $(a,b,c)$ up to order such that $\max\{a,b,c\}=b$.
 \end{conjecture}
By generalizing the method in the case of the Markov numbers, Conjecture \ref{conj:gen-Markov} is proved to be correct for any $k$ when $b$ is prime (\cite{gyo-maru}*{Theorem 1.6}). This has led to increasing expectations that the theory of the Markov equations can be organized from the perspective of the $k$-GM equations.

We will construct a theory on positive integer solutions of the $ k $-GM equation in this paper, mainly from combinatorial and geometric perspectives, including the well-known case $ k = 0 $. In below, we will introduce the main results.

\subsection{Two ways of $SL(2,\mathbb Z)$-matrixization of $k$-GM triples}
We consider the following two operations, the \emph{Vieta jumpings}:
\begin{align}\label{intro:vieta-jumping}
(a,b,c) \mapsto \left(a, \frac{a^2+kab+b^2}{c}, b\right), \quad (a,b,c) \mapsto \left(b, \frac{b^2+kbc+c^2}{a}, c\right).
\end{align}
By repeatedly applying these operations from $(1,k+2,1)$, any $k$-GM triple such that the second component is strictly maximum can be obtained (\cite{gyomatsu}). The following is  the tree of $0$-GM triples.
\begin{align*}
\begin{xy}(0,0)*+{(1,2,1)}="1",(25,16)*+{(2,5,1)}="2",(25,-16)*+{(1,5,2)}="3", (50,24)*+{(5,13,1)}="4",(50,8)*+{(2,29,5)}="5",(50,-8)*+{(5,29,2)}="6",(50,-24)*+{(1,13,5)}="7", (85,28)*+{(13,34,1)\cdots}="8",(85,20)*+{(5,194,13)\cdots}="9",(85,12)*+{(29,433,5)\cdots}="10",(85,4)*+{(2,169,29)\cdots}="11",(85,-4)*+{(29,169,2)\dots}="12",(85,-12)*+{(5,433,29)\cdots}="13",(85,-20)*+{(13,194,5)\cdots}="14",(85,-28)*+{(1,34,13)\cdots}="15",\ar@{-}"1";"2"\ar@{-}"1";"3"\ar@{-}"2";"4"\ar@{-}"2";"5"\ar@{-}"3";"6"\ar@{-}"3";"7"\ar@{-}"4";"8"\ar@{-}"4";"9"\ar@{-}"5";"10"\ar@{-}"5";"11"\ar@{-}"6";"12"\ar@{-}"6";"13"\ar@{-}"7";"14"\ar@{-}"7";"15"
\end{xy}
\end{align*}
In contrast, we consider the following two operations, the inverse of the above Vieta jumpings: 
\begin{align}\label{intro:inverse-vieta-jumping}
(a,b,c) \mapsto \left(a, c, \frac{a^2+kac+c^2}{b}\right) ,\quad (a,b,c) \mapsto \left(\frac{a^2+kac+c^2}{b},a,c\right).
\end{align}
By repeatedly applying these operations from $(1,1,1)$, any $k$-GM triple such that the second component is not strictly maximum can be obtained. The following tree is for $k=0$ case.
\begin{align*}
\begin{xy}(0,0)*+{(1,1,1)}="1",(25,16)*+{(2,1,1)}="2",(25,-16)*+{(1,1,2)}="3", (50,24)*+{(5,2,1)}="4",(50,8)*+{(2,1,5)}="5",(50,-8)*+{(5,1,2)}="6",(50,-24)*+{(1,2,5)}="7", (85,28)*+{(13,5,1)\cdots}="8",(85,20)*+{(5,1,13)\cdots}="9",(85,12)*+{(29,2,5)\cdots}="10",(85,4)*+{(2,5,29)\cdots}="11",(85,-4)*+{(29,5,2)\cdots}="12",(85,-12)*+{(5,2,29)\cdots}="13",(85,-20)*+{(13,1,5)\cdots}="14",(85,-28)*+{(1,5,13)\cdots}="15",\ar@{-}"1";"2"\ar@{-}"1";"3"\ar@{-}"2";"4"\ar@{-}"2";"5"\ar@{-}"3";"6"\ar@{-}"3";"7"\ar@{-}"4";"8"\ar@{-}"4";"9"\ar@{-}"5";"10"\ar@{-}"5";"11"\ar@{-}"6";"12"\ar@{-}"6";"13"\ar@{-}"7";"14"\ar@{-}"7";"15"
\end{xy}
\end{align*}

Cohn introduced the \emph{Cohn triple} in \cite{cohn}, which enriches tree structures of (0-generalized) Markov triples. This triple consists of elements in $SL(2,\mathbb Z)$, whose $(1,2)$-entries form a Markov triple. This concept can be regarded as a $2\times 2$ matrixization of the Markov triple. As a further generalization, the first and second authors introduced the \emph{$k$-generalized Cohn triple} in \cite{gyo-maru}. It is defined as a triple $(P,Q,R)\in SL(2,\mathbb Z)^3$ which satisfies
\begin{itemize}\setlength{\leftskip}{-15pt}
     \item $(\mathrm{tr} (P),\mathrm{tr}(Q),\mathrm{tr}(R))=((3+3k)p_{12}-k,(3+3k)q_{12}-k,(3+3k)r_{12}-k)$, where $p_{12},q_{12},r_{12}$ are the $(1,2)$-entries of $P,Q,R$ respectively,
     \item $Q=PR-S$, where $S=\begin{bmatrix}
         k&0\\3k^2+3k&k
     \end{bmatrix}$,
     \item $(p_{12},q_{12},r_{12})$ is a $k$-GM triple.
\end{itemize}
In this paper, we abbreviate the $k$-generalized Cohn triple as the \emph{$k$-GC triple}. The following theorem is essentially proved in \cite{gyo-maru}:
\begin{theorem}[Proposition \ref{pr:all-cohn-triple}, Corollary \ref{cor:inverse-all-cone}]\label{thm:intro-cohn}The following statements hold:
\begin{itemize} \setlength{\leftskip}{-15pt}
\item[(1)] 
Every $k$-GC triple $(P,Q,R)$ with $q_{12}> \max\{p_{12},r_{12}\}$ is obtained by applying 
\begin{equation}\label{eq:k-gen-cohn-jumping}
    (P,Q,R)\mapsto (P,PQ-S,Q)\quad \text{or} \quad (P,Q,R)\mapsto (Q,QR-S,R)
\end{equation}
successively to a $k$-GC triple with $(p_{12},q_{12},r_{12})=(1,k+2,1)$. Moreover, the transformations of $(1,2)$-entries in \eqref{eq:k-gen-cohn-jumping} coincide with the operations \eqref{intro:vieta-jumping} of the $k$-GM triples. 
%That is, the following two diagrams commute:
%\[\begin{tikzcd}
%	{(P,Q,R)} && {(a,b,c)} \\
%	{(P,PQ-S,Q)} && {\left(a,\frac{a^2+kab+b^2}{c},b\right)}
%	\arrow["{(1,2)\text{-entry}}", maps to, from=1-1, to=1-3]
%	\arrow[maps to, from=1-1, to=2-1]
%	\arrow[maps to, from=1-3, to=2-3]
%	\arrow["{(1,2)\text{-entry}}"', maps to, from=2-1, to=2-3]
%\end{tikzcd}
%\begin{tikzcd}
%	{(P,Q,R)} && {(a,b,c)} \\
%	{(Q,QR-S,R)} && {\left(b,\frac{b^2+kbc+c^2}{a},c\right),}
%	\arrow["{(1,2)\text{-entry}}", maps to, from=1-1, to=1-3]
%	\arrow[maps to, from=1-1, to=2-1]
%	\arrow[maps to, from=1-3, to=2-3]
%	\arrow["{(1,2)\text{-entry}}"', maps to, from=2-1, to=2-3]
%\end{tikzcd}\]
\item[(2)] 
Every $k$-GC triple $(P,Q,R)$ with $q_{12}\leq \max\{p_{12},r_{12}\}$ is obtained by applying 
\begin{equation}\label{eq:inverse-k-gen-cohn-jumping}
    (P,Q,R)\mapsto (P,R,P^{-1}(R+S))\quad \text{or} \quad (P,Q,R)\mapsto ((P+S)R^{-1},P,R)
\end{equation}
successively to a $k$-GC triple with $(p_{12},q_{12},r_{12})=(1,1,1)$. Moreover, the transformations of $(1,2)$-entries in \eqref{eq:inverse-k-gen-cohn-jumping} coincide with the operations \eqref{intro:inverse-vieta-jumping} of the $k$-GM triples. 
%That is, the following two diagrams commute:
%\[\begin{tikzcd}[scale cd=0.9]
%	{(P,Q,R)} && {(a,b,c)} \\
%	{(P,R,P^{-1}(R+S))} && {\left(a,c,\frac{a^2+kac+c^2}{b}\right)}
%	\arrow["{(1,2)\text{-entry}}", maps to, from=1-1, to=1-3]
%	\arrow[maps to, from=1-1, to=2-1]
%	\arrow[maps to, from=1-3, to=2-3]
%	\arrow["{(1,2)\text{-entry}}"', maps to, from=2-1, to=2-3]
%\end{tikzcd}
%\begin{tikzcd}[scale cd=0.9]
%	{(P,Q,R)} && {(a,b,c)} \\
%	{((P+S)R^{-1},P,R)} && {\left(\frac{a^2+kac+c^2}{b},a,c\right).}
%	\arrow["{(1,2)\text{-entry}}", maps to, from=1-1, to=1-3]
%	\arrow[maps to, from=1-1, to=2-1]
%	\arrow[maps to, from=1-3, to=2-3]
%	\arrow["{(1,2)\text{-entry}}"', maps to, from=2-1, to=2-3]
%\end{tikzcd}\]
\end{itemize}
\end{theorem}
In addition to the aforementioned $k$-GC triple, we introduce another $2\times 2$ matrixization. The \emph{$k$-Markov-monodromy triple} $(X,Y,Z)\in SL(2,\mathbb Z)^3$ is defined as a triple satisfying these conditions:
\begin{itemize}\setlength{\leftskip}{-15pt}
     \item $\mathrm{tr} (X)=\mathrm{tr}(Y)=\mathrm{tr}(Z)=-k$,
     \item $XYZ=T$, where $T=\begin{bmatrix}
         -1&0\\3k+3&-1
     \end{bmatrix}$,
     \item $(x_{12},y_{12},z_{12})$ is a $k$-GM triple, where $x_{12},y_{12},z_{12}$ are the $(1,2)$-entries of $X,Y,Z$ respectively.
\end{itemize}
In this paper, we abbreviate this triple as the \emph{$k$-MM triple}. 
The background of the $k$-MM triple is as follows. The following trace identity in $SL(2,\mathbb C)$ holds: 
\[
x^2 + y^2 + z^2 + (ad+bc)x + (bd+ca)y + (cd+ab)z + a^2 + b^2 + c^2 + d^2 + abcd - 4 = xyz,
\] for any $(X,Y,Z)\in SL(2,\mathbb C)$, where $a:=-\mathrm{tr}(X), b:=-\mathrm{tr}(Y), c:=-\mathrm{tr}(Z), d:=-\mathrm{tr}(XYZ), x:=-\mathrm{tr}(YZ), y:=-\mathrm{tr}(ZX), z:=-\mathrm{tr}(XY)$ (for detail, see \cite{luo} or \cite{nana}). The \emph{second Markov equation} 
\[
x^2 + y^2 + z^2 = xyz,
\]
which is an equation closely related to the Markov equation, can be restored by substituting the trace identity with $a=b=c=0, d=2$. Inspired by this, the \emph{$k$-generalized second Markov equation} was defined by setting $a=b=c=k, d=2$, that is,
\begin{align}\label{eq:generalized-second-Markov}
x^2 + y^2 + z^2 + (2k+k^2)(x+y+z) + 2k^3 + 3k^2 = xyz.
\end{align}
It is stated in \cite{gyo-maru} that a triple of traces of a $k$-GC triple $(P,Q,R)$ is a solution to \eqref{eq:generalized-second-Markov}.
By these facts, we expect that there exists a suitable $SL(2,\mathbb C)$-triple $(X,Y,Z)$ satisfying the following conditions for $(P,Q,R)$:
\begin{itemize}\setlength{\leftskip}{15pt}
    \item [(MM-1)]$\mathrm{tr}(P)=-\mathrm{tr}(YZ),\ \mathrm{tr}(Q)=-\mathrm{tr}(ZX),\ \mathrm{tr}(R)=-\mathrm{tr}(XY)$,
    \item [(MM-2)]$\mathrm{tr} (X)=\mathrm{tr}(Y)=\mathrm{tr}(Z)=-k$,
    \item[(MM-3)]  $\mathrm{tr} (XYZ)=-2$.
\end{itemize}
The $k$-MM triple is introduced as a triple satisfying the above conditions. We can see that $(X,Y,Z)$ satisfies the conditions (MM-2) and (MM-3) immediately by definition. An explanation of the fact that this triple has property (MM-1) is deferred to the next subsection. 

Here, we will explain that $k$-MM triples have properties closely resembling those of $k$-GC triples. In fact, $k$-MM triples have the following property, which runs in parallel with Theorem \ref{thm:intro-cohn}: 
\begin{theorem}[Proposition \ref{pr:all-Markov-monodromy-triple}, Corollary \ref{pr:all-Markov-monodromy-triple2}]\label{thm:intro-Markov-monodromy}The following statements hold:
\begin{itemize}\setlength{\leftskip}{-15pt}
\item [(1)]
Every $k$-MM triple $(X,Y,Z)$ with $y_{12}> \max\{x_{12},z_{12}\}$ is obtained by applying 
\begin{equation}\label{eq:k-Markov-monodromy-jumping}
    (X,Y,Z)\mapsto (X,YZY^{-1},Y)\quad \text{or} \quad (X,Y,Z)\mapsto (Y,Y^{-1}XY,Z)
\end{equation}
successively to a $k$-MM triple with $(x_{12},y_{12},z_{12})=(1,1,1)$. Moreover, the transformations of $(1,2)$-entries in \eqref{eq:k-Markov-monodromy-jumping} coincide with the operations \eqref{intro:vieta-jumping} of the $k$-GM triples. 
%That is, the following two diagrams commute:
%\[\begin{tikzcd}[scale cd=0.9]
%	{(X,Y,Z)} && {(a,b,c)} \\
%	{(X,YZY^{-1},Y)} && {\left(a,\frac{a^2+kab+b^2}{c},b\right)}
%	\arrow["{(1,2)\text{-entry}}", maps to, from=1-1, to=1-3]
%	\arrow[maps to, from=1-1, to=2-1]
%	\arrow[maps to, from=1-3, to=2-3]
%	\arrow["{(1,2)\text{-entry}}"', maps to, from=2-1, to=2-3]
%\end{tikzcd}
%\begin{tikzcd}[scale cd=1]
%	{(X,Y,Z)} && {(a,b,c)} \\
%	{(Y,Y^{-1}XY,Z)} && {\left(b,\frac{b^2+kbc+c^2}{a},c\right).}
%	\arrow["{(1,2)\text{-entry}}", maps to, from=1-1, to=1-3]
%	\arrow[maps to, from=1-1, to=2-1]
%	\arrow[maps to, from=1-3, to=2-3]
%	\arrow["{(1,2)\text{-entry}}"', maps to, from=2-1, to=2-3]
%\end{tikzcd}\]
\item[(2)] Every $k$-MM triple $(X,Y,Z)$ with $y_{12}\leq \max\{x_{12},z_{12}\}$ is obtained by applying 
\begin{equation}\label{eq:inverse-k-Markov-monodromy-jumping}
    (X,Y,Z)\mapsto (X,Z,Z^{-1}YZ)\quad \text{or} \quad (X,Y,Z)\mapsto (XYX^{-1},X,Z)
\end{equation}
successively to a $k$-MM triple with $(x_{12},y_{12},z_{12})=(1,k+2,1)$. Moreover, the transformations of $(1,2)$-entries in \eqref{eq:inverse-k-Markov-monodromy-jumping} coincide with the operations \eqref{intro:inverse-vieta-jumping} of the $k$-GM triples. 
%That is, the following two diagrams commute:
%\[\begin{tikzcd}[scale cd=1]
%	{(X,Y,Z)} && {(a,b,c)} \\
%	{(X,Z,Z^{-1}YZ)} && {\left(a,c,\frac{a^2+kac+c^2}{b}\right)}
%	\arrow["{(1,2)\text{-entry}}", maps to, from=1-1, to=1-3]
%	\arrow[maps to, from=1-1, to=2-1]
%	\arrow[maps to, from=1-3, to=2-3]
%	\arrow["{(1,2)\text{-entry}}"', maps to, from=2-1, to=2-3]
%\end{tikzcd}
%\begin{tikzcd}[scale cd=1]
%	{(X,Y,Z)} && {(a,b,c)} \\
%	{(XYX^{-1},X,Z)}&& {\left(\frac{a^2+kac+c^2}{b}, a,c\right).}
%	\arrow["{(1,2)\text{-entry}}", maps to, from=1-1, to=1-3]
%	\arrow[maps to, from=1-1, to=2-1]
%	\arrow[maps to, from=1-3, to=2-3]
%	\arrow["{(1,2)\text{-entry}}"', maps to, from=2-1, to=2-3]
%\end{tikzcd}\]
\end{itemize}
\end{theorem}
To the best of the authors' knowledge, the definition and tree structure of $k$-MM triples are not known concepts even in the case $k=0$, unlike $k$-GC triples.
\subsection{Relations between $k$-GC triples and $k$-MM triples}
As stated in the previous subsection, the $k$-GC triple and the $k$-MM triple have an analogy. We explicitly provide the correspondences between these two triples, $\Psi$ and $\Phi$, as follows. Let $M(2,\mathbb Z)$ be the set of $2\times 2$ matrices whose entries are integers. First, we consider the following map
$\psi\colon M(2,\mathbb Z) \to M(2,\mathbb Z)$:
\[\psi\colon \begin{bmatrix}
    m_{11}&m_{12}\\m_{21}&m_{22}
\end{bmatrix}\mapsto \begin{bmatrix}
    -m_{11}+m_{12}k-k&m_{12}\\m_{21}-(k+3)m_{11}+k(2k+3)(m_{12}-1) & -m_{22}+(2k+3)m_{12}-k
\end{bmatrix}.\] 
This map is a bijection. More strongly, the following holds:
\begin{theorem}[Proposition \ref{pr:Markov-monodromy-cohn-triple}, Theorem \ref{thm:BT-CT2}]The following statements hold:
\begin{itemize}\setlength{\leftskip}{-15pt}
\item [(1)] the map $\Psi\colon M(2,\mathbb Z)^3 \to M(2,\mathbb Z)^3,\ (X,Y,Z)\mapsto(\psi(X),\psi(Y),\psi(Z))$ induces a bijection from the set of $k$-MM triples to the set of $k$-GC triples,
\item [(2)] the map $\Psi$ is compatible with the opetation \eqref{eq:k-Markov-monodromy-jumping} of $k$-MM triples and  the operation \eqref{eq:k-gen-cohn-jumping} of $k$-GC triples, that is, the following two diagrams commute:  
\[\begin{tikzcd}
	{(X,Y,Z)} &  {(P,Q,R)} \\
	{(X,YZY^{-1},Y)} &  {(P,PQ-S,Q)}
	\arrow["\Psi", maps to, from=1-1, to=1-2]
	\arrow[maps to, from=1-1, to=2-1]
	\arrow[maps to, from=1-2, to=2-2]
	\arrow["\Psi"', maps to, from=2-1, to=2-2]
\end{tikzcd}
\begin{tikzcd}
	{(X,Y,Z)} & {(P,Q,R)} \\
	{(Y,Y^{-1}XY,Z)} & {(Q,QR-S,R).}
	\arrow["\Psi", maps to, from=1-1, to=1-2]
	\arrow[maps to, from=1-1, to=2-1]
	\arrow[maps to, from=1-2, to=2-2]
	\arrow["\Psi"', maps to, from=2-1, to=2-2]
\end{tikzcd}\]
\end{itemize}
\end{theorem}

%\begin{example}\label{ex:Psi}
%Compare the first tree in Example \ref{ex:intro-Markov-monodromy} with the first tree in Example \ref{ex:intro-cohn}. We can see that $\Psi$ gives a correspondence between vertices of these trees. The same thing can be said for a pair of the second trees.
%\end{example}

We will define another map $\Phi$ as 
\[\Phi \colon M(2,\mathbb Z)^3 \to M(2,\mathbb Z)^3,\ (X,Y,Z)\mapsto(-(YZ)^{-1},-(XZ)^{-1},-(XY)^{-1}).\]
This map gives another relation between $k$-MM triples and $k$-GC triples. 
\begin{theorem}[Corollaries \ref{cor:bijection-triples}, \ref{cor:tree-iso-BT-CTdag}]\label{intro:k-Markov-monodromy-decomp}
The following statements hold: 
\begin{itemize}\setlength{\leftskip}{-15pt}
\item [(1)]
the map $\Phi$ induces a bijection from the set of $k$-MM triples to the set of $k$-GC triples,
\item[(2)] the map $\Phi$ is compatible with the operation \eqref{eq:k-Markov-monodromy-jumping} of $k$-MM triples and the operation \eqref{eq:inverse-k-gen-cohn-jumping} of $k$-GC triples, that is, the following two diagrams commute:  
\[\begin{tikzcd}
	{(X,Y,Z)} &  {(P,Q,R)} \\
	{(X,YZY^{-1},Y)} &  {(P,R,P^{-1}(R+S))}
	\arrow["\Phi", maps to, from=1-1, to=1-2]
	\arrow[maps to, from=1-1, to=2-1]
	\arrow[maps to, from=1-2, to=2-2]
	\arrow["\Phi"', maps to, from=2-1, to=2-2]
\end{tikzcd}
\begin{tikzcd}
	{(X,Y,Z)} & {(P,Q,R)} \\
	{(Y,Y^{-1}XY,Z)} & {((P+S)R^{-1},P,R).}
	\arrow["\Phi", maps to, from=1-1, to=1-2]
	\arrow[maps to, from=1-1, to=2-1]
	\arrow[maps to, from=1-2, to=2-2]
	\arrow["\Phi"', maps to, from=2-1, to=2-2]
\end{tikzcd}\]
\end{itemize}
\end{theorem}
Theorem \ref{intro:k-Markov-monodromy-decomp} (1) implies that $k$-MM triples have the property (MM-1) in the previous subsection. 

%\begin{example}\label{ex:Phi}
%Compare the first tree in Example \ref{ex:intro-Markov-monodromy} with the second tree in Example \ref{ex:intro-cohn}. We can see that $\Phi$ gives a correspondence between vertices of these trees. The same thing can be said for the second tree in Example \ref{ex:intro-Markov-monodromy} and the first tree in Example \ref{ex:intro-cohn}.
%\end{example}

Let us compare $\Psi$ and $\Phi$. The map $\Psi$ preserves Vieta jumpings, whereas $\Phi$ transfers Vieta jumpings to the inverses. Furthermore, these two maps have the following relation. 
\begin{theorem}[Corollary \ref{cor:Markov-monodromy-decom-algorithm}]
The composition map $(\Phi\circ\Psi^{-1})^2$ (resp. $(\Psi\circ\Phi^{-1})^2$) induces the identity map on the set of $k$-GC triples (resp. $k$-MM triples).
\end{theorem}

The correspondence $\Phi^{-1}\colon (P,Q,R)\mapsto (X,Y,Z)$ is called the \emph{Markov-monodromy decomposition}. While the explicit form of the map cannot be directly described from its definition, the theorem above implies that $\Phi^{-1} = \Psi^{-1} \circ \Phi \circ \Psi^{-1}$, and it is possible to compute this right-hand side explicitly.

\subsection{Realization of Markov triples as fixed points of $2$-MM triples}

We consider the case $k=2$. When we regard a $2$-MM matrix (i.e., a component of a $2$-MM triple) as a M\"obius transformation of $\mathbb RP^1=\mathbb R\cup \{\infty\}$, it is of parabolic type. Therefore, there is a unique fixed point, which is contained in $\mathbb Q\cup \{\infty\}$, of each $2$-MM matrix. We have the following result:

\begin{theorem}[Corollary \ref{parabolic-markov}]
Let $(X,Y,Z)$ be a $2$-MM triple. The following statements hold: 
\begin{itemize}\setlength{\leftskip}{-15pt}
\item [(1)] if $\dfrac{p}{p'},\dfrac{q}{q'},\dfrac{r}{r'}$ are irreducible fractions of the fixed points of $X,Y,Z$ respectively, then $(|p|,|q|,|r|)$ is a Markov triple, where we regard $\infty$ as $\dfrac{1}{0}$, 
\item[(2)] the correspondence $(X,Y,Z)\mapsto (|p|,|q|,|r|)$ is compatible with the opelation \eqref{eq:k-Markov-monodromy-jumping} of $2$-MM triples and the opelation \eqref{intro:vieta-jumping} of Markov triples.  
\end{itemize}
\end{theorem}

By using the above theorem, we obtain the following algorithm to calculate Markov numbers. We consider the following tree $\mathrm{P}\mathbb{T}(\ell)$ for an integer $\ell$:

\begin{itemize}\setlength{\leftskip}{-15pt}
    \item [(1)] the root vertex is 
        \[\left(\begin{bmatrix}
            1\\-\ell-1
        \end{bmatrix},\begin{bmatrix}
            2\\-2\ell+1
        \end{bmatrix},\begin{bmatrix}
            1\\-\ell+2
        \end{bmatrix}\right),\]
    \item [(2)] for a vertex $\left(\begin{bmatrix}
        p\\p'
    \end{bmatrix},\begin{bmatrix}
        q\\q'
    \end{bmatrix},\begin{bmatrix}
        r\\r'
    \end{bmatrix}\right)$, we consider the following two children of it:
\[\begin{xy}(0,0)*+{\left(\begin{bmatrix}
        p\\p'
    \end{bmatrix},\begin{bmatrix}
        q\\q'
    \end{bmatrix},\begin{bmatrix}
        r\\r'
    \end{bmatrix}\right)}="1",(-40,-15)*+{\left(\begin{bmatrix}
        p\\p'
    \end{bmatrix},\begin{bmatrix}
        q^2r'-qq'r-r\\-q'^2r+qq'r'-r'
    \end{bmatrix},\begin{bmatrix}
        q\\q'
    \end{bmatrix}\right)}="2",(40,-15)*+{\left(\begin{bmatrix}
        q\\q'
    \end{bmatrix},\begin{bmatrix}
        -q^2p'+qq'p-p\\q'^2p-qq'p'-p'
    \end{bmatrix},\begin{bmatrix}
        r\\r'
    \end{bmatrix}\right).}="3", \ar@{-}"1";"2"\ar@{-}"1";"3"
\end{xy}\]
\end{itemize}
We have the following theorem:
\begin{theorem}[Theorem \ref{thm:PT-LLMT}]
For a vertex $\left(\begin{bmatrix}
        p\\p'
    \end{bmatrix},\begin{bmatrix}
        q\\q'
    \end{bmatrix},\begin{bmatrix}
        r\\r'
    \end{bmatrix}\right)$ in the tree $\mathrm{P}\mathbb{T}(\ell)$, the following statements hold: 
\begin{itemize}\setlength{\leftskip}{-15pt}
\item [(1)] $(p,q,r)$ is a Markov triple,
\item[(2)] the transformations
\[ (p,q,r)\mapsto (p,q^2r'-qq'r-r,q), \quad (p,q,r)\mapsto (q,-q^2p'+qq'p-p,r)\] coincide with the Vieta jumpings of a Markov triple.
\end{itemize}
\end{theorem}
We denote by $\mathrm{LP}\mathbb{T}(\ell)$ the full subtree of $\mathrm{P}\mathbb{T}(\ell)$ whose root is the left child of the root of $\mathrm{P}\mathbb{T}(\ell)$. We have the following conjecture:
\begin{conjecture}\label{conj:LPT}
For some (in fact, for all) $\ell\in \mathbb Z$, the upper entries of second components of all vertices of $\mathrm{LP}\mathbb{T}(\ell)$ are distinct.      
\end{conjecture}
From the following viewpoint, Conjecture \ref{conj:LPT} is important:
 \begin{proposition} Conjecture \ref{conj:LPT} is equivalent to Conjecture \ref{conj:markov}. 
\end{proposition}
In this paper, we will prove a weak version of Conjecture \ref{conj:LPT}.

\begin{theorem}[Theorem \ref{thm:p/p'<q/q'<r/r'}]
We fix $\ell\in \mathbb Z$. The second components of all vertices of $\mathrm{LP}\mathbb{T}(\ell)$ are distinct.    
\end{theorem}

\subsection{Calculation algorithm of $k$-GM number from irreducible fraction}
In this paper, we also introduce a calculation algorithm of $k$-GM number from an irreducible fraction. In the case $k=0$, some mathematicians discovered methods constructing a Markov number from an irreducible fraction,  for example, Propp \cites{propp} and {{\c{C}}anak\c{c}\i} and Schiffler \cite{cs18}. Moreover, Banaian and Sen found a method which can be applied to the cases $k=0,1$ in \cite{bansen}. We will introduce a generalization of the Banaian--Sen's method that can be applied to arbitrary $k$. Furthermore, we prove that this method yields $k$-GM numbers by using $k$-MM triples and $k$-GC triples. 

Here we only provide an overview of the method and its results. For a given irreducible fraction $ t\in (0,1] $, we consider a line segment with slope $ t $ in $ \mathbb{R}^2 $. Using this line segment, we construct a figure called a \emph{pre-snake graph} (see Section 7.2). For the components of this pre-snake graph, we assign signs $ \{+,-\} $ according to a certain rule, and from these signs, we construct a continued fraction $ F^+(k,t) $. To describe the result, we use the \emph{Farey triple}. For a triple of irreducible fractions $\left(\dfrac{a}{b},\dfrac{c}{d},\dfrac{e}{f}\right)$, it is called the Farey triple if $|ad-bc|=|cf-de|=|af-be|=1$. The result is the following theorem:
\begin{theorem}[Theorem \ref{thm:M_t-C_t-combinatorics}, Corollary \ref{cor:k-gen.snake-markov}] Let $m_{k,t}$ be the numerator of $F^{+}(k,t)$ for an irreducible fraction $t\in (0,1]$. The following statements hold:
\begin{itemize}\setlength{\leftskip}{-15pt}
\item [(1)] $m_{k,t}$ is a $k$-GM number,
\item [(2)] for any $k$-GM number $b\neq 1$, there exists $t\in (0,1]$ such that $b=m_{k,t}$,
\item[(3)] $(m_{k,r},m_{k,t},m_{k,s})$ is a $k$-GM triple if and only if $(r,t,s)$ is a Farey triple. 
\end{itemize}
\end{theorem}

We note that we do not know whether the uniqueness of $m_{k,t}$ in (2) holds or not. We have the following proposition (this result is essentially given by \cite{gyo-maru}*{Corollary 4.2}):

\begin{proposition}
The following condition is equivalent to Conjecture \ref{conj:gen-Markov}: the map $(0,1]\cap\mathbb Q\to \mathbb Z_{\geq 0},\ t\mapsto m_{k,t}$ is an injective map.   
\end{proposition}

The denominator of $ F^+(k,t) $ also carries significant meaning. Let $(r,t,s)$ be a Farey triple with $r<t<s$. We consider solutions $x$ to equations
\begin{align*}
    m_{k,r}x&\equiv m_{k,s} \mod m_{k,t},\\
     m_{k,r}x&\equiv -m_{k,s} \mod m_{k,t},\\
    m_{k,s}x&\equiv  m_{k,r} \mod m_{k,t},\\
    m_{k,s}x&\equiv - m_{k,r} \mod m_{k,t}.
\end{align*} Each solution is unique in the range $\left(0,m_t\right)$ in this situation. These numbers are called the \emph{characteristic numbers} and we denote them by $u^+_{k,t}, u^-_{k,t},v^+_{k,t}, v^-_{k,t}$, respectively. Note that it seems that $u^+_{k,t}, u^-_{k,t},v^+_{k,t}, v^-_{k,t}$ are depend on $k$ and a Farey triple $(r,t,s)$, but since $t$ determines a Farey triple $(r,t,s)$ with $r<t<s$ uniquely, it depends only on $k$ and $t$. We will prove the following theorem in this paper:
\begin{theorem}[Theorem \ref{thm:u_t-denominator}]
 For any $k\in \mathbb Z_{\geq 0}$ and an irreducible fraction $t\in (0,1]$, we have $F^+(k,t)=\dfrac{m_{k,t}}{u^+_{k,t}}$.   
\end{theorem}

We will introduce the results on the characteristic numbers in Section 7.4.

\subsection{$k$-GM numbers and HJ-continued fractions}
%\subsection{Property of HJ-continued fractions obtained from $k$-GM numbers}
Section 8 deals with negative type continued fractions for $k$-GM numbers. 
Let $r$ and $a$ are positive integers such that $1 \leq a < r$  and ${\rm gcd}(r, a)=1$. 
Then the Hirzebruch-Jung continued fraction (shortly, HJ-continued fraction) of $r/a$ is defined by
\[\frac{r}{a}=b_1-\dfrac{1}{b_2-\dfrac{1}{\ddots-\dfrac{\ddots}{b_{\ell-1}-\dfrac{1}{b_\ell}}}}.\]
For simplicity of notation, we write this continued fraction by $[[b_1,\dots,b_\ell]]$. The HJ-continued fraction is closely related to the cyclic quotient singularity in algebraic geometry.
For a $k$-GM number $m_{k,t}$ and its characteristic number $u_{k,t}^+$, we consider the cyclic quotient singularity of type $\frac{1}{m_{k,t}}(1,u_{k,t}^+)$ and its minimal resolution. Then the self-intersection numbers of exceptional curves of the minimal resolution are given by the HJ-continued fraction of $m_{k,t}/u_{k,t}^+$. In other words, the characterization of continued fractions is nothing but the characterization of cyclic quotient singularities.

\begin{definition}\label{def:kwahlchains}
Let $k\in \mathbb Z_{\geq0}$. \emph{$k$-Wahl chains} are defined as follows.
\begin{itemize}\setlength{\leftskip}{-15pt}
\item [(i)] $[[k+2]]$ is a $k$-Wahl chain.
\item[(ii)] If $[[b_1,\dots,b_l]]$ is a $k$-Wahl chain, then $[[b_1+1,b_2,\dots,b_\ell,2]]$ and $[[2,b_1,\dots, b_{\ell-1},b_{\ell}+1]]$ are also $k$-Wahl chains.
\end{itemize}
\end{definition}

We will show the following theorem:

\begin{theorem}[Theorem \ref{thm:k-wahl-chain-Markov}]
Let $m_{k,t}$ be a $k$-GM number labeled with an irreducible fraction $t \in (0,1]$, and let $u_{k,t}^+$ be its characteristic number. Then the HJ-continued fraction of $m_{k,t}/u_{k,t}^+$ is a $k$-Wahl chain.
\end{theorem}

In the case $k=2$, the cyclic quotient singularity of type $\frac{1}{m_{k,t}}(1,u_{k,t}^+)$
admits a $\QQ$-Gorenstein one parameter smoothing (see \cite{hp10} and \cite{Per22}). In other words, it is a singularity of class $T$. The singularities of class $T$ are fundamental objects for understanding the deformation theory of surface singularities (\cite{KSB}). This theorem suggests a relation between $k$-GM numbers and deformation theory.
In addition, we obtain the following result, which says that the HJ-continued fraction of a $k$-GM number can be obtained from the HJ-continued fraction of a smaller $k$-GM number.

\begin{theorem}[Theorem \ref{thm:smallk-GM}]\label{thm:k-GMofHJCF}
For a $k$-GM triple $(m_{k,r},m_{k,t},m_{k,s})$ and these characteristic numbers $u^+_{k,t}, v^-_{k,r},$ and $ v^-_{k,s}$, we have 
\[
\frac{m_{k,t}}{u^+_{k,t}}=\left[\left[ \frac{m_{k,r}}{v^-_{k,r}}, 3k+4, \frac{m_{k,s}}{v^-_{k,s}} \right]\right].
\]
\end{theorem}

In the case $k=0$, the characteristic numbers $u^+_{k,t}, u^-_{k,t}, v^+_{k,t}$  and $ v^-_{k,t}$ are the same number. Theorem \ref{thm:k-GMofHJCF} is a generalization of the following proposition.

\begin{proposition}[\cite{UZ}*{Proposition 3.4}]\label{UZ}
Let $\displaystyle \frac{m_{0,r}}{u^+_{0,r}},\frac{m_{0,t}}{u^+_{0,t}},\frac{m_{0,s}}{u^+_{0,s}}$ be the fractions of $0$-Wahl chains. Then $\displaystyle \frac{m_{0,t}}{u^+_{0,t}}=\left[\left[\frac{m_{0,r}}{u^+_{0,r}},4,\frac{m_{0,s}}{u^+_{0,s}}\right]\right]$ if and only if $(m_{0,r},m_{0,t},m_{0,s})$ is a Markov triple with $m_{0,r} < m_{0,s} < m_{0,t}$.
\end{proposition}

\section{Generalized Markov equation}
In this section, we recall facts about the generalized Markov equation according to \cite{gyo-maru}*{Section 2}. Let $k\in \ZZ_{\geq0}$. We consider the following equation:
\[
    x^2+y^2+z^2+k(yz+xz+xy)=(3+3k)xyz.
\]
It is called the \emph{$k$-generalized Markov equation}, or  abbreviated as the \emph{$k$-GM equation} and we denote it by GME$(k)$. Recall that $n$ is a \emph{$k$-generalized Markov number} (or  abbreviated as the \emph{$k$-GM number}) if $n$ appears in some positive integer solutions to $\mathrm{GME}(k)$. A triple $(a,b,c)\in \mathbb Z_{\geq 1}^3$ is called a \emph{$k$-generalized Markov triple} (or  abbreviated as the \emph{$k$-GM triple}) if $(a,b,c)$ is a positive integer solution to $\mathrm{GME}(k)$. 

There is an algorithm that enumerates all $k$-GM triples.
We give a tree $\mathbb T^{k}$ with triples of positive integers as its vertices in the following steps.
\begin{itemize}\setlength{\leftskip}{-15pt}
    \item [(1)] The root vertex is $(1,1,1)$,
    \item[(2)] the triple $(1,1,1)$ has three children, $(k+2,1,1),(1,k+2,1),(1,1,k+2),$ and
    \item [(3)] the generation rule below $(k+2,1,1),(1,k+2,1),(1,1,k+2)$ is as follows: 
\begin{itemize}\setlength{\leftskip}{-15pt}
        \item [(i)] if $a$ is the maximal number in $(a,b,c)$, then $(a,b,c)$ has two children 
        \[\left(a,\dfrac{a^2+kac+c^2}{b},c\right)\text{ and } \left(a,b,\dfrac{a^2+kab+b^2}{c}\right),\] 
        \item [(ii)] if $b$ is the maximal number in $(a,b,c)$, then $(a,b,c)$ has two children
        \[\left(\dfrac{b^2+kbc+c^2}{a},b,c\right)\text{ and } \left(a,b,\dfrac{a^2+kab+b^2}{c}\right),\] 
        \item [(iii)] if $c$ is the maximal number in $(a,b,c)$, then $(a,b,c)$ has two children 
        \[\left(\dfrac{b^2+kbc+c^2}{a},b,c\right)\text{ and } \left(a,\dfrac{a^2+kac+c^2}{b},c\right).\] 
\end{itemize}
\end{itemize}
We remark that when $(a,b,c)$ is a $k$-GM triple, $\dfrac{b^2+kbc+c^2}{a}$, $\dfrac{a^2+kac+c^2}{b}$, $\dfrac{a^2+kbc+b^2}{c}$ are also integers, because 
\begin{align*}
    \dfrac{b^2+kbc+c^2}{a}=(3+3k)bc-a-kc-kb,\\
    \dfrac{a^2+kac+c^2}{b}=(3+3k)ac-b-kc-ka,\\
    \dfrac{a^2+kab+b^2}{c}=(3+3k)ab-c-kb-ka.
\end{align*}
\begin{example}
When $k=1$, $\mathbb T^{k}$ is the following.
 \begin{align*}\label{tree}
\begin{xy}(-5,0)*+{(1,1,1)}="0",(20,15)*+{(3,1,1)}="1",(20,0)*+{(1,3,1)}="1'",(20,-15)*+{(1,1,3)}="1''",(45,25)*+{(3,13,1)}="20",(45,15)*+{(3,1,13)}="21",(45,5)*+{(13,3,1)}="22",(45,-5)*+{(1,3,13)}="23",(45,-15)*+{(13,1,3)}="24",(45,-25)*+{(1,13,3)}="25",(80,33)*+{(61,13,1)\cdots}="40",(80,27)*+{(3,13,217)\cdots}="41", (80,21)*+{(61,1,13)\cdots}="42", (80,15)*+{(3,217,13)\cdots}="43", (80,9)*+{(13,61,1)\cdots}="44", (80,3)*+{(13,3,217)\cdots}="45", (80,-3)*+{(217,3,13)\cdots}="46", (80,-9)*+{(1,61,13)\cdots}="47", (80,-15)*+{(13,217,3)\cdots}="48", (80,-21)*+{(13,1,61)\cdots}="49", (80,-27)*+{(217,13,3)\cdots}="410", (80,-33)*+{(1,13,61)\cdots}="411", \ar@{-}"0";"1"\ar@{-}"0";"1'"\ar@{-}"0";"1''"\ar@{-}"1";"20"\ar@{-}"1";"21"\ar@{-}"1'";"22"\ar@{-}"1'";"23"\ar@{-}"1''";"24"\ar@{-}"1''";"25"\ar@{-}"20";"40"\ar@{-}"20";"41"\ar@{-}"21";"42"\ar@{-}"21";"43"\ar@{-}"22";"44"\ar@{-}"22";"45"\ar@{-}"23";"46"\ar@{-}"23";"47"\ar@{-}"24";"48"\ar@{-}"24";"49"\ar@{-}"25";"410"\ar@{-}"25";"411"
\end{xy}
\end{align*}   
\end{example}

\begin{theorem}[\cite{gyomatsu}*{Theorem 1}]\label{Diophantinetheorem}
Every $k$-GM triple appears exactly once in $\mathbb T^{k}$.
\end{theorem}
The operation $(a,b,c)\mapsto \left(\dfrac{b^2+kbc+c^2}{a},b,c\right)$ is called the \emph{first Vieta jumping}, $(a,b,c)\mapsto\left(a,\dfrac{a^2+kac+c^2}{b},c\right)$ the \emph{second Vieta jumping}, and $(a,b,c)\mapsto \left(a,b,\dfrac{a^2+kab+b^2}{c}\right)$ the \emph{third Vieta jumping}. 

The following are important properties of the $k$-GM triple.

\begin{proposition}[\cite{gyomatsu}*{Lemma 4}]\label{singular-solution}
If a triple $(a,b,c)$ is a $k$-GM triple with $a=b$, $b=c$ or $c=a$, then $(a,b,c)$ is any one of $(1,1,1),(k+2,1,1),(1,k+2,1)$, or $(1,1,k+2)$.   
\end{proposition}

\begin{proposition}\label{magnitude}
For any $k$-GM triple $(a,b,c)$, if $a>b\geq c$, then we have \vspace{3mm}
\begin{itemize}\setlength{\leftskip}{-15pt}
\item[(1)] $\dfrac{a^2+kac+c^2}{b}>a(>c)$\vspace{3mm},
\item[(2)] $\dfrac{a^2+kab+b^2}{c}>a(>b)$\vspace{3mm},
\item[(3)] $b\geq \dfrac{b^2+kbc+c^2}{a}$.
\end{itemize}
\end{proposition}
\begin{proof}
The case $a>b>c$ follows from \cite{gyomatsu}*{Proposition 5}. When $a>b=c$, we can check the statement directly because of Proposition \ref{singular-solution}.     
\end{proof}

\begin{proposition}[\cite{gyomatsu}*{Corollary 8}]\label{relatively-prime}
For any $k$-GM triple $(a,b,c)$, all pairs in $a,b,c$ are relatively prime.
\end{proposition}
In this paper, we introduce another equation:
\[
    x^2+y^2+z^2+(2k+k^2)(x+y+z)+2k^3+3k^2=xyz.
\]
It is called the \emph{$k$-generalized second Markov equation} and we denote it by GSME$(k)$.

By a straightforward calculation, we have the following proposition:
\begin{proposition}[\cite{gyo-maru}*{Proposition 2.4}]\label{pr:rationalsolution}
A triple $(a,b,c)$ is one of the rational solutions to $\mathrm{GME}(k)$ if and only if the triple
\[((3+3k)a-k,(3+3k)b-k,(3+3k)c-k)\] is one of the rational solutions to $\mathrm{GSME}(k)$.     
\end{proposition}
By Proposition \ref{pr:rationalsolution}, if $(a,b,c)$ is a $k$-GM triple, then $((3+3k)a-k,(3+3k)b-k,(3+3k)c-k)$ is a positive integer solution to GSME$(k)$. In the case $k=0$, the converse holds (cf. \cite{aig}*{Proposition 2.2}), but in general, this does not hold. 

\begin{example}
We set $k=4$. Then $(9,9,22)$ is a positive integer solution to $\mathrm{GSME}(4)$, but the corresponding solution $\left(\dfrac{13}{15},\dfrac{13}{15},\dfrac{26}{15}\right)$ to $\mathrm{GME}(4)$ is not a $4$-GM triple.
\end{example}

We can get the Vieta jumping of $\mathrm{GSME}(k)$ from that of  $\mathrm{GME}(k)$.

\begin{proposition}[\cite{gyo-maru}*{Proposition 2.6}]\label{pr:vietajumpGSME}
Let $(a,b,c)$ be an integer solution to $\mathrm{GSME}(k)$. Then
\[(bc-a-k^2-2k,b,c),(a, ac-b-k^2-2k,c), (a,b,ab-c-k^2-2k)\]
are also integer solutions to $\mathrm{GSME}(k)$.
\end{proposition}
We call the operation $(a,b,c)\mapsto (bc-a-k^2-2k,b,c)$ (resp. $(a, ac-b-k^2-2k,c)$,$(a,b,ab-c-k^2-2k)$) the \emph {first} (resp. \emph {second, third}) Vieta jumping.

Even if $(a,b,c)$ is a positive integer solution to $\mathrm{GSME}(k)$, $\left(\dfrac{a+k}{3+3k},\dfrac{b+k}{3+3k},\dfrac{c+k}{3+3k}\right)$ is not a $k$-GM triple in general.
It is said that a positive integer solution $(a,b,c)$ to $\mathrm{GSME}(k)$ is an \emph{induced (positive) solution (from $\mathrm{GME}(k)$)} if $\left(\dfrac{a+k}{3+3k},\dfrac{b+k}{3+3k},\dfrac{c+k}{3+3k}\right)$ is a $k$-GM triple.
\begin{proposition}[\cite{gyo-maru}*{Proposition 2.7}]
Let $(a,b,c)$ be an induced solution to $\mathrm{GSME}(k)$. Then
\[(bc-a-k^2-2k,b,c),(a, ac-b-k^2-2k,c), (a,b,ab-c-k^2-2k)\]
are also induced solutions to $\mathrm{GSME}(k)$.    
\end{proposition}

We denote by $\tilde{\mathbb T}^{k}$ the tree obtained from ${\mathbb T}^{k}$ by replacing $(a,b,c)$ with $((3+3k)a-k,(3+3k)b-k,(3+3k)c-k)$. 

\begin{corollary}
Every induced solution to $\mathrm{GSME}(k)$ appears exactly once in $\tilde{\mathbb T}^{k}$.
\end{corollary}

\begin{example}
When $k=1$, $\tilde{\mathbb T}^{k}$ is the following.
 \begin{align*}\label{tree}
\begin{xy}(-10,0)*+{(5,5,5)}="0",(20,15)*+{(17,5,5)}="1",(20,0)*+{(5,17,5)}="1'",(20,-15)*+{(5,5,17)}="1''",(50,25)*+{(17,77,5)}="20",(50,15)*+{(17,5,77)}="21",(50,5)*+{(77,17,5)}="22",(50,-5)*+{(5,17,77)}="23",(50,-15)*+{(77,5,17)}="24",(50,-25)*+{(5,77,17)}="25",(90,33)*+{(365,77,5)\cdots}="40",(90,27)*+{(17,77,1301)\cdots}="41", (90,21)*+{(365,5,77)\cdots}="42", (90,15)*+{(17,1301,77)\cdots}="43", (90,9)*+{(77,365,5)\cdots}="44", (90,3)*+{(77,17,1301)\cdots}="45", (90,-3)*+{(1301,17,77)\cdots}="46", (90,-9)*+{(5,365,77)\cdots}="47", (90,-15)*+{(77,1301,17)\cdots}="48", (90,-21)*+{(77,5,365)\cdots}="49", (90,-27)*+{(1301,77,17)\cdots}="410", (90,-33)*+{(5,77,365)\cdots}="411", \ar@{-}"0";"1"\ar@{-}"0";"1'"\ar@{-}"0";"1''"\ar@{-}"1";"20"\ar@{-}"1";"21"\ar@{-}"1'";"22"\ar@{-}"1'";"23"\ar@{-}"1''";"24"\ar@{-}"1''";"25"\ar@{-}"20";"40"\ar@{-}"20";"41"\ar@{-}"21";"42"\ar@{-}"21";"43"\ar@{-}"22";"44"\ar@{-}"22";"45"\ar@{-}"23";"46"\ar@{-}"23";"47"\ar@{-}"24";"48"\ar@{-}"24";"49"\ar@{-}"25";"410"\ar@{-}"25";"411"
\end{xy}
\end{align*}   
\end{example}    

\section{Generalized Markov tree and inverse generalized Markov tree}
In the previous section, we gave a tree consisting of all $k$-GM triples. In this section, we divide these $k$-GM triples into two trees, and we see the relation between them. 

First, we consider the following binary tree $\mathrm{M}\mathbb T(k)$:
\begin{itemize}\setlength{\leftskip}{-15pt}
\item [(1)] the root vertex is $(1,k+2,1)$,
\item [(2)] for a vertex $(a,b,c)$, there are the following two children of it:
\[\begin{xy}(0,0)*+{(a,b,c)}="1",(30,-15)*+{\left(b,\dfrac{b^2+kbc+c^2}{a},c\right).}="2",(-30,-15)*+{\left(a,\dfrac{a^2+kab+b^2}{c},b\right)}="3", \ar@{-}"1";"2"\ar@{-}"1";"3"
\end{xy}\]
\end{itemize}
 It is called the \emph{$k$-generalized Markov tree}, or abbreviated as the \emph{$k$-GM tree}. We see that for any $(a,b,c) \in \mathrm{M}\mathbb T(k)$, $b$ is the unique maximal number in $a,b,c$ according to Proposition \ref{magnitude} (note that this tree is different from $\mathbb T^{k}$).

 \begin{example}
 When $k=1$, $\mathrm{M}\mathbb T(k)$ is the following.
 \begin{align*}
\begin{xy}(0,0)*+{(1,3,1)}="1",(25,16)*+{(3,13,1)}="2",(25,-16)*+{(1,13,3)}="3", (50,24)*+{(13,61,1)}="4",(50,8)*+{(3,217,13)}="5",(50,-8)*+{(13,217,3)}="6",(50,-24)*+{(1,61,13)}="7", (85,28)*+{(61,291,1)\cdots}="8",(85,20)*+{(13,4683,61)\cdots}="9",(85,12)*+{(217,16693,13)\cdots}="10",(85,4)*+{(3,3673,217)\cdots}="11",(85,-4)*+{(217,3673,3)\cdots}="12",(85,-12)*+{(13,16693,217)\cdots}="13",(85,-20)*+{(61,4683,13)\cdots}="14",(85,-28)*+{(1,291,61)\cdots}="15",\ar@{-}"1";"2"\ar@{-}"1";"3"\ar@{-}"2";"4"\ar@{-}"2";"5"\ar@{-}"3";"6"\ar@{-}"3";"7"\ar@{-}"4";"8"\ar@{-}"4";"9"\ar@{-}"5";"10"\ar@{-}"5";"11"\ar@{-}"6";"12"\ar@{-}"6";"13"\ar@{-}"7";"14"\ar@{-}"7";"15"
\end{xy}
\end{align*}
%When $k=0$, see the first tree in Example \ref{ex:intro-markov}. 
 \end{example}
The following proposition follows from Theorem \ref{Diophantinetheorem}:
\begin{proposition}\label{prop:all-markov}
The following statements hold:
\begin{itemize}\setlength{\leftskip}{-15pt}
    \item[(1)] each vertex $(a,b,c)$ in $\mathrm{M}\mathbb T(k)$ is a $k$-GM triple with $b> \max\{a,c\}$, \item[(2)] every $k$-GM triple $(a,b,c)$ with $b> \max\{a,c\}$ appears exactly once in $\mathrm{M}\mathbb T(k)$.
\end{itemize}
\end{proposition}

Next, we define another tree whose vertices are $k$-GM triples.
We consider the following binary tree $\mathrm{M}\mathbb T^\dag(k)$:
\begin{itemize}\setlength{\leftskip}{-15pt}
\item [(1)] the root vertex is $(1,1,1)$,
\item [(2)] for a vertex $(a,b,c)$, there are the following two children of it:
\[\begin{xy}(0,0)*+{(a,b,c)}="1",(-30,-15)*+{\left(a,c,\dfrac{a^2+kac+c^2}{b}\right)}="2",(30,-15)*+{\left(\dfrac{a^2+kac+c^2}{b},a,c\right).}="3", \ar@{-}"1";"2"\ar@{-}"1";"3"
\end{xy}\]
\end{itemize}
It is called the \emph{inverse $k$-generalized Markov tree}, or abbreviated as the \emph{inverse $k$-GM tree}.  The operation taking left (resp. right) child in the inverse $k$-GM tree is the inverse of the operation taking the left (resp. right) child in the $k$-GM tree. 

\begin{example}
 When $k=1$, $\mathrm{M}\mathbb T^\dag(k)$ is the following.
 \begin{align*}
\begin{xy}(0,0)*+{(1,1,1)}="1",(25,16)*+{(3,1,1)}="2",(25,-16)*+{(1,1,3)}="3", (50,24)*+{(13,3,1)}="4",(50,8)*+{(3,1,13)}="5",(50,-8)*+{(13,1,3)}="6",(50,-24)*+{(1,3,13)}="7", (85,28)*+{(61,13,1)\cdots}="8",(85,20)*+{(13,1,61)\cdots}="9",(85,12)*+{(217,3,13)\cdots}="10",(85,4)*+{(3,13,217)\cdots}="11",(85,-4)*+{(217,13,3)\cdots}="12",(85,-12)*+{(13,3,217)\cdots}="13",(85,-20)*+{(61,1,13)\cdots}="14",(85,-28)*+{(1,13,61)\cdots}="15",\ar@{-}"1";"2"\ar@{-}"1";"3"\ar@{-}"2";"4"\ar@{-}"2";"5"\ar@{-}"3";"6"\ar@{-}"3";"7"\ar@{-}"4";"8"\ar@{-}"4";"9"\ar@{-}"5";"10"\ar@{-}"5";"11"\ar@{-}"6";"12"\ar@{-}"6";"13"\ar@{-}"7";"14"\ar@{-}"7";"15"
\end{xy}
\end{align*}
%When $k=0$, see the second tree in Example \ref{ex:intro-markov}. 
 \end{example}

Before describing the relation between $\mathrm{M}\mathbb T(k)$ and $\mathrm{M}\mathbb T^\dag(k)$, we will introduce the \emph{canonical graph isomorphism} between two trees.

\begin{definition}
Let $\mathbb T$ and $\mathbb T'$ be full planar binary trees. If a graph isomorphism $f\colon \mathbb T \to \mathbb T'$ preserves the left child and the right child, then $f$ is called the \emph{canonical graph isomorphism}.
\end{definition}
\begin{proposition}\label{pr:rho-mor}
The correspondence $\mu\colon(a,b,c)\mapsto\left(a,\dfrac{a^2+kac+c^2}{b},c\right)$ induces the canonical graph isomorphism from $\mathrm{M}\mathbb T(k)$ to $\mathrm{M}\mathbb T^\dag(k)$.
\end{proposition}

\begin{proof}
We can check that the statement holds for the root vertex in $\mathrm{M}\mathbb T(k)$ directly. We assume that the statement holds for $(a,b,c) \in \mathrm{M}\mathbb T(k)$. We denote by $\sigma_L$ (resp. $\sigma_R$)  the operation taking the left (resp. right) child in $\mathrm{M}\mathbb T(k)$, and $\sigma^{\dag}_L$ (resp. $\sigma^{\dag}_R$)  the operation taking the left (resp. right) child in $\mathrm{M}\mathbb T^{\dag}(k)$. It suffices to show $\mu\circ\sigma_L(a,b,c)=\sigma^{\dag}_L\circ\mu(a,b,c)$ and $\mu\circ \sigma_R(a,b,c)=\sigma^{\dag}_R\circ\mu(a,b,c)$. We will prove only the first statement. The left-hand side is
\[\mu\circ\sigma_L(a,b,c)=\mu\left(a,\dfrac{a^2+kab+b^2}{c},b\right)=(a,c,b),\]
and the right-hand side is
\[\sigma_L^{\dag}\circ\mu(a,b,c)=\sigma_L^{\dag}\left(a,\dfrac{a^2+kac+c^2}{b},c\right)=(a,c,b),\] as desired.
\end{proof}

\begin{remark}
Since the correspondence $\mu$ is an involution, we can also regard $\mu$ as the map  from $\mathrm{M}\mathbb T^\dag(k)$ to $\mathrm{M}\mathbb T(k)$.
\end{remark}

In parallel with the $k$-GM tree, we have the following proposition:

\begin{proposition}
The following statements hold:
\begin{itemize}\setlength{\leftskip}{-15pt}
    \item[(1)] Each vertex $(a,b,c)$ in $\mathrm{M}\mathbb T^\dag(k)$ is a $k$-GM triple with $b\leq \max\{a,c\}$.
    \item[(2)] Every $k$-GM triple $(a,b,c)$ with $b\leq \max\{a,c\}$ appears exactly once in $\mathrm{M}\mathbb T^\dag(k)$. 
\end{itemize}
\end{proposition}

\begin{proof}
First, we will prove (1). We assume that $(a,b,c) \in \mathrm{M}\mathbb T^\dag(k)$ satisfies $b> \max\{a,c\}$. By Propositions \ref{magnitude} (3) and \ref{pr:rho-mor}, $(a,b',c):=\mu(a,b,c)$ satisfies $b'\leq a$ or $b'\leq c$ and it is a vertex in $\mathrm{M}\mathbb T(k)$. It is in contradiction to Proposition \ref{prop:all-markov} (1). Second, we will prove (2). If $a> \max\{b,c\}$ or $c>\max\{a,b\}$, then $(a,b',c):=\mu(a,b,c)$ satisfies $b'> \max\{a,c\}$ and it is a vertex in $\mathrm{M}\mathbb T(k)$ by Propositions \ref{magnitude} and \ref{pr:rho-mor}. By Proposition \ref{prop:all-markov}, $(a,b',c)$ appears exactly once in $\mathrm{M}\mathbb T(k)$. Therefore, since $(a,b,c)=\mu(a,b',c)$, we have the conclusion.
\end{proof}

\section{Generalized Cohn tree and inverse generalized Cohn tree}
\subsection{$k$-generalized Cohn tree}
We will recall the $k$-generalized Cohn matrix and $k$-generalized Cohn triple according to \cite{gyo-maru}, and give some properties of them.
\begin{definition}
For $k\in \mathbb {Z}_{\geq 0}$, we define a \emph{$k$-generalized Cohn matrix} $P=\begin{bmatrix}p_{11}&p_{12}\\p_{21}&p_{22}\end{bmatrix} \in SL(2,\mathbb Z)$ as a matrix satisfying the following conditions: 
\begin{itemize}\setlength{\leftskip}{-15pt}
    \item [(1)] $p_{12}$ is a $k$-GM number, and
    \item [(2)] $\mathrm{tr}(P)=(3k+3)p_{12}-k$.
\end{itemize}
\end{definition}

\begin{definition}\label{def:gen-Cohn-triple}
For $k\in \mathbb {Z}_{\geq 0}$, we define a \emph{$k$-generalized Cohn triple $(P,Q,R)$} as a triple satisfying the following conditions:
\begin{itemize}\setlength{\leftskip}{-15pt}
    \item [(1)] $P,Q,R$ are $k$-generalized Cohn matrices, 
    \item[(2)] $Q=PR-S$, where $S=\begin{bmatrix}
        k&0\\3k^2+3k &k
    \end{bmatrix}$, and
    \item[(3)] $(p_{12},q_{12},r_{12})$ is a $k$-GM triple, where $p_{12},q_{12},r_{12}$ are the $(1,2)$-entries of $P,Q,R$, respectively.
\end{itemize}
The triple $(P,Q,R)$ is said to be \emph{associated with} $(p_{12},q_{12},r_{12})$.
\end{definition}
Note that this definition coincides with the definition of the $k$-generalized Cohn triple in Section 1.
In this paper, we abbreviate the $k$-generalized Cohn matrix as the \emph{$k$-GC matrix} and the $k$-generalized triple as the \emph{$k$-GC triple}.
By the definition of the $k$-GC matrix, (3) in Definition \ref{def:gen-Cohn-triple} can be replaced by the following condition:
\begin{itemize}\setlength{\leftskip}{-15pt}
    \item [(3')] $(\mathrm{tr}(P),\mathrm{tr}(Q),\mathrm{tr}(R))$ is an induced solution to $\mathrm{GSME}(k)$. 
\end{itemize}

The definition of $k$-GC matrix does not refer to the existence of the triple satisfying their conditions. In \cite{gyo-maru}, this existence was proved by constructing the \emph{$k$-generalized Cohn tree}. Here, we recall the argument of it. First, we prove the case $(a,b,c)=(1,1,1)$. For any $\ell\in \mathbb Z$, we set
\begin{align*}
    P_{1;\ell}&=\begin{bmatrix}
        \ell&1\\-\ell^2+2k\ell+3\ell-1&-\ell+2k+3
    \end{bmatrix}\\
    Q_{1;\ell}&=\begin{bmatrix}
        k+\ell+1&1\\k^2-\ell^2+3k+\ell+1&k-\ell+2
    \end{bmatrix}\\
    R_{1;\ell}&=\begin{bmatrix}
        2k+\ell+2&1\\-\ell^2-2k\ell+2k-\ell+1&-\ell+1
    \end{bmatrix}.
\end{align*}
\begin{proposition}[\cite{gyo-maru}*{Proposition 3.4}]\label{cohn-mat-with-111}
 The triple $(P_{1;\ell},Q_{1;\ell},R_{1;\ell})$ is a $k$-GC triple. Conversely, for a $k$-GC triple $(P,Q,R)$ satisfying $(p_{12},q_{12},r_{12})=(1,1,1)$, there exists $\ell\in \mathbb Z$ such that $(P,Q,R)=(P_{1;\ell},Q_{1;\ell},R_{1;\ell})$.
\end{proposition}
Now, we consider a binary tree, the \emph{$k$-generalized Cohn tree} $\mathrm{GC}\mathbb T(k,\ell)$ for $\ell\in \mathbb Z$:
\begin{itemize}\setlength{\leftskip}{-15pt}
\item [(1)] the root vertex is $(P_{\ell},Q_{\ell},R_{\ell}):=(P_{1;\ell},P_{1;\ell}Q_{1;\ell}-S,Q_{1;\ell})$, that is,
\begin{align*}
    P_{\ell}&=\begin{bmatrix}
        \ell&1\\-\ell^2+2k\ell+3\ell-1&-\ell+2k+3
    \end{bmatrix},\\
     Q_{\ell}&=\begin{bmatrix}
        k^2 + k\ell + 2k + 2\ell + 1 & k + 2\\
        2k^3 + k^2\ell - k\ell^2 + 6k^2 + 4k\ell - 2\ell^2 + 7k + 4\ell + 2&  2k^2 - k\ell + 6k - 2\ell + 5
    \end{bmatrix},\\
    R_{\ell}&=\begin{bmatrix}
        k+\ell+1&1\\k^2-\ell^2+3k+\ell+1&k-\ell+2
    \end{bmatrix},
\end{align*}
\item[(2)]for a vertex $(P,Q,R)$, there are the following two children of it:
\[\begin{xy}(0,0)*+{(P,Q,R)}="1",(30,-15)*+{(Q,QR-S,R).}="2",(-30,-15)*+{(P,PQ-S,Q)}="3", \ar@{-}"1";"2"\ar@{-}"1";"3"
\end{xy}\]
\end{itemize}
In this paper, we abbreviate this tree as the \emph{$k$-GC tree}.
\begin{example}\label{ex:cohn-tree}
 When $k=1$ and $\ell=-1$, $\mathrm{GC}\mathbb T(k,\ell)$ is the following.
\relsize{-2}
\begin{align*}
\begin{xy}(-20,0)*+{\left(\begin{bmatrix}-1&1\\-7&6\end{bmatrix},\begin{bmatrix}1&3\\5&16\end{bmatrix},\begin{bmatrix}1&1\\3&4\end{bmatrix}\right)}="1",(30,-16)*+{\left(\begin{bmatrix}-1&1\\-7&6\end{bmatrix},\begin{bmatrix}3&13\\17&74\end{bmatrix},\begin{bmatrix}1&3\\5&16\end{bmatrix}\right)}="2",(30,16)*+{\left(\begin{bmatrix}1&3\\5&16\end{bmatrix},\begin{bmatrix}9&13\\47&68\end{bmatrix},\begin{bmatrix}1&1\\3&4\end{bmatrix}\right)}="3", 
(90,-24)*+{\left(\begin{bmatrix}-1&1\\-7&6\end{bmatrix},\begin{bmatrix}13&61\\75&352\end{bmatrix},\begin{bmatrix}3&13\\17&74\end{bmatrix}\right)\cdots}="4",(90,-8)*+{\left(\begin{bmatrix}3&13\\17&74\end{bmatrix},\begin{bmatrix}67&217\\381&1234\end{bmatrix},\begin{bmatrix}1&3\\5&16\end{bmatrix}\right)\cdots}="5",(90,8)*+{\left(\begin{bmatrix}1&3\\5&16\end{bmatrix},\begin{bmatrix}149&217\\791&1152\end{bmatrix},\begin{bmatrix}9&13\\47&68\end{bmatrix}\right)\cdots}="6",(90,24)*+{\left(\begin{bmatrix}9&13\\47&68\end{bmatrix},\begin{bmatrix}47&61\\245&318\end{bmatrix},\begin{bmatrix}1&1\\3&4\end{bmatrix}\right)\cdots}="7",\ar@{-}"1";"2"\ar@{-}"1";"3"\ar@{-}"2";"4"\ar@{-}"2";"5"\ar@{-}"3";"6"\ar@{-}"3";"7"
\end{xy}
\end{align*}
\relsize{+2}
%When $k=0$ and $\ell=0$, see the first tree in Example \ref{ex:intro-cohn}. 
\end{example}
\begin{theorem}[\cite{gyo-maru}*{Theorem 3.5}]\label{thm:cohn-markov}
If $(P,Q,R)$ is a $k$-GC triple associated with $(a,b,c)$, then $(P,PQ-S,Q)$ (resp. $(Q,QR-S,R)$) is a $k$-GC triple associated with $\left(a,c',b\right)$ (resp. $\left(b,a',c\right)$), where $c'=\dfrac{a^2+kab+b^2}{c}$ and $a'=\dfrac{b^2+kbc+c^2}{a}$.    
\end{theorem}

Theorem \ref{thm:cohn-markov} is a theorem about moving to the lower vertex on $\mathrm{GC}\mathbb T(k,\ell)$, while the following lemma is a statement about moving to the upper vertex on $\mathrm{GC}\mathbb T(k,\ell)$.

\begin{lemma}[\cite{gyo-maru}*{Lemma 3.11}]\label{Cohn-Markov2}
If $(P,Q,R)$ is a $k$-GC triple associated with $(a,b,c)$, then $(P,R,P^{-1}(R+S))$ (resp. $((P+S)R^{-1},P,R)$) is a $k$-GC triple associated with $\left(a,c,b'\right)$ (resp. $\left(b',a,c\right)$), where $b'=\dfrac{a^2+kac+c^2}{b}$.  
\end{lemma}

By using Theorem \ref{thm:cohn-markov}, for any $k$-GM triple $(a,b,c)$, we can prove the existence of a $k$-GC triple associated with $(a,b,c)$:

\begin{theorem}[\cite{gyo-maru}*{Corollary 3.15}]\label{cor:CT-MT}
We fix $\ell\in \mathbb Z$. The correspondence from $(P,Q,R)$ in $\mathrm{GC}\TT(k,\ell)$ to $(p_{12},q_{12},r_{12})$ induces the canonical graph isomorphism between $\mathrm{GC}\TT(k,\ell)$ and $\mathrm{M}\TT(k)$. In particular, for any $k$-GM triple with $b>\max\{a,c\}$, there is a $k$-GC matrix associated with $(a,b,c)$.
\end{theorem}

Moreover, by using Lemma \ref{Cohn-Markov2}, we have a stronger result.

\begin{proposition}[\cite{gyo-maru}*{Proposition 3.17}]\label{pr:all-cohn-triple}
Let $(P,Q,R)$ be a $k$-GC triple associated with $(a,b,c)$. We assume that $b> \max\{a,c\}.$ Then, there exist a unique $\ell\in \mathbb Z$ and a unique vertex $v$ in $\mathrm{GC}\mathbb T(k,\ell)$ such that $v=(P,Q,R)$. 
\end{proposition}

When we fix $\ell\in \mathbb Z$, we have the following property for $\mathrm{GC}\mathbb T(k, \ell)$.

\begin{theorem}[\cite{gyo-maru}*{Corollary 3.16}]\label{thm:Cohn-distinct}
We fix $k\in \mathbb Z_{\geq0}$ and $\ell\in \mathbb Z$. The second entries of $k$-GC triples in $\mathrm{GC}\mathbb T(k, \ell)$ are distinct.   
\end{theorem}

Finally, we give a relation between $\mathrm{GC}\mathbb T(k, \ell)$ and $\mathrm{GC}\mathbb T(k, \ell')$. It is an analogue of \cite{aig}*{Proposition 4.15}.

\begin{proposition}\label{prop:cohnmatrix-ltol'}
Let $L=\begin{bmatrix}
    1&0\\\ell'-\ell&1
\end{bmatrix}$. The right conjugation of $P$ by $L$ induces the canonical graph isomorphism from $\mathrm{GC}\mathbb T(k, \ell)$ to $\mathrm{GC}\mathbb T(k, \ell')$.   
\end{proposition}

\begin{proof}
 We can check $P_{1;\ell'}=L^{-1}P_{1;\ell}L$, $Q_{1;\ell'}=L^{-1}Q_{1;\ell}L$, $R_{1;\ell'}=L^{-1}R_{1;\ell}L$ directly. For general cases, the statement follows inductively from 
 \[L^{-1}(PQ)L=(L^{-1}PL)(L^{-1}QL),\quad L^{-1}(QR)L=(L^{-1}QL)(L^{-1}RL),\quad L^{-1}SL=S.\]
\end{proof}

\begin{remark}\label{rmk:Markov-monodromy-structure-cohn}
Let $A$ be the set of $k$-GC triples.  We set $\tau_1,\tau_2\colon A\to A$ by 
\[\tau_1(P,Q,R)=(P,R,P^{-1}(R+S)),\quad \tau_2(P,Q,R)=(Q,QR-S,R).\] 
Then, it can be seen that in $\mathrm{GC}\mathbb T(k,\ell)$ the operation giving the left child of $(P,Q,R)$ is given by $\tau_1^{-1}$ and the operation giving the right child by $\tau_2$. Moreover, we have the braid relation $\tau_1\tau_2\tau_1=\tau_2\tau_1\tau_2$.
\end{remark}

\subsection{Inverse $k$-generalized Cohn tree}

In the discussion in the previous subsection, we assume $b> \max\{a,c\}$, but we can also construct the $k$-GC triple with $b\leq \max\{a,c\}$. Indeed, according to Lemma \ref{Cohn-Markov2}, for $(P,Q,R)$ associated with $(a,b,c)$ where $b> a> c$, $(P,R,P^{-1}(R+S))$ is a $k$-GC triple associated with $(a',b',c')$ where $a'=a$ and $b'=c$, in particular, $a'>b'$. In the same way, for $(P,Q,R)$ associated with $(a,b,c)$ where $b> c> a$, $((P+S)R^{-1},P,R)$ is a $k$-GC triple associated with $(a',b',c')$ where $c'=c$ and $b'=a$, in particular, $c'>b'$. 

In this subsection, we deal with $k$-GC triple with $b\leq \max\{a,c\}$. We can discuss this case in parallel with $b> \max\{a,c\}$ case. We consider the following binary tree $\mathrm{GC}\mathbb T^{\dag}(k,\ell)$:
\begin{itemize}\setlength{\leftskip}{-15pt}
\item [(1)] the root vertex is $(P_{1;\ell},Q_{1;\ell},R_{1;\ell})$,

\item[(2)]for a vertex $(P,Q,R)$, there are the following two children of it:
\[\begin{xy}(0,0)*+{(P,Q,R)}="1",(-30,-15)*+{(P,R,P^{-1}(R+S))}="2",(30,-15)*+{((P+S)R^{-1},P,R).}="3", \ar_{\tau_1}@{-}"1";"2"\ar^{\tau_2^{-1}}@{-}"1";"3"
\end{xy}\]
\end{itemize}

We call $\mathrm{GC}\mathbb T^{\dag}(k,\ell)$ the \emph{inverse $k$-generalized Cohn tree}. In this paper, we abbreviate this tree as the \emph{inverse $k$-GC tree}. The operation taking left (resp. right) child in the inverse $k$-GC tree is the inverse of the operation taking the left (resp. right) child in the $k$-GC tree. 

\begin{example}\label{ex:inverse-cohn-tree}
 When $k=1$ and $\ell=-1$, $\mathrm{GC}\mathbb T^\dag(k,\ell)$ is the following.
\relsize{-2}
\begin{align*}
\begin{xy}(-20,0)*+{\left(\begin{bmatrix}-1&1\\-7&6\end{bmatrix},\begin{bmatrix}1&1\\3&4\end{bmatrix},\begin{bmatrix}3&1\\5&2\end{bmatrix}\right)}="1",(30,-16)*+{\left(\begin{bmatrix}-1&1\\-7&6\end{bmatrix},\begin{bmatrix}3&1\\5&2\end{bmatrix},\begin{bmatrix}13&3\\17&4\end{bmatrix}\right)}="2",(30,16)*+{\left(\begin{bmatrix}-5&3\\-37&22\end{bmatrix},\begin{bmatrix}-1&1\\-7&6\end{bmatrix},\begin{bmatrix}3&1\\5&2\end{bmatrix}\right)}="3",(90,-24)*+{\left(\begin{bmatrix}-1&1\\-7&6\end{bmatrix},\begin{bmatrix}13&3\\17&4\end{bmatrix},\begin{bmatrix}61&13\\75&16\end{bmatrix}\right)\cdots}="4",(90,-8)*+{\left(\begin{bmatrix}
 -17&13\\-123&94   
\end{bmatrix},\begin{bmatrix}-1&1\\-7&6\end{bmatrix},\begin{bmatrix}13&3\\17&4\end{bmatrix}\right)\cdots}="5",(90,8)*+
{\left(\begin{bmatrix}-5&3\\-37&22\end{bmatrix},\begin{bmatrix}3&1\\5&2\end{bmatrix},\begin{bmatrix}55&13\\93&22\end{bmatrix}\right)\cdots}="6",(90,24)*+{\left(\begin{bmatrix}-23&13\\-177&100\end{bmatrix},\begin{bmatrix}-5&3\\-37&22\end{bmatrix},\begin{bmatrix}3&1\\5&2\end{bmatrix}\right)\cdots}="7",\ar@{-}"1";"2"\ar@{-}"1";"3"\ar@{-}"2";"4"\ar@{-}"2";"5"\ar@{-}"3";"6"\ar@{-}"3";"7"
\end{xy}
\end{align*}
\relsize{+2}
%When $k=0$ and $\ell=0$, see the second tree in Example \ref{ex:intro-markov}. 
\end{example}

By exchanging the role of Theorem \ref{thm:cohn-markov} and Lemma \ref{Cohn-Markov2}, we have the following corollaries:

\begin{corollary}\label{cor:CTdag-MTdag}
We fix $\ell\in \mathbb Z$. The correspondence from $(P,Q,R)$ in $\mathrm{GC}\TT(k,\ell)$ to $(p_{12},q_{12},r_{12})$ induces the canonical graph isomorphism between $\mathrm{GC}\TT^{\dag}(k,\ell)$ and $\mathrm{M}\TT^\dag(k)$. In particular, for any $k$-GM triple $(a,b,c)$ with $b\leq \max\{a,c\}$, there is a $k$-GC triple associated with $(a,b,c)$.
\end{corollary}

\begin{corollary}\label{cor:inverse-all-cone}
Let $(P,Q,R)$ be a $k$-GC triple associated with $(a,b,c)$. We assume that $b\leq \max\{a,c\}.$ Then, there exist a unique $\ell\in \mathbb Z$ and a unique vertex $v$ in $\mathrm{GC}\mathbb T^\dag(k,\ell)$ such that $v=(P,Q,R)$.  
\end{corollary}

In parallel with Proposition \ref{prop:cohnmatrix-ltol'}, we have the following proposition. 

\begin{proposition}
Let $L=\begin{bmatrix}
    1&0\\\ell'-\ell&1
\end{bmatrix}$. The right conjugation of $P$ by $L$ induces the canonical graph isomorphism from $\mathrm{GC}\mathbb T(k, \ell)$ to $\mathrm{GC}\mathbb T^\dag(k, \ell')$.   
\end{proposition}

\begin{remark}
 The map corresponding to $\mu$ in the Section 3, i.e., the map that gives the relation between the $k$-GC tree and the inverse $k$-GC tree, is given in Theorem \ref{thm:phi-psi-1}.   
\end{remark}

\section{Markov-monodromy tree and Markov-monodromy decomposition}

In this section, we introduce the $k$-Markov-monodromy matrix and $k$-Markov-monodromy triple and describe their properties.
\subsection{Definition and existence of Markov-monodromy triple}
\begin{definition}
For $k\in \mathbb {Z}_{\geq 0}$, we define a \emph{$k$-Markov-monodromy matrix} $X=\begin{bmatrix}x_{11}&x_{12}\\x_{21}&x_{22}\end{bmatrix} \in SL(2,\mathbb Z)$ as a matrix satisfying the following conditions: 
\begin{itemize}\setlength{\leftskip}{-15pt}
    \item [(1)] $x_{12}$ is a $k$-GM number,
    \item [(2)] $\mathrm{tr}(X)=-k.$
\end{itemize}
\end{definition}
\begin{definition}\label{def:Markov-monodromy-triple}
For $k\in \mathbb {Z}_{\geq 0}$, we define a \emph{$k$-Markov-monodromy triple $(X,Y,Z)$} as a triple satisfying the following conditions:
\begin{itemize}\setlength{\leftskip}{-15pt}
    \item [(1)] $X,Y,Z$ are $k$-Markov-monodromy matrices,
    \item[(2)] $XYZ= T$, where $T=
        \begin{bmatrix}
            -1&0\\3k+3&-1
        \end{bmatrix}$ 
    \item[(3)] $(x_{12},y_{12},z_{12})$ is a $k$-GM triple, where $x_{12},y_{12},z_{12}$ are the $(1,2)$-entries of $X,Y,Z$, respectively.
\end{itemize}
The triple $(X,Y,Z)$ is said to be \emph{associated with} $(x_{12},y_{12},z_{12})$.
\end{definition}

Note that this definition coincides with the definition of the $k$-Markov-monodromy triple in Section 1. In this paper, we abbreviate the $k$-Markov-monodromy matrix as the $k$-MM matrix and the $k$-Markov-monodromy triple as the $k$-MM triple.
Each $k$-MM matrix is related to a $k$-GC matrix by the following bijective map:

\begin{proposition}\label{pr:Cohn-Markov-monodromy}
We fix $k\in \mathbb Z_{\geq 0}$. We consider the following map $\psi\colon M(2,\mathbb Z) \to M(2,\mathbb Z)$:
\[\psi\colon \begin{bmatrix}
    m_{11}&m_{12}\\m_{21}&m_{22}
\end{bmatrix}\mapsto \begin{bmatrix}
    -m_{11}+m_{12}k-k&m_{12}\\m_{21}-(k+3)m_{11}+k(2k+3)(m_{12}-1) & -m_{22}+(2k+3)m_{12}-k
\end{bmatrix}.\]    
This map induces a bijection from the set of $k$-MM matrices to the set of $k$-GC matrices. 
\end{proposition}

\begin{proof}
First, we prove $\psi\colon M(2,\mathbb Z) \to M(2,\mathbb Z)$ is a bijection. We can construct the inverse map of $\psi$ as
\[\psi^{-1}\colon \begin{bmatrix}
    m_{11}&m_{12}\\m_{21}&m_{22}
\end{bmatrix}\mapsto \begin{bmatrix}
    -m_{11}+m_{12}k-k&m_{12}\\m_{21}-(k+3)m_{11}-k^2(m_{12}-1) & -m_{22}+(2k+3)m_{12}-k
\end{bmatrix}.\]
Second, for a $k$-MM matrix $X$, we prove that $\psi(X)$ is a $k$-GC matrix. Accoding to $X\in SL(2,\ZZ)$ and $x_{11}+x_{22}=-k$, we have 
\begin{align*}
\det(\psi(X))&=x_{11}x_{22}-kx_{11}x_{12}+kx_{11}-kx_{12}x_{22}-k^2x_{12}+kx_{22}+k^2-x_{21}x_{12}\\
&=x_{11}x_{22}-x_{21}x_{12}=1.
\end{align*}
Therefore, we have $\psi(X)\in SL(2,\ZZ)$. Moreover, we have
\begin{align*}
    \mathrm{tr}(\psi(X))=-x_{11}-x_{22}+(3k+3)x_{12}-2k=(3k+3)x_{12}-k.
\end{align*}
Therefore, $\psi(X)$ is a $k$-GC matrix. Finally, for a $k$-GC matrix $P$, we prove that $\psi^{-1}(P)$ is a $k$-MM matrix. According to $P\in SL(2,\ZZ)$ and $p_{11}+p_{22}= (3k+3)p_{12}-k,$ we have 
\begin{align*}
\det(\psi^{-1}(P))=&p_{11}p_{22}-kp_{11}p_{12}+kp_{11}-kp_{12}p_{22}+k(3k+3)p_{12}^2-k^2p_{12}+kp_{22}\\
&-k(3k+3)p_{12}+k^2-p_{21}p_{12}\\
=&p_{11}p_{22}-p_{21}p_{12}=1.
\end{align*}
Therefore, we have $\psi^{-1}(P)\in SL(2,\ZZ)$. Moreover, we have
\begin{align*}
    \mathrm{tr}(\psi^{-1}(P))=-p_{11}-p_{22}+(3k+3)p_{12}-2k=-k.
\end{align*}
Therefore, $\psi^{-1}(P)$ is a $k$-MM matrix.
\end{proof}

Moreover, $\psi$ induces a more strong bijection:

\begin{proposition}\label{pr:Markov-monodromy-cohn-triple}
For a $k$-MM triple $(X,Y,Z)$, $(\psi(X),\psi(Y),\psi(Z))$ is a $k$-GC triple. Conversely, for a $k$-GC triple $(P,Q,R)$, $(\psi^{-1}(P),\psi^{-1}(Q),\psi^{-1}(R))$ is a $k$-MM triple. 
\end{proposition}

\begin{proof}
First, we prove the former statement. By Proposition \ref{pr:Cohn-Markov-monodromy}, it suffices to show that $\psi(Y)=\psi(X)\psi(Z)-S$. By assumption, we have $\psi(Y)=\psi(X^{-1}TZ^{-1})$. Therefore, it is enough to show that $\psi(X^{-1}TZ^{-1})=\psi(X)\psi(Z)-S$. By the definition of the $k$-MM matrix, we can set
\[X=\begin{bmatrix}
    x_{11}& x_{12}\\x_{21}&-x_{11}-k
\end{bmatrix},\quad Z=\begin{bmatrix}
    z_{11}& z_{12}\\z_{21}&-z_{11}-k
\end{bmatrix}.\]
Now, we have $X^{-1}TZ^{-1}=\begin{bmatrix}
    m_{11}&m_{12}\\\ast &m_{22}
\end{bmatrix}$, where
\begin{align*}
m_{11}=&(3k+3)x_{12}z_{11}+(3k^2+3k)x_{12}-x_{11}z_{11} + 3x_{12}z_{11} - x_{12}z_{21} - kx_{11} - kz_{11} - k^2,\\
m_{12}=&(3k+3)x_{12}z_{12} + x_{12}z_{11} - x_{11}z_{12} - kz_{12},\\
m_{22}=& -(3k+3)x_{11}z_{12} - x_{11}z_{11}- x_{21}z_{12}.
\end{align*}
By applying $\psi$, we have $\psi(X^{-1}TZ^{-1})=\begin{bmatrix}
    m'_{11}&m'_{12}\\\ast &m'_{22}
\end{bmatrix}$, where
\begin{align*}
m'_{11}=&(3k^2+3k)x_{12}z_{12} - (2k+3)x_{12}z_{11} - kx_{11}z_{12} - (3k^2+3)x_{12} - k^2z_{12} + x_{11}z_{11}\\& + x_{12}z_{21} + kx_{11} + kz_{11} + k^2 - k,\\
m'_{12}=&(3k+3)x_{12}z_{12} + x_{12}z_{11} - x_{11}z_{12} - kz_{12},\\
m'_{22}=& 3(2k+3)(k+1)x_{12}z_{12} + (2k+3)x_{12}z_{11} + kx_{11}z_{12} - (2k^2+3k)z_{12} + x_{11}z_{11} \\&+ x_{21}z_{12} - k.
\end{align*}
On the other hand, by a direct calculation, we have $\psi(X)\psi(Z)-S=\begin{bmatrix}
    m'_{11}& m'_{12}\\ \ast & m'_{22}
\end{bmatrix}$. 
Moreover, we have $\det(\psi(X^{-1}TZ^{-1}))=1$ from Proposition \ref{pr:Cohn-Markov-monodromy}, and \begin{align*}\det(\psi(X)\psi(Z)-S)=&\det(X^{-1}TZ^{-1})+k\mathrm{tr}(X^{-1}TZ^{-1})+k^2\\=&\det(X^{-1}TZ^{-1})+k\mathrm{tr}(Y)+k^2\\=&\det(X^{-1}TZ^{-1})=1\end{align*}
by a direct calculation. Therefore, we have $\psi(X^{-1}TZ^{-1})=\psi(X)\psi(Z)-S$. 
Second, we prove the latter statement. By Proposition \ref{pr:Cohn-Markov-monodromy}, it suffices to show that \[\psi^{-1}(P)\psi^{-1}(Q)\psi^{-1}(R)=T.\] By assumption, it is enough to show that \[\psi^{-1}(PR-S)=(\psi^{-1}(P))^{-1}T(\psi^{-1}(R))^{-1}.\]
By the definition of the $k$-GC matrix, we can set
\[P=\begin{bmatrix}
    p_{11}& p_{12}\\p_{21}&-p_{11}+(3k+3)p_{12}-k
\end{bmatrix},\quad R=\begin{bmatrix}
    r_{11}& r_{12}\\r_{21}&-r_{11}+(3k+3)r_{12}-k
\end{bmatrix}.\]
By a direct calculation, we have $PR-S=\begin{bmatrix}
    n_{11}&n_{12}\\\ast & n_{22}
\end{bmatrix}$, where
\begin{align*}
n_{11}=& p_{11}r_{11} + p_{12}r_{21} - k, \\
n_{12}=&(3k+3)p_{12}r_{12} - kp_{12} - p_{12}r_{11} + p_{11}r_{12},\\
n_{22}=&9k^2p_{12}r_{12} - 3k^2p_{12} - 3kp_{12}r_{11} - 3k^2r_{12} - 3kp_{11}r_{12} + 18kp_{12}r_{12} + k^2 + kp_{11} \\&- 3kp_{12} + kr_{11} + p_{11}r_{11} - 3p_{12}r_{11} - 3kr_{12} - 3p_{11}r_{12} + 9p_{12}r_{12} + p_{21}r_{12} - k.
\end{align*}
By applying $\psi$, we have $\psi^{-1}(PR-S)=\begin{bmatrix}
    n'_{11}&n'_{12}\\\ast & n'_{22}
\end{bmatrix}$, where
\begin{align*}
n'_{11}=&(3k^2+3k)p_{12}r_{12} - k^2p_{12} - kp_{12}r_{11} + kp_{11}r_{12} - p_{11}r_{11} - p_{12}r_{21},\\
n'_{12}=&(3k+3)p_{12}r_{12} - kp_{12} - p_{12}r_{11} + p_{11}r_{12},\\
n'_{22}=&-3k^2p_{12}r_{12} + k^2p_{12} + kp_{12}r_{11} + 3k^2r_{12} + 5kp_{11}r_{12} - 3kp_{12}r_{12} - k^2 - kp_{11} - kr_{11} \\&- p_{11}r_{11} + 3kr_{12} + 6p_{11}r_{12} - p_{21}r_{12}.
\end{align*}
On the other hand, by a direct calculation, we have \[(\psi^{-1}(P))^{-1}T(\psi^{-1}(R))^{-1}=\begin{bmatrix}
 n'_{11}&n'_{12}\\ \ast & n'_{22}    
\end{bmatrix}.\]
Moreover, since $\det(\psi^{-1}(PR-S))=\det((\psi^{-1}(P))^{-1}T(\psi^{-1}(R))^{-1})=1$ holds, we have the desired equality.
\end{proof}
We set
$(X_{1;\ell},Y_{1;\ell},Z_{1;\ell}):=(\psi^{-1}(P_{1;-\ell}),\psi^{-1}(Q_{1;-\ell}),\psi^{-1}(R_{1;-\ell}))$, that is,
\begin{align*}
  X_{1;\ell}&=\begin{bmatrix}
    \ell &1\\-\ell^2-k\ell-1&-k-\ell
\end{bmatrix},\\
    Y_{1;\ell}&=\begin{bmatrix}
    -k+\ell-1 &1\\-\ell^2+k\ell+2\ell-k-2&-\ell+1
\end{bmatrix},\\
  Z_{1;\ell}&=\begin{bmatrix}
    -2k+\ell-2 &1\\-2k^2-\ell^2+3k\ell-6k+4\ell-5& k-\ell+2
\end{bmatrix}.
\end{align*}

We fix $k\in \ZZ_{\geq 0}$ and $\ell\in \mathbb Z$. We consider the following tree $\mathrm{MM}\mathbb T(k,\ell)$:
\begin{itemize}\setlength{\leftskip}{-15pt}
\item [(1)] the root vertex is $(X_\ell,Y_\ell,Z_\ell):=(X_{1;\ell},Y_{1;\ell}Z_{1;\ell}Y^{-1}_{1;\ell}, Y_{1;\ell})$, that is, 
\begin{align*}
  X_{\ell}&=\begin{bmatrix}
    \ell &1\\-\ell^2-k\ell-1&-k-\ell
\end{bmatrix},\\
    Y_\ell&=\begin{bmatrix}
    k\ell-k+2\ell-1 &k+2\\-k\ell^2-2\ell^2+k\ell+2\ell-1&-k\ell-2\ell+1
\end{bmatrix},\\
  Z_\ell&=\begin{bmatrix}
    -k+\ell-1 &1\\-\ell^2+k\ell+2\ell-k-2&-\ell+1
\end{bmatrix},
\end{align*}
\item[(2)]for a vertex $(X,Y,Z)$, there are the following two children of it:
\[\begin{xy}(0,0)*+{(X,Y,Z)}="1",(-30,-15)*+{(X,YZY^{-1},Y)}="2",(30,-15)*+{(Y,Y^{-1}XY,Z).}="3", \ar@{-}"1";"2"\ar@{-}"1";"3"
\end{xy}\]
\end{itemize}
We call $\mathrm{MM}\mathbb T(k,\ell)$ the \emph{$k$-Markov-monodromy tree}. In this paper, we abbreviate this tree as the \emph{$k$-MM tree}.
\begin{example}\label{ex:Markov-monodromy-tree}
 When $k=1$ and $\ell=0$, $\mathrm{MM}\mathbb T(k,\ell)$ is the following.
\relsize{-2}
\begin{align*}
\begin{xy}(-20,0)*+{\left(\begin{bmatrix}0&1\\-1&-1\end{bmatrix},\begin{bmatrix}-2&3\\-1&1\end{bmatrix},\begin{bmatrix}-2&1\\-3&1\end{bmatrix}\right)}="1",(30,-16)*+{\left(\begin{bmatrix}0&1\\-1&-1\end{bmatrix},\begin{bmatrix}-4&13\\-1&3\end{bmatrix},\begin{bmatrix}-2&3\\-1&1\end{bmatrix}\right)}="2",(30,16)*+{\left(\begin{bmatrix}-2&3\\-1&1\end{bmatrix},\begin{bmatrix}-10&13\\-7&9\end{bmatrix},\begin{bmatrix}-2&1\\-3&1\end{bmatrix}\right)}="3", (90,-24)*+{\left(\begin{bmatrix}0&1\\-1&-1\end{bmatrix},\begin{bmatrix}-14&61\\-3&13\end{bmatrix},\begin{bmatrix}-4&13\\-1&3\end{bmatrix}\right)\cdots}="4",(90,-8)*+{\left(\begin{bmatrix}-4&13\\-1&3\end{bmatrix},\begin{bmatrix}-68&217\\-21&67\end{bmatrix},\begin{bmatrix}-2&3\\-1&1\end{bmatrix}\right)\cdots}="5",(90,8)*+{\left(\begin{bmatrix}-2&3\\-1&1\end{bmatrix},\begin{bmatrix}-150&217\\-103&149\end{bmatrix},\begin{bmatrix}-10&13\\-7&9\end{bmatrix}\right)\cdots}="6",(90,24)*+{\left(\begin{bmatrix}-10&13\\-7&9\end{bmatrix},\begin{bmatrix}-48&61\\-37&47\end{bmatrix},\begin{bmatrix}-2&1\\-3&1\end{bmatrix}\right)\cdots}="7",\ar@{-}"1";"2"\ar@{-}"1";"3"\ar@{-}"2";"4"\ar@{-}"2";"5"\ar@{-}"3";"6"\ar@{-}"3";"7"
\end{xy}
\end{align*}
\relsize{+2}
%When $k=0$ and $\ell=0$, see the first tree in Example \ref{ex:intro-Markov-monodromy}. 
\end{example}

\begin{theorem}\label{thm:BT-CT2}
For $(X,Y,Z)$ in $\mathrm{MM}\mathbb T(k,\ell)$, \[\Psi(X,Y,Z):=(\psi(X),\psi(Y),\psi(Z))\] is a $k$-GC triple in $\mathrm{GC}\mathbb T(k,-\ell)$. In particular, $(X,Y,Z)$ is a $k$-MM triple. Moreover, this correspondence induces the canonical graph isomorphism between $\mathrm{MM}\mathbb T(k,\ell)$ and $\mathrm{GC}\mathbb T(k,-\ell)$.
\end{theorem}

The following lemma is a key to prove Theorem \ref{thm:BT-CT2}.

\begin{lemma}\label{lem:psi-cohn}
For a $k$-MM triple $(X,Y,Z)$, we have 
\begin{align*}
\Psi(X,YZY^{-1},Y)&=(\psi(X),\psi(X)\psi(Y)-S,\psi(Y)),\\
\Psi(Y,Y^{-1}XY,Z)&=(\psi(Y),\psi(Y)\psi(Z)-S,\psi(Z)).
\end{align*}
\end{lemma}

\begin{proof}
It is enough to show \[\psi(YZY^{-1})=\psi(X)\psi(Y)-S,\quad \psi(Y^{-1}XY)=\psi(Y)\psi(Z)-S.\] We will only prove the former. By Proposition \ref{pr:Markov-monodromy-cohn-triple}, we have $\psi(Y)=\psi(X)\psi(Z)-S$. Therefore, we have
\[\psi(X)\psi(Y)-S=\psi(Y)\psi(Z)^{-1}\psi(Y)+S\psi(Z)^{-1}\psi(Y)-S.\]
Therefore, we will show that
\begin{equation}\label{eq:BT-CT2}
\psi(YZY^{-1})=\psi(Y)\psi(Z)^{-1}\psi(Y)+S\psi(Z)^{-1}\psi(Y)-S.
\end{equation}
We set $Y=\begin{bmatrix}
    y_{11}&y_{12}\\y_{21}&-y_{11}-k
\end{bmatrix},Z=\begin{bmatrix}
    z_{11}&z_{12}\\z_{21}&-z_{11}-k
\end{bmatrix}$, $YZY^{-1}=\begin{bmatrix}
    \alpha_{11}&\alpha_{12}\\\alpha_{21}&\alpha_{22}
\end{bmatrix}$, and the right-hand side of \eqref{eq:BT-CT2} by $\begin{bmatrix}
    \beta_{11}&\beta_{12}\\\beta_{21}&\beta_{22}
\end{bmatrix}$.
Then, by a direct calculation, we have
\begin{align*}
    \alpha_{11}=&   -y_{11}^2z_{11} + y_{12}y_{21}z_{11} - y_{11}y_{21}z_{12} - y_{11}y_{12}z_{21} - ky_{11}z_{11} + ky_{12}y_{21} - ky_{12}z_{21},\\  
    \alpha_{12}=&-2y_{11}y_{12}z_{11} + y_{11}^2z_{12} - y_{12}^2z_{21} - ky_{11}y_{12},\\
    \beta_{11}=&-2ky_{11}y_{12}z_{11} + ky_{11}^2z_{12} - ky_{12}^2z_{21} - k^2y_{11}y_{12} + y_{11}^2z_{11} - y_{12}y_{21}z_{11} + y_{11}y_{21}z_{12}\\
    & + y_{11}y_{12}z_{21}+ ky_{11}z_{11}- ky_{12}y_{21} + ky_{12}z_{21} - k,\\
    \beta_{12}=&-2y_{11}y_{12}z_{11} + y_{11}^2z_{12} - y_{12}^2z_{21} - ky_{11}y_{12},\\
    \beta_{22}=&-4ky_{11}z_{11}y_{12} + 2ky_{11}^2z_{12} - 2ky_{12}^2z_{21} - 2k^2y_{11}y_{12} - y_{11}^2z_{11} - 6y_{11}z_{11}y_{12} + 3y_{11}^2z_{12}\\& + z_{11}y_{12}y_{21} - y_{11}z_{12}y_{21}- y_{11}y_{12}z_{21} - 3y_{12}^2z_{21} - ky_{11}^2 - ky_{11}z_{11} - 3ky_{11}y_{12} - ky_{12}z_{21} \\&-k^2y_{11} - k.
\end{align*}
Therefore, we have $-\alpha_{11}+\alpha_{12}k-k=\beta_{11}$ and $\alpha_{12}=\beta_{12}$. Moreover, since $Y\in SL(2,\ZZ)$, we have
\begin{align*}
\beta_{22}-{\alpha}_{11}-(2k+3)\alpha_{12}=-ky_{11}^2 - ky_{12}y_{21} - k^2y_{11} - k=0,
\end{align*}
and this finishes the proof.
\end{proof}
\begin{proof}[Proof of Theorem \ref{thm:BT-CT2}]
By definition, $(P_{1;-\ell},Q_{1;-\ell},R_{1;-\ell})=\Psi(X_{1;\ell},Y_{1;\ell},Z_{1;\ell})$. Therefore, by Lemma \ref{lem:psi-cohn}, we have $(P_{1;-\ell},P_{1;-\ell}Q_{1;-\ell}-S,Q_{1;-\ell})=\Psi(X_{1;\ell},Y_{1;\ell}Z_{1;\ell}Y_{1;\ell}^{-1},Y_{1;\ell})$. This implies that the root of $\mathrm{MM}\mathbb T(k,\ell)$ maps to the root of $\mathrm{GC}\mathbb T(k,-\ell)$ by $\Psi$. We assume that $(X,Y,Z)$ is in $\mathrm{MM}\mathbb T(k,\ell)$. Then, by Lemma \ref{lem:psi-cohn} and Proposition \ref{pr:Markov-monodromy-cohn-triple},  $(\psi(X),\psi(YZY^{-1}),\psi(Y))$ and $(\psi(Y),\psi(Y^{-1}XY),\psi(Z))$ are $k$-GC triples and thus $(X,YZY^{-1},Y)$ and $(Y,Y^{-1}XY,Z)$ are $k$-MM triples.
\end{proof}

%\begin{example}
%See Example \ref{ex:Psi} for specific examples of Theorem \ref{thm:BT-CT2}.
%\end{example}

In pallarel with $k$-GC triple, we have the following proposition:

\begin{corollary}\label{cor:BT-MT}
We fix $\ell\in \mathbb Z$. The correspondence from $(X,Y,Z)$ in $\mathrm{MM}\TT(k,\ell)$ to $(x_{12},y_{12},z_{12})$ induces the canonical graph isomorphism between $\mathrm{MM}\TT(k,\ell)$ and $\mathrm{M}\TT(k)$. In particular, for a $k$-GM triple $(a,b,c)$ with $b> \max\{a,c\}$, there is a $k$-MM triple associated with $(a,b,c)$. 
\end{corollary}

\begin{proof}
 It follows from Corollary \ref{cor:CT-MT} and Theorem \ref{thm:BT-CT2}.   
\end{proof}

\begin{proposition}\label{pr:all-Markov-monodromy-triple}
Let $(X,Y,Z)$ be a $k$-MM triple associated with $(a,b,c)$. We assume that $b> \max\{a,c\}.$ Then, there exist a unique $\ell\in \mathbb Z$ and a unique vertex $v$ in $\mathrm{MM}\mathbb T(k,\ell)$ such that $v=(X,Y,Z)$.  
\end{proposition}

\begin{proof}
Let $(X,Y,Z)$ be a $k$-MM triple associated with $(a,b,c)$. Then, $\Psi(X,Y,Z)$ is a $k$-GC triple. By Proposition \ref{pr:all-cohn-triple}, there exists a unique $\ell$ such that $\Psi(X,Y,Z)\in \mathrm{GC}\mathbb T(k,-\ell)$. Therefore, by Theorem \ref{thm:BT-CT2}, $(X,Y,Z)$ is in $\mathrm{MM}\mathbb T(k,\ell)$.
\end{proof}

\begin{corollary}\label{thm:Markov-monodromy-distinct}
We fix $k\in \mathbb Z_{\geq0}$ and $\ell\in \mathbb Z$. The second entries of $k$-MM triples in $\mathrm{MM}\mathbb T(k, \ell)$ are distinct.   
\end{corollary}

\begin{proof}
It follows from Theorem \ref{thm:Cohn-distinct} and Theorem \ref{thm:BT-CT2}.   
\end{proof}

In parallel with the $k$-GC triple, we have the following relation between $\mathrm{MM}\mathbb T(k,\ell)$ and $\mathrm{MM}\mathbb T(k,\ell')$. 

\begin{proposition}\label{Markov-monodromy-matrix-ltol'}
Let $L=\begin{bmatrix}
    1&0\\\ell'-\ell&1
\end{bmatrix}$. The right conjugation of $X$ by $L$ induces the canonical graph isomorphism from $\mathrm{MM}\mathbb T(k, \ell)$ to $\mathrm{MM}\mathbb T(k, \ell')$.   
\end{proposition}

\begin{proof}
 We can check $X_{1;\ell'}=L^{-1}X_{1;\ell}L$, $Y_{1;\ell'}=L^{-1}Y_{1;\ell}L$, $Z_{1;\ell'}=L^{-1}Z_{1;\ell}L$ directly. For general cases, the statement follows inductively from 
 \begin{align*}
    L^{-1}(YZY^{-1})L&=(L^{-1}YL)(L^{-1}ZL)(L^{-1}Y^{-1}L),\\ L^{-1}(Y^{-1}XY)L&=(L^{-1}Y^{-1}L)(L^{-1}XL)(L^{-1}YL).
\end{align*}
\end{proof}

\begin{remark}
Let $A$ be the set of $k$-MM triples. We set $\sigma_1,\sigma_2\colon A \to A$ by
\[\sigma_1(X,Y,Z)=(X,Z,Z^{-1}YZ),\quad \sigma_2(X,Y,Z)= (Y,Y^{-1}XY,Z).\] 
Then, it can be seen that in $\mathrm{MM}\mathbb T(k,\ell)$ the operation giving the left child of $(X,Y,Z)$ is given by $\sigma_1^{-1}$ and the operation giving the right child is given by $\sigma_2$. Moreover, in parallel with Remark \ref{rmk:Markov-monodromy-structure-cohn}, we have the braid relation $\sigma_1\sigma_2\sigma_1=\sigma_2\sigma_1\sigma_2$.
\end{remark}

\subsection{Inverse $k$-Markov-monodromy tree}
We consider the following tree $\mathrm{MM}\mathbb T^{\dag}(k,\ell)$:
\begin{itemize}\setlength{\leftskip}{-15pt}
\item [(1)] the root vertex is $(X_{1;\ell},Y_{1;\ell},Z_{1;\ell})$,
\item[(2)]for a vertex $(X,Y,Z)$, we consider the following two children of it:
\[\begin{xy}(0,0)*+{(X,Y,Z)}="1",(-30,-15)*+{(X,Z,Z^{-1}YZ)}="2",(30,-15)*+{(XYX^{-1},X,Z).}="3", \ar_{\sigma_1}@{-}"1";"2"\ar^{\sigma_2^{-1}}@{-}"1";"3"
\end{xy}\]
\end{itemize}
We call $\mathrm{MM}\mathbb T^{\dag}(k,\ell)$ the \emph{inverse Markov-monodromy tree}. In this paper, we abbreviate this tree as the \emph{inverse $k$-MM tree}.

\begin{example}\label{ex:inverse-Markov-monodromy-tree}
 When $k=1$ and $\ell=0$, $\mathrm{MM}\mathbb T^\dag(k,\ell)$ is the following.
\relsize{-2}
\begin{align*}
\begin{xy}(-20,0)*+{\left(\begin{bmatrix}0&1\\-1&-1\end{bmatrix},\begin{bmatrix}-2&1\\-3&1\end{bmatrix},\begin{bmatrix}-4&1\\-13&3\end{bmatrix}\right)}="1",(30,-16)*+{\left(\begin{bmatrix}0&1\\-1&-1\end{bmatrix},\begin{bmatrix}-4&1\\-13&3\end{bmatrix},\begin{bmatrix}-14&3\\-61&13\end{bmatrix}\right)}="2",(30,16)*+{\left(\begin{bmatrix}4&3\\-7&-5\end{bmatrix},\begin{bmatrix}0&1\\-1&-1\end{bmatrix},\begin{bmatrix}-4&1\\-13&3\end{bmatrix}\right)}="3", 
(90,-24)*+{\left(\begin{bmatrix}0&1\\-1&-1\end{bmatrix},\begin{bmatrix}-14&3\\-61&13\end{bmatrix},\begin{bmatrix}-62&13\\-291&61\end{bmatrix}\right)\cdots}="4",(90,-8)*+{\left(\begin{bmatrix}16&13\\
-21&-17\end{bmatrix},\begin{bmatrix}0&1\\-1&-1\end{bmatrix},\begin{bmatrix}-14&3\\-61&13\end{bmatrix}\right)\cdots}="5",(90,8)*+{\left(\begin{bmatrix}4&3\\-7&-5\end{bmatrix},\begin{bmatrix}-4&1\\-13&3\end{bmatrix},\begin{bmatrix}-56&13\\-237&55
\end{bmatrix}\right)\cdots}="6",(90,24)*+{\left(\begin{bmatrix}22& 13\\-39&-23\end{bmatrix},\begin{bmatrix}4&3\\-7&-5\end{bmatrix},\begin{bmatrix}-4&1\\-13&3\end{bmatrix}\right)\cdots}="7",\ar@{-}"1";"2"\ar@{-}"1";"3"\ar@{-}"2";"4"\ar@{-}"2";"5"\ar@{-}"3";"6"\ar@{-}"3";"7"
\end{xy}
\end{align*}
\relsize{+2}
%When $k=0$ and $\ell=0$, see the second tree in Example \ref{ex:intro-Markov-monodromy}. 
\end{example}

\begin{theorem}\label{thm:BT-CT2dag}
For $(X,Y,Z)$ in $\mathrm{MM}\mathbb T^\dag(k,\ell)$, $\Psi(X,Y,Z)$ is a $k$-GC matrix in $\mathrm{GC}\mathbb T^{\dag}(k,-\ell)$. In particular, $(X,Y,Z)$ is a $k$-MM triple. Moreover, this correspondence induces the canonical graph isomorphism between $\mathrm{MM}\mathbb T^\dag(k,\ell)$ and $\mathrm{GC}\mathbb T^{\dag}(k,-\ell)$.
\end{theorem}

The following lemma is a key to prove Theorem \ref{thm:BT-CT2dag}.

\begin{lemma}\label{lem:psi-cohn2}
For a $k$-MM triple $(X,Y,Z)$, we have 
\begin{align*}
\Psi(X,Z,Z^{-1}YZ)&=(\psi(X),\psi(Z),\psi(X)^{-1}(\psi(Z)+S))\\
\Psi(XYX^{-1},X,Z)&=((\psi(X)+S)\psi(Z)^{-1},\psi(X),\psi(Z)).
\end{align*}
\end{lemma}

We omit the proofs of the above theorem and lemma, because it is almost the same as those of Theorem \ref{thm:BT-CT2} and Lemma \ref{lem:psi-cohn}. The following is a list of properties that hold in parallel with the $k$-MM tree case.

\begin{corollary}
We fix $\ell\in \mathbb Z$. The correspondence from $(X,Y,Z)$ in $\mathrm{MM}\TT^\dag(k,\ell)$ to $(x_{12},y_{12},z_{12})$ induces the canonical graph isomorphism between $\mathrm{MM}\TT^\dag(k,\ell)$ and $\mathrm{M}\TT^\dag(k)$. In particular, for a $k$-GM triple $(a,b,c)$ with $b\leq \max\{a,c\}$, there is a $k$-MM triple associated with $(a,b,c)$. 
\end{corollary}

\begin{proposition}\label{pr:all-Markov-monodromy-triple2}
Let $(X,Y,Z)$ be a $k$-MM triple associated with $(a,b,c)$. We assume that $b\leq \max\{a,c\}.$ Then, there exist a unique $\ell\in \mathbb Z$ and a unique vertex $v$ in $\mathrm{MM}\mathbb T^\dag(k,\ell)$ such that $v=(X,Y,Z)$.
\end{proposition}

\begin{proposition}\label{Markov-monodromy-matrix-ltol'2}
Let $L=\begin{bmatrix}
    1&0\\\ell'-\ell&1
\end{bmatrix}$. The  right conjugation of $X$ by $L$ induces the canonical graph isomorphism from $\mathrm{MM}\mathbb T^\dag(k, \ell)$ to $\mathrm{MM}\mathbb T^\dag(k, \ell')$. 
\end{proposition}

\subsection{Markov-monodromy decomposition of generalized Cohn matrix}
We introduce another connection between $k$-GC triples and $k$-MM triples. 
\begin{definition}\label{def:Markov-monodromy-dec}
 We fix $k\in \mathbb Z_{\geq 0}$. For a $k$-GC triple $(P,Q,R)$, we consider a triple $(X,Y,Z)$ satisfying the following conditions:
\begin{itemize}\setlength{\leftskip}{-15pt}
\item [(1)] $X,Y,Z\in SL(2,\mathbb Z)$,
\item [(2)] $P=-Z^{-1}Y^{-1},\quad Q=-Z^{-1}X^{-1},\quad R=-Y^{-1}X^{-1}$,
\item[(3)] $\mathrm{tr}(X)=\mathrm{tr}(Y)=\mathrm{tr}(Z)$.
\end{itemize}
This triple $(X,Y,Z)$ of $(P,Q,R)$ is called a \emph{Markov-monodromy decomposition of} $(P,Q,R)$. 
\end{definition}
In this paper, we abbreviate the Markov-monodromy decomposition as the \emph{MM decomposition}.
We note that we can replace (3) in Definition \ref{def:Markov-monodromy-dec} with
\begin{itemize}\setlength{\leftskip}{-15pt}
    \item [(3')] $X+X^{-1}=Y+Y^{-1}=Z+Z^{-1}$.
\end{itemize}

\begin{lemma}\label{thm:Markov-monodromy-uniqueness-initial}
 The $k$-MM triples $(X_\ell,Y_\ell,Z_\ell)$ i.e., the root of $\mathrm{MM}\mathbb T(k,\ell)$, and $(-X_\ell,-Y_\ell,-Z_\ell)$ are MM decompositions of the $k$-GC triple $(P_{1;\ell},Q_{1;\ell},R_{1;\ell})$. Moreover, there are no other MM decompositions of $(P_{1;\ell},Q_{1;\ell},R_{1;\ell})$ than these two.
\end{lemma}

To prove Lemma \ref{thm:Markov-monodromy-uniqueness-initial}, we use the following lemma.

\begin{lemma}[\cite{sul93}]\label{lem:sqrt-matrix}
Let $Y\in M(2,\mathbb C)$. We assume that $Y$ is not a scalar matrix. The following statements hold:
\begin{itemize}\setlength{\leftskip}{-15pt}
    \item [(1)] if $\mathrm{tr}(Y^2)^2\neq 4\det (Y^2)$, then we have $Y=\pm\dfrac{Y^2+\varepsilon I}{\sqrt{\mathrm{tr}(Y^2)+2\varepsilon} }$, where $\varepsilon=\pm 1$,
    \item [(2)] if $\mathrm{tr}(Y^2)^2= 4\det (Y^2)$,  then we have $Y=\pm\dfrac{1}{2}(Y^2+I)$.
\end{itemize}
\end{lemma}

\begin{proof}[Proof of Lemma \ref{thm:Markov-monodromy-uniqueness-initial}]
  We can check directly that $(X_\ell,Y_\ell,Z_\ell)$ and $(-X_\ell,-Y_\ell,-Z_\ell)$ are MM decompositions of $(P_{1;\ell},Q_{1;\ell},R_{1;\ell})$. Let $(X,Y,Z)$ be an MM decomposition of $(P_{1;\ell},Q_{1;\ell},R_{1;\ell})$. Then we have 
  \[Y^{2}=-P_{1;\ell}^{-1}Q_{1;\ell}R_{1;\ell}^{-1}=\begin{bmatrix}
      -k^2\ell+k^2-2k\ell+k-1& -k^2-2k\\k^2\ell^2-k^2\ell+2k\ell^2-2k\ell+k &k^2\ell+2k\ell-k-1
  \end{bmatrix}.\]
We will calculate $Y$ according to Lemma \ref{lem:sqrt-matrix}. 

First, we consider the cases $k\neq 0,2$. Then, $Y^2$ is not a diagonal matrix, and $\mathrm{tr}(Y^2)^2\neq 4\det (Y^2)$. Then we have
\[Y=\pm\dfrac{Y^2+\varepsilon I}{\sqrt{\mathrm{tr}(Y^2)+2\varepsilon} }=\pm\dfrac{Y^2+\varepsilon I}{\sqrt{k^2-2+2\varepsilon} },\]
where $\varepsilon=\pm 1$. When $\varepsilon=-1$, the denominator $\sqrt{k^2-4}$ is not an integer. Therefore, by the condition (1) in Definition \ref{def:Markov-monodromy-dec}, $\varepsilon$ must be $1$ and we have
\[Y=\pm\dfrac{1}{k}(Y^2+I)=\mp Y_\ell.\]
Then, we have
\[X=-R^{-1}(\mp Y_\ell^{-1})=\mp X_{\ell},\quad Z=-(\mp Y_{\ell}^{-1})P^{-1}=\mp Z_{\ell},\]as desired. 

Second, we consider the case $k=2$. Then $Y^2$ is not a diagonal matrix, and $\mathrm{tr}(Y^2)^2= 4\det (Y^2)$. Then we have
\[Y=\pm\dfrac{1}{2}(Y^2+I)=\mp Y_\ell.\]
The rest of the discussion is the same as for $k\neq 0,2$. 

Finally, we consider the case $k=0$. Then we have $Y^2=-I$ and $Y=\begin{bmatrix}
    \alpha& \beta \\ \dfrac{-1-\alpha^2}{\beta}&-\alpha
\end{bmatrix}$, where $\alpha\in\ZZ$ and $\beta\in\ZZ\setminus\{0\}$.
Then we have
\begin{align*}   
\mathrm{tr}(X)&=\mathrm{tr}(-R^{-1}Y^{-1})=-2\ell\alpha-\alpha+\dfrac{\alpha^2}{\beta}+\dfrac{1}{\beta}+\ell^2\beta+\ell\beta-\beta,\\
\mathrm{tr}(Z)&=\mathrm{tr}(-Y^{-1}P^{-1})=-2\ell\alpha+3\alpha+\dfrac{\alpha^2}{\beta}+\dfrac{1}{\beta}+\ell^2\beta-3\ell\beta+\beta.
\end{align*}
By the condition (3) in Definition \ref{def:Markov-monodromy-dec}, we have $\mathrm{tr}(X)=\mathrm{tr}(Y)=\mathrm{tr}(Z)=0$ and thus we get $\alpha=(2\ell-1)\varepsilon$ and $\beta=2\varepsilon$, where $\varepsilon=\pm1$. If $\varepsilon=1$, then we have $(X,Y,Z)=(X_\ell,Y_\ell,Z_{\ell})$, and $\varepsilon=-1$, then we have $(X,Y,Z)=(-X_\ell,-Y_\ell,-Z_{\ell})$.
\end{proof}

We will prove the following theorem:

\begin{theorem}\label{thm:Markov-monodromy-uniqueness}
For any $k$-GC triple $(P,Q,R)$ associated with $(a,b,c)$, if $b> \max\{a,c\}$ holds, then there is an MM decomposition $(X,Y,Z)$ of $(P,Q,R)$ and it is unique up to sign. 
\end{theorem}

To prove Theorem \ref{thm:Markov-monodromy-uniqueness}, we prepare some propositions. The next one provides the existence of an MM decomposition of any $k$-GC triple:

\begin{proposition}\label{thm:BT-CT}
If $(X,Y,Z)$ is an MM decomposition of $(P,Q,R)$, then $(X,Z,Z^{-1}YZ)$ (resp. $(XYX^{-1},X,Z)$) is an MM decomposition of $(P,PQ-S,Q)$ (resp. $(Q,QR-S,R)$).   
\end{proposition}

\begin{proof}
By the assumption $Q=PR-S$, we have
\[-Z^{-1}X^{-1}=Z^{-1}Y^{-2}X^{-1}-S.\]
By multiplying the above equality by $Z$ from the left and by $X$ from the right, we have
\begin{equation}\label{eq:ZSX}
    ZSX=Y^{-2}+I.
\end{equation}
 We denote by $(X',Y',Z')$ (resp. $(X'',Y'',Z'')$) the left (resp. right) child of $(X,Y,Z)$ in $\mathrm{MM}\TT^\dag(k,\ell)$. We can easily see that  $(X',Y',Z')$ and $(X'',Y'',Z'')$ satisfy (1) and (3') in Definition \ref{def:Markov-monodromy-dec}.  We will prove that they satisfy (2). It suffices to show that
\begin{align*}
    -(Y'Z')^{-1}&=-(YZ)^{-1},\\
    -(X'Z')^{-1}&=(YZ)^{-1}(XZ)^{-1}-S,\\
    -(X'Y')^{-1}&=-(XZ)^{-1},\\
    -(Y''Z'')^{-1}&=-(XZ)^{-1},\\
    -(X''Z'')^{-1}&=(XZ)^{-1}(XY)^{-1}-S,\\
    -(X''Y'')^{-1}&=-(XY)^{-1}.
\end{align*}  
All but the second and fifth equality are clear. We only prove the second equality. It suffices to show that
\[-Z^{-1}Y^{-1}ZX^{-1}=Z^{-1}Y^{-1}Z^{-1}X^{-1}-S.\]
By multiplying the above equality by $Z$ from the left and by $X$ from the right and applying \eqref{eq:ZSX}, we have
\[-Y^{-1}Z=Y^{-1}Z^{-1}-Y^{-2}-I.\]
This equality can be obtained from (3') in Definition \ref{def:Markov-monodromy-dec}.
\end{proof}

In parallel with Lemma \ref{Cohn-Markov2}, we have the following lemma.
\begin{lemma}\label{lem:inverseBT-CT}
If $(X,Y,Z)$ is an MM decomposition of $(P,Q,R)$, then $(X,YZY^{-1},Y)$ (resp. $(Y,Y^{-1}XY,Z)$) is an MM decomposition of $(P,R,P^{-1}(R+S))$ (resp. $((P+S)R^{-1},P,R)$). 
\end{lemma}

\begin{proof}
Similar to the proof Theorem \ref{thm:BT-CT}, we have \eqref{eq:ZSX}. We set $(X',Y',Z'):=(X,YZY^{-1},Y)$ and $(X'',Y'',Z''):=(Y,Y^{-1}XY,Z)$.  We can  easily see that  $(X',Y',Z')$ and $(X'',Y'',Z'')$ satisfy (1) and (3) in Definition \ref{def:Markov-monodromy-dec}.  We will prove that they satisfy (2). It suffices to show that
\[-(X'Y')^{-1}=-(YZ)(-(XY)^{-1}+S),\quad -(Y'Z')^{-1}=(-(YZ)^{-1}+S)(-XY).\]
We will only prove the former equality.
It suffices to show that
\[-YZ^{-1}Y^{-1}X^{-1}=YZY^{-1}X^{-1}-YZS.\]
By multiplying the above equality by $Y^{-1}$ from the left and by $X$ from the right and applying \eqref{eq:ZSX}, we have
\[-Z^{-1}Y^{-1}=ZY^{-1}-Y^{-2}-I.\]
This equality follows from (3') in Definition \ref{def:Markov-monodromy-dec}.
\end{proof}

\begin{proof}[Proof of Theorem \ref{thm:Markov-monodromy-uniqueness}]
By Proposition \ref{pr:all-cohn-triple}, there exists $\ell\in \ZZ$ such that $(P,Q,R)\in \mathrm{GC}\mathbb T(k,\ell)$. The existence follows from Lemma \ref{thm:Markov-monodromy-uniqueness-initial} and Proposition \ref{thm:BT-CT}. We assume that $(X_1,Y_1,Z_1)$ and $(X_2,Y_2,Z_2)$ are MM decompositions of $(P,Q,R)$. From Lemma \ref{lem:inverseBT-CT} by applying $(X,Y,Z)\mapsto(X,YZY^{-1},Y)$ and $(X,Y,Z)\mapsto (Y,Y^{-1}XY,Z)$ to $(X_1,Y_1,Z_1)$ and $(X_2,Y_2,Z_2)$ repeatedly, we get $(X'_1,Y'_1,Z'_1)$ and $(X'_2,Y'_2,Z'_2)$ such that they are MM decompositions of a $k$-GC triple associated with $(1,1,1)$. By Lemma \ref{thm:Markov-monodromy-uniqueness-initial}, we have $(X'_1,Y'_1,Z'_1)=\pm(X'_2,Y'_2,Z'_2)$. Therefore, we have $(X_1,Y_1,Z_1)=\pm(X_2,Y_2,Z_2)$. 
\end{proof}

Moreover, the following theorem holds:

\begin{theorem}\label{thm:Markov-monodromy-decomposition}
Let $(P,Q,R)$ be a $k$-GC triple associated $(a,b,c)$ with $b> \max\{a,c\}$. For an MM decomposition $(X,Y,Z)$ of a $k$-GC triple $(P,Q,R)$, if $x_{12}>0$, then it is in $\mathrm{MM}\mathbb T^{\dag}(k,\ell)$. In particular, $(X,Y,Z)$ is a $k$-MM triple.
\end{theorem}

To prove it, we will prove the following lemma:

\begin{lemma}\label{thm:Markov-monodromy-uniqueness-initial2}
The $k$-MM triple $(X_{1;\ell},Y_{1;\ell},Z_{1;\ell})$ (given in the definition of $\mathrm{MM}\TT^\dag(k,\ell)$) and $(-X_{1;\ell},-Y_{1;\ell},-Z_{1;\ell})$ are an MM decompositions of the root of  $\mathrm{GC}\TT(k,\ell)$. Moreover, there are no other MM decompositions of it than these two.
\end{lemma}

\begin{proof}
We can check that the triple $(X_{1;\ell},Y_{1;\ell},Z_{1;\ell})$ and $(-X_{1;\ell},-Y_{1;\ell},-Z_{1;\ell})$ give the MM decomposition of the root of $\mathrm{GC}\TT(k,\ell)$ directly, or by using Lemma \ref{thm:Markov-monodromy-uniqueness-initial} and Proposition \ref{thm:BT-CT}.
The latter statement is the special case of Theorem \ref{thm:Markov-monodromy-uniqueness}.
\end{proof}

\begin{proof}[Proof of Theorem \ref{thm:Markov-monodromy-decomposition}]
By Lemma \ref{thm:Markov-monodromy-uniqueness-initial2}, the statement holds when $(P,Q,R)=(P_{\ell},Q_{\ell},R_{\ell})$. Moreover, by Theorem \ref{thm:Markov-monodromy-uniqueness} and Proposition \ref{thm:BT-CT}, we have the conclusion.
\end{proof}

The mapping $\Phi$ given next is the inverse operation of the MM decomposition. 

\begin{corollary}\label{cor:tree-iso-BT-CTdag}
We set $\Phi\colon GL(2,\mathbb C)^3\to GL(2,\mathbb C)^3$ by \[\Phi(X,Y,Z)= (-(YZ)^{-1},-(XZ)^{-1},-(XY)^{-1}).\] The map $\Phi$ induces the canonical graph isomorphism from $\mathrm{MM}\mathbb T^{\dag}(k,\ell)$ to $\mathrm{GC}\mathbb T(k,\ell)$. 
\end{corollary}

%\begin{example}
% See Example \ref{ex:Phi} for specific examples of Corollary \ref{cor:tree-iso-BT-CTdag}.   
%\end{example}

In the previous discussions, we assume that $b> \max\{a,c\}$. We can do the same under the assumption $b\leq \max\{a,c\}$ by considering $\mathrm{M}\mathbb T^{\dag}(k,\ell)$, $\mathrm{GC}\mathbb T^\dag(k,\ell)$ and $\mathrm{MM}\mathbb T(k,\ell)$ instead of $\mathrm{M}\mathbb T(k,\ell)$, $\mathrm{GC}\mathbb T(k,\ell)$ and $\mathrm{MM}\mathbb T^{\dag}(k,\ell)$:

\begin{corollary}\label{cor:tree-iso-BT-CT^dag2}
For any $k$-GC triple $(P,Q,R)$ associated with $(a,b,c)$, we assume that $b\leq \max\{a,c\}$. Then there is an MM decomposition of $(P,Q,R)$ and it is unique up to sign. Moreover, there exists $\ell\in \mathbb Z$ such that one of the MM decompositions of $(P,Q,R)$ is in  $\mathrm{MM}\mathbb T(k,\ell)$. 
\end{corollary}

\begin{proof}
We can prove the statement in the same way as Theorem \ref{thm:Markov-monodromy-uniqueness} (note that the role of Proposition \ref{thm:BT-CT} and Lemma \ref{lem:inverseBT-CT} are exchanged).
\end{proof}

\begin{corollary}
Let $(P,Q,R)$ be a $k$-GC triple associated with $(a,b,c)$ with $b\leq \max\{a,c\}$. For an MM decomposition $(X,Y,Z)$ of a $k$-GC triple $(P,Q,R)$, if $x_{12}>0$, then it is in $\mathrm{MM}\mathbb T(k,\ell)$. In particular, $(X,Y,Z)$ is a $k$-MM triple.
\end{corollary}

From Corollaries \ref{cor:tree-iso-BT-CTdag} and \ref{cor:tree-iso-BT-CT^dag2}, we have the following:

\begin{corollary}\label{cor:bijection-triples}
The map $\Phi$ induces a bijection from the set of $k$-MM triples to the set of $k$-GC triples.     
\end{corollary}

%\begin{remark}
%From the viewpoint of Teichm\"uller theory, the following identities have been found by \cite{luo} and \cite{nana} to hold for any matrices $A, B, C$ in $SL(2,\mathbb C)$: we set $a:=-\mathrm{tr}(A), b:=-\mathrm{tr}(B),c:=-\mathrm{tr}(C), d:=-\mathrm{tr}(ABC), x:=-\mathrm{tr}(AB), y:=-\mathrm{tr}(BC), z:=-\mathrm{tr}(CA)$. Then we have
%\[x^2+y^2+z^2+(ad+bc)x+(bd+ca)y+(cd+ab)z+a^2+b^2+c^2+d^2+abcd-4=xyz.\]
%By substituting $(A,B,C)=(X,Y,Z)$, where $(X,Y,Z)$ is a $k$-MM triple, then we have 
%\begin{align*}
%    \mathrm{tr}(YZ)^2+\mathrm{tr}(ZX)^2+\mathrm{tr}(XY)^2-(2k+k^2)(\mathrm{tr}(YZ)+\mathrm{tr}(ZX)+&\mathrm{tr}(XY))+3k^2+2k^3\\
%&=-\mathrm{tr}(YZ)\mathrm{tr}(ZX)\mathrm{tr}(XY).
%\end{align*}
%Moreover, by using (the inverse of) the MM decomposition, we have
%\[ \mathrm{tr}(P)^2+\mathrm{tr}(Q)^2+\mathrm{tr}(R)^2+(2k+k^2)(\mathrm{tr}(P)+\mathrm{tr}(Q)+\mathrm{tr}(R))+3k^2+2k^3=\mathrm{tr}(P)\mathrm{tr}(Q)\mathrm{tr}(R).\]
%This equality shows the fact that the traces of the $k$-%generalized Cohn triple satisfies GSME$(k)$.
%\end{remark}

Now, we consider compositions of them, $\Phi\circ \Psi^{-1}$ and  $\Psi^{-1}\circ \Phi$. They are $k$-GC triple version and $k$-MM triple version of the map $\mu$ in Proposition \ref{pr:rho-mor}.

\begin{theorem}\label{thm:phi-psi-1}The following statements hold:
\begin{itemize}\setlength{\leftskip}{-15pt}
    \item[(1)] For the graph isomorphism $\Phi\circ \Psi^{-1}\colon \mathrm{GC}\mathbb T(k,\ell)\to \mathrm{GC}\mathbb T^\dag(k,-\ell)$, we have the following commutative diagram:
\[\begin{xy}(0,0)*+{\mathrm{GC}\mathbb T(k,\ell)}="1",(40,0)*+{\mathrm{GC}\mathbb T^\dag(k,-\ell)}="2",(0,-20)*+{\mathrm{M}\mathbb T(k)}="3",(40,-20)*+{\mathrm{M}\mathbb T^\dag(k),}="4", \ar^{\Phi\circ \Psi^{-1}}@{->}"1";"2"\ar@{->}"1";"3"\ar^{\mu}@{->}"3";"4"\ar@{->}"2";"4"
\end{xy}\]
where the vertical arrows are induced by the correspondence from $(X,Y,Z)$ to $(x_{12},y_{12},z_{12})$.
\item[(2)] For the graph isomorphism $\Phi\circ \Psi^{-1}\colon \mathrm{GC}\mathbb T^{\dag}(k,\ell)\to \mathrm{GC}\mathbb T(k,-\ell)$, we have the following commutative diagram:
\[\begin{xy}(0,0)*+{\mathrm{GC}\mathbb T^\dag(k,\ell)}="1",(40,0)*+{\mathrm{GC}\mathbb T(k,-\ell)}="2",(0,-20)*+{\mathrm{M}\mathbb T^{\dag}(k)}="3",(40,-20)*+{\mathrm{M}\mathbb T(k),}="4", \ar^{\Phi\circ \Psi^{-1}}@{->}"1";"2"\ar@{->}"1";"3"\ar^{\mu}@{->}"3";"4"\ar@{->}"2";"4"
\end{xy}\]
where the vertical arrows are induced by the correspondence from $(X,Y,Z)$ to $(x_{12},y_{12},z_{12})$.
\end{itemize}
\end{theorem}
\begin{proof}
The statement (1) follows from Corollaries \ref{cor:CTdag-MTdag}, \ref{cor:BT-MT} and Lemma \ref{lem:inverseBT-CT}. The statement (2) can be proved in the same way.   
\end{proof}

By using Theorems \ref{thm:Cohn-distinct}, \ref{thm:phi-psi-1}, we have the following corollary:

\begin{corollary}
We fix $k\in \mathbb Z_{\geq0}$ and $\ell\in \mathbb Z$. The second entries of $k$-GC triples in $\mathrm{GC}\mathbb T^{\dag}(k, \ell)$ are distinct.   
\end{corollary}

The following theorem is proved in the same way as the above theorem:

\begin{theorem}\label{thm:phi-psi-2}The following statements hold:
\begin{itemize}\setlength{\leftskip}{-15pt}
    \item[(1)] For the graph isomorphism $\Psi^{-1}\circ \Phi\colon \mathrm{MM}\mathbb T(k,\ell)\to \mathrm{MM}\mathbb T^\dag(k,-\ell)$, we have the following commutative diagram:
\[\begin{xy}(0,0)*+{\mathrm{MM}\mathbb T(k,\ell)}="1",(40,0)*+{\mathrm{MM}\mathbb T^\dag(k,-\ell)}="2",(0,-20)*+{\mathrm{M}\mathbb T(k)}="3",(40,-20)*+{\mathrm{M}\mathbb T^\dag(k),}="4", \ar^{\Psi^{-1}\circ \Phi}@{->}"1";"2"\ar@{->}"1";"3"\ar^{\mu}@{->}"3";"4"\ar@{->}"2";"4"
\end{xy}\]
where the vertical arrows are induced by the correspondence from $(X,Y,Z)$ to $(x_{12},y_{12},z_{12})$.
\item[(2)] For the graph isomorphism $\Psi^{-1}\circ \Phi\colon \mathrm{MM}\mathbb T^{\dag}(k,\ell)\to \mathrm{MM}\mathbb T(k,-\ell)$, we have the following commutative diagram:
\[\begin{xy}(0,0)*+{\mathrm{MM}\mathbb T^\dag(k,\ell)}="1",(40,0)*+{\mathrm{MM}\mathbb T(k,-\ell)}="2",(0,-20)*+{\mathrm{M}\mathbb T^{\dag}(k)}="3",(40,-20)*+{\mathrm{M}\mathbb T(k),}="4", \ar^{\Psi^{-1}\circ \Phi}@{->}"1";"2"\ar@{->}"1";"3"\ar^{\mu}@{->}"3";"4"\ar@{->}"2";"4"
\end{xy}\]
where the vertical arrows are induced by the correspondence from $(X,Y,Z)$ to $(x_{12},y_{12},z_{12})$.
\end{itemize}
\end{theorem}

\begin{corollary}\label{thm:Markov-monodromy-distinct2}
We fix $k\in \mathbb Z_{\geq0}$ and $\ell\in \mathbb Z$. The second entries of $k$-MM triples in $\mathrm{MM}\mathbb T^\dag(k, \ell)$ are distinct.   
\end{corollary} 

Moreover, we also have the following result from Theorems \ref{thm:phi-psi-1} and \ref{thm:phi-psi-2}:

\begin{corollary}\label{cor:Markov-monodromy-decom-algorithm}
    We have $(\Psi^{-1}\circ\Phi)^2=\mathrm{id}$ and $(\Phi\circ\Psi^{-1})^2=\mathrm{id}$. In particular, the MM decomposition $\Phi^{-1}$ is given by $\Psi^{-1}\circ \Phi\circ \Psi^{-1}$.
\end{corollary}

Note that Corollary \ref{cor:Markov-monodromy-decom-algorithm} implies that the MM decomposition of $(P, Q, R)$ can be computed with a certain algorithm.

\subsection{Interpretation as representation of fundamental group of 4-punctured sphere}

In this subsection, we will discuss the relation between $k$-MM triples and $SL(2,\mathbb{C})$-representations of the fundamental group $\pi_1(S_4^2)$ of the 4-punctured sphere. Note that $\pi_1(S_4^2)$ has the following presentation:
\[\pi_1(S_4^2)=\langle \alpha,\beta,\gamma, \delta \mid \alpha\beta\gamma\delta=1\rangle.\]

Let $\mathrm{Rep}(S_4^2)$ be the set of $SL(2,\mathbb{C})$-representations of $\pi_1(S_4^2)$. Since $\{\alpha,\beta,\gamma\}$ is a free generator of $\pi_1(S_4^2)$, a representation $\rho \in \mathrm{Rep}(S_4^2)$ is determined by the choice of $\rho(\alpha), \rho(\beta),\rho(\gamma)$. We define
\begin{align*}
&a=-\mathrm{tr}\rho(\alpha),\  b=-\mathrm{tr}\rho(\beta),\   c=-\mathrm{tr}\rho(\gamma),\  d=-\mathrm{tr}\rho(\delta),\\
&x=-\mathrm{tr}\rho(\alpha\beta),\  y=-\mathrm{tr}\rho(\beta\gamma),\ z=-\mathrm{tr}\rho(\gamma\alpha).
\end{align*}
We set the map $\chi\colon \mathrm{Rep}(S_4^2)\to \mathbb C^7$ by $\chi(\rho)=(x,y,z,a,b,c,d)$. Now, we have the following properties (see e.g. \cites{CaLo,gol09}):
\begin{itemize}\setlength{\leftskip}{-15pt}
\item[(1)] the algebra of polynomial functions on $\mathrm{Rep}(S_4^2)$ which are invariant under conjugation is generated by $a,b,c,d,x,y,z$,
\item[(2)] $a,b,c,d,x,y,z$ satisfy the following equality:
\begin{align}\label{eq:equation-rep}
    x^2+y^2+z^2+(ab+cd)x+(bc+ad)y+(ac+bd)z+a^2+b^2+c^2+d^2+abcd-4=xyz,
\end{align}
\item[(3)] let $\mathrm{Rep}(S^2_4)/\!/SL(2,\mathbb{C})$ be the GIT quotient of $\mathrm{Rep}(S^2_4)$ by the conjugate action of $SL(2,\mathbb C)$ (about the GIT quotient, see e.g. \cite{mar16}). We set
\[H:=\{(x,y,z,a,b,c,d)\in \mathbb C^7 \mid (x,y,z,a,b,c,d) \text{ satisfies \eqref{eq:equation-rep}}\}.\] Then $\chi$ induces a homeomorphism $\tilde{\chi}$ between $\mathrm{Rep}(S^2_4)/\!/SL(2,C)$ and $H$. The variety $\chi(S_4^2):=H$ is called the \emph{character variety}.
\end{itemize}

Now, we consider the representation $\rho_{X,Y,Z}$ which satisfies that $\rho(\alpha)=X,\rho(\beta)=Y,\rho(\gamma)=Z$, where $(X,Y,Z)$ is a $k$-MM triple. Since $XYZ=T$, we have $\rho_{XYZ}(\delta)=T^{-1}$. We set $(P,Q,R)=\Phi(X,Y,Z)$. Then, by the definition of $\Phi$, we have
\begin{align}\label{eq:basis-transformation}
    -P=\rho_{X,Y,Z}(\gamma^{-1}\beta^{-1}), -Q=\rho_{X,Y,Z}(\gamma^{-1}\alpha^{-1}),-R=\rho_{X,Y,Z}(\beta^{-1}\alpha^{-1}),
\end{align}

\begin{theorem}
 Let $(X,Y,Z)$ be a $k$-MM triple associated with $x_{12},y_{12},z_{12}$. We have \[\chi(\rho_{X,Y,Z})=((3+3k)x_{12}-k,(3+3k)y_{12}'-k,(3+3k)z_{12}-k,k,k,k,2),\]
 where $y'_{12}=\dfrac{x_{12}^2+kx_{12}z_{12}+z_{12}^2}{y_{12}}$.
\end{theorem}

\begin{proof}
Let $(P,Q,R):=\Phi(X,Y,Z)$. Note that $(P,Q,R)$ is a $k$-GC triple associated with $\mu(x_{12},y_{12},z_{12})=(x_{12},y'_{12},z_{12})$. Under this situation, we have
\begin{align*}
x=\mathrm{tr}(P),\  y=\mathrm{tr}(Q),\ z=\mathrm{tr}(R),\  a=k,\  b=k,\   c=k,\  d=2,
\end{align*}
as desired.
\end{proof}

We consider the intersection of $\chi(S_4^2)$ and hypersurface $a=k$, $b=k$, $c=k$, $d=2$. This can be identified with
\[\bar{H}=\{(x,y,z)\in \mathbb {C}^3\mid (x,y,z) \text{ is a solution to $\mathrm{GSME}(k)$}\},\]
by the projection $p\colon(x,y,z,a,b,c,d)\mapsto (x,y,z)$. Therefore, we have the following theorem:

\begin{theorem}
We fix $k\in \mathbb Z_{\geq 0}$ and $\ell\in \mathbb Z$.
Let \begin{align*}
    &M(k,\ell):=\{(X,Y,Z)\in SL(2,\mathbb Z)^3\mid \text{$(X,Y,Z)$ is a vertex in $\mathrm{MM}\mathbb {T}(k,\ell)$ or $\mathrm{MM}\mathbb{T}^\dag(k,\ell)$}\},\\
&\mathrm{Rep}(S^2_4)/\!/SL(2,\mathbb C)|_{M(k,\ell)}:=\{[\rho]\in \mathrm{Rep}(S^2_4)/\!/SL(2,\mathbb C)\mid \text{$\exists(X,Y,Z)\in M(k,\ell)$ s.t. $\rho=\rho_{X,Y,Z}$}\},\\
&\bar{H}_{\mathrm{IS}}:=\{(x,y,z)\in \bar{H}\mid \text{ $(x,y,z)$ is an induced solution to $\mathrm{GSME}(k)$}\}.
\end{align*}

Then the maps
\begin{align*}
&f\colon M(k,\ell) \to \mathrm{Rep}(S^2_4)/\!/SL(2,\mathbb C)|_{M(k,\ell)}, \quad f(X,Y,Z)= [\rho_{X,Y,Z}],\\
&\overline{\chi}|_{M(k,\ell)}\colon \mathrm{Rep}(S^2_4)/\!/SL(2,\mathbb C)|_{M(k,\ell)}\to \bar{H}_{\mathrm{IS}}, \quad \overline{\chi}|_{M(k,\ell)}([\rho])=p\circ\overline{\chi}(\rho)
\end{align*}
are bijections.
\end{theorem}

This theorem allows us to interpret both $k$-GM triples (or their induced solutions) and $k$-MM triples, as the same point in the character variety, with the former emerging when viewed as an algebraic variety $H$ and the latter when viewed as a quotient of the $SL(2, \mathbb{C})$ representation of $\pi_1(S_4^2)$.

The authors do not know the interpretation the another map $\Psi$, which provides a bijection between $k$-GC triples and $k$-MM triples.

\begin{question}
What is the meaning of the map $\Psi$ in the context of the representation of $\pi_1(S_4^2)$?    
\end{question}

Let 
\[\Gamma_2^\ast:=\{M\in PGL(2,\mathbb Z)\mid M \equiv I_2 \mod 2\},\]
where $I_2$ is the identity matrix. The dynamics of $\Gamma_2^\ast$ on $\chi(S_4^2)|_{a,b,c,d}$ is related to the Painlev\'e VI equation (for example, see \cite{CaLo}). It is interesting to look for relation between $k$-GM numbers and the Painlev\'e VI equation.

\begin{remark}
 Several papers which deal with $SL(2,\mathbb{C})$-representations of $\pi_1(S_4^2)$, the signs of $x,y,z,a,b,c,d$ are opposite to the setting in this paper. To adapt the description in this paper to this setting, it suffices to simply multiply the $k$-GC matrix by $-1$. This is the more natural setting for this subsection.
\end{remark}

\subsection{Sign of entries of $k$-MM matrix in $k$-MM tree}
 In this subsection, we discuss the sign of entries of $k$-MM matrices, in particular, the second components of vertices in $\mathrm{MM}\mathbb T(k,\ell)$. Clearly, for each $k$-MM matrix $Y$, the sign of $(1,2)$-entry of $Y$ is positive. We begin with the following lemma.

\begin{lemma}\label{lem:y21<0}
For $(X,Y,Z)\in \mathrm{MM}\mathbb T(k,\ell)$, we have $y_{21}<0$. 
\end{lemma}

\begin{proof}
First, we prove the case that $(X,Y,Z)$ is the root of $\mathrm{MM}\mathbb T(k,\ell)$. Now, $y_{21}$ is given by $-k\ell^2-2\ell^2+k\ell+2\ell -1$ and it is less than $0$ clearly.
Next, we prove the case that $(X,Y,Z)$ is not the root. First, we prove the case $k\neq 2$. We assume that $y_{21}\geq 0$. Since $y_{12}>0$, we have $y_{11}y_{22}=y_{12}y_{21}+1>0$. Therefore, $y_{11}$ and $y_{22}$ are nonzero and have the same sign.  When $k=0$, since $y_{11}+y_{22}=0$, it is contradiction and this finishes the proof for $k=0$ case. We assume that $k\neq 0$. If $y_{11}>0$ and $y_{22}>0$ hold, then it conflicts with $y_{11}+y_{22}=-k$, therefore we have $y_{11}<0$ and $y_{22}<0$. If $y_{21}=0$ holds, since we have $\det (Y)=1$, we have $y_{11}=y_{22}=-1$, and it is in contradiction to $y_{11}+y_{22}=-k$. Therefore, we have $y_{21}\geq 1$. Since $\dfrac{(-y_{11})+(-y_{22})}{2}=\dfrac{k}{2}$ holds, we have $y_{11}y_{22}\leq \dfrac{k^2}{4}$ by arithmetic-geometric mean. Therefore, we have $y_{12}y_{21}\leq \dfrac{k^2}{4}-1$. However, we have $y_{12}\geq 2k^2+6k+5$ because $(X,Y,Z)$ is not the root and the smallest $k$-GM number that is larger than  $k+2$ is $2k^2+6k+5$. It conflicts with $y_{21}\geq 1$. Therefore, we have $y_{21}<0$.

Next, we prove the case $k=2$. We can prove it in parallel with the case $k\neq 2$ other than the step of proving $y_{21}\neq 0$. Therefore, it is enough to show that $y_{21}\neq 0$. We assume that $y_{21}=0$, then we have $Y=\begin{bmatrix}
  -1 &y_{12}\\ 0&-1  
\end{bmatrix}$. Since $XYZ=\begin{bmatrix}
    -1&0\\9&-1
\end{bmatrix}$ holds by definition of the $2$-MM triple, we have the following equality:
\begin{align}
    -x_{11}z_{11}+x_{11}y_{12}z_{21}-x_{12}z_{21}&=-1,\label{eq:XYZ-1}\\
    -x_{11}z_{12}+x_{11}y_{12}z_{22}-x_{12}z_{22}&=0,\label{eq:XYZ-2}\\
    -x_{21}z_{11}+x_{21}y_{12}z_{21}-x_{22}z_{21}&=9,\label{eq:XYZ-3}\\
    -x_{21}z_{12}+x_{21}y_{12}z_{22}-x_{22}z_{22}&=-1.\label{eq:XYZ-4}
\end{align}
Then, we have 
\begin{align*}
    x_{11}z_{11}=x_{11}z_{12}\dfrac{z_{11}}{z_{12}}\overset{\text{\eqref{eq:XYZ-2}}}{=}x_{11}y_{12}\dfrac{z_{11}z_{22}}{z_{12}}-x_{12}\dfrac{z_{11}z_{22}}{z_{12}}\overset{\text{$\det(Z)=1$}}{=}x_{11}y_{12}z_{21}+\dfrac{x_{11}y_{12}}{z_{12}}-x_{12}z_{21}-\dfrac{x_{12}}{z_{12}}.
\end{align*}
Therefore, by substituting the above equality with \eqref{eq:XYZ-1}, we have 
\[x_{11}y_{12}-x_{12}=z_{12}.\]
Moreover, by substituting the above equality with \eqref{eq:XYZ-2}, we have \[-x_{11}z_{12}+z_{12}z_{22}=0,\] 
and it leads to $x_{11}=z_{22}$.
Moreover, we have 
\begin{align*}
x_{12}z_{11}&=x_{12}z_{12}\dfrac{z_{11}}{z_{12}}\overset{\text{\eqref{eq:XYZ-4}}}{=}x_{21}y_{12}\dfrac{z_{11}z_{22}}{z_{12}}-x_{22}\dfrac{z_{11}z_{12}}{z_{22}}+\dfrac{z_{11}}{z_{12}}\\
&\overset{\det(Z)=1}{=}x_{21}y_{12}z_{21}+\dfrac{x_{21}y_{12}}{z_{12}}-x_{22}z_{21}-\dfrac{x_{22}}{z_{12}}+\dfrac{z_{11}}{z_{12}}.
\end{align*}
Therefore, by substituting the above equation with \eqref{eq:XYZ-3}, we have 
\[x_{21}y_{12}-x_{22}=-9z_{12}-z_{11}.\]
Moreover, by substituting the above with \eqref{eq:XYZ-4}, we have \[-x_{21}z_{12}-9z_{12}z_{22}-z_{11}z_{22}=-1,\] 
and therefore we have $x_{21}=\dfrac{-9z_{12}z_{22}-z_{11}z_{22}+1}{z_{12}}=-9z_{22}-z_{21}$. By the above argument, we have
\[X=\begin{bmatrix}
    z_{22}& x_{12}\\-9z_{22}-z_{21} &-z_{22}-2
\end{bmatrix}.\]
We note that $-9z_{22}-z_{21}\neq 0$. Indeed, if $-9z_{22}-z_{21}= 0$ hold, we have
\[X=\begin{bmatrix}
    -1&x_{12}\\0&-1,
\end{bmatrix}, \quad Z=\begin{bmatrix}
    -1 &z_{12}\\9&-1
\end{bmatrix},\]
and we have $x_{12}+y_{12}+z_{12}=0$ by $XYZ=T$, and it is a contradiction.
Since $\det(X)=\det(Z)=1$, we have $x_{12}=\dfrac{1-z_{22}(-z_{22}-2)}{9x_{22}+z_{21}}=-\dfrac{z_{12}z_{21}}{9x_{22}+z_{21}}$ (note that $z_{11}=-z_{22}-2$). Therefore, $\left(-\dfrac{z_{12}z_{21}}{9z_{22}+z_{21}},y_{12},z_{12}\right)$ is a $2$-GM triple. By Proposition \ref{relatively-prime}, $-\dfrac{z_{12}z_{21}}{9z_{22}+z_{21}}$ and $z_{12}$ are relatively prime, and hence there exists $\alpha\in \mathbb Z\setminus \{0\}$ such that $-(9z_{22}+z_{21})=\alpha z_{12}$. Therefore, we have $x_{12}=\dfrac{z_{21}}{\alpha}$, and 
\[X=\begin{bmatrix}
    z_{22}&\dfrac{z_{21}}{\alpha}\\-9z_{22}-z_{21}&-z_{22}-2
\end{bmatrix}, \quad Z=\begin{bmatrix}
    -z_{22}-2&\dfrac{-9z_{22}-z_{21}}{\alpha}\\z_{21}&z_{22}
\end{bmatrix}.\]
Since $\det (X)=1$, we have
\[-1=z_{22}^2+2z_{22}-\dfrac{9z_{21}z_{22}}{\alpha}-\dfrac{z_{21}^2}{\alpha},\] and since $Y=X^{-1}\begin{bmatrix}
    -1&0\\9&-1
\end{bmatrix}Z^{-1}$, we get
\begin{align*}
    y_{12}&=\left(z_{22}+2-\dfrac{9z_{21}}{\alpha}\right)\dfrac{9z_{22}+z_{21}}{\alpha}+\dfrac{z_{21}(-z_{22}-2)}{\alpha}\\
    &=\dfrac{9}{\alpha}\left(z_{22}^2+2z_{22}-\dfrac{9z_{21}z_{22}}{\alpha}-\dfrac{z_{21}^2}{\alpha}\right)=-\dfrac{9}{\alpha}.
  \end{align*} 
However, it conflicts with $y_{12}\geq 2k^2+6k+5=25$.  
\end{proof}

\begin{remark}
There is a $2$-MM triple such that $y_{21}=0$. Indeed, $(X_{1;2},Y_{1;2},Z_{1;2})$ for $k=2$ satisfies $y_{21}=0$. Clearly, $(X_{1;2},Y_{1;2},Z_{1;2})$ is not in $\mathrm{MM}\mathbb T(k,\ell)$. 
\end{remark}
\begin{corollary}\label{cor:y_11y_22<0}
For $(X,Y,Z)\in \mathrm{MM}\mathbb T(k,\ell)$, if $(X,Y,Z)$ is not the root, we have $y_{11}y_{22}<0$. 
\end{corollary}
\begin{proof}
Since $\det (Y)=1$, we have $y_{11}y_{22}-y_{12}y_{21}=1$. Since $y_{12}>1$ and $y_{21}<0$ hold by Lemma \ref{lem:y21<0}, we have $y_{11}y_{22}=1+y_{12}y_{21}<0$.    
\end{proof}
\begin{proposition}\label{prop:sign-Y-inductive}
   Let $(X,Y,Z)\in \mathrm{MM}\mathbb T(k,\ell)$. We set
   \[YZY^{-1}=\begin{bmatrix}
       y'_{11} & y'_{12}\\ y'_{21} & y'_{22}
   \end{bmatrix},\quad  Y^{-1}XY=\begin{bmatrix}
       y''_{11} & y''_{12}\\ y''_{21} & y''_{22}
   \end{bmatrix}.\]
\begin{itemize}\setlength{\leftskip}{-15pt}
    \item [(1)] If $y_{11}<0$ and $y_{22}>0$, then we have $y'_{11}<0$, $y'_{22}>0$, $y''_{11}<0$ and $y''_{22}>0$,
    \item [(2)] if $y_{11}>0$ and $y_{22}<0$, then we have $y'_{11}>0$, $y'_{22}<0$, $y''_{11}>0$ and $y''_{22}<0$.
\end{itemize}
\end{proposition}
\begin{proof}
 We prove only (1). First, we prove $y'_{11}<0$ and $y'_{22}>0$. By Corollary \ref{cor:y_11y_22<0}, it is enough to show $y'_{11}<0$. Since $XYZ=T$, we have $YZY^{-1}=X^{-1}TY^{-1}$. Then, we have
 \begin{align}
     y'_{11}&=-x_{22}y_{22}-(3k+3)x_{12}y_{22}-x_{12}y_{21},\label{eq:y'11}\\
     y'_{12}&=x_{22}y_{12}+(3k+3)x_{12}y_{12}+x_{12}y_{11}
     =\dfrac{x_{12}^2+kx_{12}y_{12}+y_{12}^2}{z_{12}},\label{eq:y'12}
 \end{align}
 where the second equality of \eqref{eq:y'12} comes from the Vieta jumping of $\textrm{GME}(k)$. Then, we have
 \begin{align*}
     x_{22}y_{22}&=x_{22}y_{12}\dfrac{y_{22}}{y_{12}}\\
     &\overset{\text{\eqref{eq:y'12}}}{=}-(3k+3)x_{12}y_{22}-x_{12}\dfrac{y_{11}y_{22}}{y_{12}}+\dfrac{y_{22}(x_{12}^2+kx_{12}y_{12}+y_{12}^2)}{y_{12}z_{12}}\\
    &\overset{\det(Y)=1}{=}-(3k+3)x_{12}y_{22}-x_{12}y_{21}-\dfrac{x_{12}}{y_{12}}+\dfrac{y_{22}(x_{12}^2+kx_{12}y_{12}+y_{12}^2)}{y_{12}z_{12}}\\
    &>-(3k+3)x_{12}y_{22}-x_{12}y_{21}+\dfrac{y_{22}x_{12}^2+(y_{22}k-1)x_{12}y_{12}+y_{22}y_{12}^2}{y_{12}^2},
 \end{align*}
 where the last inequality follows from $y_{12}>z_{12}$. Therefore, by the above inequality and \eqref{eq:y'11}, we have
 \begin{align*}
     y_{11}'<-\dfrac{y_{22}x_{12}^2+(y_{22}k-1)x_{12}y_{12}+y_{22}y_{12}^2}{y_{12}^2}<-((2+k)y_{22}-1)\dfrac{x_{12}^2}{y_{12}^2}<0,
 \end{align*}
 as desired (we note that $x_{12}<y_{12}$). Second, we prove $y''_{11}<0$ and $y''_{22}>0$. By Corollary \ref{cor:y_11y_22<0}, it is enough to show $y''_{22}>0$. Since $XYZ=T$, we have $Y^{-1}XY=Y^{-1}TZ^{-1}$. Then, we have
 \begin{align}
     y''_{22}&=-y_{21}z_{12}-(3k+3)y_{11}z_{12}-y_{11}z_{11},\label{eq:y''22}\\
     y''_{12}&=y_{22}z_{12}+(3k+3)y_{12}z_{12}+y_{12}z_{11}
     =\dfrac{y_{12}^2+ky_{12}z_{12}+z_{12}^2}{x_{12}}.\label{eq:y''12}
 \end{align}
 Then, we have
 \begin{align*}
     y_{11}z_{11}&=y_{12}z_{11}\dfrac{y_{11}}{y_{12}}\\
     &\overset{\text{\eqref{eq:y''12}}}{=}-z_{12}\dfrac{y_{11}y_{22}}{y_{12}}-(3k+3)y_{11}z_{12}+\dfrac{y_{11}(y_{12}^2+ky_{12}z_{12}+z_{12}^2)}{x_{12}y_{12}}\\
    &\overset{\det(Y)=1}{=}-y_{21}z_{12}-\dfrac{z_{12}}{y_{12}}-(3k+3)y_{11}z_{12}+\dfrac{y_{11}(y_{12}^2+ky_{12}z_{12}+z_{12}^2)}{x_{12}y_{12}}\\
    &<-y_{21}z_{12}-(3k+3)y_{11}z_{12}+\dfrac{y_{11}y_{12}^2+(y_{11}k-1)y_{12}z_{12}+y_{11}z_{12}^2}{y^2_{12}},
 \end{align*}
 where the last inequality follows from $y_{12}>x_{12}$. Therefore, by the above ineqation and \eqref{eq:y''22}, we have
 \begin{align*}
     y_{22}''>-\dfrac{y_{11}y_{12}^2+(y_{11}k-1)y_{12}z_{12}+y_{11}z_{12}^2}{y_{12}^2}>0,
 \end{align*}
 as desired.
\end{proof}

\begin{corollary}\label{cor:sign-Markov-monodromy}
 For $(X,Y,Z)\in \mathrm{MM}\mathbb T(k,\ell)$, we have
 \[\begin{cases}y_{11}>0,y_{12}>0, y_{21}<0,y_{22}<0 \quad &\text {if $\ell\geq 1$,}\\ y_{11}<0,y_{12}>0, y_{21}<0,y_{22}>0\quad &\text {if $\ell\leq 0$}.\end{cases}\]  
\end{corollary}
\begin{proof}
 It is directly checked for the root of $\mathrm{MM}\mathbb T(k,\ell)$. By Proposition \ref{prop:sign-Y-inductive} and Lemma \ref{lem:y21<0}, the statement follows inductively.   
\end{proof}

\begin{remark}\label{rem:aboutXZ}
 The signs of entries of $X$ and $Z$ except for $X_{\ell}(=X_{1;\ell})$ and $Z_\ell(=Y_{1;\ell})$ are the same as those of entries of $Y$ because each matrix coincides with the second component of some $k$-MM triple in $\mathrm{MM}\mathbb T(k,\ell)$. For $X_{\ell}$, we have the following instead of Corollary \ref{cor:sign-Markov-monodromy}: if $k> 2$, we have
 \[\begin{cases}x_{11}>0, x_{21}<0,x_{22}<0 \quad &\text {if $\ell\geq 1$,}\\ 
 x_{11}=0, x_{21}=-1,x_{22}<0 \quad &\text {if $\ell=0$,}\\ 
 x_{11}< 0, x_{21}>0,x_{22}< 0\quad &\text {if $-1\leq \ell\leq -k+1$,}\\
  x_{11}<0, x_{21}=-1,x_{22}=0 \quad &\text {if $\ell=-k$,}\\ 
 x_{11}<0, x_{21}<0,x_{22}>0\quad &\text {if $\ell\leq -k-1$},\end{cases}\]
 and if $k= 0$, then we have
\[\begin{cases}x_{11}>0,x_{21}<0,x_{22}<0 \quad &\text {if $\ell\geq 1$,}\\ 
 x_{11}= 0,x_{21}=-1,x_{22}= 0\quad &\text {if $\ell=0$,}\\
 x_{11}<0,x_{21}<0,x_{22}>0\quad &\text {if $\ell\leq-1$},\end{cases}\]
 and if $k= 1$, then we have 
 \[\begin{cases}x_{11}>0, x_{21}<0,x_{22}<0 \quad &\text {if $\ell\geq 1$,}\\ 
 x_{11}=0, x_{21}=-1,x_{22}=-1\quad &\text {if $\ell=0$,}\\ 
  x_{11}=-1, x_{21}=-1,x_{22}=0 \quad &\text {if $\ell=-1$,}\\ 
 x_{11}<0, x_{21}<0,x_{22}>0\quad &\text {if $\ell\leq -2$},\end{cases}\]
 and if $k= 2$, then we have 
 \[\begin{cases}x_{11}>0, x_{21}<0,x_{22}<0 \quad &\text {if $\ell\geq 1$,}\\ 
 x_{11}=0, x_{21}=-1,x_{22}=-2 \quad &\text {if $\ell=0$,}\\ 
  x_{11}=-1, x_{21}=0,x_{22}=-1\quad &\text {if $\ell=-1$,}\\
  x_{11}=-2, x_{21}=-1,x_{22}=0 \quad &\text {if $\ell=-2$,}\\ 
 x_{11}<0, x_{21}<0,x_{22}>0\quad &\text {if $\ell\leq -3$}.\end{cases}\]
We note that $Z_{\ell}=X_{-k+\ell+1}$ (thus we omit the signs of $Z_{\ell}$).
\end{remark}

\subsection{M\"obius transformation of $k$-MM matrix to $0$ or $\infty$}
We regard a $k$-MM matrix as a M\"obius transformation on $\mathbb RP^1=\mathbb R \cup \{\infty\}$, and we will study the configuration of the image of $\infty$ and $0$. These will play an important role in the next section.

We consider the tree $\mathrm{LMM}\mathbb T(k,\ell)$, which is the full subtree of $\mathrm{MM}\mathbb T(k,\ell)$ whose root is $\sigma_1 ^{-1}(X,Y,Z)$, that is,
\begin{align*}
  \tilde{X}_\ell&=\begin{bmatrix}
    \ell &1\\-\ell^2-k\ell-1&-k-\ell
\end{bmatrix},\\
    \tilde{Y}_\ell&=\begin{bmatrix}
      2k^2\ell + 6k\ell - 2k + 5\ell - 2  &2k^2 + 6k + 5\\-2k^2\ell^2 - 6k\ell^2 + 3k\ell - 5\ell^2 + 4\ell - 1 &-2k^2\ell - 6k\ell + k - 5\ell + 2
\end{bmatrix},\\
  \tilde{Z}_\ell&=\begin{bmatrix}
    k\ell-k+2\ell-1 &k+2\\-k\ell^2-2\ell^2+k\ell+2\ell-1&-k\ell-2\ell+1
\end{bmatrix}.
\end{align*}
Similarly, we consider the tree $\mathrm{LM}\mathbb T(k)$, which is the full subtree of $\mathrm{M}\mathbb T(k)$ whose root is the left child of the root of $\mathrm{M}\mathbb T(k)$, that is, $(1,2k^2+6k+5,k+2)$.
The restriction of the correspondence $(X,Y,Z)\mapsto (x_{12},y_{12},z_{12})$ to $\mathrm{LMM}\mathbb T(k,\ell)\to \mathrm{LM}\mathbb T(k)$ gives a bijection from vertices in $\mathrm{LMM}\mathbb T(k,\ell)$ and all $k$-GM triples but $(1,1,1)$ and $(1,k+2,1)$ up to order.

\begin{lemma}\label{lem:Y-inf-Z-inf}
We fix $\ell\in \mathbb Z_{\leq 0}$. For a vertex $(X,Y,Z)\in \mathrm{LMM}\mathbb T(k,\ell)$, $Y^{-1}(\infty)$, and $Z(\infty)$ are not $\infty$ and $Z(\infty)<Y^{-1}(\infty)$ holds with respect to the standard order of $\mathbb R$.   
\end{lemma}

\begin{proof}
 Since $y_{21}\neq 0$ and $z_{21}\neq 0$ by Lemma \ref{lem:y21<0} and Remark \ref{rem:aboutXZ}, $Y^{-1}(\infty)$ and $Z(\infty)$ are not $\infty$. 
 We note that $Z(\infty)=\dfrac{z_{11}}{z_{21}}$, and  $Y^{-1}(\infty)=-\dfrac{y_{22}}{y_{21}}$. Since
 \[Y^{-1}(\infty)-Z(\infty)=\dfrac{-y_{22}z_{21}-y_{21}z_{11}}{y_{21}z_{21}},\]
 and $y_{21}z_{21}>0$ by Lemma \ref{lem:y21<0} and Remark \ref{rem:aboutXZ}, it suffices to show that $-y_{22}z_{21}-y_{21}z_{11}>0$. Since $y_{22}z_{21}+y_{21}z_{11}$ is the $(2,1)$-entry of $YZ$, we will prove that it is negative. To prove that the $(2,1)$-entry of $YZ$ is negative for any $(X,Y,Z)$, we use the induction on distance from the root in $\mathrm{LMM}\mathbb T(k,\ell)$. First, we will check this for the root vertex. For $\ell=0$, we have
 \[\tilde{Y}_{0}\tilde{Z}_0=\begin{bmatrix}
     -2k-3&1\\-1&0
 \end{bmatrix},\]
and the $(2,1)$-entry of $YZ$ is negative. Since 
$\tilde{Y}_{\ell}\tilde{Z}_\ell=L^{-1}\tilde{Y}_{0}\tilde{Z}_0 L$ by Proposition \ref{Markov-monodromy-matrix-ltol'}, where $L=\begin{bmatrix}
    1&0\\\ell&1 
\end{bmatrix}$,  we have
\[\tilde{Y}_{\ell}\tilde{Z}_\ell=\begin{bmatrix}
     -2k-3+\ell&1\\2\ell(k+2)-1&-\ell
 \end{bmatrix}\]
for any $\ell\in \mathbb Z_{\leq 0}$. Therefore, we have $y_{22}z_{21}+y_{21}z_{11}=2\ell(k+2)-1<0$, as desired. 
We assume that $(X,Y,Z)\in \mathrm{LMM}\mathbb T(k,\ell)$ satisfies that the $(2,1)$-entry of $YZ$ is negative. Then, we will show that so does $(X,YZY^{-1},Y)$ and $(Y,Y^{-1}XY,Z)$. The former is clear because $YZY^{-1}Y=YZ$. We will prove the latter. Since $XYZ=T$, we have
\[Y^{-1}XYZ=Y^{-1}T=\begin{bmatrix}
    y_{22} &-y_{12} \\-y_{21} &y_{11}
\end{bmatrix}\begin{bmatrix}
    -1 &0 \\3k+3 &-1
\end{bmatrix}=\begin{bmatrix}
    \ast &\ast\\ y_{21}+y_{11}(3k+3) &\ast
\end{bmatrix}.\]
Since $y_{11}<0$ and $y_{21}<0$ by Corollary \ref{cor:sign-Markov-monodromy}, the $(2,1)$-entry of $Y^{-1}XYZ$ is negative. 
\end{proof}

The above lemma does not hold for $\ell\geq 1$. In the case $\ell\geq 1$, the following lemma holds instead of Lemma \ref{lem:Y-inf-Z-inf}: 

\begin{lemma}\label{lem:Y-0-Z-0}
We fix $\ell\in \mathbb Z_{\geq 1}$. For a vertex $(X,Y,Z)\in \mathrm{LMM}\mathbb T(k,\ell)$, $Y^{-1}(0)$ and $Z(0)$ are not $\infty$, and $Z(0)<Y^{-1}(0)$ holds.   
\end{lemma}
We will omit the proof because it is almost the same as that of Lemma \ref{lem:Y-inf-Z-inf}.
%\begin{proof}
%We note that $Z(0)=\dfrac{z_{12}}{z_{22}}$, and  $Y^{-1}(0)=-\dfrac{y_{12}}{y_{11}}$. Since
% \[Y^{-1}(0)-Z(0)=\dfrac{-y_{12}z_{22}-y_{11}z_{12}}{y_{11}z_{22}},\]
% and $y_{11}z_{22}<0$ by Lemma \ref{cor:sign-Markov-monodromy} and Remark \ref{rem:aboutXZ}, it suffices to show that $-y_{12}z_{22}-y_{11}z_{12}<0$. Since $y_{12}z_{22}+y_{11}z_{12}$ is the $(1,2)$-entry of $YZ$, we will prove that it is positive. To prove that the $(1,2)$-entry of $YZ$ is positive for any $(X,Y,Z)$, we use the induction on distance from the root in $\mathrm{LLB}\mathbb T(k,\ell)$. First, we will check that for the root vertex. From the proof of Lemma \ref{lem:Y-inf-Z-inf}, we have
%\[\tilde{Y}_{\ell}\tilde{Z}_\ell=\begin{bmatrix}
%     -2k-3+\ell&1\\2\ell(k+2)-1&-\ell
% \end{bmatrix}.\]
%Therefore, we have $y_{12}z_{22}+y_{11}z_{12}=1>0$, as desired. We assume that $(X,Y,Z)\in \mathrm{LLB}\mathbb T(k,\ell)$ satisfies that. Then, we will show that so does $(X,YZY^{-1},Y)$ and $(Y,Y^{-1}XY,Z)$. The former is clear because $YZY^{-1}Y=YZ$. We will prove the latter. Since $XYZ=T$, we have
%\[Y^{-1}XYZ=Y^{-1}T=\begin{bmatrix}
%    y_{22} &-y_{12} \\-y_{21} &y_{11}
%\end{bmatrix}\begin{bmatrix}
%    -1 &0 \\3k+3 &-1
%\end{bmatrix}=\begin{bmatrix}
%    \ast &y_{12}\\ \ast &\ast
%\end{bmatrix}.\]
%Since $y_{12}>0$, the $(1,2)$-entry of $Y^{-1}XYZ$ is positive.
%\end{proof}

The next two Lemmas are related to M\"obius transformations by $X$ and $Y$. The proofs are similar to those of Lemma \ref{lem:Y-inf-Z-inf}, so they will be omitted.

\begin{lemma}\label{lem:Y-0-X-0}
We fix $\ell\in \mathbb Z_{\leq 0}$. For a vertex $(X,Y,Z)\in \mathrm{LMM}\mathbb T(k,\ell)$, if $X\neq\tilde{X}_0$, then $X^{-1}(0)$ and $Y(0)$ are not $\infty$, and $Y(0)<X^{-1}(0)$ holds. 
\end{lemma}

\begin{lemma}\label{lem:X-inf-Y-inf}
We fix $\ell\in \mathbb Z_{\geq 1}$. For a vertex $(X,Y,Z)\in \mathrm{LMM}\mathbb T(k,\ell)$, $X^{-1}(\infty)$ and $Y(\infty)$ are not $\infty$, and $Y(\infty)<X^{-1}(\infty)$ holds.   
\end{lemma}

%\begin{proof}
%We note that $Y(\infty)=\dfrac{y_{11}}{y_{21}}$, and  $X^{-1}(\infty)=-\dfrac{x_{22}}{x_{21}}$. Since
%\[X^{-1}(\infty)-Y(\infty)=\dfrac{-x_{22}y_{21}-x_{21}y_{11}}{x_{21}y_{21}},\]
%and $x_{21}y_{21}>0$ by Lemma \ref{lem:y21<0} and Remark \ref{rem:aboutXZ}, it suffices to show that $-x_{22}y_{21}-x_{21}y_{11}>0$. Since $x_{22}y_{21}+x_{21}y_{11}$ is the $(2,1)$-entry of $XY$, we will prove that it is negative. To prove that the $(2,1)$-entry of $XY$ is negative for any $(X,Y,Z)$, we use the induction on distance from the root in $\mathrm{LLB}\mathbb T(k,\ell)$. First, we will check that for the root vertex. From the proof of Lemma \ref{lem:Y-0-X-0}, we have
%\[\tilde{X}_{\ell}\tilde{Y}_\ell=\begin{bmatrix}
 %    (k+2)\ell-1&k+2\\-(2+k)\ell^2-(3k^2+8k+4)\ell+3k+2&-(k+2)\ell-3k^2-8k-5
%\end{bmatrix}.\]
%Therefore, we have 
%\[y_{22}z_{21}+y_{21}z_{11}=-(2+k)\ell^2-%(3k^2+8k+4)\ell+3k+2\leq -3k^2-6k-4<0,\]
%as desired. 
%We assume that $(X,Y,Z)\in \mathrm{LLB}\mathbb T(k,\ell)$ satisfies that. Then, we will show that so does $(X,YZY^{-1},Y)$ and $(Y,Y^{-1}XY,Z)$. The former is clear because $YZY^{-1}Y=YZ$. We will prove the latter. Since $XYZ=T$, we have
%\[Y^{-1}XYZ=Y^{-1}T=\begin{bmatrix}
%    y_{22} &-y_{12} \\-y_{21} &y_{11}
%\end{bmatrix}\begin{bmatrix}
%    -1 &0 \\3k+3 &-1
%\end{bmatrix}=\begin{bmatrix}
%    \ast &\ast\\ y_{21}+y_{11}(3k+3) &\ast
%\end{bmatrix}.\]
%Since $y_{11}<0$ by Corollary \ref{cor:sign-Markov-monodromy}, the $(2,1)$-entry of $Y^{-1}XYZ$ is negative. 
%\end{proof}

\section{Parabolic case $k=2$ and classical Markov numbers}

In this section, we consider the case $k=2$. Let $X$ be a $2$-MM matrix. We regard $X$ as a M\"obius transformation, and we act it on $\mathbb RP^1$. Since $(\mathrm{tr} X)^2=4$, $X$ is of parabolic type. Therefore, $X$ has only one fixed point $p_{X}$ in $\mathbb RP^1$.

\subsection{Fixed points and classical Markov numbers}
In this subsection, we calculate fixed points of $2$-MM matrices and introduce relation between fixed points and classical Markov numbers ($0$-GM numbers).
\begin{proposition}\label{prop:fixed-point}
 Let $X$ be a $2$-MM matrix included in $\mathrm{MM}\mathbb T(2,\ell)$ with $x_{21}\neq 0$ and $p_X$ the fixed point of $X$ in $\mathbb R$. Then, we have
 \[p_X=\begin{cases}
     \sqrt{-\dfrac{x_{12}}{x_{21}}} \quad &\text{if $x_{11}<x_{22}$},\vspace{2mm}\\
     -\sqrt{-\dfrac{x_{12}}{x_{21}}} \quad &\text{if $x_{11}>x_{22}$}.
 \end{cases}\]
\end{proposition}

\begin{proof}
 The fixed point $p_X$ satisfies $-x_{21}p_X^2+(x_{11}-x_{22})p_X+x_{12}=0$ (we note that $x_{21}<0$). Since $X$ is of parabolic type, there is the unique solution to $-x_{21}z^2+(x_{11}-x_{22})z+x_{12}=0$. Therefore, we get $(\sqrt{-x_{21}}p_X-\sqrt{x_{12}})^2=0$ or $(\sqrt{-x_{21}}p_X+\sqrt{x_{12}})^2=0$. If $x_{11}<x_{22}$ holds, then we have the former, and otherwise, we have the latter.  
\end{proof}

\begin{remark}
    When $x_{21}=0$, the fixed point of $X$ is $\infty$ in $\mathbb R P^1$. There is the unique $2$-MM matrix in $\cup_\ell\mathrm{MM}\mathbb T(2,\ell)$ such that $x_{21}=0$. Indeed, since Proposition \ref{lem:y21<0} and the generation rule of $\mathrm{MM}\mathbb T(2,\ell)$, only ${X}_{1;-1}(=Y_{1;k}=Z_{1;2k+1})$ meets the condition (see also Remark \ref{rem:aboutXZ}).
\end{remark}

Comparing the coefficient of $p_X$ of $-x_{21}p_X^2+(x_{11}-x_{22})p_X+x_{12}$ and $(\sqrt{-x_{21}}p_X\pm\sqrt{x_{12}})^2$, we see that $\sqrt{-x_{12}x_{21}}$ is an integer (note that since $x_{11}+x_{12}=-2$, $x_{11}-x_{12}$ is also even). Moreover, the following fact is known:

\begin{proposition}[\cite{gyomatsu}*{Theorem 11}]\label{squre-markov}
If $(a,b,c)$ is a $0$-GM triple, then $(a^2,b^2,c^2)$ is a $2$-GM triple. Conversely, if $(A,B,C)$ is a $2$-GM triple, then $(\sqrt{A},\sqrt{B},\sqrt{C})$ is a $0$-GM triple. Moreover, the correspondence $(a,b,c)\mapsto(a^2,b^2,c^2)$ induces the canonical graph isomorphism from $\mathrm{M}\mathbb{T}(0)$ to $\mathrm{M}\mathbb{T}(2)$. 
\end{proposition}

By Proposition \ref{squre-markov}, $\sqrt{x_{12}}$ is an integer (in particular, a classical Markov number) and thus $\sqrt{-x_{21}}$ is also an integer. Let us look further at the relation between these two numbers. Before describing proposition, we extend the definition of relatively prime. 

\begin{definition}
Non-negative integers $a$ and $b$ with $(a,b)\neq (0,0)$ are said to be \emph{relatively prime} if there are no $a'$ and $b'\in \mathbb Z_{\geq 0}$ such that $ca'=a$ and $cb'=b$ for any $c\in \mathbb Z_{>1}$.
\end{definition}

If $a$ and $b$ are both positive integers, then the above definition is the same as the usual sense. We consider the case $a=0$. If $b=1$, then $a$ and $b$ are relatively prime, and otherwise, $a$ and $b$ are not relatively prime.

\begin{proposition}\label{prop:relative-prime-y12-y21}
For $(X,Y,Z)\in \mathrm{MM}\mathbb T(2,\ell)$, 
\begin{itemize}\setlength{\leftskip}{-15pt}
    \item [(1)] $\sqrt{x_{12}}$ and $\sqrt{-x_{21}}$ are relative prime,
    \item [(2)] $\sqrt{y_{12}}$ and $\sqrt{-y_{21}}$ are relative prime,
    \item [(3)] $\sqrt{z_{12}}$ and $\sqrt{-z_{21}}$ are relative prime.
\end{itemize}
\end{proposition}

\begin{proof}
We only prove (2). We assume that $\sqrt{y_{12}}$ and $\sqrt{-y_{21}}$ have a non-trivial common divisor $c\in \mathbb Z_{>1}$, that is, there exists $a',b'\in \mathbb Z_{\geq 0}$ such that $da'=\sqrt{y_{12}}$ and $cb'=\sqrt{-y_{21}}$. Since $y_{11}-y_{22}=\pm 2\sqrt{-y_{12}y_{21}}=\pm2c^2a'b'$, $c$ is a divisor of $|y_{11}-y_{22}|$.  Now, we set
 \[Z^{-1}YZ=\begin{bmatrix}
     \alpha_{11}&\alpha_{12}\\ \alpha_{21}&\alpha_{22}
 \end{bmatrix},\quad XYX^{-1}=\begin{bmatrix}
     \beta_{11}&\beta_{12}\\ \beta_{21}&\beta_{22}
 \end{bmatrix}.\]

Then, we have
\begin{align*}
\alpha_{12}&=z_{12}z_{22}y_{11}-z_{12}^2y_{21}+z_{22}^2y_{12}-z_{12}z_{22}y_{22}=(z_{12}z_{22})(y_{11}-y_{22})-z_{12}^2y_{21}+z_{22}^2y_{12}, \\  
\alpha_{21}&=z_{11}z_{21}y_{11}-z_{11}^2y_{21}+z_{21}^2y_{12}-z_{11}z_{21}y_{22}=(z_{11}z_{21})(y_{11}-y_{22})-z_{11}^2y_{21}+z_{21}^2y_{12}, \\
\beta_{12}&=-x_{11}x_{12}y_{11}-x_{12}^2y_{21}+x_{11}^2y_{12}+x_{11}x_{12}y_{22}=-(x_{11}x_{12})(y_{11}-y_{22})-x_{12}^2y_{21}+x_{11}^2y_{12},\\
\beta_{21}&=-x_{21}x_{22}y_{11}-x_{22}^2y_{21}+x_{21}^2y_{12}+x_{21}x_{22}y_{22}=-(x_{21}x_{22})(y_{11}-y_{22})-x_{22}^2y_{21}+x_{21}^2y_{12}.
\end{align*}
 Therefore, $c$ is a common divisor of $\alpha_{12}$ and $-\alpha_{21}$ (resp. $\beta_{12}$ and $-\beta_{21}$). We consider going upstream from $(X,Y,Z)$ to the root in $\mathrm{MM}\mathbb T(2,\ell)$. We apply $\sigma_1\colon (X',Y',Z')\mapsto (X',Z',Z'^{-1}Y'Z')$ or $\sigma_2^{-1}\colon (X',Y',Z')\mapsto (X'Y'X'^{-1},X',Z')$ to $(X,Y,Z)$ repeatedly. If there exists a matrix in $(X',Y',Z')$ such that the $(1,2)$-entry and (the absolute value of) the $(2,1)$-entry of it have a common divisor $c$, then so do $\sigma_1 (X',Y',Z')$ and $\sigma_2^{-1}(X',Y',Z')$ by the previous argument. Therefore, by assumption, there exists a matrix in the root $({X}_\ell,{Y}_\ell,{Z}_\ell)$ such that the $(1,2)$-entry and (the absolute value of) the $(2,1)$-entry of it have a common divisor $c$. However, since
 \[X_\ell=\begin{bmatrix}\ell& 1^2\\ -(\ell+1)^2 &-\ell-2
     \end{bmatrix}, \ Y_\ell=\begin{bmatrix}
     4\ell-3& 2^2\\ -(2\ell-1)^2 &-4\ell+1
 \end{bmatrix},Z_\ell=\begin{bmatrix}
     \ell-3& 1^2\\ -(\ell-2)^2 &-\ell+1
 \end{bmatrix},\]
it is a contradiction.
\end{proof}

We will define the irreducible fraction.

\begin{definition}\label{irreducible-fraction}
Let $q\in\QQ_{\geq 0}\cup\{\infty\}$ and $n$ and $d\in\ZZ_{\geq 0}$. The symbol $\dfrac{n}{d}$ is called the \emph{reduced expression} of $q$ if $n$ and $d$ are relatively prime and $q = \dfrac{n}{d}$, where $\dfrac{n}{d}$ is regarded as $\infty$ when $d = 0$ and $n > 0$.
Moreover, a fraction $\dfrac{n}{d}$ is said to be \emph{irreducible} if there exists $q\in\QQ_{\geq 0}\cup\{\infty\}$ such that $\dfrac{n}{d}$ is the reduced expression of $q$.
\end{definition}

From the above argument, we can see the following relation between $\mathrm{MM}\mathbb T(2,\ell)$ and $\mathrm{M}\mathbb T(0)$.

\begin{corollary}\label{parabolic-markov}
For $(X,Y,Z) \in \mathrm{MM}\mathbb T(2,\ell)$, we denote by $\left(\dfrac{p}{p'},\dfrac{q}{q'},\dfrac{r}{r'}\right)$ fixed points of $(X,Y,Z)$. If $\left|\dfrac{p}{p'}\right|,\left|\dfrac{q}{q'}\right|,\left|\dfrac{r}{r'}\right|$ are reduced expressions, then $(|p|,|q|,|r|)$ is a $0$-GM Markov triple. Moreover, the correspondence $(X,Y,Z) \mapsto (|p|,|q|,|r|)$ induces the canonical graph isomorphism from $\mathrm{MM}\mathbb T(2,\ell)$ to $\mathrm{M}\mathbb T(0)$.
\end{corollary}

\begin{proof}
 Since the former statement follows from the latter statement, we will prove the latter statement. By Proposition \ref{pr:Cohn-Markov-monodromy} and the definition of $k$-MM triples, $(X,Y,Z)\mapsto (x_{12},y_{12},z_{12})$ induces the canonical graph isomorphism from $\mathrm{MM}\mathbb T(2,\ell)$ to $\mathrm{M}\mathbb T(2)$. Moreover, by Proposition \ref{squre-markov}, $(x_{12},y_{12},z_{12})\mapsto (\sqrt{x_{12}},\sqrt{y_{12}},\sqrt{z_{12}})$ induces the canonical graph isomorphism from $\mathrm{M}\mathbb T(2)$ to $\mathrm{M}\mathbb T(0)$. On the other hand, the absolute values of fixed points of $(X,Y,Z)$ are $\left(\dfrac{\sqrt{x_{12}}}{\sqrt{-x_{21}}}, \dfrac{\sqrt{y_{12}}}{\sqrt{-y_{21}}},\dfrac{\sqrt{z_{12}}}{\sqrt{-z_{21}}} \right)$ by Proposition \ref{prop:fixed-point}. Moreover, by Proposition \ref{prop:relative-prime-y12-y21}, they are reduced expressions. This finishes the proof.
\end{proof}

\begin{remark}
 Since the reduced expression of $\infty$ is $\dfrac{1}{0}$, we can include the case $x_{21}=0$ in Corollary \ref{parabolic-markov}.   
\end{remark}

From here to the end of this subsection, we will consider refining Corollary \ref{parabolic-markov}. Let $p_X,p_Z$ be the fixed points of $X,Z$ in $\mathbb RP^1$ respectively. Then the fixed point of $YZY^{-1}$ is $Y(p_Z)=\dfrac{y_{11}p_Z+y_{12}}{y_{21}p_Z+y_{22}}$, and that of $Y^{-1}XY$ is $Y^{-1}(p_X)=\dfrac{y_{22}p_X-y_{12}}{-y_{21}p_X+y_{11}}$. From Propositions \ref{prop:fixed-point} and \ref{prop:relative-prime-y12-y21}, we have $y_{12} = q^2, \quad y_{21} = -q'^2$.
Furthermore, from the definition of $k$-MM matrices, we have $y_{11}(-y_{11} - k) + q^2q'^2 = 1$
and thus $y_{11}^2 + 2y_{11} - q^2q'^2 + 1 = 0$.
Similarly, we obtain $
y_{22}^2 + 2y_{22} - q^2q'^2 + 1 = 0$.
Solving these, we find that $y_{11}$ and $y_{22}$ are either $-1 \pm qq'$. Considering the sign of $q'$ and by Proposition \ref{prop:relative-prime-y12-y21}, we determine that $y_{11} = -1 - qq'$ and $y_{22} = -1 + qq'$. Substituting these into the fixed point $\dfrac{y_{11}r + y_{12}r'}{y_{21}r + y_{22}r'}$ of $YZY^{-1}$, we find that the fixed point of $YZY^{-1}$ is given by $\dfrac{q^2r' - qq'r - r}{-q'^2r + qq'r' - r'}$. Similarly, the fixed point of $Y^{-1}XY$ for the right child $(Y, Y^{-1}XY, Z)$ of $(X, Y, Z)$ is $\dfrac{-q^2p' + qq'p - p}{q'^2p - qq'p' - p'}$. Thus, from given a $k$-MM triple $(X,Y,Z)$ in $\mathrm{MM}\mathbb T(2,\ell)$ and their fixed points $\left(\dfrac{p}{p'},\dfrac{q}{q'},\dfrac{r}{r'}\right)$, the fixed points of $(X,YZY^{-1},Y)$ and $(Y,Y^{-1}ZY,Z)$ are expressed as

\[\left(\dfrac{p}{p'},\dfrac{q^2r'-qq'r-r}{-q'^2r+qq'r'-r'},\dfrac{q}{q'}\right),\quad \left(\dfrac{q}{q'}\dfrac{-q^2p'+qq'p-p}{q'^2p-qq'p'-p'},\dfrac{r}{r'}\right).
\]

Based on this, we define a \emph{parabolic fixed point tree} $\mathrm{P}\mathbb T(\ell)$ as follows:
we fix $\ell\in \mathbb Z$.
\begin{itemize}\setlength{\leftskip}{-15pt}
    \item [(1)] The root vertex is 
        \[\left(\begin{bmatrix}
            1\\-\ell-1
        \end{bmatrix},\begin{bmatrix}
            2\\-2\ell+1
        \end{bmatrix},\begin{bmatrix}
            1\\-\ell+2
        \end{bmatrix}\right),\]
    \item [(2)] for a vertex $\left(\begin{bmatrix}
        p\\p'
    \end{bmatrix},\begin{bmatrix}
        q\\q'
    \end{bmatrix},\begin{bmatrix}
        r\\r'
    \end{bmatrix}\right)$, we consider the following two children of it:
\[\begin{xy}(0,0)*+{\left(\begin{bmatrix}
        p\\p'
    \end{bmatrix},\begin{bmatrix}
        q\\q'
    \end{bmatrix},\begin{bmatrix}
        r\\r'
    \end{bmatrix}\right)}="1",(-40,-15)*+{\left(\begin{bmatrix}
        p\\p'
    \end{bmatrix},\begin{bmatrix}
        q^2r'-qq'r-r\\-q'^2r+qq'r'-r'
    \end{bmatrix},\begin{bmatrix}
        q\\q'
    \end{bmatrix}\right)}="2",(40,-15)*+{\left(\begin{bmatrix}
        q\\q'
    \end{bmatrix},\begin{bmatrix}
        -q^2p'+qq'p-p\\q'^2p-qq'p'-p'
    \end{bmatrix},\begin{bmatrix}
        r\\r'
    \end{bmatrix}\right).}="3", \ar@{-}"1";"2"\ar@{-}"1";"3"
\end{xy}\]
\end{itemize}

\begin{example}\label{ex:PT}
When $\ell=0$, the tree $\mathrm{P}\mathbb{T}(\ell)$ is given by the following.
\begin{align*}
\begin{xy}(-30,0)*+{\left(\begin{bmatrix}1\\-1\end{bmatrix},\begin{bmatrix}2\\1\end{bmatrix},\begin{bmatrix}1\\2\end{bmatrix}\right)}="1",(10,-16)*+{\left(\begin{bmatrix}1\\-1\end{bmatrix},\begin{bmatrix}5\\1\end{bmatrix},\begin{bmatrix}2\\1\end{bmatrix}\right)}="2",(10,16)*+{\left(\begin{bmatrix}2\\1\end{bmatrix},\begin{bmatrix}5\\4\end{bmatrix},\begin{bmatrix}1\\2\end{bmatrix}\right)}="3", 
(60,-24)*+{\left(\begin{bmatrix}1\\-1\end{bmatrix},\begin{bmatrix}13\\2\end{bmatrix},\begin{bmatrix}5\\1\end{bmatrix}\right)\cdots}="4",(60,-8)*+{\left(\begin{bmatrix}5\\1\end{bmatrix},\begin{bmatrix}29\\7\end{bmatrix},\begin{bmatrix}2\\1\end{bmatrix}\right)\cdots}="5",(60,8)*+{\left(\begin{bmatrix}2\\1\end{bmatrix},\begin{bmatrix}29\\22\end{bmatrix},\begin{bmatrix}5\\4\end{bmatrix}\right)\cdots}="6",(60,24)*+{\left(\begin{bmatrix}5\\4\end{bmatrix},\begin{bmatrix}13\\11\end{bmatrix},\begin{bmatrix}1\\2\end{bmatrix}\right)\cdots}="7",\ar@{-}"1";"2"\ar@{-}"1";"3"\ar@{-}"2";"4"\ar@{-}"2";"5"\ar@{-}"3";"6"\ar@{-}"3";"7"
\end{xy}
\end{align*}
\end{example}
\begin{proposition}\label{pr:pqr-positivity}
 For a vertex $\left(\begin{bmatrix}
        p\\p'
    \end{bmatrix},\begin{bmatrix}
        q\\q'
    \end{bmatrix},\begin{bmatrix}
        r\\r'
    \end{bmatrix}\right)$ in $\mathrm{P}\mathbb T(\ell)$, 
    \begin{itemize}\setlength{\leftskip}{-15pt}
        \item [(1)] $p,q,r>0$ hold, 
        \item [(2)] $p$ and $|p'|$, $q$ and $|q'|$, $r$ and $|r'|$ are relatively prime.
    \end{itemize}
\end{proposition}

To prove Proposition \ref{pr:pqr-positivity} (1), the following lemma is essential:

\begin{lemma}\label{lem:qp'-q'p}
For a vertex $\left(\begin{bmatrix}
        p\\p'
    \end{bmatrix},\begin{bmatrix}
        q\\q'
    \end{bmatrix},\begin{bmatrix}
        r\\r'
    \end{bmatrix}\right)$ in $\mathrm{P}\mathbb T(\ell)$, 
\begin{itemize}\setlength{\leftskip}{-15pt}  
 \item[(1)] $qp'-q'p<-1$, $qr'-q'r>1$, $rp'-r'p<-1$ hold, 
 \item[(2)] $q>r$ and $q>p$ hold.
\end{itemize}  
\end{lemma}
\begin{proof}
    First, we prove (1). We can see that the root of $\mathrm{P}\mathbb T(\ell)$ satisfies (1) by a direct calculation. We assume that $\left(\begin{bmatrix}
        p\\p'
    \end{bmatrix},\begin{bmatrix}
        q\\q'
    \end{bmatrix},\begin{bmatrix}
        r\\r'
    \end{bmatrix}\right)$ satisfies (1). We will prove the statement (1) for the left child of $\left(\begin{bmatrix}
        p\\p'
    \end{bmatrix},\begin{bmatrix}
        q\\q'
    \end{bmatrix},\begin{bmatrix}
        r\\r'
    \end{bmatrix}\right)$. We set \[\left(\begin{bmatrix}
        \tilde{p}\\\tilde{p}'
    \end{bmatrix},\begin{bmatrix}
\tilde{q}\\\tilde{q}'
    \end{bmatrix},\begin{bmatrix}
        \tilde{r}\\\tilde{r}'
    \end{bmatrix}\right):=\left(\begin{bmatrix}
        p\\p'
    \end{bmatrix},\begin{bmatrix}
        q^2r'-qq'r-r\\-q'^2r+qq'r'-r'
    \end{bmatrix},\begin{bmatrix}
        q\\q'
    \end{bmatrix}\right).\] 
Then, we have
\begin{align*}
    \tilde{q}\tilde{p}'-\tilde{q}'\tilde{p}&=(qr'-q'r-1)(qp'-q'p)+(rp'-r'p)<-1,\\
    \tilde{q}\tilde{r}'-\tilde{q}'\tilde{r}&=qr'-q'r>1,\\
    \tilde{r}\tilde{p}'-\tilde{r}'\tilde{p}&=qp'-q'p<-1,
\end{align*}
as desired. We can prove (1) for the right child in the same way. Second, we will prove (2). We assume that $\left(\begin{bmatrix}
        p\\p'
    \end{bmatrix},\begin{bmatrix}
        q\\q'
    \end{bmatrix},\begin{bmatrix}
        r\\r'
    \end{bmatrix}\right)$ satisfies (2). We will prove that the left child of $\left(\begin{bmatrix}
        p\\p'
    \end{bmatrix},\begin{bmatrix}
        q\\q'
    \end{bmatrix},\begin{bmatrix}
        r\\r'
    \end{bmatrix}\right)$ satisfies (2). By (1), we have
\begin{align*}
    \tilde{q}-\tilde{r}&=q^2r'-qq'r-r-q=q(qr'-q'r-1)-r>q-r>0,\\
    \tilde{q}-\tilde{p}&=q^2r'-qq'r-r-p>q(qr'-q'r-1)-p>q-p>0,
    \end{align*}
as desired. We can prove (2) for the right child in the same way.
\end{proof}

\begin{proof}[Proof of Proposition \ref{pr:pqr-positivity}]
The statement (1) follows from Lemma \ref{lem:qp'-q'p} (2) and the generation rule of $\mathrm{P}\mathbb T(\ell)$. We will prove (2). We consider the correspondences
\begin{align*}
\left(\begin{bmatrix}
        p\\p'
    \end{bmatrix},\begin{bmatrix}
        q\\q'
    \end{bmatrix},\begin{bmatrix}
        r\\r'
    \end{bmatrix}\right)&\mapsto \left(\begin{bmatrix}
        p\\p'
    \end{bmatrix},\begin{bmatrix}
        r\\r'
    \end{bmatrix},\begin{bmatrix}
        -r^2q'+rr'q-q\\r'^2q-rr'q'-q'
    \end{bmatrix}\right),\\
\left(\begin{bmatrix}
        p\\p'
    \end{bmatrix},\begin{bmatrix}
        q\\q'
    \end{bmatrix},\begin{bmatrix}
        r\\r'
    \end{bmatrix}\right)&\mapsto \left(\begin{bmatrix}
        p^2q'-pp'q-q\\-p'^2q+pp'q'-q'
    \end{bmatrix},\begin{bmatrix}
        p\\p'
    \end{bmatrix},\begin{bmatrix}
        r\\r'
    \end{bmatrix}\right).
\end{align*}
They are the inverses of operations of taking the left child and taking the right child in $\mathrm{P}\mathbb T(\ell)$,  respectively. We assume that $p$ and $p'$ have a non-trivial common divisor $d$. Then, by using these inverses and going upstream from $\left(\begin{bmatrix}
        p\\p'
    \end{bmatrix},\begin{bmatrix}
        q\\q'
    \end{bmatrix},\begin{bmatrix}
        r\\r'
    \end{bmatrix}\right)$ in $\mathrm{P}\mathbb T(\ell)$, we can see that there exists a vector in the root vertex in $\mathrm{P}\mathbb T(\ell)$ such that two entries have a common divisor $d$ (cf. the proof of Proposition \ref{prop:relative-prime-y12-y21}). This is a contradiction.
\end{proof}

From the above argument, the tree $\mathrm{P}\mathbb T(\ell)$ gives a way to compute Markov numbers.

\begin{theorem}\label{thm:PT-LLMT}
We fix $\ell\in \mathbb Z$. The correspondence $\left(\begin{bmatrix}
        p\\p'
    \end{bmatrix},\begin{bmatrix}
        q\\q'
    \end{bmatrix},\begin{bmatrix}
        r\\r'
    \end{bmatrix}\right)\mapsto (p,q,r)$ induces the canonical graph isomorphism from $\mathrm{P}\mathbb T(\ell)$ to $\mathrm{M}\mathbb T(0)$.     
\end{theorem}

\begin{proof}
By the definition of $\mathrm{P}\mathbb T(\ell)$, the map $\left(\begin{bmatrix}
        p\\p'
    \end{bmatrix},\begin{bmatrix}
        q\\q'
    \end{bmatrix},\begin{bmatrix}
        r\\r'
    \end{bmatrix}\right)\mapsto\left(\dfrac{p}{p'},\dfrac{q}{q'},\dfrac{r}{r'}\right)$ induces a bijection from vertices in $\mathrm{P}\mathbb T(\ell)$ to fixed point triples of $2$-MM triples in $\mathrm{MM}\mathbb T(2,\ell)$. By Proposition \ref{pr:pqr-positivity} (2) and Corollary \ref{parabolic-markov},  the map $\left(\begin{bmatrix}
        p\\p'
    \end{bmatrix},\begin{bmatrix}
        q\\q'
    \end{bmatrix},\begin{bmatrix}
        r\\r'
    \end{bmatrix}\right)\mapsto\left(|p|,|q|,|r|\right)$ induces the canonical graph isomorphism from $\mathrm{P}\mathbb T(\ell)$ to $\mathrm{M}\mathbb T(0)$. Moreover, by Proposition \ref{pr:pqr-positivity} (1), we have $|p|=p$, $|q|=q$, $|r|=r$.    
\end{proof}

Vertices in $\mathrm{P}\mathbb T(\ell)$ also have the following meaning.

\begin{proposition}\label{thm:eigen-vector}
For the vertex $\left(\begin{bmatrix}
        p\\p'
    \end{bmatrix},\begin{bmatrix}
        q\\q'
    \end{bmatrix},\begin{bmatrix}
        r\\r'
    \end{bmatrix}\right)$ in $\mathrm{P}\mathbb T(\ell)$, $\begin{bmatrix}
        p\\p'
    \end{bmatrix},\begin{bmatrix}
        q\\q'
    \end{bmatrix},\begin{bmatrix}
        r\\r'
    \end{bmatrix}$ are eigenvectors of the corresponding $2$-MM matrices in $(X,Y,Z)\in \mathrm{MM}\mathbb T(2,\ell)$. Moreover, their eigenvalues are $-1$.
\end{proposition}

\begin{proof}
We will prove the former statement. Let $X=\begin{bmatrix}
     a&b\\c&d
 \end{bmatrix}$. Since $\dfrac{p}{p'}$ is the fixed point of the action of $X$ as a M\"obius transformation, we have
 \[\dfrac{ap+bp'}{cp+dp'}=\dfrac{p}{p'}.\]
 Therefore, there exists $e\in \mathbb R\setminus\{0\}$ such that 
 \[ap+bp'=ep, \quad cp+dp'=ep'.\]
 Hence we have
 \[X\begin{bmatrix}
     p\\p'
 \end{bmatrix}=\begin{bmatrix}
     ap+bp'\\cp+dp'
 \end{bmatrix}=e\begin{bmatrix}
     p\\p'
 \end{bmatrix}, \]
as desired. The latter statement follows from $\mathrm{tr}(X)=\mathrm{tr}(Y)=\mathrm{tr}(Z)=-2$.
\end{proof}

\begin{remark}
When $\ell=0$, $p',q',r'$ in each vertex of $\mathrm{P}\mathbb {T}(\ell)$ is described by using the number of perfect matchings of a certain graph. See Theorem \ref{thm:p'q'r'}.
\end{remark}

\subsection{Configuration of fixed point}
Let $\mathrm{LP}\mathbb T(\ell)$ be the full subtree of $\mathrm{P}\mathbb T(\ell)$ with the root
\[\left(\begin{bmatrix}
            1\\-\ell-1
        \end{bmatrix},\begin{bmatrix}
            5\\-5\ell+1
        \end{bmatrix},\begin{bmatrix}
            2\\-2\ell+1
        \end{bmatrix}\right),\]
which is the left child of the root of $\mathrm{P}\mathbb T(\ell)$. We discuss the positions of the fixed points of $2$-MM triples in $\mathrm{LP}\mathbb T(\ell)$. First, we consider the fixed points of $Y$ and $Z$. 

\begin{proposition}\label{pr:q/q'>r/r'}
 For a vertex $\left(\begin{bmatrix}
        p\\p'
    \end{bmatrix},\begin{bmatrix}
        q\\q'
    \end{bmatrix},\begin{bmatrix}
        r\\r'
    \end{bmatrix}\right)$ in $\mathrm{LP}\mathbb T(\ell)$, $\dfrac{q}{q'}>\dfrac{r}{r'}$ holds.
\end{proposition}

 We will begin with an easy lemma. This lemma gives a different presentation of the fixed point than in Proposition \ref{prop:fixed-point}.

 \begin{lemma}\label{lem:fixed-point2}
If $X=\begin{bmatrix}
x_{11}&x_{12}\\x_{21}&x_{22}    
\end{bmatrix} \in SL(2,\mathbb R)$ has the unique fixed point $p_X$ on $\mathbb R$, then, we have
 \[p_X=\dfrac{x_{11}-x_{22}}{2x_{21}}.\]
 \end{lemma}

\begin{proof}
The fixed point $p_X$ satisfies $x_{21}p_X^2+(x_{22}-x_{11})p_X+x_{12}=0$. Completing the square, we obtain the following equation:
\[x_{21}\left(p_X-\dfrac{x_{11}-x_{22}}{2x_{21}}\right)^2-\left(\dfrac{x_{11}-x_{22}}{2x_{21}}\right)^2+x_{12}=0.\]
Since this equation has a double root, we have $-\left(\dfrac{x_{11}-x_{22}}{2x_{21}}\right)^2+x_{12}=0$ and $p_X=\dfrac{x_{11}-x_{22}}{2x_{21}}$.
\end{proof} 

\begin{proof}[Proof of Proposition \ref{pr:q/q'>r/r'}]
 We take $(X,Y,Z) \in \mathrm{LMM}\mathbb T(2,\ell)$ such that the fixed point of $X$ (resp. $Y,Z$) is $\dfrac{p}{p'}$ (resp. $\dfrac{q}{q'}$, $\dfrac{r}{r'}$). 
First, we prove the statement in the case $\ell\leq 0$. By Lemma \ref{lem:fixed-point2}, we have
 \[\dfrac{q}{q'}=\dfrac{y_{11}-y_{22}}{2y_{21}}=\dfrac{-y_{22}-1}{y_{21}},\quad \dfrac{r}{r'}=\dfrac{z_{11}-z_{22}}{2z_{21}}=\dfrac{z_{11}+1}{z_{21}}.\]
 Therefore, we have $\dfrac{r}{r'}<Z(\infty)$ and $Y^{-1}(\infty)<\dfrac{q}{q'}$ (we note that $y_{21},z_{21}<0$). Therefore, by Lemma \ref{lem:Y-inf-Z-inf}, we have $\dfrac{q}{q'}>\dfrac{r}{r'}$.

 Second, we prove the statement in the case $\ell\geq 1$. Now, $\dfrac{r}{r'}\leq Z(0)$ and $Y^{-1}(0)<\dfrac{q}{q'}$ hold. Indeed, we have
 \begin{align*}
     Z(0)-\dfrac{r}{r'}&=\dfrac{z_{12}}{z_{22}}-\dfrac{z_{11}+1}{z_{21}}=\dfrac{-1-z_{22}}{z_{21}z_{22}}\geq 0,\\
     \dfrac{q}{q'}-Y^{-1}(0)&=\dfrac{-y_{22}-1}{y_{21}}+\dfrac{y_{12}}{y_{11}}=\dfrac{-1-y_{11}}{y_{11}y_{21}}\geq 0
 \end{align*}
 (we note that $y_{11}>0$, $y_{21},z_{21},z_{22}<0$). Therefore, by Lemma \ref{lem:Y-0-Z-0}, we have $\dfrac{q}{q'}>\dfrac{r}{r'}$.
\end{proof}

Next, we discuss the fixed points of $X$ and $Y$. As we saw in Remark \ref{rem:aboutXZ}, a case separation is required when $X=\tilde{X}_\ell$. 

\begin{proposition}\label{pr:p/p'>q/q'}
 For a vertex $\left(\begin{bmatrix}
        p\\p'
    \end{bmatrix},\begin{bmatrix}
        q\\q'
    \end{bmatrix},\begin{bmatrix}
        r\\r'
    \end{bmatrix}\right)$ in $\mathrm{LP}\mathbb T(\ell)$ with $(\ell,p)\neq (0,1)$, $\dfrac{p}{p'}>\dfrac{q}{q'}$ holds, where $\dfrac{1}{0}$ is regarded as a fraction that is larger than any real numbers.
\end{proposition}

\begin{proof}
 We take $(X,Y,Z) \in \mathrm{LMM}\mathbb T(2,\ell)$ such that the fixed point of $X$ (resp. $Y,Z$) is $\dfrac{p}{p'}$ (resp. $\dfrac{q}{q'}$, $\dfrac{r}{r'}$). We note that $p=1$ if and only if $X=\tilde X_{\ell}$.
First, we prove the statement in the case $\ell\leq 0$.  By Lemma \ref{lem:fixed-point2}, we have
 \[\dfrac{p}{p'}=\dfrac{x_{11}-x_{22}}{2x_{21}}=\dfrac{-x_{22}-1}{x_{21}},\quad \dfrac{q}{q'}=\dfrac{y_{11}-y_{22}}{2y_{21}}=\dfrac{y_{11}+1}{y_{21}}.\] Now, if $(\ell,p)\neq (0,1), (-1,1)$, then $\dfrac{q}{q'}\leq Y(0)$ and $X^{-1}(0)<\dfrac{p}{p'}$ hold. Indeed, we have
 \begin{align*}
     Y(0)-\dfrac{q}{q'}&=\dfrac{y_{12}}{y_{22}}-\dfrac{y_{11}+1}{y_{21}}=\dfrac{-1-y_{22}}{y_{21}y_{22}}> 0,\\
     \dfrac{p}{p'}-X^{-1}(0)&=\dfrac{-x_{22}-1}{x_{21}}+\dfrac{x_{12}}{x_{11}}=\dfrac{-1-x_{11}}{x_{11}x_{21}}\geq 0
 \end{align*}
 (we note that $y_{22}>0$, $x_{21},y_{21},x_{11}<0$). 
 Therefore, by Lemma \ref{lem:Y-0-X-0}, we have $\dfrac{p}{p'}>\dfrac{q}{q'}$. When $(\ell,p)=(-1,1)$, since $\dfrac{p}{p'}=\dfrac{1}{0}$, we have $\dfrac{p}{p'}>\dfrac{q}{q'}$.  
 Second, we prove the case $\ell\geq 1$.  We have $\dfrac{q}{q'}<Y(\infty)$ and $X^{-1}(\infty)<\dfrac{p}{p'}$, (we note that $x_{21},y_{21}<0$). Therefore, by Lemma \ref{lem:X-inf-Y-inf}, we have $\dfrac{p}{p'}>\dfrac{q}{q'}$.
 \end{proof}

We summarize the configuration of $\dfrac{p}{p'}$, $\dfrac{q}{q'}$, and $\dfrac{r}{r'}$, including the case where $X=\tilde X_\ell$.

\begin{theorem}\label{thm:p/p'<q/q'<r/r'}
 For a vertex $\left(\begin{bmatrix}
        p\\p'
    \end{bmatrix},\begin{bmatrix}
        q\\q'
    \end{bmatrix},\begin{bmatrix}
        r\\r'
    \end{bmatrix}\right)$ in $\mathrm{LP}\mathbb T(\ell)$,
\begin{itemize}\setlength{\leftskip}{-15pt}   
\item[(1)] if $\ell\geq 1$, then $\dfrac{r}{r'}<\dfrac{q}{q'}<\dfrac{p}{p'}<0$ holds,\vspace{2mm}
\item[(2)] if $\ell= 0$ and $p \neq 1$, then $0<\dfrac{r}{r'}<\dfrac{q}{q'}<\dfrac{p}{p'}$ holds,\vspace{2mm}
\item[(3)] if $\ell= 0$ and $p=1$, then $\dfrac{p}{p'}<0<\dfrac{r}{r'}<\dfrac{q}{q'}$ holds,\vspace{2mm}
\item[(4)] if $\ell\leq -1$, then $0<\dfrac{r}{r'}<\dfrac{q}{q'}<\dfrac{p}{p'}$ holds.
\end{itemize} 
In particular, for any $\ell\in \mathbb Z$, the second components of all vertices of $\mathrm{LP}\mathbb{T}(\ell)$ are distinct. 
\end{theorem}
\begin{proof} 
The statements (1),(2),(4) follow from Remark \ref{rem:aboutXZ} and Propositions \ref{prop:fixed-point}, \ref{pr:q/q'>r/r'}, \ref{pr:p/p'>q/q'}. The statement (3) can be checked easily.
\end{proof}

\subsection{Determinants of matrix constructed by fixed points}

By Lemma \ref{lem:qp'-q'p} (1), the determinants of $\begin{bmatrix}
    q&r\\q'&r'
\end{bmatrix}$, $\begin{bmatrix}
    p&r\\p'&r'
\end{bmatrix}$, $\begin{bmatrix}
    p&q\\p'&q'
\end{bmatrix}$
are larger than $1$. We will see these determinants in more detail. 

\begin{theorem}\label{thm:iso-PT-MTdag}
For a vertex $\left(\begin{bmatrix}
        p\\p'
    \end{bmatrix},\begin{bmatrix}
        q\\q'
    \end{bmatrix},\begin{bmatrix}
        r\\r'
    \end{bmatrix}\right)$ in $\mathrm{P}\mathbb T(\ell)$, we have 
    \[\left(\det\begin{bmatrix}
    q&r\\q'&r'
\end{bmatrix},\det\begin{bmatrix}
    p&r\\p'&r'
\end{bmatrix},\det\begin{bmatrix}
    p&q\\p'&q'
\end{bmatrix}\right)=3\mu(p,q,r),\]
where $\mu$ is the same notation as in Proposition \ref{pr:rho-mor}.
In particular, the correspondence \[\left(\begin{bmatrix}
        p\\p'
    \end{bmatrix},\begin{bmatrix}
        q\\q'
    \end{bmatrix},\begin{bmatrix}
        r\\r'
    \end{bmatrix}\right)\mapsto\dfrac{1}{3}\left(\det\begin{bmatrix}
    q&r\\q'&r'
\end{bmatrix},\det\begin{bmatrix}
    p&r\\p'&r'
\end{bmatrix},\det\begin{bmatrix}
    p&q\\p'&q'
\end{bmatrix}\right)\] induces the canonical graph isomorphism from $\mathrm{P}\mathbb{T}(\ell)$ to $\mathrm{M}\mathbb{T}^\dag (0)$.
\end{theorem}

\begin{proof}
When $ \left(\begin{bmatrix}
        p\\p'
    \end{bmatrix},\begin{bmatrix}
        q\\q'
    \end{bmatrix},\begin{bmatrix}
        r\\r'
    \end{bmatrix}\right)=\left(\begin{bmatrix}
            1\\-\ell-1
        \end{bmatrix},\begin{bmatrix}
           2\\-2\ell+1
        \end{bmatrix},\begin{bmatrix}
           1\\-\ell+2
        \end{bmatrix}\right)$, we can check the statement directly.
Next, we assume that $\left(\begin{bmatrix}
        p\\p'
    \end{bmatrix},\begin{bmatrix}
        q\\q'
    \end{bmatrix},\begin{bmatrix}
        r\\r'
    \end{bmatrix}\right)$ satisfies the statement. Then, for the left child $\left(\begin{bmatrix}
        p\\p'
    \end{bmatrix},\begin{bmatrix}
        q^2r'-qq'r-r\\-q'^2r+qq'r'-r'
    \end{bmatrix},\begin{bmatrix}
        q\\q'
    \end{bmatrix}\right)$, we have
\begin{align*}
\det\begin{bmatrix}
        q^2r'-qq'r-r & q\\-q'^2r+qq'r'-r' &q'
    \end{bmatrix}&=qr'-q'r=\det\begin{bmatrix}
    q&r\\q'&r'\end{bmatrix},\\
  \det\begin{bmatrix}
        p&q^2r'-qq'r-r\\p'&-q'^2r+qq'r'-r'
    \end{bmatrix}&=(pq'-p'q)(qr'-qr)-(pr'-p'r)\\&=\det\begin{bmatrix}
    p&q\\p'&q'\end{bmatrix}\det\begin{bmatrix}
    q&r\\q'&r'\end{bmatrix}-\det\begin{bmatrix}
    p&r\\p'&r'\end{bmatrix}.
\end{align*}
By assumption, $\left(\det\begin{bmatrix}
    q&r\\q'&r'
\end{bmatrix},\det\begin{bmatrix}
    p&r\\p'&r'
\end{bmatrix},\det\begin{bmatrix}
    p&q\\p'&q'
\end{bmatrix}\right)$ is a solution to $\mathrm{GSME}(0)$, and thus so is $\left(\det\begin{bmatrix}
    q&r\\q'&r'
\end{bmatrix},\det\begin{bmatrix}
    p&q\\p'&q'\end{bmatrix},\det\begin{bmatrix}
    p&q\\p'&q'\end{bmatrix}\det\begin{bmatrix}
    q&r\\q'&r'\end{bmatrix}-\det\begin{bmatrix}
    p&r\\p'&r'\end{bmatrix}\right)$ because it is a permutation of the Vieta jumping. The permutation rule is consistent with the operation taking the left child of $\mathrm{M}\TT^\dag(0)$. The same is true for the right child. Therefore, we can prove that the statement holds inductively.
\end{proof}

We define the \emph{inverse parabolic fixed point tree} $\mathrm{P}\TT^\dag(\ell)$:
\begin{itemize}\setlength{\leftskip}{-15pt}
    \item [(1)] the root vertex is  \[\left(\begin{bmatrix}
            1\\-\ell-1
        \end{bmatrix},\begin{bmatrix}
            1\\-\ell+2
        \end{bmatrix},\begin{bmatrix}
            1\\-\ell+5
        \end{bmatrix}\right),\]
     \item [(2)] for a vertex $\left(\begin{bmatrix}
        p\\p'
    \end{bmatrix},\begin{bmatrix}
        q\\q'
    \end{bmatrix},\begin{bmatrix}
        r\\r'
    \end{bmatrix}\right)$, we consider the following two children of it:
\[\begin{xy}(0,0)*+{\left(\begin{bmatrix}
        p\\p'
    \end{bmatrix},\begin{bmatrix}
        q\\q'
    \end{bmatrix},\begin{bmatrix}
        r\\r'
    \end{bmatrix}\right)}="1",(-40,-15)*+{\left(\begin{bmatrix}
        p\\p'
    \end{bmatrix},\begin{bmatrix}
        r\\r'
    \end{bmatrix},\begin{bmatrix}
        -r^2q'+rr'q-q\\r'^2q-rr'q'-q'
    \end{bmatrix}\right)}="2",(40,-15)*+{\left(\begin{bmatrix}
        p^2q'-pp'q-q\\-p'^2q+pp'q'-q'
    \end{bmatrix},\begin{bmatrix}
        p\\p'
    \end{bmatrix},\begin{bmatrix}
        r\\r'
    \end{bmatrix}\right).}="3", \ar@{-}"1";"2"\ar@{-}"1";"3"
\end{xy}\]    
\end{itemize}

\begin{example}
When $\ell=0$, the tree $\mathrm{P}\mathbb{T}^\dag(\ell)$ is given by the following.
\begin{align*}
\begin{xy}(-30,0)*+{\left(\begin{bmatrix}1\\-1\end{bmatrix},\begin{bmatrix}1\\2\end{bmatrix},\begin{bmatrix}1\\5\end{bmatrix}\right)}="1",(10,-16)*+{\left(\begin{bmatrix}1\\-1\end{bmatrix},\begin{bmatrix}1\\5\end{bmatrix},\begin{bmatrix}2\\13\end{bmatrix}\right)}="2",(10,16)*+{\left(\begin{bmatrix}2\\-5\end{bmatrix},\begin{bmatrix}1\\-1\end{bmatrix},\begin{bmatrix}1\\5\end{bmatrix}\right)}="3", 
(60,-24)*+{\left(\begin{bmatrix}1\\-1\end{bmatrix},\begin{bmatrix}2\\13\end{bmatrix},\begin{bmatrix}5\\34\end{bmatrix}\right)\dots}="4",(60,-8)*+{\left(\begin{bmatrix}5\\-11\end{bmatrix},\begin{bmatrix}1\\-1\end{bmatrix},\begin{bmatrix}2\\13\end{bmatrix}\right)\dots}="5",(60,8)*+{\left(\begin{bmatrix}2\\-5\end{bmatrix},\begin{bmatrix}1\\5\end{bmatrix},\begin{bmatrix}5\\31\end{bmatrix}\right)\dots}="6",(60,24)*+{\left(\begin{bmatrix}5\\-14\end{bmatrix},\begin{bmatrix}2\\-5\end{bmatrix},\begin{bmatrix}1\\5\end{bmatrix}\right)\dots}="7",\ar@{-}"1";"2"\ar@{-}"1";"3"\ar@{-}"2";"4"\ar@{-}"2";"5"\ar@{-}"3";"6"\ar@{-}"3";"7"
\end{xy}
\end{align*}
\end{example}

In parallel with $\mathrm{P}\TT(\ell)$, we have the following properties:

\begin{proposition}\label{pr:pqr-positivity-dual}
 For a vertex $\left(\begin{bmatrix}
        p\\p'
    \end{bmatrix},\begin{bmatrix}
        q\\q'
    \end{bmatrix},\begin{bmatrix}
        r\\r'
    \end{bmatrix}\right)$ in $\mathrm{P}\mathbb T^\dag(\ell)$, 
    \begin{itemize}\setlength{\leftskip}{-15pt}
        \item [(1)] $p,q,r>0$ hold, 
        \item [(2)] $p$ and $|p'|$ (resp. $q$ and $|q'|$, $r$ and $|r'|$) are relatively prime.
    \end{itemize}
\end{proposition}

\begin{theorem}
We fix $\ell\in \mathbb Z$. The correspondence $\left(\begin{bmatrix}
        p\\p'
    \end{bmatrix},\begin{bmatrix}
        q\\q'
    \end{bmatrix},\begin{bmatrix}
        r\\r'
    \end{bmatrix}\right)\mapsto (p,q,r)$ induces the canonical graph isomorphism from $\mathrm{P}\mathbb T^\dag(\ell)$ to $\mathrm{M}\mathbb T^\dag(0)$.     
\end{theorem}

\begin{theorem}\label{thm:iso-PTdag-MT}
For a vertex $\left(\begin{bmatrix}
        p\\p'
    \end{bmatrix},\begin{bmatrix}
        q\\q'
    \end{bmatrix},\begin{bmatrix}
        r\\r'
    \end{bmatrix}\right)$ in $\mathrm{P}\mathbb T^\dag(\ell)$, 
    \[\left(\det\begin{bmatrix}
    q&r\\q'&r'
\end{bmatrix},\det\begin{bmatrix}
    p&r\\p'&r'
\end{bmatrix},\det\begin{bmatrix}
    p&q\\p'&q'
\end{bmatrix}\right)=3\mu(p,q,r).\]
In particular, the correspondence \[\left(\begin{bmatrix}
        p\\p'
    \end{bmatrix},\begin{bmatrix}
        q\\q'
    \end{bmatrix},\begin{bmatrix}
        r\\r'
    \end{bmatrix}\right)\mapsto\dfrac{1}{3}\left(\det\begin{bmatrix}
    q&r\\q'&r'
\end{bmatrix},\det\begin{bmatrix}
    p&r\\p'&r'
\end{bmatrix},\det\begin{bmatrix}
    p&q\\p'&q'
\end{bmatrix}\right)\] induces the canonical graph isomorphism from $\mathrm{P}\mathbb{T}^\dag(\ell)$ to $\mathrm{M}\mathbb{T} (0)$.
\end{theorem}

By the above argument, the tree $\mathrm{M}\mathbb T^\dag(0)$ (resp. $\mathrm{M}\mathbb T(0)$) can be constructed from the determinants of a matrices composed of fixed points of $k$-MM matrices in $\mathrm{MM}\mathbb T(2,\ell)$ (resp. $\mathrm{MM}\mathbb T^\dag(2,\ell)$). Does there exist algebraic or geometric theory that can effectively explain phenomena of Theorems \ref{thm:iso-PT-MTdag} and \ref{thm:iso-PTdag-MT}?

\section{Combinatorics of $k$-GM number}
In this section, we provide a way to calculate certain $k$-GC matrices and certain $k$-MM matrices by using a combinatorial method.

\subsection{Farey tree and fraction labeling}

We fix $k\in \mathbb Z_{\geq 0}$. In this subsection, we recall the Farey tree, and we label $k$-GC matrices in $\mathrm{GC}\mathbb T(k,\ell)$ and $k$-MM matrices in $\mathrm{MM}\mathbb T(k,\ell)$ with irreducible fractions.  

\begin{definition}
For $\dfrac{a}{b}$ and $\dfrac{c}{d}$, we denote $ad-bc$ by $\det\left(\dfrac{a}{b},\dfrac{c}{d}\right)$. A triple $\left(\dfrac{a}{b},\dfrac{c}{d},\dfrac{e}{f}\right)$ is called a \emph{Farey triple} if the following conditions hold:
\begin{itemize}\setlength{\leftskip}{-15pt}
    \item [(1)] $\dfrac{a}{b},\dfrac{c}{d}$ and $\dfrac{e}{f}$ are irreducible fractions, and
    \item [(2)] $\left|\det\middle(\dfrac{a}{b},\dfrac{c}{d}\middle)\middle|=\middle|\det\middle(\dfrac{c}{d},\dfrac{e}{f}\middle)\middle|=\middle|\det\middle(\dfrac{e}{f},\dfrac{a}{b}\middle)\right|=1$.
\end{itemize}
\end{definition}
We define the \emph{Farey tree} $\mathrm{F}\mathbb T$ as follows:
\begin{itemize}\setlength{\leftskip}{-15pt}
\item [(1)] the root vertex is $\left(\dfrac{0}{1},\dfrac{1}{1},\dfrac{1}{0}\right)$, and
\item[(2)]every vertex $\left(\dfrac{a}{b},\dfrac{c}{d},\dfrac{e}{f}\right)$ has the following two children:
\[\begin{xy}(0,0)*+{\left(\dfrac{a}{b},\dfrac{c}{d},\dfrac{e}{f}\right)}="1",(-30,-15)*+{\left(\dfrac{a}{b},\dfrac{a+c}{b+d},\dfrac{c}{d}\right)}="2",(30,-15)*+{\left(\dfrac{c}{d},\dfrac{c+e}{d+f},\dfrac{e}{f}\right).}="3", \ar@{-}"1";"2"\ar@{-}"1";"3"
\end{xy}\]
\end{itemize}
The first few vertices of $\mathrm{F}\mathbb{T}$ are given by the following.
\relsize{+1}
\begin{align*}
\begin{xy}(0,0)*+{\left(\frac{0}{1},\frac{1}{1},\frac{1}{0}\right)}="1",(20,-14)*+{\left(\frac{0}{1},\frac{1}{2},\frac{1}{1}\right)}="2",(20,14)*+{\left(\frac{1}{1},\frac{2}{1},\frac{1}{0}\right)}="3", 
(50,-24)*+{\left(\frac{0}{1},\frac{1}{3},\frac{1}{2}\right)}="4",(50,-8)*+{\left(\frac{1}{2},\frac{2}{3},\frac{1}{1}\right)}="5",(50,8)*+{\left(\frac{1}{1},\frac{3}{2},\frac{2}{1}\right)}="6",(50,24)*+{\left(\frac{2}{1},\frac{3}{1},\frac{1}{0}\right)}="7",(85,-28)*+{\left(\frac{0}{1},\frac{1}{4},\frac{1}{3}\right)\cdots}="8",(85,-20)*+{\left(\frac{1}{3},\frac{2}{5},\frac{1}{2}\right)\cdots}="9",(85,-12)*+{\left(\frac{1}{2},\frac{3}{5},\frac{2}{3}\right)\cdots}="10",(85,-4)*+{\left(\frac{2}{3},\frac{3}{4},\frac{1}{1}\right)\cdots}="11",(85,4)*+{\left(\frac{1}{1},\frac{4}{3},\frac{3}{2}\right)\cdots}="12",(85,12)*+{\left(\frac{3}{2},\frac{5}{3},\frac{2}{1}\right)\cdots}="13",(85,20)*+{\left(\frac{2}{1},\frac{5}{2},\frac{3}{1}\right)\cdots}="14",(85,28)*+{\left(\frac{3}{1},\frac{4}{1},\frac{1}{0}\right)\cdots}="15",\ar@{-}"1";"2"\ar@{-}"1";"3"\ar@{-}"2";"4"\ar@{-}"2";"5"\ar@{-}"3";"6"\ar@{-}"3";"7"\ar@{-}"4";"8"\ar@{-}"4";"9"\ar@{-}"5";"10"\ar@{-}"5";"11"\ar@{-}"6";"12"\ar@{-}"6";"13"\ar@{-}"7";"14"\ar@{-}"7";"15"
\end{xy}
\end{align*}
\relsize{-1}
\begin{proposition}[see \cite{aig}*{Section 3.2}]\label{prop:property-farey}\indent
\begin{itemize}\setlength{\leftskip}{-15pt}
    \item [(1)]If $\left(\dfrac{a}{b},\dfrac{c}{d},\dfrac{e}{f}\right)$ is a Farey triple, then so are $\left(\dfrac{a}{b},\dfrac{a+c}{b+d},\dfrac{c}{d}\right)$ and $\left(\dfrac{c}{d},\dfrac{c+e}{d+f},\dfrac{e}{f}\right)$. In particular, each vertex in $\mathrm{F}\mathbb T$ is a Farey triple. 
    \item [(2)] For every irreducible fraction $\dfrac{a}{b} \in \mathbb Q_{> 0}$, there exists a unique Farey triple $F$ in $\mathrm{F}\mathbb T$ such that $\dfrac{a}{b}$ is the second entry of $F$.
    \item [(3)] For $\left(\dfrac{a}{b},\dfrac{c}{d},\dfrac{e}{f}\right)$ in $\mathrm{F}\mathbb T$, $\dfrac{a}{b}<\dfrac{c}{d}<\dfrac{e}{f}$ holds. 
\end{itemize}
\end{proposition}

By using the canonical graph isomorphism from the Farey tree to the $k$-GC tree, we provide the correspondence from Farey triple in $\mathrm{F}\mathbb T$ to $k$-GC triple in $\mathrm{GC}\mathbb T(k,\ell)$. This correspondence induces the map from irreducible fractions in $(0,\infty)$ to $k$-GC matrices which are the second components of $k$-GC triples in $\mathrm{GC}\mathbb T(k,\ell)$. This map is called the \emph{fraction labeling to $k$-GC matrices}. We denote by $C_t(k,\ell)$ the $k$-GC matrix labeled with a fraction $t$. In the same way, we define the map from the set of irreducible fractions to the set of $k$-MM matrices in $\mathrm{MM}\mathbb T(k,\ell)$. This map is called the \emph{fraction labeling to $k$-MM matrices}. We denote by $M_t(k,\ell)$ the $k$-MM matrix labeled with a fraction $t$.  

All $k$-GC matrices (or $k$-MM matrices) labeled with irreducible fractions between $0$ and $1$ are included in $\mathrm{LGC}\mathbb T(k, \ell)$ (or $\mathrm{LMM}\mathbb T(k, \ell)$), where $\mathrm{LGC}\mathbb T(k, \ell)$ is the full subtree of $\mathrm{GC}\mathbb T(k, \ell)$ whose root is the left child of the root of $\mathrm{GC}\mathbb T(k, \ell)$.   Moreover, if we restrict $\mathrm{M}\mathbb T(k)$ to $\mathrm{LM}\mathbb T(k)$, then all $k$-GM triples but $(1,1,1)$ and $(1,k+2,1)$ appear exactly once without overlap (here, triples that differ only in order are regarded as the same solution). Therefore, the correspondence between a Farey triple $(r, t, s)$ and the $(1,2)$-entries of the $k$-GC triple $(C_r(k,\ell), C_t(k,\ell), C_s(k,\ell))$ (or the $k$-MM triple $(M_r(k,\ell), M_t(k,\ell), M_s(k,\ell))$) induces a bijection from the set of Farey triples in $[0,1]^3$ to the set of all $k$-GM triples but $(1,1,1)$ and $(1,k+2,1)$. If we take the second entries of $(r,t,s)$, then this bijection gives a fraction labeling of each $k$-GM number except for $1$ and $k+2$. We call it the \emph{fraction labeling to $k$-GM numbers}, and for every irreducible fraction $t\in (0,1)$, we denote by $m_{k,t}$ the corresponding $k$-GM number. Also, we set $m_{k,\frac{0}{1}}=1$ and $m_{k,\frac{1}{1}}=k+2$. Note that whether this labeling is injective is an open problem (equivalent to Conjecture \ref{conj:gen-Markov}).

\subsection{Continued fraction and $k$-GM snake graph}

In this subsection, we construct pre-snake graphs, and construct continued fractions from them. Moreover, we will calculate $k$-GM numbers by using these continued fractions. 

We will begin with recalling the relation between the snake graphs and continued fractions.  
We set
\[[a_1,\dots,a_\ell]:=a_1+\dfrac{1}{a_2+\dfrac{1}{\ddots+\dfrac{\ddots}{a_{\ell-1}+\dfrac{1}{a_\ell}}}}\]
and $[\ ]:=1$.

In this paper, we assume $a_i\in \mathbb Z_{\geq 1}$. First, we will recall the \emph{snake graph} associated with a continued fraction $[a_1,\dots,a_\ell]$ according to \cite{cs18}. For a given continued fraction $[a_1,\dots,a_\ell]$ with $(\ell,a_1)\neq (1,1)$, 
arrange $(a_1 + \cdots + a_\ell)$ signs such that the first $a_1$ signs are $-$, the next $a_2$ signs are $+$, the following $a_3$ signs are $+$, and so on, alternating between $-$ and $+$. Remove the first and last signs in the $(a_1 + \cdots + a_\ell)$-tuple of signs given above, leaving $(a_1 + \cdots + a_\ell-2)$-tuple. We denote this tuple by $S$. Using it, arrange the two types of tiles with signs as shown in Figure \ref{fig:signedtile} by connecting their edges in a way that satisfies the following rules:

\begin{itemize}\setlength{\leftskip}{-15pt}
\item[$\bullet$] The first (the leftmost) tile is the one on the left in Figure \ref{fig:signedtile},
\item [$\bullet$] a new tile is placed either of the right or the above of the preceding tile, 
\item [$\bullet$] the signs on the adjoining parts of the two adjacent tiles coincide,
\item [$\bullet$] $S$ coincides with the sequence the signs on the adjoining parts arranged in order from the leftmost side.
\end{itemize}
This graph is called the \emph{snake graph}. For the empty continued fraction $[\ ]$, we set $\mathcal G[\ ]= \emptyset$, and for the continued fraction $[1]$, we set that  $\mathcal G[1]$ is a line segment. We denote by $\mathcal{G} [a_1,\dots,a_\ell]$ the snake graph obtained from a continued fraction $[a_1,\dots,a_\ell]$. We remark that $\mathcal{G} [2]$ is constructed by a single tile. 

\begin{figure}[ht]
    \centering
    \includegraphics[scale=0.3]{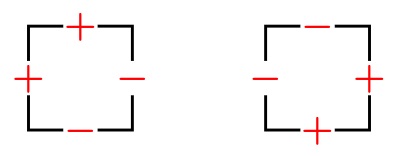}
    \caption{Signed tiles}
    \label{fig:signedtile}
\end{figure}

\begin{example}
For a continued fraction $[2,4,2,1]$, the snake graph associated with $[2,4,2,1]$ is given as in Figure \ref{fig:ex-snakegraph}. Indeed, the signs located inside the connected tiles are arranged from the leftmost to the rightmost as follows: there are ($2-1$) consecutive ``$-$" signs, followed by 4 ``$+$" signs, then 2 ``$-$" signs, and finally $(1-1)$ (therefore, no) consecutive ``$+$" signs.
\begin{figure}[ht]
    \centering
    \includegraphics[scale=0.12]{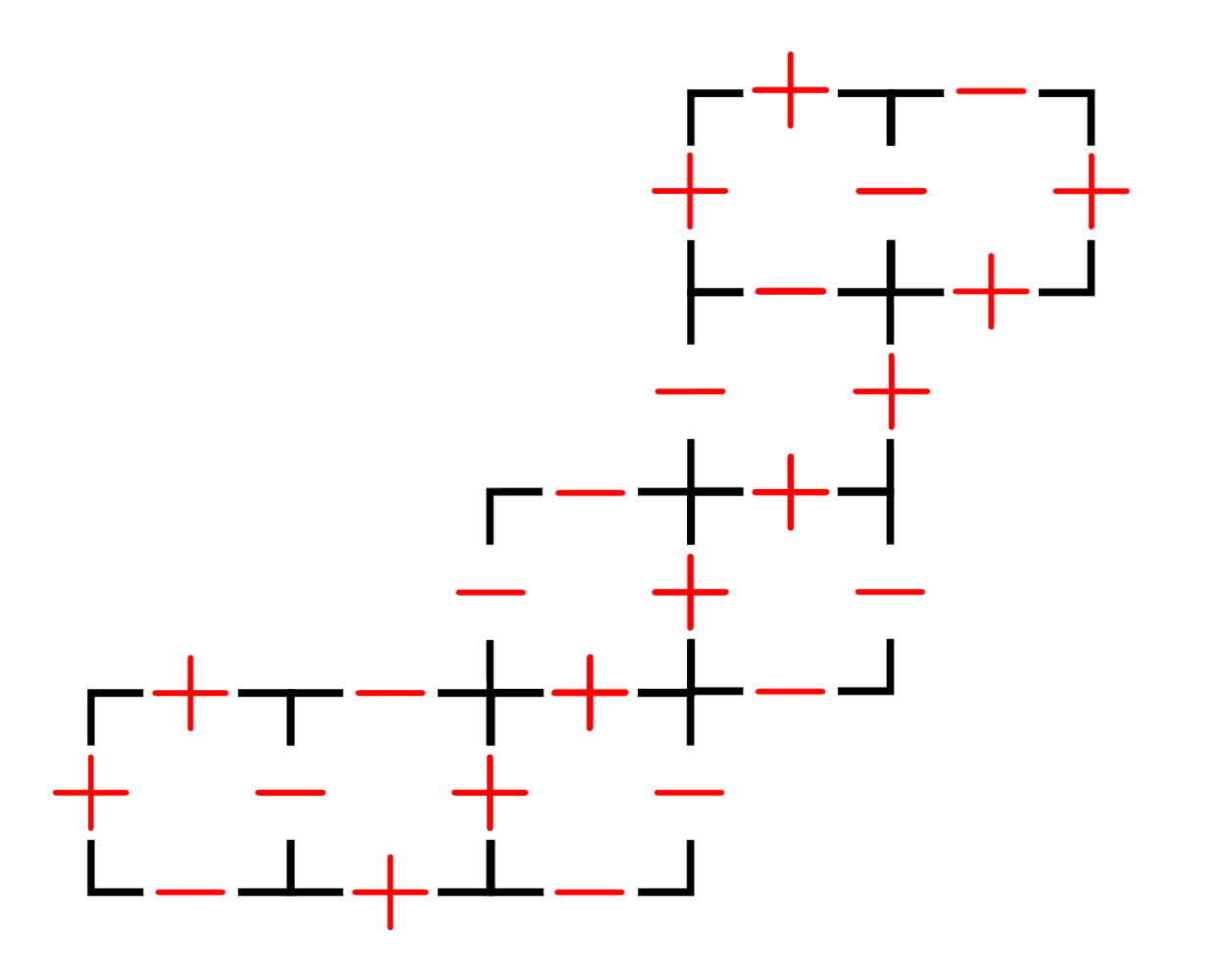}
    \caption{Snake graph associated with $[2,4,2,1]$}
    \label{fig:ex-snakegraph}
\end{figure}
\end{example}

\begin{remark}
The sign on the right edge in each tile is different from one on the upper edge. Therefore, for a continued fraction $[a_1,\dots,a_\ell]$, there is a unique snake graph associated with $[a_1,\dots,a_\ell]$.
\end{remark}

Let $G$ be an undirected graph. We recall that a subset $P$ of the edge set of $G$ is called a \emph{perfect matching} of $G$ if each vertex of $G$ is incident to exactly one edge in $P$. We denote by $m(\mathcal G[a_1,\dots,a_\ell])$ the number of perfect matchings of $\mathcal{G} [a_1,\dots,a_\ell]$. We set $m(\mathcal{G}[\ ])=1$.

\begin{example}
By an induction argument, we have $m(\mathcal G [n])=n$. The Figure \ref{fig:list-of-perfect-matching} is the list of perfect matchings of $\mathcal{G}[5]$. 
\begin{figure}[ht]
    \centering
\scalebox{3}{\begin{tikzpicture}[baseline=0mm]
\coordinate(1) at (0,0){}; 
\coordinate(2) at (0.2,0){}; 
\coordinate(3) at (0.4,0){}; 
\coordinate(5) at (-0.2,-0.2){}; 
\coordinate(6) at (0,-0.2){}; 
\coordinate(7) at (0.2,-0.2){};
\coordinate(8) at (0.4,-0.2){};
\coordinate(10) at (-0.2,-0.4){}; 
\coordinate(11) at (0,-0.4){}; 
\coordinate(12) at (0.2,-0.4){};
\draw[very thick, red](1) to (2);
\draw(2) to (3);
\draw(5) to (6);
\draw(6) to (7);
\draw(7) to (8);
\draw(10) to (11);
\draw(11) to (12);
\draw[very thick, red](3) to (8);
\draw(2) to (7);
\draw[very thick, red](7) to (12);
\draw(1) to (6);
\draw[very thick, red](6) to (11);
\draw[very thick, red](5) to (10);
\end{tikzpicture}}\hspace{3mm}
\scalebox{3}{\begin{tikzpicture}[baseline=0mm]
\coordinate(1) at (0,0){}; 
\coordinate(2) at (0.2,0){}; 
\coordinate(3) at (0.4,0){}; 
\coordinate(5) at (-0.2,-0.2){}; 
\coordinate(6) at (0,-0.2){}; 
\coordinate(7) at (0.2,-0.2){};
\coordinate(8) at (0.4,-0.2){};
\coordinate(10) at (-0.2,-0.4){}; 
\coordinate(11) at (0,-0.4){}; 
\coordinate(12) at (0.2,-0.4){};
\draw[very thick, red](1) to (2);
\draw(2) to (3);
\draw[very thick, red](5) to (6);
\draw(6) to (7);
\draw(7) to (8);
\draw[very thick, red](10) to (11);
\draw(11) to (12);
\draw[very thick, red](3) to (8);
\draw(2) to (7);
\draw[very thick, red](7) to (12);
\draw(1) to (6);
\draw(6) to (11);
\draw(5) to (10);
\end{tikzpicture}}
\hspace{3mm}
\scalebox{3}{\begin{tikzpicture}[baseline=0mm]
\coordinate(1) at (0,0){}; 
\coordinate(2) at (0.2,0){}; 
\coordinate(3) at (0.4,0){}; 
\coordinate(5) at (-0.2,-0.2){}; 
\coordinate(6) at (0,-0.2){}; 
\coordinate(7) at (0.2,-0.2){};
\coordinate(8) at (0.4,-0.2){};
\coordinate(10) at (-0.2,-0.4){}; 
\coordinate(11) at (0,-0.4){}; 
\coordinate(12) at (0.2,-0.4){};
\draw[very thick, red](1) to (2);
\draw(2) to (3);
\draw(5) to (6);
\draw[very thick, red](6) to (7);
\draw(7) to (8);
\draw(10) to (11);
\draw[very thick, red](11) to (12);
\draw[very thick, red](3) to (8);
\draw(2) to (7);
\draw(7) to (12);
\draw(1) to (6);
\draw(6) to (11);
\draw[very thick, red](5) to (10);
\end{tikzpicture}}
\hspace{3mm}
\scalebox{3}{\begin{tikzpicture}[baseline=0mm]
\coordinate(1) at (0,0){}; 
\coordinate(2) at (0.2,0){}; 
\coordinate(3) at (0.4,0){}; 
\coordinate(5) at (-0.2,-0.2){}; 
\coordinate(6) at (0,-0.2){}; 
\coordinate(7) at (0.2,-0.2){};
\coordinate(8) at (0.4,-0.2){};
\coordinate(10) at (-0.2,-0.4){}; 
\coordinate(11) at (0,-0.4){}; 
\coordinate(12) at (0.2,-0.4){};
\draw(1) to (2);
\draw(2) to (3);
\draw(5) to (6);
\draw(6) to (7);
\draw(7) to (8);
\draw(10) to (11);
\draw[very thick, red](11) to (12);
\draw[very thick, red](3) to (8);
\draw[very thick, red](2) to (7);
\draw(7) to (12);
\draw[very thick, red](1) to (6);
\draw(6) to (11);
\draw[very thick, red](5) to (10);
\end{tikzpicture}}
\hspace{3mm}
\scalebox{3}{\begin{tikzpicture}[baseline=0mm]
\coordinate(1) at (0,0){}; 
\coordinate(2) at (0.2,0){}; 
\coordinate(3) at (0.4,0){}; 
\coordinate(5) at (-0.2,-0.2){}; 
\coordinate(6) at (0,-0.2){}; 
\coordinate(7) at (0.2,-0.2){};
\coordinate(8) at (0.4,-0.2){};
\coordinate(10) at (-0.2,-0.4){}; 
\coordinate(11) at (0,-0.4){}; 
\coordinate(12) at (0.2,-0.4){};
\draw(1) to (2);
\draw[very thick, red](2) to (3);
\draw(5) to (6);
\draw(6) to (7);
\draw[very thick, red](7) to (8);
\draw(10) to (11);
\draw[very thick, red](11) to (12);
\draw(3) to (8);
\draw(2) to (7);
\draw(7) to (12);
\draw[very thick, red](1) to (6);
\draw(6) to (11);
\draw[very thick, red](5) to (10);
\end{tikzpicture}}
\caption{List of perfect matchings of $\mathcal{G}[5]$}\label{fig:list-of-perfect-matching}
\end{figure}
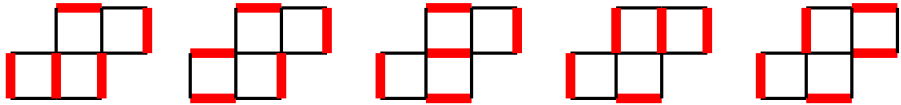
\end{example}

In \cite{cs18}, they give the following relation between a continued fraction and the corresponding snake graph:

\begin{theorem}[\cite{cs18}*{Theorem 3.4}]\label{thm:snakegraph-continuedfraction}
 The following equality holds:
 \[[a_1.\dots,a_\ell]=\dfrac{m(\mathcal G[a_1,\dots,a_\ell])}{m(\mathcal G[a_2,\dots,a_\ell])}.\]
\end{theorem}
We note that the claim of Theorem \ref{thm:snakegraph-continuedfraction} contains the case that the denominator of right-hand side in the equality is $m(\mathcal G[\ ])$.

Next, we will construct the pre-snake graph from an irreducible fraction $t$. For a given irreducible fraction $t\in (0,1]$, we construct the \emph{pre-snake graph} associated with $t$ as follows:

\begin{itemize}\setlength{\leftskip}{-15pt}
    \item [(1)] in the 2-dimensional plane $\mathbb R^2$, for a shortest line segment of slope $t$ whose endpoints are distinct points in $\mathbb Z^2$ (we denote the line segment by $L_t$), consider a graph consisting of all unit squares with integer lattice vertices through which the line segment passes, and
    \item [(2)]for each unit square in the graph given in (1), draw a diagonal edge connecting the upper left and lower right vertices.
\end{itemize}
We denote by $\mathcal {PG}(t)$ the pre-snake graph associated with $t$. We note that $\mathcal {PG}(t)$ does not contain $L_t$.

For example, the pre-snake graph associated with $\dfrac{2}{5}$ is given as in Figure \ref{fig:ex-presnakegraph}.

\begin{figure}[ht]
    \centering
    \includegraphics[scale=0.25]{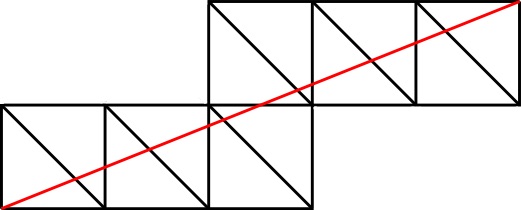}
    \caption{Pre-snake graph associated with $t=\dfrac{2}{5}$}
    \label{fig:ex-presnakegraph}
\end{figure}

By using $\mathcal {PG}(t)$, we construct a continued fraction $F^{+}(k,t)$ associated with $k\in \mathbb Z_{\geq 0}$ and $t$ as follows:

\begin{itemize}\setlength{\leftskip}{-15pt}
    \item [(1)] we set the orientation of $L_t$ from left to right. For each right-angled triangle in the pre-snake graph, assign a sign in $\{+,-\}$ as follows:
    \begin{itemize}\setlength{\leftskip}{-15pt}
        \item [(i)]  we assign $-$ to the following triangles (see Figure \ref{fig:minus-righttriangles}): 
        \begin{itemize}\setlength{\leftskip}{-30pt}
            \item[$\bullet$] the left-most triangle,
            \item[$\bullet$] triangles whose left-hand side part of $L_t$ is a quadrilateral,
        \end{itemize}
        \begin{figure}[ht]
    \centering
\begin{tikzpicture}
\draw (0,0) -- (1,0) -- (0,1) -- cycle;
\draw[red] (0,0) -- (0.6,0.6);
\end{tikzpicture}
\hspace{0.7cm}
\begin{tikzpicture}[baseline=0mm]
\draw (0,0) -- (1,0) -- (0,1) -- cycle;
\draw[red] (0,-0.1) -- (1.1,0.4);
\end{tikzpicture}\hspace{0.4cm}
\begin{tikzpicture}[baseline=0mm]
\draw (1,0) -- (1,1) -- (0,1) -- cycle;
\draw[red] (1.1,0.6) -- (0.6,0.3);
\end{tikzpicture}
    \caption{Right-angled triangles with $-$}
    \label{fig:minus-righttriangles}
\end{figure}
        \item [(ii)]  we assign $+$ to the others than (i) (see Figure \ref{fig:plus-righttriangles}),
 \begin{figure}[ht]
    \centering
\rotatebox{180}{\begin{tikzpicture}[baseline=0mm]
\draw (1,0) -- (1,1) -- (0,1) -- cycle;
\draw[red] (1.1,0.6) -- (0.6,0.3);
\end{tikzpicture}}
\hspace{0.3cm}
\rotatebox{180}{
\begin{tikzpicture}[baseline=0mm]
\draw (0,0) -- (1,0) -- (0,1) -- cycle;
\draw[red] (0,-0.1) -- (1.1,0.4);
\end{tikzpicture}}\hspace{0.5cm}
\rotatebox{180}{\begin{tikzpicture}
\draw (0,0) -- (1,0) -- (0,1) -- cycle;
\draw[red] (0,0) -- (0.6,0.6);
\end{tikzpicture}}
    \caption{Right-angled triangles with $+$}
    \label{fig:plus-righttriangles} 
\end{figure}       
    \end{itemize}
    \item [(2)] for each edge in the pre-snake graph that intersects with the interior of $L_t$, assign $k$ signs as follows:
    \begin{itemize}\setlength{\leftskip}{-15pt}
        \item [(i)] we assign $k$ minus signs ($-$) to each edge whose midpoint is not on the right side of $L_t$ (see Figure \ref{fig:minus-edge}).
        \begin{figure}[ht]
    \centering
\begin{tikzpicture}[baseline=0mm]
\draw (-0.5,0) -- (0.5,0);
\fill (0,0) circle (1.5pt);
\draw[red] (-0.5,-0.3) -- (0.5,0.3);
\end{tikzpicture}
\hspace{0.7cm}
\begin{tikzpicture}[baseline=0mm]
\draw (-0.5,0) -- (0.5,0);
\fill (0,0) circle (1.5pt);
\draw[red] (-0.5,-0.3) -- (0.5,0.1);
\end{tikzpicture}
\hspace{0.7cm}
\begin{tikzpicture}[baseline=0mm]
\draw (0,-0.5) -- (0,0.5);
\fill (0,0) circle (1.5pt);
\draw[red] (-0.5,-0.3) -- (0.5,0.3);
\end{tikzpicture}\hspace{0.6cm}
\begin{tikzpicture}[baseline=0mm]
\draw (0,-0.5) -- (0,0.5);
\fill (0,0) circle (1.5pt);
\draw[red] (-0.5,-0.3) -- (0.5,0);
\end{tikzpicture}\hspace{0.6cm}
\begin{tikzpicture}[baseline=0mm]
\draw (-0.5,0.5) -- (0.5,-0.5);
\fill (0,0) circle (1.5pt);
\draw[red] (-0.5,-0.3) -- (0.5,0.3);
\end{tikzpicture}\hspace{0.6cm}
\begin{tikzpicture}[baseline=0mm]
\draw (-0.5,0.5) -- (0.5,-0.5);
\fill (0,0) circle (1.5pt);
\draw[red] (-0.5,-0.3) -- (0.5,0);
\end{tikzpicture}
    \caption{Edges with $-$}
    \label{fig:minus-edge}
\end{figure}
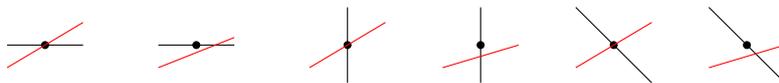
     \item [(ii)] we assign $k$ plus signs $(+)$ to the others than (i) (see Figure \ref{fig:plus-edge}),  
 \begin{figure}[ht]
    \centering
\begin{tikzpicture}[baseline=0mm]
\draw (-0.5,0) -- (0.5,0);
\fill (0,0) circle (1.5pt);
\draw[red] (-0.5,-0.1) -- (0.5,0.3);
\end{tikzpicture}
\hspace{0.7cm}
\begin{tikzpicture}[baseline=0mm]
\draw (0,-0.5) -- (0,0.5);
\fill (0,0) circle (1.5pt);
\draw[red] (-0.5,-0) -- (0.5,0.3);
\end{tikzpicture}\hspace{0.6cm}
\begin{tikzpicture}[baseline=0mm]
\draw (-0.5,0.5) -- (0.5,-0.5);
\fill (0,0) circle (1.5pt);
\draw[red] (-0.5,-0) -- (0.5,0.3);
\end{tikzpicture}
    \caption{Edges with $+$}
    \label{fig:plus-edge}
\end{figure}
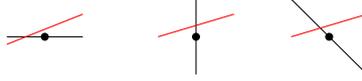
\end{itemize}
    \item [(3)] arrange the signs assigned to the triangles and edges in the order in which $L_t$ passes through them,
    \item [(4)] the sequence of integers $(a_1, \dots, a_\ell)$ is constructed from the numbers of consecutive occurrences of the same sign in the sequence of signs in (3), and we define $F^+(k,t)$ as the continued fraction $[a_1, \dots, a_\ell]$,
\end{itemize}
and we set $F^{+}\left(k,0/1\right):=1$ for any $k$. For an irreducible fraction $t\in [0,1]$, $\mathcal G (F^+(k,t))$ is called the \emph{$k$-GM snake graph} associated with $t$.
\begin{example}\label{ex:continued-fraction-Markov}
Let $t=\dfrac{2}{5}$. The signs assigned to triangles and edges in $\mathcal {PG}(2/5)$ are given as in Figure \ref{fig:ex-signed-presnakegraph}. For each $k=0,1,2,3$, $F^+(k,2/5)$ has the following expression:
\begin{align*}
    F^+(0,2/5)&=[2,1,1,2,2,1,1,2]=\frac{194}{75},\\
    F^+(1,2/5)&=[4,2,1,4,5,1,2,4]=\frac{4683}{1075},\\
    F^+(2,2/5)&=[6,3,1,6,8,1,3,6]=\frac{37636}{6013},\\
    F^+(3,2/5)&=[8,4,1,8,11,1,4,8]=\frac{176405}{21501}.
\end{align*}
\begin{figure}[ht]
    \centering
    \includegraphics[scale=0.08]{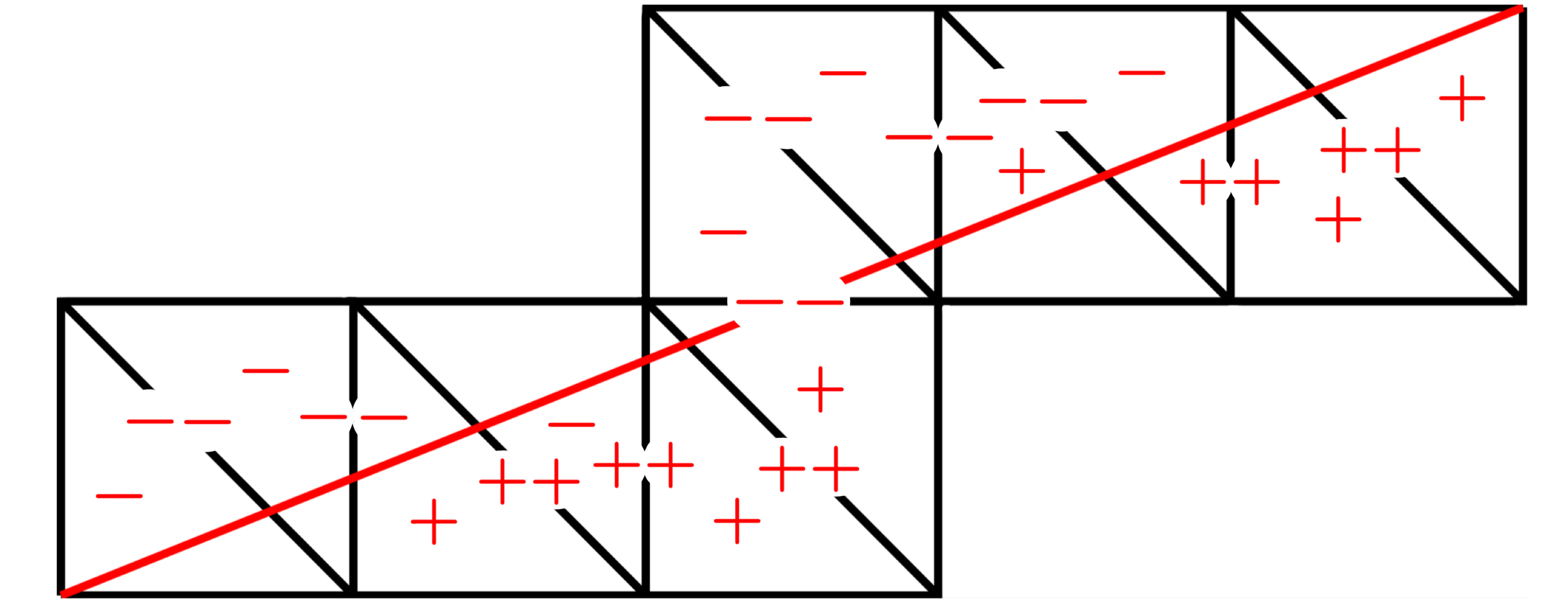}
    \caption{Signs assigned to triangles and edges in  $\mathcal {PG}(2/5)$ when $k=2$}
    \label{fig:ex-signed-presnakegraph}
\end{figure}
We can see that the numerator of $F^+(k,t)$ equals to $m_{k,t}$.
\end{example}
We note that if $F^+(k,t)=[a_1,\dots,a_\ell]$ holds for $t\in (0,1]$, then $\ell$ is even because the sequence of signs associated with $\mathcal {PG}(t)$ starts from $-$ and ends at $+$. Moreover, because of the point-symmetry of the pre-snake graph, we have the following lemma:

\begin{lemma}\label{lem:semi-palindrome}
A continued fraction $F^{+}(k,t)$ with $t\in (0,1]$ has the expression
\begin{align}
    F^{+}(k,t)=[a_1,\dots,a_{\frac{\ell}{2}},a_{\frac{\ell}{2}}+k,a_{\frac{\ell}{2}-1},\dots,a_1] \text{ or }[a_1,\dots,a_{\frac{\ell}{2}},a_{\frac{\ell}{2}}-k,a_{\frac{\ell}{2}-1},\dots,a_1],\label{eq:semi-palindrome-expression}    
\end{align}
where $a_i$ is the number of successive signs in the sign sequence obtained from (3) in the construction of $F^{+}(k,t)$. Moreover, if $\dfrac{\ell}{2}$ is even, then $F^{+}(k,t)$ takes the form given by the first expression in \eqref{eq:semi-palindrome-expression}; otherwise, $F^{+}(k,t)$ takes the form given by the second expression.
\end{lemma}

The expression \eqref{eq:semi-palindrome-expression} of a continued fraction is called the \emph{canonical semi-palindrome expression}. In this paper, when expressing $F^+(k,t)$ as a continued fraction, we use the canonical semi-palindrome expression.

\begin{remark}\label{rmk:reverse-sign}
Let $F^+(k, t) = [a_1, \dots, a_\ell] $. We consider constructing the continued fraction by using the same procedure as obtaining $ F^+(k, t) $ after changing the $ k $ negative signs associated with the center edge in $ \mathcal{PG}(t)$ to positive signs. Then we obtain the continued fraction $[a_\ell, \dots, a_1]$ by Lemma \ref{lem:semi-palindrome}.   
\end{remark}

Our goal in this subsection is the following theorem:

\begin{theorem}\label{thm:M_t-C_t-combinatorics}
For any $t\in (0,1]$, we set $F^{+}(k,t)=[a_1,\dots,a_\ell]$. The following equalities hold:
\begin{itemize}\setlength{\leftskip}{-15pt}
    \item[(1)] 

    $M_t(k,0)=\begin{bmatrix}
        -m(\mathcal{G}[a_1,\dots,a_{\ell-1}]) & m(\mathcal{G}[a_1,\dots,a_{\ell}])\\
        -m(\mathcal{G}[a_2,\dots,a_{\ell-1}]) &m(\mathcal{G}[a_2,\dots,a_{\ell}])
    \end{bmatrix},$\vspace{3mm}
    \item[(2)]$C_t(k,-k)=\begin{bmatrix}
        m(\mathcal{G}[a_2,\dots,a_{\ell}]) & m(\mathcal{G}[a_1,\dots,a_{\ell}])\\[0.5em]
        \hspace{-1cm}(3k+3)m(\mathcal{G}[a_2,\dots,a_{\ell}]) & \hspace{-1cm}(3k+3)m(\mathcal{G}[a_1,\dots,a_{\ell}])\\
        \hspace{2cm}-m(\mathcal{G}[a_2,\dots,a_{\ell-1}])&  \hspace{2cm} - m(\mathcal{G}[a_1,\dots,a_{\ell-1}])
    \end{bmatrix}.$
\end{itemize}
In particular, we have $m(\mathcal{G}[a_1,\dots,a_{\ell}])=m_{k,t}$.
\end{theorem}

Combining Theorem \ref{thm:M_t-C_t-combinatorics} and Theorem \ref{thm:snakegraph-continuedfraction}, we have the following corollary:

\begin{corollary}\label{cor:k-gen.snake-markov}
We denote by $N(k,t)$ the numerator of $F^{+}(k,t)$. The following statements hold:
\begin{itemize}\setlength{\leftskip}{-15pt}
    \item[(1)] for any irreducible fraction $t\in [0,1]$, we have $N(k,t)=m_{k,t}$,
    \item[(2)] for any $k$-GM number $b$, there exists $t\in [0,1]$ such that $b=N(k,t)$,
    \item[(3)] for any $r,t,s\in [0,1]$, $(N(k,r),N(k,t),N(k,s))$ is in $\mathrm{LM}\mathbb T(k)$ if and only if $(r,t,s)$ is in $\mathrm{F}\mathbb T$. 
\end{itemize}
\end{corollary}
\subsection{Proof of Theorem \ref{thm:M_t-C_t-combinatorics}}
To prove Theorem \ref{thm:M_t-C_t-combinatorics}, we will give a relation between $F^{+}(k,r)$, $F^{+}(k,t)$, and $F^{+}(k,s)$ for a Farey triple $(r,t,s)$.
\begin{proposition}\label{prop:presnake-relation}
 For a Farey triple $(r,t,s)$ with $t\in (0,1)$, the following three statements hold:
 \begin{itemize}\setlength{\leftskip}{-15pt}
     \item[(1)] We assume that $r=\dfrac{0}{1}$ and $s\neq \dfrac{1}{1}$. If $F^{+}(k,s)=[b_{1},\dots,b_{m}]$, then we have\[F^{+}(k,t)=[2k+2,1,b_{m}-1,b_{m-1},\dots,b_{1}].\]
      \item[(2)] We assume that $r\neq \dfrac{0}{1}$ and $s= \dfrac{1}{1}$. If $F^{+}(k,r)=[a_{1},\dots,a_{\ell}]$, then we have\[F^{+}(k,t)=[a_{\ell},\dots,a_{1},3k+2,k+2].\]
     \item[(3)] We assume that $r\neq \dfrac{0}{1}$ and $s\neq \dfrac{1}{1}$. If $F^{+}(k,r)=[a_{1},\dots,a_{\ell}]$ and $F^{+}(k,s)=[b_{1},\dots,b_{m}]$, then we have\[F^{+}(k,t)=[a_{\ell},\dots,a_{1},3k+2,1,b_{m}-1,b_{m-1},\dots,m_1].\]
 \end{itemize}
\end{proposition}

\begin{proof}[Proof of Proposition \ref{prop:presnake-relation} (1) and (2)]
First, we will prove (1). Under the assumption $r=\dfrac{0}{1}$ and $s\neq \dfrac{1}{1}$, there exists $p\in \mathbb Z_{> 1}$ such that $s=\dfrac{1}{p}$. We will prove the statement for $t=\dfrac{1}{p+1}$. Since $p\geq 2$, the first $2+2k$ signs in $\mathcal{PG}(1/(p+1))$ are $-$, and the next sign is $+$ (see Figure \ref{fig:presnake-1-8}). 
\begin{figure}[ht]
    \centering
    \includegraphics[scale=0.25]{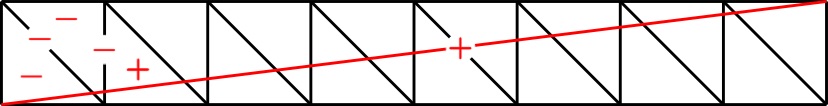}
    \caption{$\mathcal{PG}(1/(p+1))$ with $k=1$ and $p=7$}
    \label{fig:presnake-1-8}
\end{figure}

Let us compare the sign sequence of $\mathcal{PG}(1/(p+1))$ after removing the first tile and that of $\mathcal{PG}(1/p)$ (compare Figures \ref{fig:presnake-1-8} and \ref{fig:presnake-1-7}). We denote by $\mathcal{SPG}(1/(p+1))$ the former graph.
\begin{figure}[ht]
    \centering
    \includegraphics[scale=0.25]{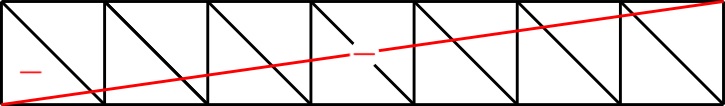}
    \caption{$\mathcal{PG}(1/p)$ with $k=1$ and $p=7$}
    \label{fig:presnake-1-7}
\end{figure}

 We will prove that only the signs associated with the first triangles and the central edges in $\mathcal{SPG}(1/(p+1))$ and $\mathcal{PG}(1/p)$ differ. It is clear that signs at these two places differ. We will prove that all other signs coincide. Clearly, the signs assigned to right-angled triangles coincide. We will consider signs assigned to vertical edges. The height of the intersection point between the $(a+1)$-th vertical edge from the left of $\mathcal{PG}(1/p)$ and the line segment $L_\frac{1}{p}$ is $\dfrac{a}{p}$. Moreover, the height of the intersection point between the $(a+1)$-th vertical edge from the left of $\mathcal{SPG}(1/(p+1))$ and the line segment $L_\frac{1}{p+1}$ is $\dfrac{a+1}{p+1}$. Since $\dfrac{a}{p} < \dfrac{a+1}{p+1}$, it is sufficient to show $\dfrac{a+1}{p+1} \leq \dfrac{1}{2}$ if $\dfrac{a}{p} < \dfrac{1}{2}$. Since $2a\leq p-1$, we have
\[\dfrac{1}{2}-\dfrac{a+1}{p+1}=\dfrac{p-2a-1}{2(p+1)}\geq 0,\]
as desired. We can prove about signs assigned to diagonal edges in the same way. Therefore, only the signs associated with the first triangles and the central edges in $\mathcal{SPG}(1/(p+1))$ and $\mathcal{PG}(1/p)$ differ. By Remark \ref{rmk:reverse-sign}, the continued fraction constructed from the sign sequence of $\mathcal{SPG}(1/(p+1))$ is $[1,b_m-1,b_{m-1},\dots, b_1]$. Combining the continued fraction given by the $2+2k$ signs associated with the initial tile in $\mathcal {PG}(1/(p+1))$, we obtain the claim.

Second, we will prove (2). Under the assumption $r\neq\dfrac{0}{1}$ and $s=\dfrac{1}{1}$, there exists $p\in \mathbb Z_{> 1}$ such that $s=\dfrac{p}{p+1}$. We will prove the statement for $t=\dfrac{p+1}{p+2}$. Since $p\geq 2$, the last $4k+4$ signs in $\mathcal{PG}((p+1)/(p+2))$ are $3k+2$ minus signs and $k+2$ plus signs (see Figure \ref{fig:presnake-4-5}).
\begin{figure}[ht]
    \centering
    \includegraphics[scale=0.25]{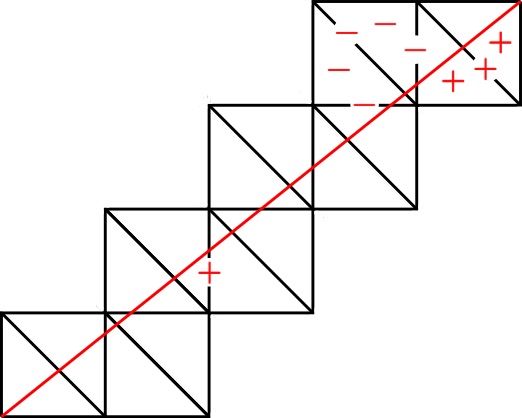}
    \caption{$\mathcal{PG}((p+1)/(p+2))$ with $k=1$ and $p=3$}
    \label{fig:presnake-4-5}
\end{figure}

Let us compare the sign sequence of $\mathcal{PG}((p+1)/(p+2))$ after removing the last two tiles and that of $\mathcal{PG}(p/(p+1))$ (compare Figures \ref{fig:presnake-4-5} and \ref{fig:presnake-3-4}). We denote by $\mathcal{SPG}((p+1)/(p+2))$ the former graph.
\begin{figure}[ht]
    \centering
    \includegraphics[scale=0.25]{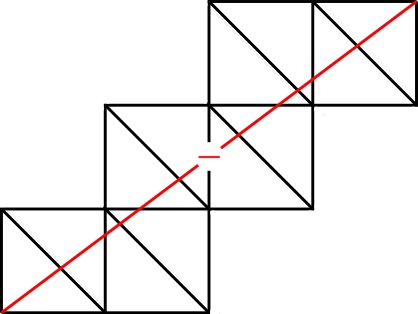}
    \caption{$\mathcal{PG}(p/(p+1))$ with $k=1$ and $p=3$}
    \label{fig:presnake-3-4}
\end{figure}

We will prove that only the signs associated with the central edges in $\mathcal{SPG}((p+1)/(p+2))$ and $\mathcal{PG}(p/(p+1))$ differ. It is clear that signs at this places differ. We will prove that all other signs coincide. Clearly, the signs assigned to right-angled triangles coincide. The height of the intersection point of the $(a+1)$-th vertical edge from the left of $\mathcal{PG}(p/(p+1))$ and the line segment $L_\frac{p}{p+1}$ is $\dfrac{ap}{p+1}$, and the height of the intersection point of the $(a+1)$-th vertical edge from the left of $\mathcal{SPG}((p+1)/(p+2))$ and the line segment $L_\frac{p+1}{p+2}$ is $\dfrac{a(p+1)}{p+2}$. It is sufficient to show that if $\dfrac{ap}{p+1}-(a-1)<\dfrac{1}{2}$, then $\dfrac{a(p+1)}{p+2}-(a-1)\leq\dfrac{1}{2}$ holds. Since $2a\geq p+2$, we have
\[\dfrac{1}{2}-\dfrac{a(p+1)}{p+2}+(a-1)=\dfrac{-p+2a-2}{2(p+2)}\geq 0,\]
as desired. We can prove about signs assigned to diagonal edges and horizontal edges in the same way. The discussion afterward is the same as in case (1). 
\end{proof}

To prove Proposition \ref{prop:presnake-relation} (3), we recall the Christoffel word. We denote by $\{A,B\}^*$ the set of words consisting of $A$ and $B$. Let $\dfrac{a}{b}$ be an irreducible fraction. For $1\leq i\leq b+1$, we denote by $y_i$ the height of the intersection point of $L_t$ and the $i$-th vertical line from the left in $\mathcal{PG}(t)$. We denote the integer part of $x$ by $\lfloor x \rfloor$. The \emph{Christoffel word $\mathrm{ch}_{a/b}$ associated with $\dfrac{a}{b}$} is defined as $\mathrm{ch}_{a/b}:=w_1\cdots w_{b}\in\{A,B\}^\ast$, where
\begin{align*}
    w_i=\begin{cases}
        A & \text{if $\lfloor y_{i+1} \rfloor-\lfloor y_i \rfloor=0$},\\
        B & \text{if $\lfloor y_{i+1} \rfloor-\lfloor y_i \rfloor=1$}.
    \end{cases}
\end{align*}

\begin{example}
The Christoffel word $\mathrm{ch}_{2/5}$ is $AABAB$. See also Figure \ref{fig:christoffel}.
\begin{figure}[ht]
    \centering
    \begin{tikzpicture}
    % Draw grid
    \draw[thick] (0,0) grid (5,2);
    % Draw line with slope 2/5
    \draw[red,thick] (0,0) -- (5,2);
    \draw[line width=2pt] (0,0) -- (2,0)-- (3,1) -- (4,1) -- (5,2);
    \node at (0.5,-0.5) {$A$};
    \node at (1.5,-0.5) {$A$};
    \node at (2.5,-0.5) {$B$};
    \node at (3.5,-0.5) {$A$};
    \node at (4.5,-0.5) {$B$};
\end{tikzpicture}
    \caption{Christoffel word $\mathrm{ch}_{2/5}$}
    \label{fig:christoffel}
\end{figure}
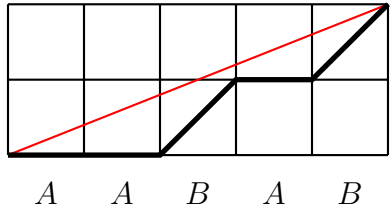
\end{example}

In \cite{aig}*{Theorem 7.6}, the following theorem about the Christoffel word is proved by using the argument based on the \emph{Cohn word}.

\begin{theorem}[\cite{aig}*{Theorem 7.6}]\label{thm:christoffel}
For $(r,t,s)\in \mathrm{F}\mathbb{T}$, we have
\[\mathrm{ch}_r\cdot\mathrm{ch}_s=\mathrm{ch}_t,\]
where $\cdot$ means the concatenation of words. Moreover, if $r=\dfrac{a}{b}$ and $s=\dfrac{c}{d}$ (thus $t=\dfrac{a+c}{b+d}$) hold, then we have $y_{b+1}-\lfloor y_{b+1} \rfloor=\dfrac{1}{b+d}$, where $y_i$ is the height of the intersection point of $L_t$ and the $i$-th vertical line from the left in $\mathcal{PG}(t)$.
\end{theorem}

From the above theorem, we have the following decomposition of the pre-snake graph.

\begin{corollary}\label{lem:decomposition-lemma}
 For $(r,t,s)\in \mathrm{F}\mathbb{T}$ with $t\in (0,1)$, $\mathcal{PG}(t)$ is decomposed into $\mathcal{PG}(r)$, a tile, and $\mathcal{PG}(s)$ in the order from the lower left to the upper right.
\end{corollary}

\begin{proof}[Proof of Proposition \ref{prop:presnake-relation} (3)]
By Corollary \ref{lem:decomposition-lemma}, $\mathcal{PG}(t)$ is decomposed into $\mathcal{PG}(r)$, a tile, and $\mathcal{PG}(s)$ (see Figure \ref{fig:decomposition}). We denote by $\mathcal{SPG}(r)$ (resp. $\mathcal{SPG}(s)$) the $\mathcal{PG}(r)$- (resp.  $\mathcal{PG}(s)$-)part in $\mathcal{PG}(t)$. Note that signs assigned to the tile of the second component in the decomposition are all $-$.
\begin{figure}[ht]
    \centering
    \includegraphics[scale=0.1]{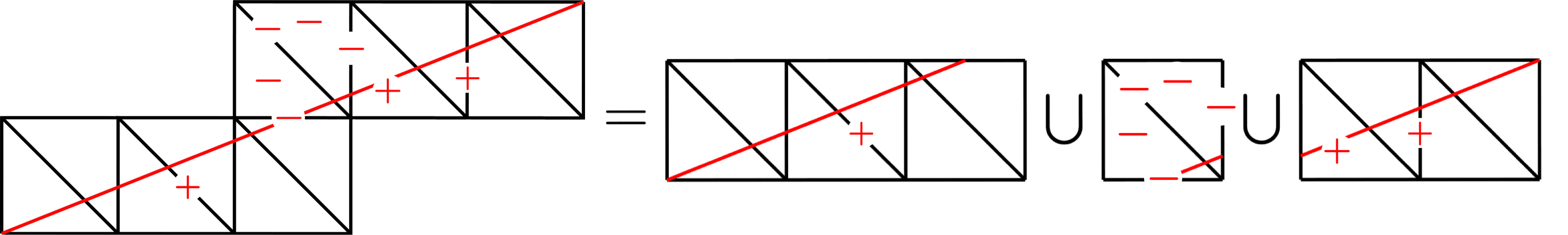}
    \caption{Decomposition of pre-snake graph with $r=\dfrac{1}{3}$, $t=\dfrac{2}{5}$, $s=\dfrac{1}{2}$}
    \label{fig:decomposition}
\end{figure}

We assume that $r=\dfrac{a}{b}, s=\dfrac{c}{d}$. By Theorem \ref{thm:christoffel}, the length from the intersection point with the first vertical edge of $\mathcal{SPG}(s)$ and $L_{t}$ to the bottom endpoint is $\dfrac{1}{b+d}$, and the length from the intersection point with the rightmost horizontal edge of $\mathcal{SPG}(r)$ and $L_{t}$ to the right endpoint of $\mathcal{SPG}(r)$ is $\dfrac{1}{a+c}$. Comparing the sign sequence of $\mathcal{SPG}(r)$ and that of $\mathcal{PG}(r)$, we can see that all signs except for ones at the center edges coincide by using the same argument in (2). Furthermore, comparing the sign sequence of $\mathcal{SPG}(s)$ and that of $\mathcal{PG}(s)$, we can see that all signs except for ones at the first triangles and center edges coincide by using the same argument in (1). Therefore, we obtain the conclusion.
\end{proof}

We have the following corollary of Proposition \ref{prop:presnake-relation}:

\begin{corollary}
 For a Farey triple $(r,t,s)$ with $t\in (0,1)$, the following two statements hold:
 \begin{itemize}\setlength{\leftskip}{-15pt}
     \item[(1)] We assume that $r=\dfrac{0}{1}$ and $s\neq \dfrac{1}{1}$. If $F^{+}(k,t)=[b_{1},\dots,b_{m}]$, then we have\[F^{+}(k,s)=[b_{m},b_{m-1},\dots,b_{3}+1].\]
      \item[(2)] We assume that $r\neq \dfrac{0}{1}$ and $s= \dfrac{1}{1}$. If $F^{+}(k,t)=[a_{1},\dots,a_{\ell}]$, then we have\[F^{+}(k,r)=[a_{\ell-2},a_{\ell-3}\dots,a_1].\]
\end{itemize}      
\end{corollary}

By using Proposition \ref{prop:presnake-relation}, we prove Theorem \ref{thm:M_t-C_t-combinatorics} (1).

\begin{proof}[Proof of Theorem \ref{thm:M_t-C_t-combinatorics} (1)]
In this proof, we abbreviate $M_t(k,0)$ to $M_t$, and $m(\mathcal{G}[a_1,\dots,a_\ell])$ to $m(a_1,\dots,a_\ell)$. We will prove the following four cases: (0) $r=\dfrac{0}{1}, s=\dfrac{1}{1}$, (1) $r=\dfrac{0}{1}, s\neq\dfrac{1}{1}$ (2) $r\neq\dfrac{0}{1}, s=\dfrac{1}{1}$, (3) $r\neq \dfrac{0}{1}, s\neq\dfrac{1}{1}$.

We prove the case (0). Now, $t=\dfrac{1}{2}$ holds.  Since $M_{\frac{1}{2}}=\tilde{Y}_0$ in Section 5, we have
\[M_{\frac{1}{2}}=\begin{bmatrix} -(2k+2)&2k^2+6k+5\\-1&k+2 \end{bmatrix}.\]
On the other hand, we have $F^{+}(k,1/2)=[2k+2,k+2]$. Since 
\[[2k+2,k+2]=\dfrac{2k^2+6k+5}{k+2}\]
hold, by $m(k+2)=k+2$ and Theorem \ref{thm:snakegraph-continuedfraction}, we have  \[m(2k+2,k+2)=2k^2+6k+5.\] Therefore, we have 
\[M_{\frac{1}{2}}=\begin{bmatrix} -m(2k+2)&m(2k+2,k+2)\\-m(\ )&m(k+2)\end{bmatrix},\]
as desired.

Next, we will prove the case (1). There exists $p\in \mathbb Z_{>1}$ such that $s=\dfrac{1}{p}$ and $t=\dfrac{1}{p+1}$. We will prove the statement by using induction on $p$. When $p=2$, $M_{\frac{1}{2}}$ satisfies the statement by the argument in the case (0). We assume that $M_{\frac{1}{p}}$ satisfies the statement, and prove that $M_{\frac{1}{p+1}}$ also satisfies the statement. We set
\[M_{\frac{1}{p}}=\begin{bmatrix}
    -m(b_1,\dots,b_{m-1})& m(b_1,\dots,b_{m})\\ -m(b_2,\dots,b_{m-1})& m(b_2,\dots,b_{m})\end{bmatrix}.\]
Since $M_{\frac{0}{1}}=\tilde{X}_0=\begin{bmatrix}
    0&1\\-1&k
\end{bmatrix}$ and $M_{\frac{1}{p+1}}=M^{-1}_{\frac{0}{1}}TM_{\frac{1}{p}}^{-1}$ hold, where $T=\begin{bmatrix}
    -1&0\\3k+3&-1
\end{bmatrix}$, we have
\[M_{\frac{1}{p+1}}=\begin{bmatrix}
        \hspace{-1cm}-((2k+3)m(b_2,\dots,b_{m}) & \hspace{-1cm}(2k+3)m(b_1,\dots,b_{m})\\
        \hspace{2cm}-m(b_2,\dots,b_{m-1}))&  \hspace{2cm} - m(b_1,\dots,b_{m-1})\\[0.5em]
        -m(b_2,\dots,b_{m}) & m(b_1,\dots,b_{m})
    \end{bmatrix}. \]
On the other hand, since $F^+(k,1/(p+1))=[2k+2,1,b_{m}-1,b_{m-1},\dots,b_{1}]$ by Proposition \ref{prop:presnake-relation}, it suffices to show the following four equalities: 
\begin{align}
    m(2k+2,1,b_{m}-1,b_{m-1},\dots,b_{2})&=(2k+3)m(b_2,\dots,b_{m})-m(b_2,\dots,b_{m-1}),\label{eq:1-1}\\
    m(2k+2,1,b_{m}-1,b_{m-1},\dots,b_{1})&=(2k+3)m(b_1,\dots,b_{m})-m(b_1,\dots,b_{m-1}),\label{eq:1-2}\\
    m(1,b_{m}-1,b_{m-1},\dots,b_{2})&=m(b_2,\dots,b_m),\label{eq:1-3}\\
    m(1,b_{m}-1,b_{m-1},\dots,b_{1})&=m(b_1,\dots,b_m).\label{eq:1-4}
\end{align}
First, we prove \eqref{eq:1-3} and \eqref{eq:1-4}. Since the graph $\mathcal{G}[1,b_m-1,b_{m-1},\dots,b_2]$ is congruent to $\mathcal{G}[b_2,\dots,b_{m-1},b_{m}-1,1]$, we have
\[m(1,b_m-1,b_{m-1},\dots,b_2)=m(b_2,\dots,b_{m-1},b_{m}-1,1).\]
Moreover, by the construction of the snake graph, $\mathcal{G}[b_2,\dots,b_m]$ coincides with $\mathcal{G}[b_2,\dots,b_{m}-1,1]$ and thus \[m(b_2,\dots,b_m)=m(b_2,\dots,b_{m}-1,1).\]
Therefore, we have \eqref{eq:1-3}. We can obtain \eqref{eq:1-4} by the same argument. Second, we prove \eqref{eq:1-1} and \eqref{eq:1-2}. The snake graph $\mathcal{G}[2k+2,1,b_m-1,b_{m-1}\dots,b_2]$ is given as in Figure \ref{fig:case1} when $b_m$ is even, and in Figure \ref{fig:case1a} when $b_m$ is odd. Since there is no difference in the argument in either case, the discussion will proceed in the case where $b_m$ is even (and we will use Figure \ref{fig:case1} and will not use Figure \ref{fig:case1a}). 
The diagram consisting of the first $2k+2$ tiles is called the \emph{tail} in $\mathcal{G}[2k+2,1,b_m-1,b_{m-1}\dots,b_2]$. The graph obtained by removing the tail from $\mathcal {G}[2k+2,1, b_{m}-1, b_{m-1}, ..., b_2]$ is isomorphic to $\mathcal {G}[b_{m}, ..., b_2]$. The left ``$\cdots$" part and the middle ``$\cdots$" part in Figure \ref{fig:case1} form  staircases because they consist of the same signs in succession. 
\begin{figure}[ht]
    \centering
    \includegraphics[scale=0.2]{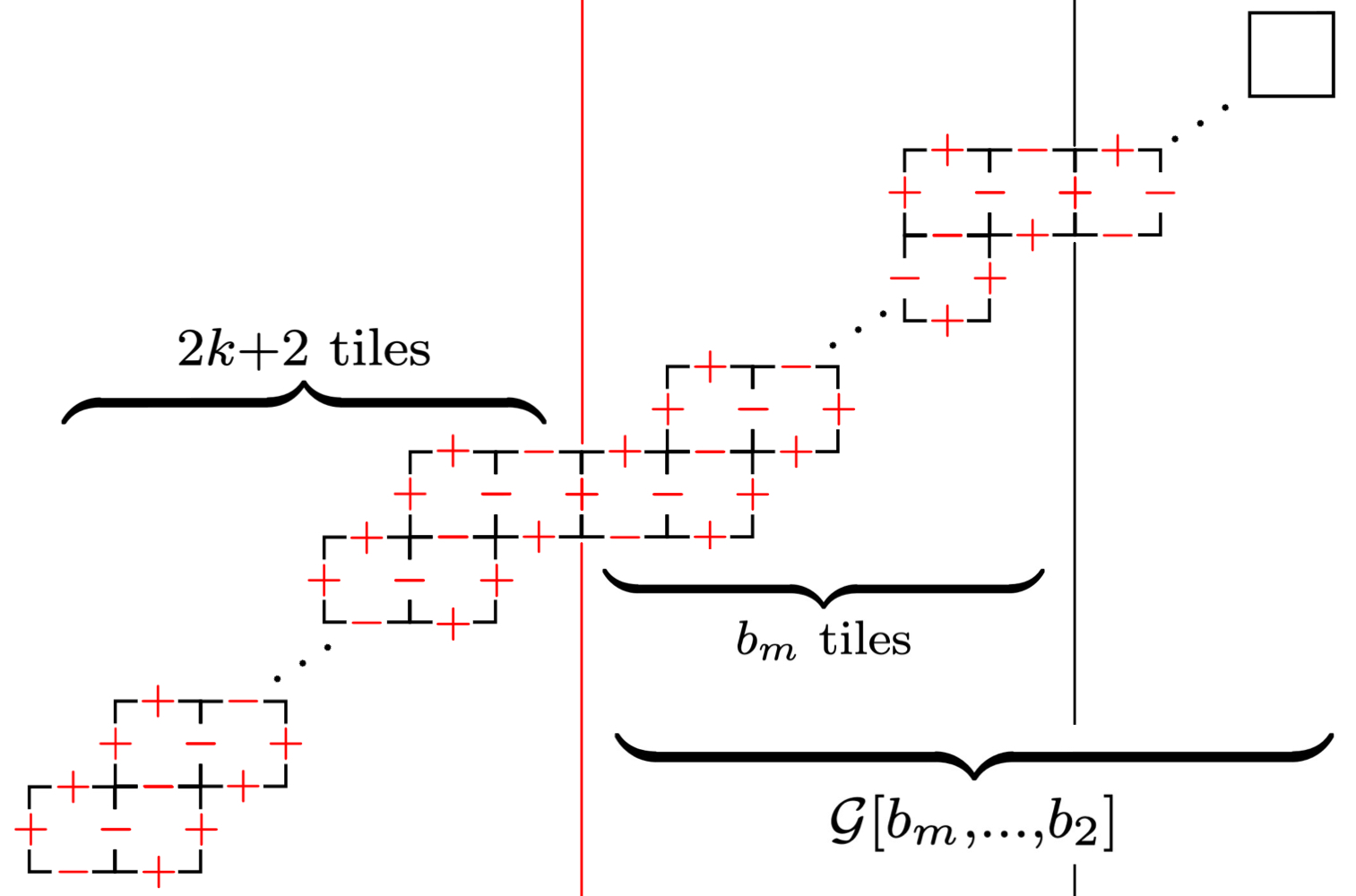}
    \caption{Snake graph $\mathcal{G}[2k+2,1,b_{m}-1,b_{m-1},\dots,b_{2}]$ when $b_m$ is even}
    \label{fig:case1}
\end{figure}
\begin{figure}[ht]
    \centering
    \includegraphics[scale=0.24]{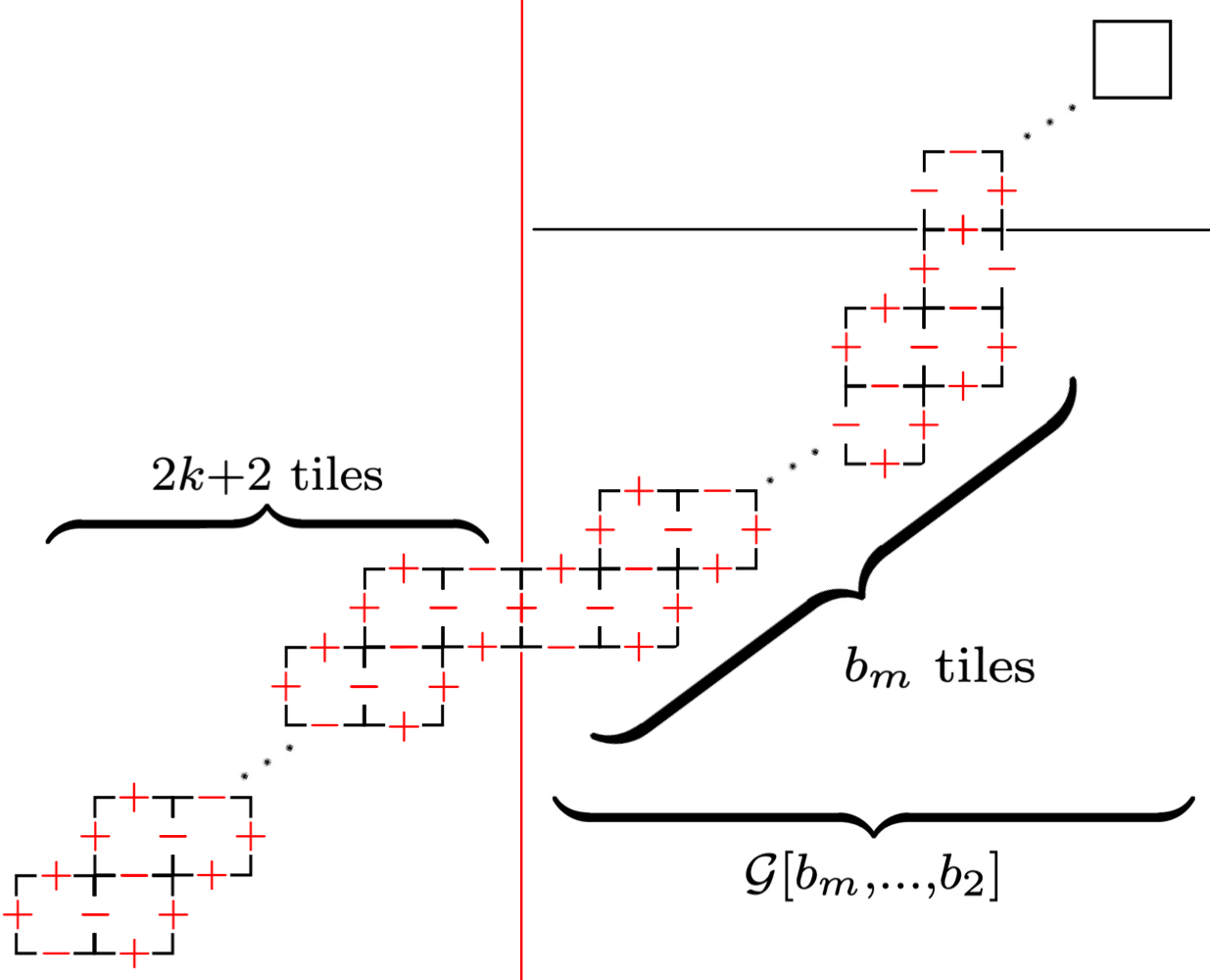}
    \caption{Snake graph $\mathcal{G}[2k+2,1,b_{m}-1,b_{m-1},\dots,b_{2}]$ when $b_m$ is odd}
    \label{fig:case1a}
\end{figure}

We will count the number of perfect matchings of $\mathcal{G}[2k+2,1,b_m-1,b_{m-1}\dots,b_2]$. Any perfect matching of this snake graph belongs to exactly one of the two cases described below:
\begin{itemize}\setlength{\leftskip}{-5pt}
    \item [(1-I)] it contains a perfect matching of the $\mathcal {G}[b_{m}, ..., b_2]$-part,
    \item [(1-II)] it contains the upper and the lower edges in the rightmost tile in the tail.
\end{itemize}
It is not possible that any perfect matching does not belong to either (1-I) or (1-II). Let us first explain this.
Suppose a perfect matching $P$ includes the lower edge of the rightmost tile in the tail and does not include the upper edge of that tile. Then, a subset of vertices of the tail obtained by removing the upper rightmost vertex from the set of vertices of the tail will be covered by a subset of $P$, but this is contradictory since the cardinality of the vertex set is an odd number. The same goes for the case that $P$ includes the upper edge of the rightmost tile of the tail but does not include the lower edge of that tile.

We will count the number of the perfect matchings belonging to (1-I). 
In this case, the perfect matching refers to combinations of perfect matchings within the red region and the blue region in Figure \ref{fig:case1-c1}.

\begin{figure}[ht]
    \centering
    \includegraphics[scale=0.2]{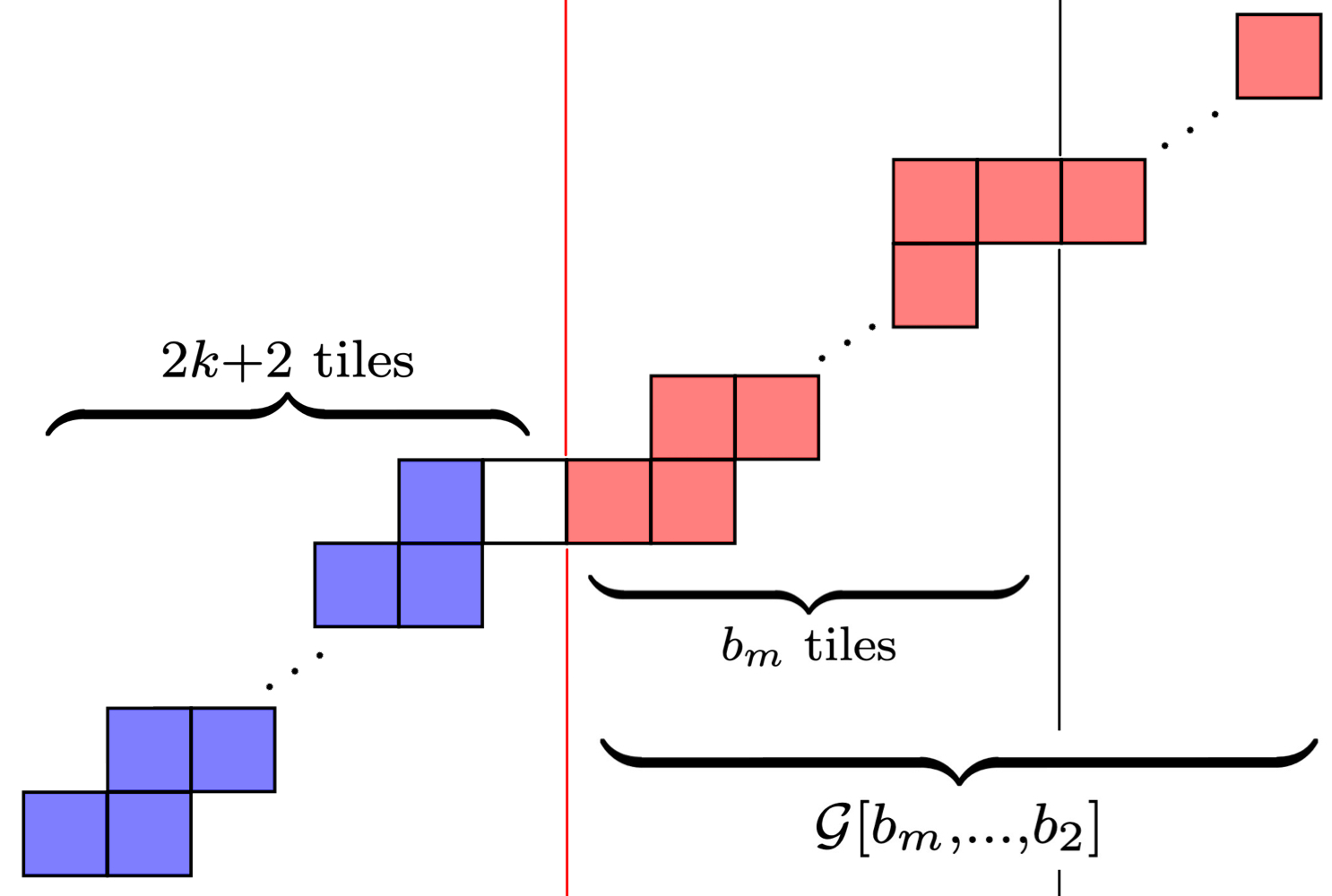}
    \caption{Case (1-I)}
    \label{fig:case1-c1}
\end{figure}
Therefore, the number of such matchings is given by the product of the number of perfect matchings in $\mathcal{G}[b_m,\dots,b_2]$ and the number of perfect matchings in the graph consisting of the first $(2k+1)$ tiles of the tail. Since the latter graph is $\mathcal{G}[2k+2]$, the number of perfect matchings belonging to (1-I) is
\[m(b_m,\dots,b_2)m(2k+2)=(2k+2)m(b_2,\dots,b_m).\]
Next, we will count the number of the perfect matchings belonging to (1-II). In this case, edges of the tail in the perfect matching is uniquely determined. Therefore, this number coincides with the number of perfect matchings of the graph removed the leftmost tile from $\mathcal{G}[b_m,\dots,b_2]$ (see Figure \ref{fig:case1-c2}). We denote this graph by $\mathcal{G'}$.
\begin{figure}[ht]
    \centering
    \includegraphics[scale=0.2]{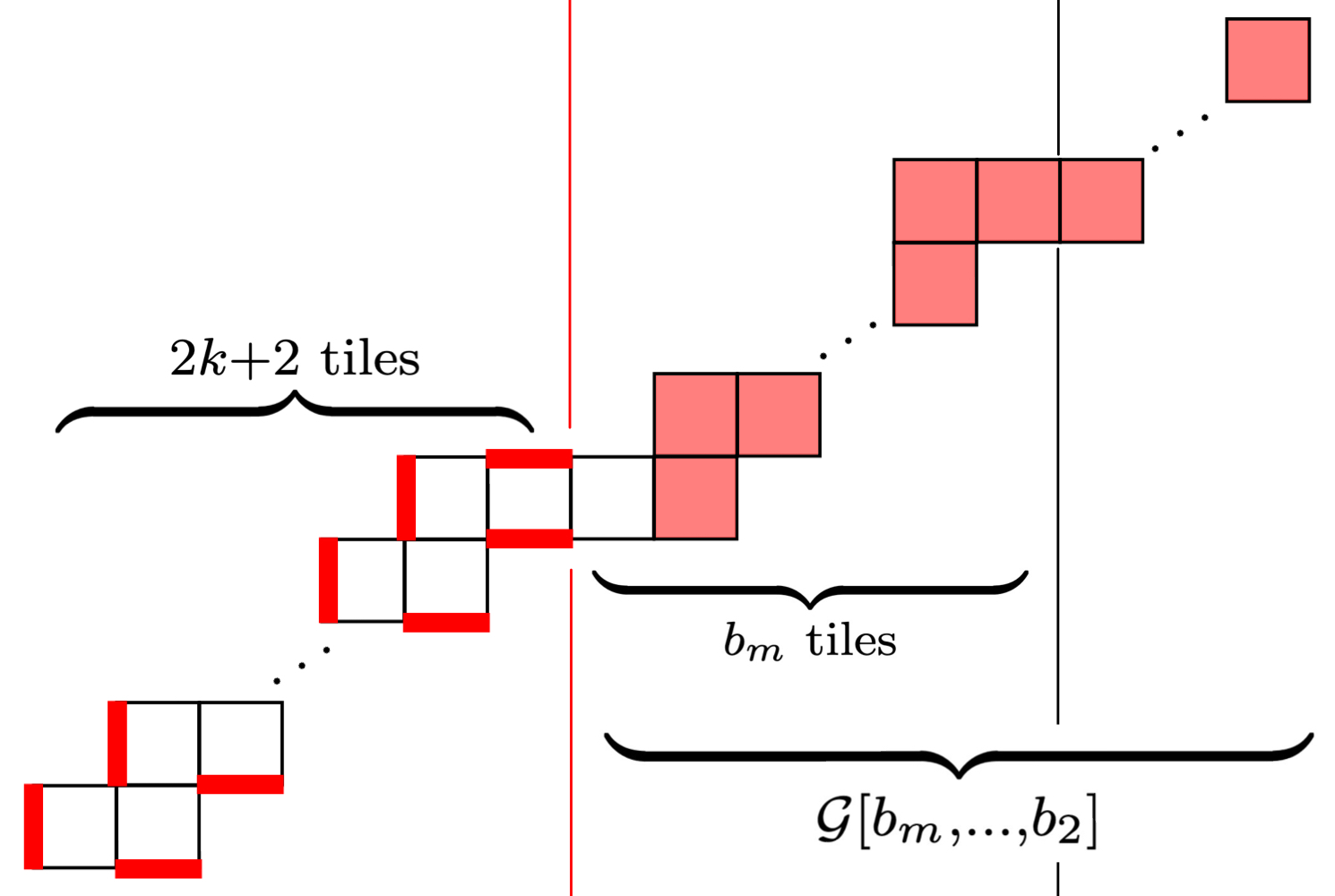}
    \caption{Case (1-II)}
    \label{fig:case1-c2}
\end{figure}

We will calculate $m(\mathcal G')$. Any perfect matching of $\mathcal{G}[b_m,\dots,b_2]$ belongs to exactly one of the two cases described below:
\begin{itemize}\setlength{\leftskip}{5pt}
    \item [(1-II-i)] it contains the left most vertical edge,
    \item [(1-II-ii)] it contains the upper and the lower edges in the leftmost tile.
\end{itemize}
The number of perfect matchings belonging to (1-II-i) coincides with $m(\mathcal G')$. The number of perfect matchings belonging to (1-II-ii) coincides with $m(b_{m-1},\dots,b_2)$. Indeed, edges in a perfect matchings belonging to (1-II-ii) other than $\mathcal{G}[b_{m-1},\dots,m_2]$-part are determined uniquely (see Figure \ref{fig:case1-c2-1}). Therefore, we have
\[m(b_{m},\dots,b_2)=m(\mathcal{G}')+m(b_{m-1},\dots,b_2),\]
hence
\[m(\mathcal{G}')=m(b_{m},\dots,b_2)-m(b_{m-1},\dots,b_2)=m(b_{2},\dots,b_m)-m(b_{2},\dots,b_{m-1}).\]
\begin{figure}[ht]
    \centering
    \includegraphics[scale=0.15]{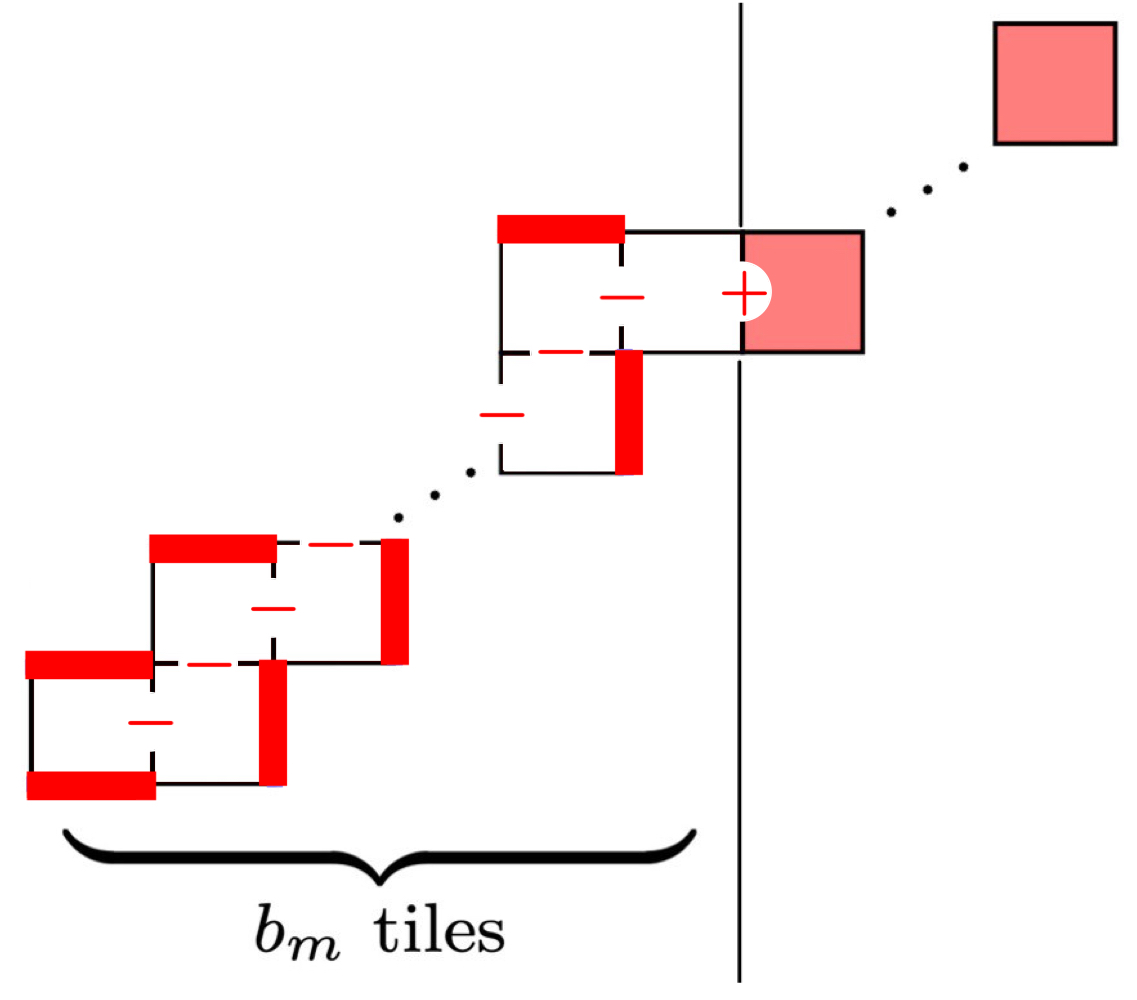}
    \caption{Case (1-II-ii)}
    \label{fig:case1-c2-1}
\end{figure} 
Therefore, combining results on (1-I) and (1-II), we have
\[m(2k+2,1,b_{m}-1,b_{m-1},\dots,b_{2})=(2k+3)m(b_2,\dots,b_{m})-m(b_2,\dots,b_{m-1}),\]
as desired. We can prove \eqref{eq:1-2} in the same way.

Next, we will prove the case (2). There exists $p\in \mathbb Z_{\geq 1}$ such that $r=\dfrac{p}{p+1}$ and $t=\dfrac{p+1}{p+2}$. We will prove the statement by using induction on $p$. When $p=1$, $M_{\frac{1}{2}}$ satisfies the statement by the argument in the case (0). We assume that $M_{\frac{p}{p+1}}$ satisfies the statement, and prove that $M_{\frac{p+1}{p+2}}$ also satisfies the statement. We set
\[M_{\frac{p}{p+1}}=\begin{bmatrix}
    -m(a_1,\dots,a_{\ell-1})& m(a_1,\dots,a_{\ell})\\ -m(a_2,\dots,a_{\ell-1})& m(a_2,\dots,a_{\ell})\end{bmatrix}.\]
Since $M_{\frac{1}{1}}=\tilde{Z}_0=\begin{bmatrix}
    -(k+1)&k+2\\-1&1
\end{bmatrix}$ and $M_{\frac{p+1}{p+2}}=M^{-1}_{\frac{p}{p+1}}TM_{\frac{1}{1}}^{-1}$, by the definition of the $k$-MM triple, we have
\[M_{\frac{1}{p+1}}=\begin{bmatrix}
        \hspace{-0.6cm}-((3k+2)m(a_1,\dots,a_{\ell}) & \hspace{-1cm}(3k^2+8k+5)[a_1,\dots,a_{\ell}]\\
        \hspace{2cm}+m(a_2,\dots,a_{\ell}))&  \hspace{2cm} +(k+2)m(a_2,\dots,a_{\ell})\\[0.5em] \hspace{-0.6cm}-((3k+2)m(a_1,\dots,a_{\ell-1}) & \hspace{-1cm}(3k^2+8k+5)[a_1,\dots,a_{\ell-1}]\\
        \hspace{2cm}+m(a_2,\dots,a_{\ell-1}))&  \hspace{2cm} +(k+2)m(a_2,\dots,a_{\ell-1})
    \end{bmatrix}. \]
On the other hand, since $F^+(k,(p+1)/(p+2))=[a_\ell,\dots,a_1,3k+2,k+2]$ by Proposition \ref{prop:presnake-relation}, it suffices to show the following four equalities: 
\begin{align}
    m(a_{\ell},\dots,a_{1},3k+2)&=(3k+2)m(a_1,\dots,a_{\ell})+m(a_2,\dots,a_{\ell}),\label{eq:2-1}\\
    m(a_{\ell},\dots,a_{1},3k+2,k+2)&=(3k^2+8k+5)m(a_1,\dots,a_{\ell})+(k+2)m(a_2,\dots,a_{\ell}),\label{eq:2-2}\\
    m(a_{\ell-1},\dots,a_{1},3k+2)&=(3k+2)m(a_1,\dots,a_{\ell-1})+m(a_2,\dots,a_{\ell-1}),\label{eq:2-3}\\
     m(a_{\ell-1},\dots,a_{1},3k+2,k+2)&=(3k^2+8k+5)m(a_1,\dots,a_{\ell-1})+(k+2)m(a_2,\dots,a_{\ell-1}).\label{eq:2-4}
\end{align}
Since we can prove \eqref{eq:2-3} (resp. \eqref{eq:2-4}) in the same way as \eqref{eq:2-1} (resp. \eqref{eq:2-2}), we only prove \eqref{eq:2-1} and \eqref{eq:2-2}. First, we will prove \eqref{eq:2-1}. The Figure \ref{fig:case2} is the snake graph $\mathcal G[a_{\ell},\dots,a_{1},3k+2]$ in the case that $a_\ell+\cdots+a_2$ is even, and $a_1$ and $3k+2$ are odd. Since we can also apply the same argument to other cases, we only prove the above case. The diagram consisting of the last $3k+2$ tiles is called the \emph{tail}. The graph obtained by removing the tail from $\mathcal {G}[a_{\ell},\dots,a_{1},3k+2]$ is congruent to $\mathcal {G}[a_{\ell}, ..., a_1]$. The  middle ``$\cdots$" part and the right ``$\cdots$" part form staircases because they consist of the same signs in succession.
\begin{figure}
    \centering
    \includegraphics[scale=0.2]{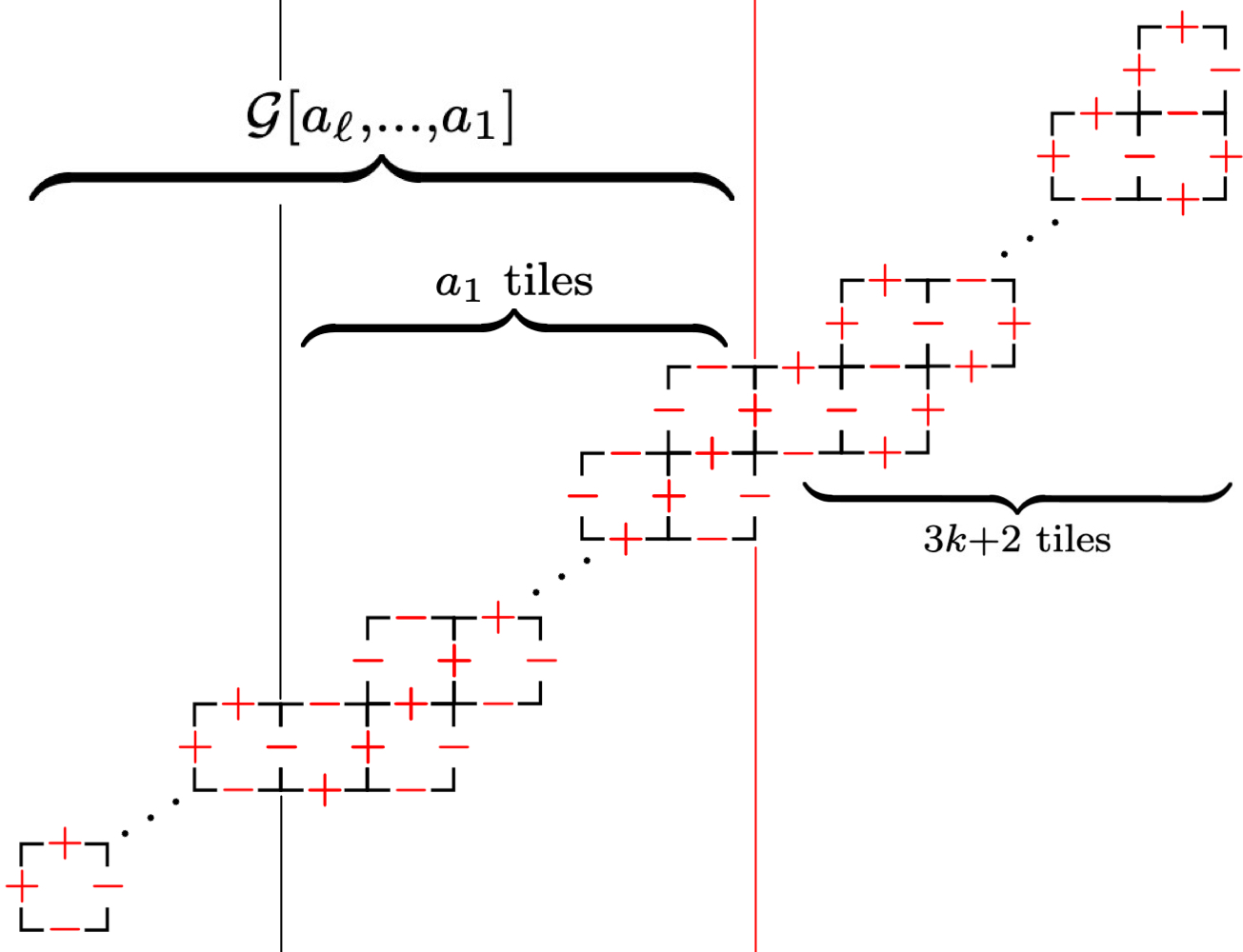}
    \caption{Snake graph $\mathcal G[a_\ell,\dots,a_1,3k+2]$}
    \label{fig:case2}
\end{figure}

We will count the number of perfect matchings of $\mathcal G[a_\ell,\dots,a_1,3k+2]$. Any perfect matching of this snake graph belongs to exactly one of the two cases described below:
\begin{itemize}\setlength{\leftskip}{5pt}
    \item [(2-a-I)] it contains a perfect matching of the $\mathcal {G}[a_{\ell}, ..., a_1]$-part,
    \item [(2-a-II)] it contains the upper and the lower edges in the leftmost tile in the tail.
\end{itemize}

We will count the number of the perfect matchings belonging to (2-a-I). In the same way as the case (1-I), we can see that this number coincides with the product of the number of perfect matchings in $\mathcal{G}[a_\ell,\dots,a_1]$ and the number of perfect matchings in the graph consisting of the last $(3k+1)$ tiles of the tail. Since the latter graph is $\mathcal{G}[3k+2]$, the number of perfect matchings belonging to (2-a-I) is
\[m(a_\ell,\dots,a_1)m(3k+2)=(3k+2)m(a_1,\dots,a_\ell).\]
We will count the number of the perfect matchings belonging to (2-a-II). If a perfect matching contains the upper and the lower edges in the leftmost tile in the tail, then edges of the tail and the rightmost $a_1$ tiles of the $\mathcal{G}[a_\ell,\dots,a_1]$-part in the perfect matching are uniquely determined. Therefore, this number coincides with the number of perfect matchings of $\mathcal{G}[a_\ell,\dots,a_2]=\mathcal{G}[a_2,\dots,a_\ell]$ (see Figure \ref{fig:case2-c2}).
Therefore, combining results on (2-a-I) and (2-a-II), we have
\[m(a_{\ell},\dots,a_{1},3k+2)=(3k+2)m(a_1,\dots,a_{\ell})+m(a_2,\dots,a_\ell),\]
\begin{figure}[ht]
    \centering
    \includegraphics[scale=0.2]{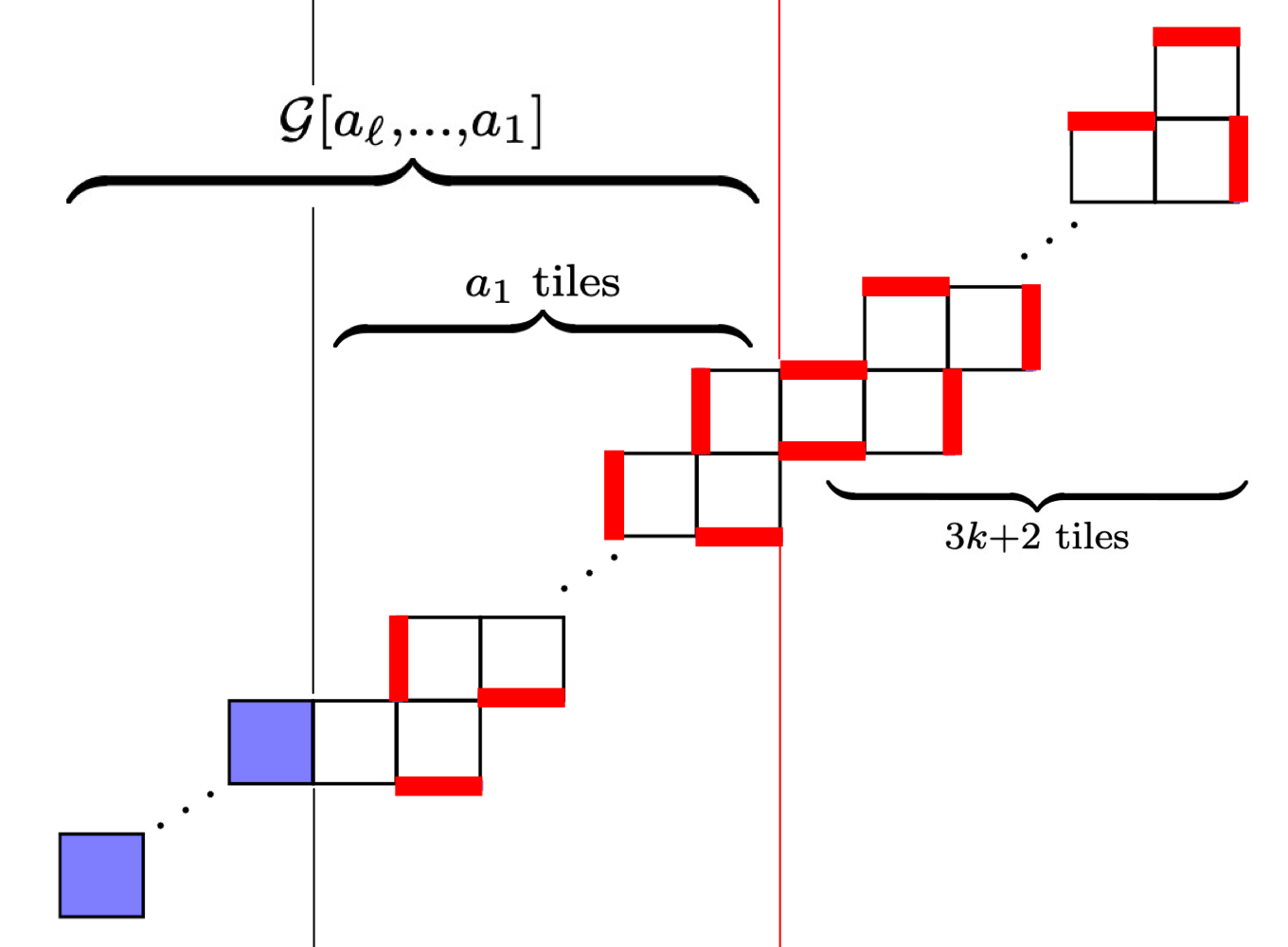}
    \caption{Case (2-a-II)}
    \label{fig:case2-c2}
\end{figure}
as desired. Next, we will prove \eqref{eq:2-2}. The snake graph $\mathcal G[a_{\ell},\dots,a_{1},3k+2,k+2]$ is given as in Figure \ref{fig:case22}. It is the figure in the case that $a_\ell+\cdots+a_2$ is even, and $a_1$ and $3k+2$ are odd. In the other cases, the shape of the snake graph is changed, but we can also apply the same argument. Here, we only prove the statement in the case that $a_\ell+\cdots+a_2$ is even, and $a_1$ and $3k+2$ are odd. The diagram consisting of $3k+2$ tiles between the second line and third line is called the \emph{first tail}, and the diagram consisting of last $k+2$ tiles is called the \emph{second tail} in Figure \ref{fig:case22}. The graph obtained by removing the first and the second tails from $\mathcal {G}[a_{\ell},\dots,a_{1},3k+2,k+2]$ is congruent to $\mathcal {G}[a_{\ell}, ..., a_1]$. The ``$\cdots$" part except for leftmost one form a staircase because they consist of the same signs in succession.
\begin{figure}
    \centering
    \includegraphics[scale=0.25]{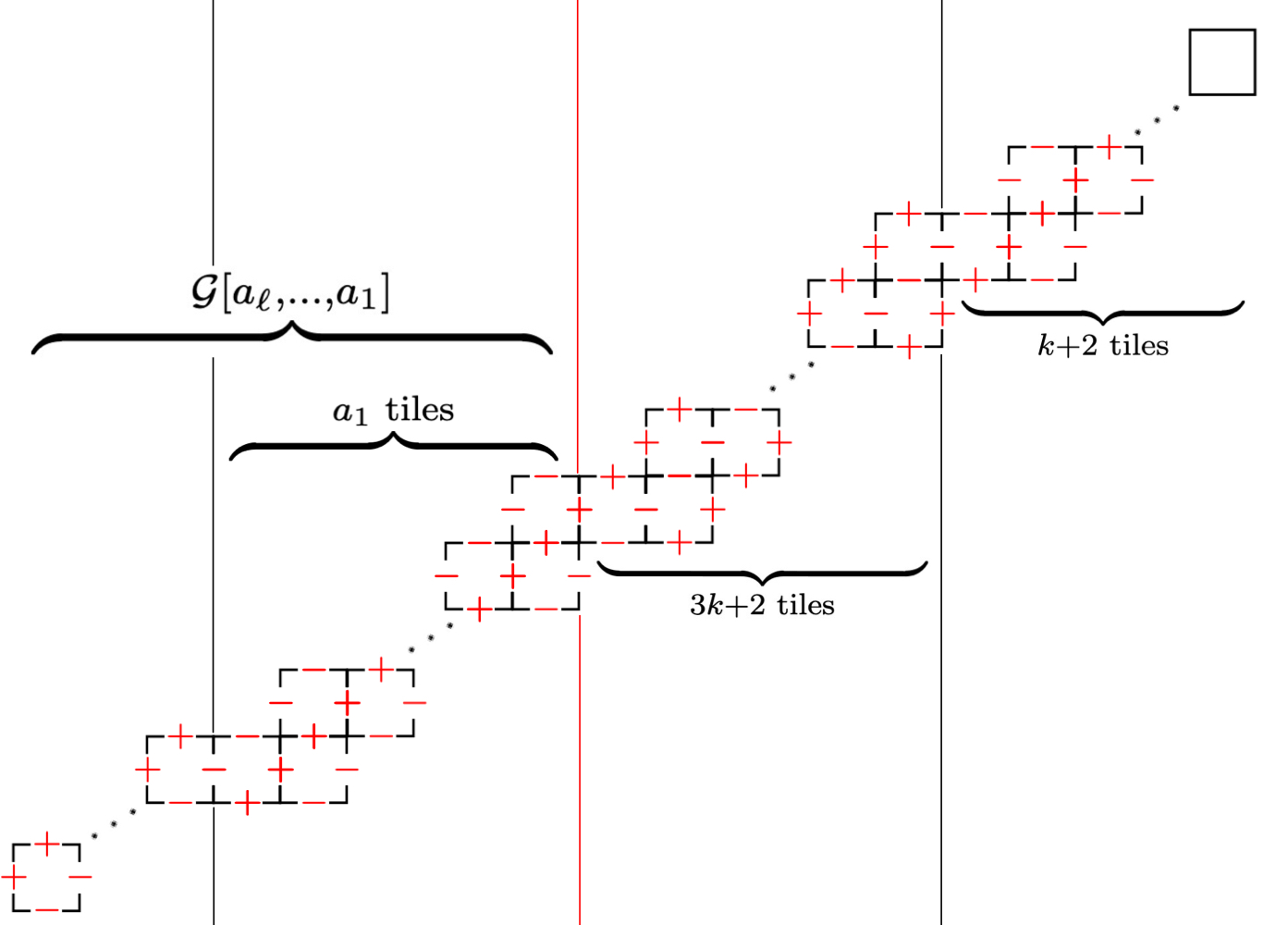}
    \caption{Snake graph $\mathcal{G}[a_\ell,\dots,a_1,3k+2,k+2]$}
    \label{fig:case22}
\end{figure}

We will count the number of perfect matchings of $\mathcal G[a_\ell,\dots,a_1,3k+2,k+2]$. Any perfect matching of this snake graph belongs to exactly one of the two cases described below:
\begin{itemize}\setlength{\leftskip}{5pt}
    \item [(2-b-I)] it contains a perfect matching of the $\mathcal {G}[a_{\ell}, ..., a_1]$-part,
    \item [(2-b-II)] it contains the upper and the lower edges in the leftmost tile in the second tail.
\end{itemize}

We will count the number of the perfect matchings belonging to (2-b-I). In the same way as the case (1-I), we can see that this number coincides with the product of the number of perfect matchings in $\mathcal{G}[a_\ell,\dots,a_1]$ and the number of perfect matchings in the graph removing the leftmost tile from the union of the first and the second tail. Since the latter graph is $m(k+2,3k+2)=3k^2+8k+5$ (by using Theorem \ref{thm:snakegraph-continuedfraction}), the number of perfect matchings belonging to (2-b-I) is
\[m(a_\ell,\dots,a_1)m(k+2,3k+2)=(3k^2+8k+5)m(a_1,\dots,a_\ell).\]
We will count the number of the perfect matchings belonging to (2-b-II). If a perfect matching contains the upper and the lower edges in the leftmost tile in the first tail, then edges of the first tail and the rightmost $a_1$ tiles of the $\mathcal{G}[a_\ell,\dots,a_1]$-part in the perfect matching is uniquely determined. Therefore, this number coincides with the product of the number of perfect matchings of $\mathcal{G}[a_\ell,\dots,a_2]$ and the number of perfect matchings of graph consisting of the last $k+1$ tiles  (see Figure \ref{fig:case22-c2}). 
\begin{figure}[ht]
    \centering
    \includegraphics[scale=0.25]{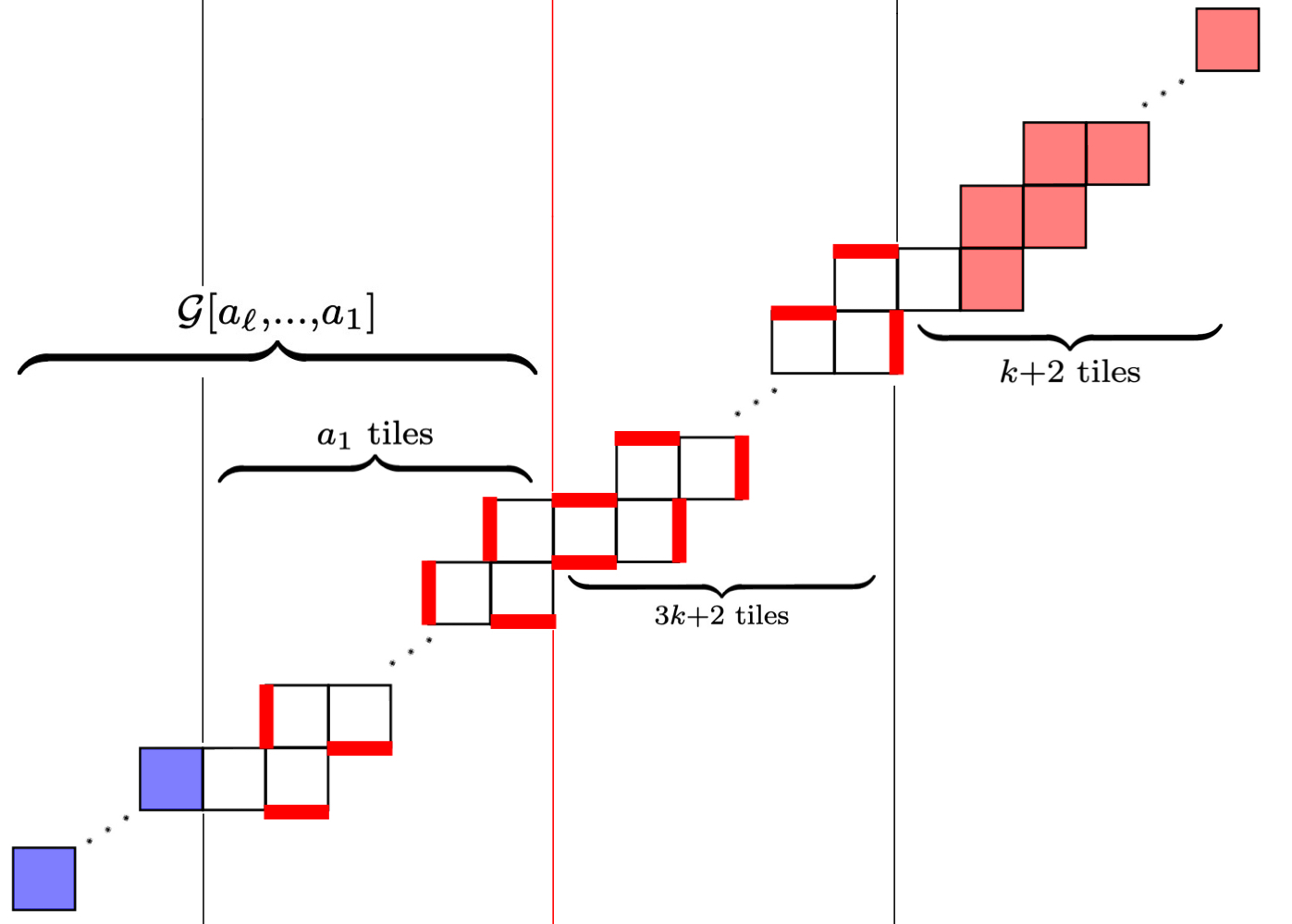}
    \caption{Case (2-b-II)}
    \label{fig:case22-c2}
\end{figure}

The latter number is $m(k+2)=k+2$, and therefore the number of perfect matchings belonging to (2-b-II) is
\[m(a_\ell,\dots,a_2)m(k+2)=(k+2)m(a_2,\dots,a_\ell).\]
Combining results on (2-b-I) and (2-b-II), we have
\[m(a_{\ell},\dots,a_{1},3k+2,k+2)=(3k^2+8k+5)m(a_1,\dots,a_{\ell})+(k+2)m(a_2,\dots,a_\ell),\]
as desired. 

Finally, we will prove the case (3). By the results of (1) and (2), it suffices to show that $M_t$ satisfies the statement under the assumption that $M_r$ and $M_s$ satisfy the statement. We set
\[M_{r}=\begin{bmatrix}
    -m(a_1,\dots,a_{\ell-1})& m(a_1,\dots,a_{\ell})\\ -m(a_2,\dots,a_{\ell-1})& m(a_2,\dots,a_{\ell})\end{bmatrix}, M_s=\begin{bmatrix}
    -m(b_1,\dots,b_{m-1})& m(b_1,\dots,b_{m})\\ -m(b_2,\dots,b_{m-1})& m(b_2,\dots,b_{m})\end{bmatrix}.\]
Since $M_t=M_r^{-1}TM_s$, where $T=\begin{bmatrix}
    -1&0\\3k+3& -1
\end{bmatrix}$, the $(1,1)$-entry of $M_t$ is
\[-(((3k+3)m(a_1,\dots,a_\ell)+m(a_2,\dots,a_\ell))m(b_2,\dots,b_m)-m(a_1,\dots,a_\ell)m(b_2,\dots,b_{m-1})),\]
the $(1,2)$-entry is
\[((3k+3)m(a_1,\dots,a_\ell)+m(a_2,\dots,a_\ell))m(b_1,\dots,b_m)-m(a_1,\dots,a_\ell)m(b_1,\dots,b_{m-1}),\]
the $(2,1)$-entry is
\[-(((3k+3)m(a_1,\dots,a_{\ell-1})+m(a_2,\dots,a_{\ell-1}))m(b_2,\dots,b_m)-m(a_1,\dots,a_{\ell-1})m(b_2,\dots,b_{m-1})),\]
the $(2,2)$-entry is
\[((3k+3)m(a_1,\dots,a_{\ell-1})+m(a_2,\dots,a_{\ell-1}))m(b_1,\dots,b_m)-m(a_1,\dots,a_{\ell-1})m(b_1,\dots,b_{m-1}).\]
On the other hand, since $F^+(k,t)=[a_\ell,\dots,a_1,3k+2,1,b_{m}-1,b_{m-1},\dots,b_1]$ by Proposition \ref{prop:presnake-relation}, it suffices to show the following equality:
\begin{align}
    &m(a_x,\dots,a_1,3k+2,1,b_{m}-1,b_{m-1},\dots,b_y)\label{eq:3-1}\\
    &=((3k+3)m(a_1,\dots,a_x)+m(a_2,\dots,a_x))m(b_y,\dots,b_m)-m(a_1,\dots,a_x)m(b_y,\dots,b_{m-1})\nonumber,
\end{align}
where $x\in\{\ell-1,\ell\}$ and $y\in\{1,2\}$. Since the following argument can be applied to any pair of $x$ and $y$, we will only prove \eqref{eq:3-1} in the case that $x = \ell$ and $y = 2$. The snake graph $\mathcal{G}[a_\ell,\dots, a_1,3k+2,1,b_m-1,b_{m-1}\dots,b_2]$ is given as in Figure \ref{fig:case3}. The diagram consisting of $3k+3$ tiles between the second and third vertical lines in Figure \ref{fig:case3} is called the \emph{joint}. We can divide $\mathcal{G}[a_\ell,\dots, a_1,3k+2,1,b_m-1,b_{m-1}\dots,b_2]$ into three parts, the $\mathcal{G}[a_\ell,\dots, a_1]$-part, the joint, and the $\mathcal{G}[b_m,\dots, b_2]$-part. The ``$\cdots$" part except for the leftmost one and the rightmost one form a staircase because they consist of the same signs in succession.
\begin{figure}[ht]
    \centering
    \includegraphics[scale=0.25]{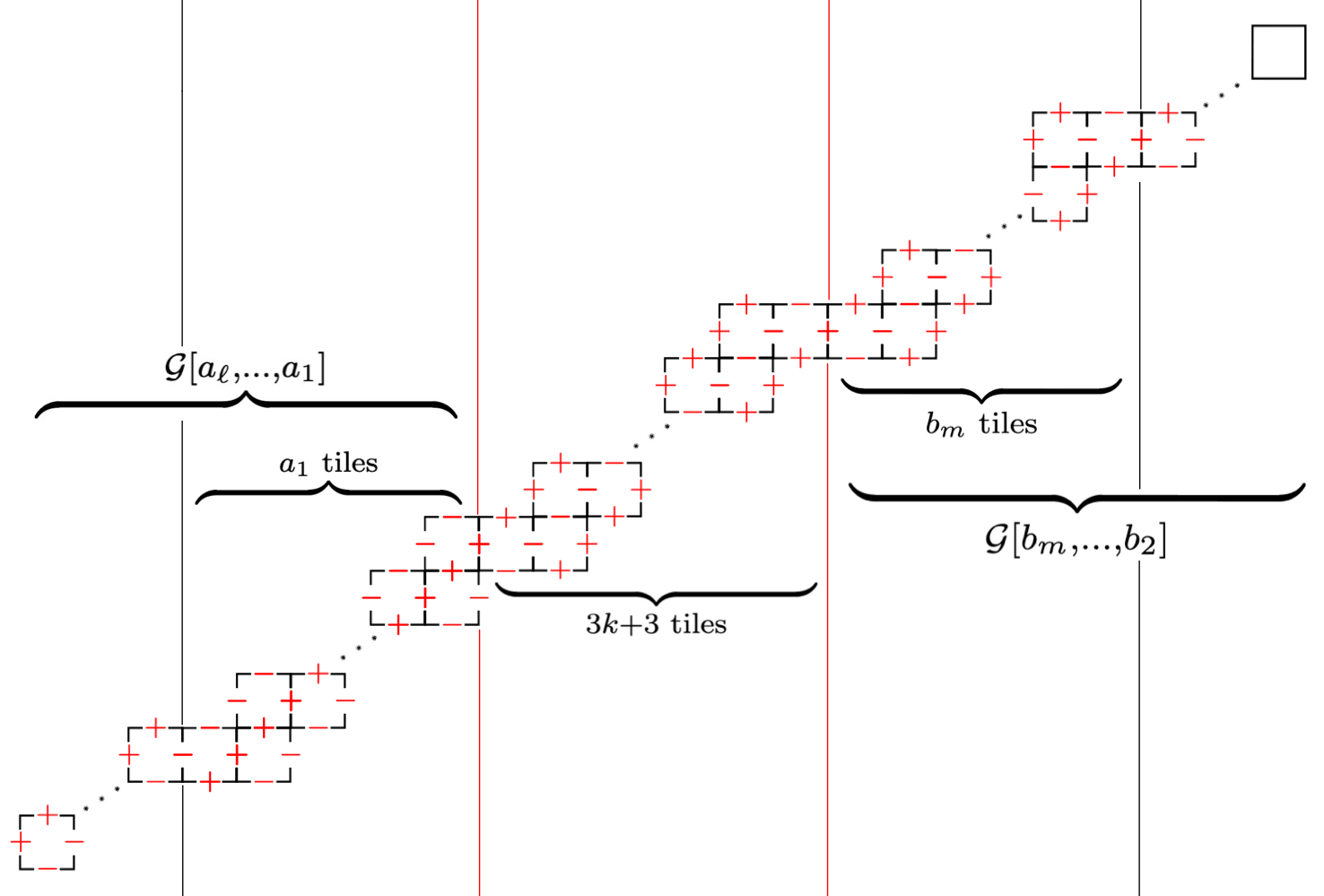}
    \caption{snake graph $\mathcal{G}[a_\ell,\dots, a_1,3k+2,1,b_m-1,b_{m-1}\dots,b_2]$}
    \label{fig:case3}
\end{figure}
We will count the number of perfect matchings of $\mathcal G[a_\ell,\dots,a_1,3k+2,1,b_m-1,b_{m-1},\dots,b_2]$. Any perfect matching of this snake graph belongs to exactly one of the three cases described below:
\begin{itemize}\setlength{\leftskip}{5pt}
    \item [(3-I)] it contains perfect matchings of the $\mathcal {G}[a_{\ell}, ..., a_1]$-part and $\mathcal {G}[b_{m}, ..., b_2]$-part,
    \item [(3-II)] it contains an edge of the leftmost tile in the joint such that its intersection with the $\mathcal {G}[a_{\ell}, ..., a_1]$-part is only one point,
    \item [(3-III)] it contains an edge of the rightmost tile in the joint such that its intersection with the $\mathcal {G}[b_{m}, ..., b_2]$-part is only one point.
\end{itemize}
We will count the number of the perfect matchings belonging to (3-I). 
In this case, the perfect matching refers to combinations of perfect matchings within the three color-coded regions in Figure \ref{fig:case3-c1}.
\begin{figure}[ht]
    \centering
    \includegraphics[scale=0.25]{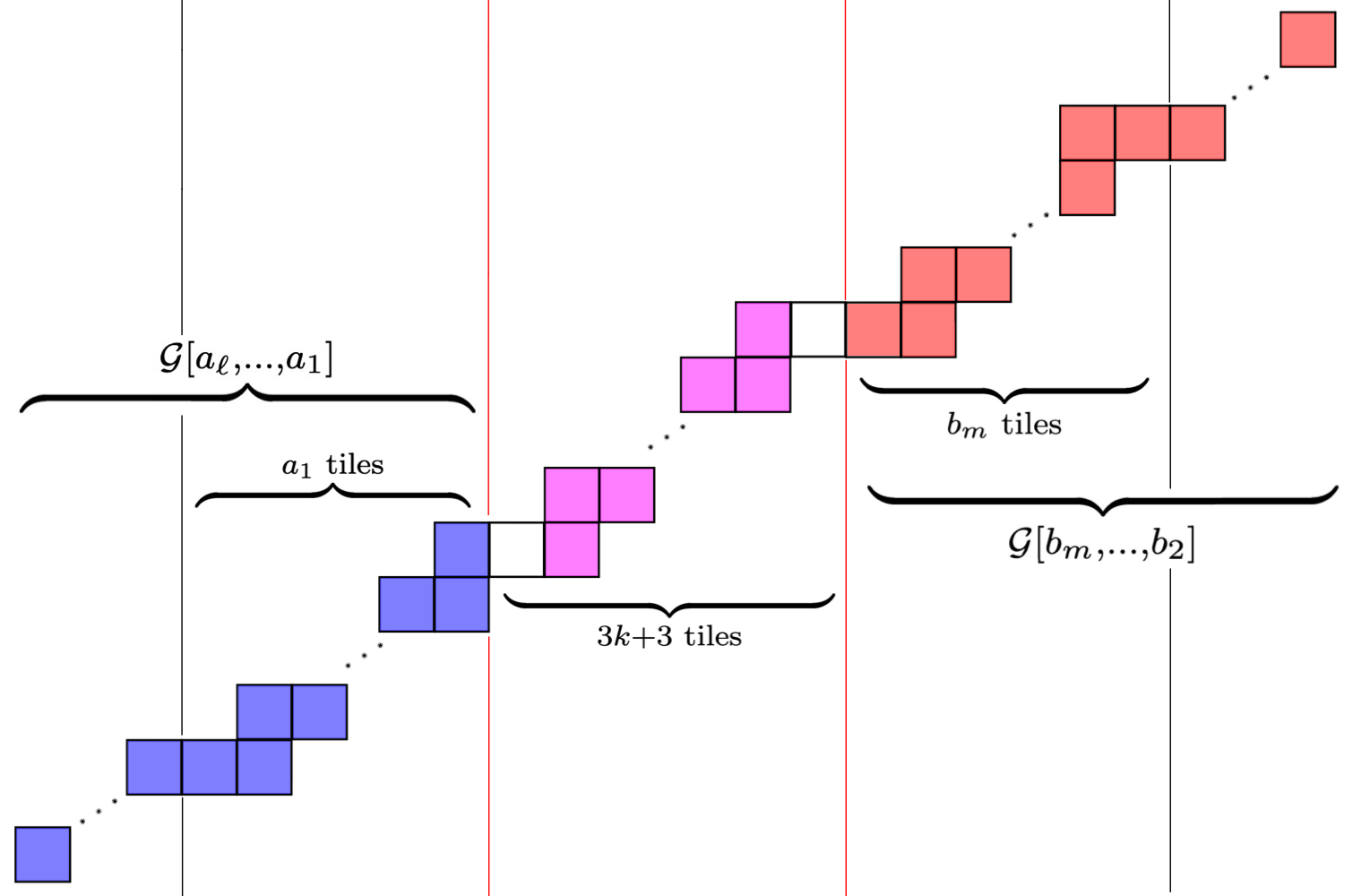}
    \caption{Case (3-I)}
    \label{fig:case3-c1}
\end{figure}
Therefore, the number of such matchings is given by the product of the number of perfect matchings in $\mathcal{G}[a_\ell,\dots,a_1]$, that of $\mathcal{G}[b_m,\dots,b_2]$, and that of the graph obtained from the joint by removing the leftmost and the rightmost tiles. Since the last graph is $\mathcal{G}[3k+2]$, the number of perfect matchings belonging to (3-I) is
\[m(a_\ell,\dots,a_1)m(3k+2)m(b_m,\dots,b_2)=(3k+2)m(a_1,\dots,a_\ell)m(b_2,\dots,b_m).\]
We will count the number of the perfect matchings belonging to (3-II). If a perfect matching contains the upper and the lower edges in the leftmost tile in the joint, then edges of the joint and the rightmost $a_1$ tiles of the $\mathcal{G}[a_\ell,\dots,a_1]$-part in the perfect matching is uniquely determined. Therefore, this number coincides with the product of the number of perfect matchings of $\mathcal{G}[a_\ell,\dots,a_2]$, and that of $\mathcal{G}[b_m,\dots,b_2]$ (see Figure \ref{fig:case3-c2}). 
\begin{figure}[ht]
    \centering
    \includegraphics[scale=0.25]{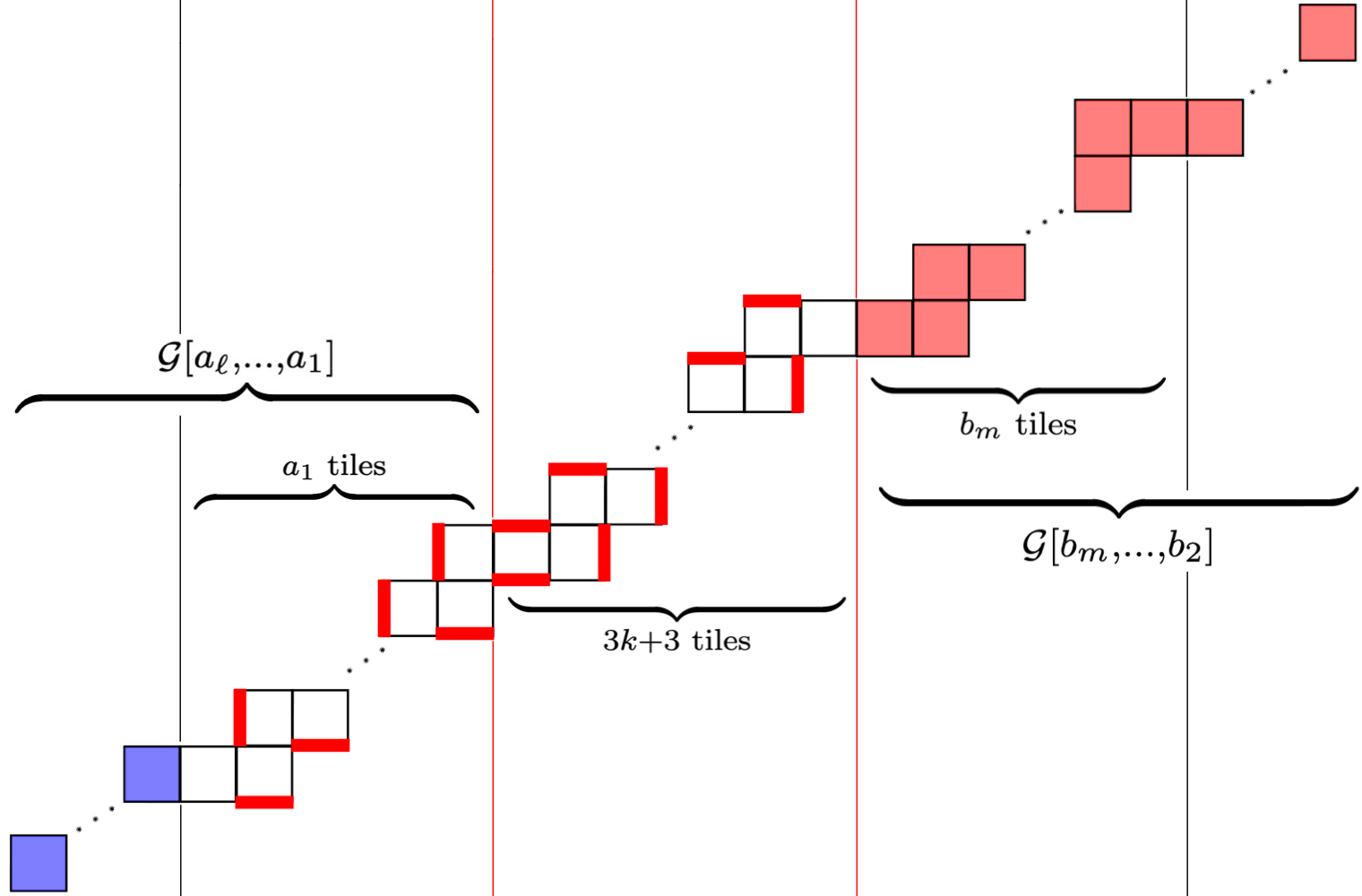}
    \caption{Case (3-II)}
    \label{fig:case3-c2}
\end{figure}
Therefore, the number of perfect matchings belonging to (3-II) is
\[m(a_\ell,\dots,a_2)m(b_m,\dots,b_2)=m(a_2,\dots,a_\ell)m(b_2,\dots,b_m).\]

Next, we will count the number of the perfect mathings belonging to (3-III). In this case, edges of the joint in a perfect matching is uniquely determined. Therefore, this number coincides with the product of the number of perfect matchings of $\mathcal{G}[a_\ell,\dots,a_1]$ and that of the graph removed the leftmost tile from $\mathcal{G}[b_m,\dots,b_2]$ (see Figure \ref{fig:case3-c3}).
\begin{figure}[ht]
    \centering
    \includegraphics[scale=0.25]{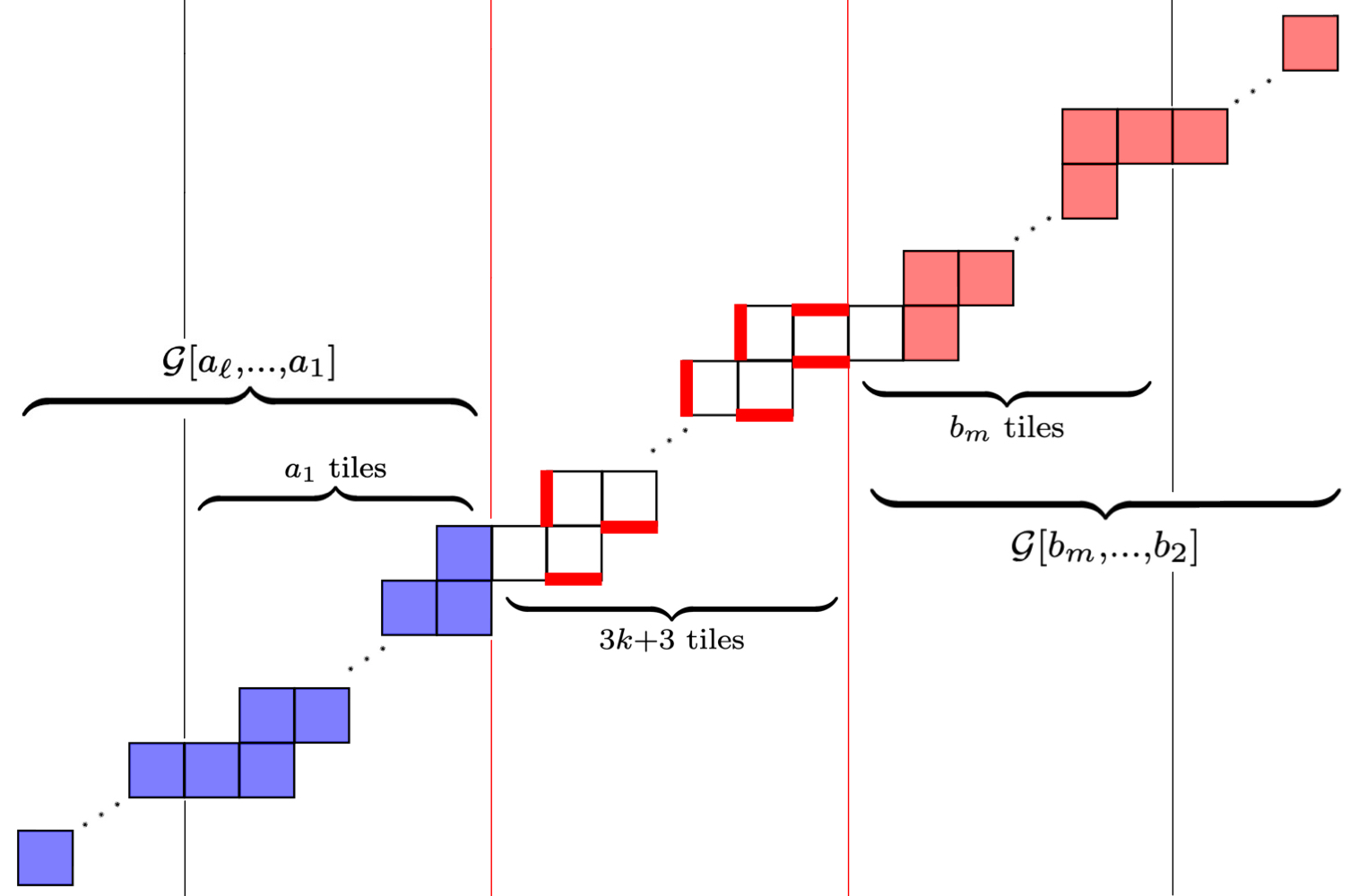}
    \caption{Case (3-III)}
    \label{fig:case3-c3}
\end{figure}
By the argument in the case (1), the latter number is $m(b_{2},\dots,b_m)-m(b_{2},\dots,b_{m-1})$. Therefore, the number of perfect matchings belonging to (3-III) is
\[m(a_1,\dots,a_\ell)(m(b_2,\dots,b_m)-m(b_2,\dots,b_{m-1})).\]
Combining the results on (3-I), (3-II) and (3-III), we have
\begin{align*}
   &m(a_\ell,\dots,a_1,3k+2,1,b_{m}-1,b_{m-1},\dots,b_2)\\
   &=((3k+3)m(a_1,\dots,a_\ell)+m(a_2,\dots,a_\ell))m(b_2,\dots,b_m)-m(a_1,\dots,a_\ell)m(b_2,\dots,b_{m-1}), 
\end{align*}
as desired. This finishes the proof.
\end{proof}

From Theorem \ref{thm:M_t-C_t-combinatorics} (1), we have the following corollary:

\begin{corollary}\label{cor:trdet}
Let $F^+(k,t)=[a_1,\dots,a_\ell]$. The following equalities hold:
\begin{itemize}\setlength{\leftskip}{-15pt}
    \item [(1)] $m(a_1,\dots,a_{\ell-1})-m(a_2,\dots,a_{\ell})=k,$
    \item [(2)] $m(a_1,\dots,a_{\ell-1})m(a_2,\dots,a_{\ell})-m(a_2,\dots,a_{\ell-1})m(a_1,\dots,a_{\ell})=-1.$
\end{itemize}
\end{corollary}

\begin{proof}
It follows from Theorem \ref{thm:M_t-C_t-combinatorics} (1), $\mathrm{tr}(M_t(k,0))=-k$, and $\det(M_t(k,0))=1$.    
\end{proof}

\begin{remark}
Corollary \ref{cor:trdet} (2) is already proved by \cite{cs18} in the more general situation, the case where $[a_1,\dots,a_\ell]$ is simply a continuous fraction. More precisely, this statement is obtained from \cite{cs18}*{Theorem 5.2}(b) by substituting $i=2$ and $j=n-3$. Moreover, Corollary \ref{cor:trdet} (1) can be proved in the case where $[a_1,\dots,a_\ell]$ is a semi-palindrome continuous fraction. When $k=1$, we can see the proof in \cite{bansen}*{Lemma 5}.
\end{remark}

Next, we will prove Theorem \ref{thm:M_t-C_t-combinatorics} (2).

\begin{proof}[Proof of Theorem \ref{thm:M_t-C_t-combinatorics} (2)]
In this proof, we abbreviate $C_t(k,-k)$ to $C_t$, and $m(\mathcal{G}[a_1,\dots,a_\ell])$ to $m(a_1,\dots,a_\ell)$. First, we show that it suffices to satisfy the assertion regarding the $(1,1)$-entry and the $(1,2)$-entry of the matrix to prove this theorem. Assuming the $(1,1)$-entry is  $m(a_2,\dots,a_\ell)$ and $(1,2)$-entry is $m(a_1,\dots,a_\ell)$, we prove that the $(2,1)$-entry $c_{21}$ is $(3k+3)m(a_2,\dots,a_{\ell})-m(a_2,\dots,a_{\ell-1})$ and the $(2,2)$-entry $c_{22}$ is $(3k+3)m(a_1,\dots,a_{\ell})-m(a_1,\dots,a_{\ell-1})$. First, we prove the latter assertion. From the condition of the trace of the $k$-GC matrix, we have
\[m(a_2,\dots,a_\ell)+c_{22}=(3+3k)m(a_1,\dots,a_{\ell})-k.\]
Therefore, from Corollary \ref{cor:trdet} (1), we have
\[c_{22}=(3+3k)m(a_1,\dots,a_{\ell})-m(a_1,\dots,a_{\ell-1}),\]
as desired. Next, we prove the former assertion.  From the condition of the determinant of the $k$-GC matrix, we have
\[m(a_2,\dots,a_\ell)((3+3k)m(a_1,\dots,a_{\ell})-m(a_1,\dots,a_{\ell-1}))-c_{21}m(a_1,\dots,a_\ell)=1.\]
Therefore, by Corollary \ref{cor:trdet} (2), we have
\begin{align*}
c_{21}&=\dfrac{(3+3k)m(a_1,\dots,a_{\ell})m(a_2,\dots,a_\ell)-m(a_1,\dots,a_{\ell-1})m(a_2,\dots,a_\ell)-1}{m(a_1,\dots,a_\ell)}\\
&=(3k+3)m(a_2,\dots,a_{\ell})-m(a_2,\dots,a_{\ell-1}),
\end{align*}
as desired.
We will prove that the assertion regarding the $(1,1)$-entry and the $(1,2)$-entry of the matrix is satisfied in the following four cases: (0) $r=\dfrac{0}{1}, s=\dfrac{1}{1}$, (1) $r=\dfrac{0}{1}, s\neq\dfrac{1}{1}$ (2) $r\neq\dfrac{0}{1}, s=\dfrac{1}{1}$, (3) $r\neq \dfrac{0}{1}, s\neq\dfrac{1}{1}$.

We prove the case (0). Now, $t=\dfrac{1}{2}$ holds. By a direct calculation, we have
\[C_{\frac{1}{2}}=\begin{bmatrix} k+2&2k^2+6k+5\\3k^2+9k+5&6k^3+24k^2+31k+13 \end{bmatrix}.\]
Moreover, we have $F^{+}(k,1/2)=[2k+2,k+2]$. Since 
\[[2k+2,k+2]=\dfrac{2k^2+6k+5}{k+2}\]
hold, by Theorem \ref{thm:snakegraph-continuedfraction}, we have \[m(2k+2,k+2)=2k^2+6k+5.\] 

Next, we will prove the case (1). There exists $p\in \mathbb Z_{>1}$ such that $s=\dfrac{1}{p}$ and $t=\dfrac{1}{p+1}$. We will prove the statement for $C_{\frac{1}{p}}$ by using induction on $p$. When $p=2$, $C_{\frac{1}{2}}$ satisfies the statement by the argument in the case (0).  We assume that $C_{\frac{1}{p}}$ satisfies the statement, and prove that $C_{\frac{1}{p+1}}$ also satisfies the statement. We set
$F^+(k,1/p)=[b_{1},\dots,b_m]$ and \[C_{\frac{1}{p}}=\begin{bmatrix}
        m(b_2,\dots,b_{m}) & m(b_1,\dots,b_{m})\\[0.5em]\hspace{-1cm}(3k+3)m(b_2,\dots,b_{m}) & \hspace{-1cm}(3k+3)m(b_1,\dots,b_{m})\\
        \hspace{2cm}-m(b_2,\dots,b_{m-1})&  \hspace{2cm} - m(b_1,\dots,b_{m-1})
    \end{bmatrix}.\]
Since $C_{\frac{0}{1}}=\begin{bmatrix}
    -k&1\\-(3k^2+3k+1)&3k+3
\end{bmatrix}$ and $C_{\frac{1}{p+1}}=C_{\frac{0}{1}}C_{\frac{1}{p}}-S$, where $S=\begin{bmatrix}
    k&0\\3k^2+3k&k
\end{bmatrix}$, we have
\[C_{\frac{1}{p+1}}=\begin{bmatrix}
       \hspace{-5mm}(2k+3)m(b_2,\dots,b_{m})& \hspace{-5mm}(2k+3)m(b_1,\dots,b_{m})\\
       \hspace{1cm}-m(b_2,\dots,b_{m-1})-k & \hspace{1cm}-m(b_1,\dots,b_{m-1}) \\[0.5em]
        \ast & \ast
    \end{bmatrix}. \]
Since $F^+(k,1/(p+1))=[2k+2,1,b_{m}-1,b_{m-1},\dots,b_{1}]$ by Proposition \ref{prop:presnake-relation} (1), it suffices to show the following two equalities: 
\begin{align}
    m(1,b_{m}-1,b_{m-1},\dots,b_{1})&=(2k+3)m(b_2,\dots,b_{m})-m(b_2,\dots,b_{m-1})-k,\label{eq:1-1-2}\\
    m(2k+2,1,b_{m}-1,b_{m-1},\dots,b_{1})&=(2k+3)m(b_1,\dots,b_{m})-m(b_1,\dots,b_{m-1}).\label{eq:1-2-2}
\end{align}
The equality \eqref{eq:1-2-2} coincides with \eqref{eq:1-2}, and it is already proved in the proof of Theorem \ref{thm:M_t-C_t-combinatorics} (1). Next, we will prove \eqref{eq:1-1-2}.  When $p=2$, we have \eqref{eq:1-1-2} by a direct calculation. We assume that $p\geq 3$. The left-hand side of \eqref{eq:1-1-2} equals to $m(b_{1},\dots,b_{m})$, and it is the $(1,2)$-entry of $C_{\frac{1}{p}}$. Applying Proposition \ref{prop:presnake-relation} (1) to $F^+(k,1/p)$ and $F^+(k,1/(p-1))$, we have $F^+(k,1/(p-1))=[b_{m},\dots,b_3+1]=[b_{m},\dots,b_3,1]$. Moreover, since $b_2=1$ by Proposition \ref{prop:presnake-relation} (1), $F^+(k,1/(p-1))$ coincides with $[b_{m},\dots,b_3,b_2]$. Therefore, we have
 \[C_{\frac{1}{p-1}}=\begin{bmatrix}
        m(b_{m-1},\dots,b_{2}) & m(b_{m},\dots,b_{2})\\[0.5em]\hspace{-1cm}(3k+3)m(b_{m-1},\dots,b_{2}) & \hspace{-1cm}(3k+3)m(b_{m},\dots,b_{2})\\
        \hspace{2cm}-m(b_{m-1},\dots,b_{3})&  \hspace{2cm} - m(b_m,\dots,b_{3})
    \end{bmatrix}.\]
Comparing $(1,2)$-entries of 
\[C_{\frac{1}{p}}=C_{\frac{0}{1}}C_{\frac{1}{p-1}}-S,\]
we have
\begin{align*}
m(b_1,\dots,b_m)&=(2k+3)m(b_m,\dots,b_2)- m(b_m,\dots,b_{3})\\
&=(2k+3)m(b_2,\dots,b_m)- m(b_2,\dots,b_{m-1})-k,
\end{align*}
as desired. Note that in the last equality, we use a relation derived from the application of Corollary \ref{cor:trdet} (1) to $F^+(k,1/(p-1))=[b_m,\dots,b_2]$ (we note that $a_1=b_m,a_2=b_{m-1},\dots,a_\ell=b_2$ in Corollary \ref{cor:trdet} (1)).

Next, we will prove the case (2). There exists $p\in \mathbb Z_{\geq 1}$ such that $s=\dfrac{p}{p+1}$ and $t=\dfrac{p+1}{p+2}$. We will prove the statement by using induction on $p$. When $p=1$, $C_{\frac{1}{2}}$ satisfies the statement by the argument in the case (0).  We assume that $C_{\frac{p}{p+1}}$ satisfies the statement, and prove that $C_{\frac{p+1}{p+2}}$ also satisfies the statement. We set
\[C_{\frac{p}{p+1}}=\begin{bmatrix}
        m(a_2,\dots,a_{\ell}) & m(a_1,\dots,a_{\ell})\\[0.5em]\hspace{-1cm}(3k+3)m(a_2,\dots,a_{\ell}) & \hspace{-1cm}(3k+3)m(a_1,\dots,a_{\ell})\\
        \hspace{2cm}-m(a_2,\dots,a_{\ell-1})&  \hspace{2cm} - m(a_1,\dots,a_{\ell-1})
    \end{bmatrix}.\]
Since $C_{\frac{1}{1}}=\begin{bmatrix}
    1&k+2\\3k+2&3k^2+8k+5
\end{bmatrix}$ and $C_{\frac{p+1}{p+2}}=C_{\frac{p}{p+1}}C_{\frac{1}{1}}-S$, where $S=\begin{bmatrix}
    k&0\\3k^2+3k&k
\end{bmatrix}$, we have
\[C_{\frac{p+1}{p+2}}=\begin{bmatrix}
       \hspace{-2mm}(3k+2)m(a_1,\dots,a_{\ell})& \hspace{5mm}(3k^2+8k+5)m(a_1,\dots,a_{\ell})\\
       \hspace{1.5cm}+m(a_2,\dots,a_\ell)-k & \hspace{2.5cm}+(k+2)m(a_2,\dots,a_{\ell}) \\[0.5em]
        \ast & \ast
    \end{bmatrix}. \]
Since $F^+(k,(p+1)/(p+2))=[a_{\ell},\dots,a_{1},3k+2,k+2]$ by Proposition \ref{prop:presnake-relation} (2), it suffices to show the following two equalities: 
\begin{align}
    m(a_{\ell-1},\dots,a_{1},3k+2,k+2)&=(3k+2)m(a_1,\dots,a_{\ell})+m(a_2,\dots,a_{\ell})-k,\label{eq:2-1-2}\\
    m(a_{\ell},\dots,a_{1},3k+2,k+2)&=(3k^2+8k+5)m(a_1,\dots,a_{\ell})+(k+2)m(a_2,\dots,a_{\ell}).\label{eq:2-2-2}
\end{align}
We will prove \eqref{eq:2-1-2}. By Corollary \ref{cor:trdet} (1), it suffices to show
\[m(a_\ell,\dots,a_{1},3k+2)=(3k+2)m(a_1,\dots,a_{\ell})+m(a_2,\dots,a_{\ell}),\]
and it coincides with \eqref{eq:2-1}, and it is already proved. The equality \eqref{eq:2-2-2} coincides with \eqref{eq:2-2}, and it is also already proved.

Finally, we will prove the case (3). By the results of (1) and (2), it suffices to show that $C_t$ satisfies the statement under the assumption that $C_r$ and $C_s$ satisfy the statement. We set
\begin{align*}
    C_{r}&=\begin{bmatrix}
        m(a_2,\dots,a_{\ell}) & m(a_1,\dots,a_{\ell})\\[0.5em]\hspace{-1cm}(3k+3)m(a_2,\dots,a_{\ell}) & \hspace{-1cm}(3k+3)m(a_1,\dots,a_{\ell})\\
        \hspace{2cm}-m(a_2,\dots,a_{\ell-1})&  \hspace{2cm} - m(a_1,\dots,a_{\ell-1})
    \end{bmatrix},\\
    C_{s}&=\begin{bmatrix}
        m(b_2,\dots,b_{m}) & m(b_1,\dots,b_{m})\\[0.5em]\hspace{-1cm}(3k+3)m(b_2,\dots,b_{m}) & \hspace{-1cm}(3k+3)m(b_1,\dots,b_{m})\\
        \hspace{2cm}-m(b_2,\dots,b_{m-1})&  \hspace{2cm} - m(b_1,\dots,b_{m-1})
    \end{bmatrix}.
\end{align*}
Since $C_t=C_rC_s-S$, the $(1,1)$-entry of $C_t$ is 
\[((3k+3)m(a_1,\dots,a_\ell)+m(a_2,\dots,a_\ell))m(b_2,\dots,b_m)-m(a_1,\dots,a_\ell)m(b_2,\dots,b_{m-1})-k,\]
the $(1,2)$-entry is 
\[((3k+3)m(a_1,\dots,a_\ell))+m(a_2,\dots,a_\ell))m(b_1,\dots,b_m)-m(a_1,\dots,a_\ell)m(b_1,\dots,b_{m-1}).\]
Since $F^+(k,t)=[a_\ell,\dots, a_1,3k+2,1,b_{m}-1,b_{m-1},\dots,b_{1}]$ by Proposition \ref{prop:presnake-relation} (3), it suffices to show the following two equalities: 
\begin{align}
&m(a_{\ell-1},\dots, a_1,3k+2,1,b_{m}-1,b_{m-1},\dots,b_{1})\label{eq:3-1-2}\\
&=((3k+3)m(a_1,\dots,a_\ell)+m(a_2,\dots,a_\ell))m(b_2,\dots,b_m)-m(a_1,\dots,a_\ell)m(b_2,\dots,b_{m-1})-k,\nonumber\\
&m(a_{\ell},\dots, a_1,3k+2,1,b_{m}-1,b_{m-1},\dots,b_{1})\label{eq:3-2-2}\\
&=((3k+3)m(a_1,\dots,a_\ell)+m(a_2,\dots,a_\ell))m(b_1,\dots,b_m)-m(a_1,\dots,a_\ell)m(b_1,\dots,b_{m-1}).\nonumber
\end{align}
The equality \eqref{eq:3-2-2} coincides with \eqref{eq:3-1} for $x=\ell$ and $y=1$, and it is already proved. We will prove \eqref{eq:3-1-2}. Transposing $-k$ to the left-hand side and using Corollary \ref{cor:trdet} (1), \eqref{eq:3-1-2} is equivalent to
\begin{align*}
&m(a_{\ell},\dots, a_1,3k+2,1,b_{m}-1,b_{m-1},\dots,b_{2})\\
&=((3k+3)m(a_1,\dots,a_\ell)+m(a_2,\dots,a_\ell))m(b_2,\dots,b_m)-m(a_1,\dots,a_\ell)m(b_2,\dots,b_{m-1}).\nonumber
\end{align*}
It coincides with \eqref{eq:3-1} for $x=\ell$ and $y=2$, and it is already proved.
\end{proof}

\begin{remark}
A snake graph $\mathcal G (F^+(k,t))$ obtained from a pre-snake graph $\mathcal{PG}(t)$ coincides with the following graphs in other papers: when $k=0$,
\begin{itemize}\setlength{\leftskip}{-15pt}
    \item a \emph{domino graph} obtained from a \emph{snake graph} in \cite{aig}*{Section 7}, 
    \item a snake graph obtained from a triangulation on the once-punctured torus in \cite{msw}*{Section 4},
\end{itemize}
when $k=1$,
\begin{itemize}\setlength{\leftskip}{-15pt}
    \item a snake graph obtained from a pre-snake graph in \cite{gyo21}*{Section 3}, 
    \item a snake graph obtained from a line segment in \cite{bansen}*{Section 3}.
\end{itemize}
\end{remark}

By using Theorem \ref{thm:M_t-C_t-combinatorics}, we will give meanings of the numbers $p',q',r'$ in each vertex $\left(\begin{bmatrix}
        p\\p'
    \end{bmatrix},\begin{bmatrix}
        q\\q'
    \end{bmatrix},\begin{bmatrix}
        r\\r'
    \end{bmatrix}\right)$ in $\mathrm{P}\mathbb T(0)$.

\begin{theorem}\label{thm:p'q'r'}
Let $\dfrac{p}{p'}$ be the (unique) fixed point of the M\"obius transformation given by the $2$-MM matrix $M_t(2,0)$, where $p$ and $p'$ are relatively prime. If $F^+(2,t)=[a_1,\dots,a_\ell]$, then we have $p=\sqrt{m([\mathcal{G}[a_1,\dots,a_\ell])}$ and $p'=\sqrt{m(\mathcal G [a_2,\dots,a_{\ell-1}]})$.
\end{theorem}
\subsection{Characteristic numbers of $k$-GM triple}
In previous subsections, we see the numerator of $F^+(k,t)$ is a $k$-GM number associated with $t$. in this subsection, we will see that the denominator of $F^+(k,t)$ is the \emph{characteristic number} . 

First, we will recall the characteristic number. When we consider $k$-GM numbers labeled with $t\in [0,1]$ at a fixed $k$, we often simply denote $m_{k,t}$ by $m_t$. We fix $k\in \mathbb Z_{\geq 0}$ and a $k$-GM triple $(m_r,m_t,m_s)$ in $\mathrm{LM}\mathbb T(k)$. Note that $m_t>\max\{m_r,m_s\}$ and $m_r\neq m_s$. We consider solutions $x$ to equations
\begin{align*}
    m_rx&\equiv m_s \mod m_t,\\
    m_rx&\equiv -m_s \mod m_t,\\
    m_sx&\equiv m_r \mod m_t,\\
    m_sx&\equiv -m_r \mod m_t.
\end{align*} Since $m_r$ and $m_t$ are relatively prime from Proposition \ref{relatively-prime}, each solution is unique in the range $\left(0,m_t\right)$. These numbers are called the \emph{characteristic numbers} and we denote them by $u^+_t, u^-_t,v^+_t, v^-_t$, respectively.

\begin{remark}
    The characteristic numbers depend only on $t$ because a Farey triple $(r,t,s)$ in $\mathrm{F}\mathbb{T}$ is determined uniquely by $t$. Therefore, $u^\pm_t, v^{\pm}_t$ are often simply referred to as the characteristic numbers of $t$.
\end{remark}

When we need to emphasize $k$, we also denote them by $u^{\pm}_{k,t}, v^{\pm}_{k,t}$.

These four numbers have the following relations:

\begin{proposition}\label{prop:uv-relation}
 For characteristic numbers $u_t^\pm$ and $v_t^\pm$, the following inequalities hold:
 \begin{itemize}\setlength{\leftskip}{-15pt}
     \item [(1)] $0<u_t^+,v_t^-<\dfrac{m_t}{2}$, \quad $\dfrac{m_t}{2}<u_t^-,v_t^+<m_t$,
    \item [(2)] $u^-_t=m_t-u^+_t$, \quad $v^{+}_t=m_t-u_t^+-k$ \quad$v^-_t=u_t^++k$. 
 \end{itemize}
\end{proposition}

First, we will consider the property of $u^+_t$. There is the following characterization.
\begin{proposition}[\cite{gyo-maru}*{Lemma 4.5}]\label{lem:index-of-Ct}
   For an irreducible fraction $t\in (0,1)$, the following equality holds:\[C_{t}(k,-k)=\begin{bmatrix}
   u^+_{t}&m_{t}\\ \ast &\ast
   \end{bmatrix}.\]
\end{proposition}

By using it, we will give a sharper estimate of $u_t^+$.

\begin{lemma}[\cite{gyo-maru}*{Lemma 4.8}]\label{lem:ut+k<mt/2}
 For an irreducible fraction $t\in(0,1)$, the inequality $u_t^++k<\dfrac{m_t}{2}$ holds. 
\end{lemma}
The following lemma implies $v^-_t\equiv u^+_t +k \mod m_t$:

\begin{lemma}\label{lem:ut-u't-modmt}
The following equality holds:
\[m_s(u^+_t+k)\equiv - m_r \mod m_t.\] 
\end{lemma}

\begin{proof}
By $m_r^2+m_s^2+km_rm_s\equiv 0 \mod m_t$ and $m_ru^+_t \equiv  m_s\mod m_t$, we have 
\[m_r^2(u_t^+)^2\equiv -m_r^2-km_rm_s \mod m_t,\] and by multiplying $1/m_r$ to both sides of the congruence, we have
\[m_r(u^+_t)^2\equiv -m_r-km_s \mod m_t.\] 
This implies
\[m_su_t\equiv  m_r(u^+_t)^2\equiv - m_r- km_s \mod m_t,\] and we have
\[m_s(u^+_t+ k)\equiv - m_r \mod m_t.\]
\end{proof}

\begin{proof}[Proof of Proposition \ref{prop:uv-relation}]
By Lemma \ref{lem:ut+k<mt/2}, we have $0<u_t^++k<\dfrac{m_t}{2}$, and in particular, $0<u_t^++k<m_t$. By the uniqueness of $v^-_t$ and Lemma \ref{lem:ut-u't-modmt}, we have $v^-_t=u_t^++k$ and  $0<v^{-}_t<\dfrac{m_t}{2}$. The rest of statements are clear.
\end{proof}

Combining Theorem \ref{thm:M_t-C_t-combinatorics} (2) and Proposition \ref{lem:index-of-Ct}, the following theorem is proved:

\begin{theorem}\label{thm:u_t-denominator}
Let $t\in (0,1)$. We set $F^+(k,t)=[a_1,\dots,a_\ell]$. Then, we have $m(a_2,\dots,a_\ell)=u^+_{k,t}$. In particular, we have $F^+(k,t)=\dfrac{m_{k,t}}{u^+_{k,t}}$.  
\end{theorem}

%\begin{example}
%When $t=\dfrac{2}{5}$, $(r,t,s)=\left(\dfrac{1}{3},\dfrac{2}{5},\dfrac{1}{2}\right)$ is in $\mathrm{F}\mathbb{T}$. We have 
%\begin{align*}
% m_{0,1/3}=5,  m_{0,2/5}=29,  m_{0,1/2}=2.  
%\end{align*}
%By calculating $5x\equiv 2 \mod 29$, we have $x=12$. On the other hand, By Example \ref{ex:continued-fraction-Markov}, we have $F^+(0,2/5)=\dfrac{29}{12}$, therefore the denominator of $F^+(0,2/5)$ coincides with $u_{0,2/5}^+$.  
%\end{example}

Moreover, we can express $C_t(k,-k)$ and $M_t(k,0)$ by using the characteristic numbers:

\begin{theorem}\label{thm:u_t-Markov-monodromy-cohn}
For an irreducible fraction $t\in (0,1)$, the following equalities hold:
\begin{itemize}\setlength{\leftskip}{-15pt}
\item [(1)] $M_{t}(k,0)=\begin{bmatrix}
   -v^-_{k,t}&m_{k,t}\\ -w_{k,t} &u^+_{k,t}\end{bmatrix},$
    \item [(2)] $C_{t}(k,-k)=\begin{bmatrix}
   u^+_{k,t}&m_{k,t}\\ (3k+3)u^+_{k,t}-w_{k,t} &(3k+3)m_{k,t}-v_{k,t}^{-}\end{bmatrix}$,
\end{itemize}
where $w_{k.t}=\dfrac{ u^+_{k,t}v^-_{k,t}+1}{m_{k,t}}$.
\end{theorem}

\begin{proof}
It follows from Theorems \ref{thm:M_t-C_t-combinatorics}, \ref{thm:u_t-denominator}, Proposition \ref{prop:uv-relation} (2) and Corollary \ref{cor:trdet} (1).  
\end{proof}

%\begin{example}
%    You can check that Theorem \ref{thm:u_t-Markov-monodromy-cohn} holds in particular cases where $k=0$ or $1$ by using trees in Examples \ref{ex:intro-cohn}, \ref{ex:intro-Markov-monodromy}, \ref{ex:cohn-tree}, and \ref{ex:Markov-monodromy-tree}.
%\end{example}

In the rules for obtaining the continued fraction $F^+(k,t)$ from the pre-snake graph, by changing the sign associated with the central edge from $-$ to $+$, a new continued fraction $G^+(k,t)$ is obtained.  
Moreover, we extend a continued fraction $F^+(k,t)$  and $G^+(k,t)$ to $t\in (0,\infty)$ by allowing for cases where the slope of the line segment for constructing the pre-snake graph is greater than $1$ (the sign rule is not changed). We will prove the following theorem:

\begin{theorem}\label{thm:uv-continued-fraction}
 For any $t\in(0,1)$, the following equalities hold:
 \begin{itemize}\setlength{\leftskip}{-15pt}
     \item [(1)] $F^+(k,1/t)=\dfrac{m_{k,t}}{v^+_{k,t}}$,\vspace{2mm}
     \item [(2)] $G^+(k,t)=\dfrac{m_{k,t}}{v^-_{k,t}}$\vspace{2mm},
     \item [(3)] $G^+(k,1/t)=\dfrac{m_{k,t}}{u^-_{k,t}}$.
 \end{itemize}
\end{theorem}

By Proposition \ref{lem:semi-palindrome}, we have the following proposition.

\begin{proposition}\label{prop:F-G-relation}
 Let $t\in(0,\infty)$. The equality $F^+(k,t)=[a_1,\dots,a_\ell]$ holds, if and only if $G^+(k,t)=[a_\ell,\dots,a_1]$ holds, where $[a_1,\dots,a_\ell]$ is the canonical semi-palindrome expression. 
\end{proposition}

Moreover, we have a relation between $F^+(k,t)$ and $F^+(k,1/t)$ (resp. $G^+(k,t)$ and $G^+(k,1/t)$) as follows:

\begin{proposition}\label{prop:tto1/t}
 Let $t\in(0,1]$. If $F^+(k,t)=[a_1,\dots,a_\ell]$, then we have $F^+(k,1/t)=[1,a_{\ell}-1,a_{\ell-1},\dots,a_2,a_1-1, 1]$. The same holds true for $G^+(k,t)$ as well.
\end{proposition}
\begin{proof}
 By applying the reflection of a slope 1 to $\mathcal{PG}(t)$ along the line passing through the vertex at the bottom-left of the graph, we obtain $\mathcal{PG}(1/t)$. In this case, the associated signs on $\mathcal{PG}(t)$ change for all except those associated with the initial triangle, terminal triangle, and the central edge. If we  trace this sequence of signs in reverse, then all signs except the last one coincide with those associated with $\mathcal{PG}(t)$.   
\end{proof}

\begin{proof}[Proof of Theorem \ref{thm:uv-continued-fraction}]
We set $F^+(k,t)=[a_1,\dots,a_\ell]$. First, we prove (2). By Proposition \ref{prop:F-G-relation}, we have $G^+(k,t)=[a_\ell,\dots,a_1]$. By Theorem \ref{thm:snakegraph-continuedfraction}, we have
\[G^+(k,t)=[a_\ell,\dots,a_1]=\dfrac{m(a_\ell,\dots,a_1)}{m(a_{\ell-1},\dots,a_1)}.\]
By Corollary \ref{cor:k-gen.snake-markov} (1), we have $m(a_\ell,\dots,a_1)=m_{k,t}$. Moreover, by Corollary \ref{cor:trdet} (1), Theorem \ref{thm:u_t-denominator}, and Proposition \ref{prop:uv-relation} (2), we have 
\[m(a_{\ell-1},\dots,a_1)=m(a_{1},\dots,a_{\ell-1})=m(a_{2},\dots,a_{\ell})+k=u_{t}^++k=v^-_{t}.\]
Next, we will prove (1). By Proposition \ref{prop:tto1/t}, we have
\[F^+(k,1/t)=[1,a_\ell-1,\dots,a_2,a_1-1,1]=[1,a_\ell-1,\dots,a_2,a_1].\]
By the result of (2), we have \[[a_\ell,\dots,a_1]=\dfrac{m_{k,t}}{u^+_{k,t}+k}.\]
Therefore, by Proposition \ref{prop:uv-relation} (2), we have 
\[F^+(k,1/t)=[1,a_\ell-1,\dots,a_1]=\dfrac{m_{k,t}}{m_{k,t}-u^+_{k,t}-k}=\dfrac{m_{k,t}}{v^{+}_{k,t}}.\]
Finally, we will prove (3). By Proposition \ref{prop:F-G-relation} and the result of (2), we have
\[G^+(k,1/t)=[1,a_1-1,\dots,a_\ell-1,1]\]
(note that to use Proposition \ref{prop:F-G-relation}, we must use the canonical semi-palindrome expression of $F^+(k,1/t)$). By Theorem \ref{thm:u_t-denominator} and Proposition \ref{prop:uv-relation}, we have
\[G^+(k,1/t)=[1,a_1-1,\dots,a_\ell-1,1]=[1,a_1-1,\dots,a_\ell]=\dfrac{m_{k,t}}{m_{k,t}-u^+_{k,t}}=\dfrac{m_{k,t}}{u^{-}_{k,t}}.\]
\end{proof}

\begin{remark}
The four graphs $\mathcal{G}(F^+(k,t))$, $\mathcal{G}(F^+(k,1/t))$, $\mathcal{G}(G^+(k,t))$, and $\mathcal{G}(F^+(k,1/t))$ are congruent.
\end{remark}

While the fraction labelings of $k$-GC triples and $k$-MM matrices are considered only for the interval $[0,1]$ in the above, these are also conceivable for irreducible fractions greater than $1$. In this case, similar relations with continued fractions as those seen for the interval $[0,1]$ can be observed.

\begin{theorem}\label{thm:b_t-c_t-uv1}
   For an irreducible fraction $t\in (0,1)$, if $F^{+}(k,t)=[a_1,\dots,a_n]$, then the following equalities hold:
   \begin{itemize}\setlength{\leftskip}{-15pt}
       \item [(1)] $M_{\frac{1}{t}}(k,0)=\begin{bmatrix}
        -m(a_1-1,\dots,a_\ell) &   m(a_1,\dots,a_\ell)\\
         -m(a_1-1,\dots,a_\ell-1) &m (a_1,\dots,a_\ell-1)
       \end{bmatrix}=\begin{bmatrix}
   -u^-_{k,t}&m_{k,t}\\ -w'_{k,t} & v^+_{k,t}
   \end{bmatrix}$,
   \vspace{2mm}
   \item [(2)] $C_{\frac{1}{t}}(k,-k)=\begin{bmatrix}
   v^+_{k,t}&m_{k,t}\\ (3k+3)v^+_{k,t}-w'_{k,t} & (3k+3)m_{k,t}-u^-_{k,t}
   \end{bmatrix}$,
   \end{itemize}
where $w'_{k,t}=\dfrac{u_{k,t}^-v_{k,t}^++1}{m_{k,t}}$.   
\end{theorem}
We omit the proof of the above theorem because it is almost the same as Theorem \ref{thm:M_t-C_t-combinatorics}.
Theorems introduced so far states that only two out of the four characteristic numbers appear as the $(1,1)$-entry of the $k$-GC matrix. However, there is a case where the remaining two also appear. This occurs when $\ell$ in $C_t(k,\ell)$ is taken as $-k-1$.

\begin{theorem}\label{thm:b_t-c_t-uv2}
   For an irreducible fraction $t\in (0,1)$, the following equalities hold:
   \begin{itemize}\setlength{\leftskip}{-15pt}
    \item [(1)] $M_{t}(k,1)=\begin{bmatrix}
   v^+_{k,t}&m_{k,t}\\ -w'_{k,t} &-u^-_{k,t}
   \end{bmatrix}$,\quad $M_{\frac{1}{t}}(k,1)=\begin{bmatrix}
   u^+_{k,t}&m_{k,t}\\ -w_{k,t} &-v^-_{k,t}
   \end{bmatrix}$.
       \item [(2)] $C_{t}(k,-k-1)=\begin{bmatrix}
   -u^-_{k,t}&m_{k,t}\\ -((3k+3)u^-_{k,t}+w'_{k,t} ) &(3k+3)m_{k,t}+v_{k,t}
^+
   \end{bmatrix}$,\vspace{2mm}\\
   $C_{\frac{1}{t}}(k,-k-1)=\begin{bmatrix}
   -v^-_{k,t}&m_{k,t}\\ -((3k+3)v^-_{k,t}+w_{k,t} ) &(3k+3)m_{k,t}+u_{k,t}
^+  \end{bmatrix}$. 
   \vspace{2mm}
   \end{itemize}
\end{theorem}

By taking appropriate values of $\ell$, we can construct $k$-GC matrices whose $(1,1)$ and $(2,2)$-entries of $C_t(k,-k),C_t(k,-k-1),C_{1/t}(k,-k),C_{1/t}(k,-k-1)$ are switched respectively.  

\begin{theorem}
   For an irreducible fraction $t\in (0,1)$, the following equalities hold:
   \begin{itemize}\setlength{\leftskip}{-15pt}
       \item [(1)] $C_{t}(k,2k+2)=\begin{bmatrix}
   (3k+3)m_{k,t}+u^-_{k,t} &m_{k,t}\\ (3k+3)v_{k,t}+w'_{k,t} & v^+_{k,t}
   \end{bmatrix}$,\vspace{2mm}\\
   $C_{\frac{1}{t}}(k,2k+2)=\begin{bmatrix}
    (3k+3)m_{k,t}+v^-_{k,t} &m_{k,t}\\ (3k+3)u^+_{k,t}+w_{k,t} & u^+_{k,t}
   \end{bmatrix}$, 
   \vspace{2mm}
   \item [(2)] $C_{t}(k,2k+3)=\begin{bmatrix}
   (3k+3)m_{k,t}+u^{+}_{k,t} &m_{k,t}\\ - ((3k+3)v^{-}_{k,t}+w_{k,t}) & -v^-_{k,t}
   \end{bmatrix}$,\quad \vspace{2mm}\\
   $C_{\frac{1}{t}}(k,2k+3)=\begin{bmatrix}
  (3k+3)m_{k,t}+v^{+}_{k,t} &m_{k,t}\\ -((3k+3)u^{-}_{k,t}+w'_{k,t}) & -u^-_{k,t}
   \end{bmatrix}$. 
   \end{itemize}
\end{theorem}

\section{Quotient singularities and $k$-GM numbers}
In this section, we discuss applications of $k$-GM numbers to algebraic geometry. In particular, we consider Hirzebruch-Jung continued fractions (shortly, HJ-continued fractions) for a $k$-GM number and its characteristic numbers. For simplicity of notation, we write HJ-continued fractions as follows:
\[[[b_1,\dots,b_\ell]]=b_1-\dfrac{1}{b_2-\dfrac{1}{\ddots-\dfrac{\ddots}{b_{\ell-1}-\dfrac{1}{b_\ell,}}}}\]
where $b_1,\dots,b_{\ell}$ are integers greater than or equal to $2$.
In this section, we treat the following HJ-continued fraction.
\begin{definition}\label{def:kwahlchainstwo}
Let $k\in \mathbb Z_{\geq 0}$. \emph{$k$-Wahl chains} are defined as follows.
\begin{itemize}\setlength{\leftskip}{-15pt}
\item [(i)] $[[k+2]]$ is a $k$-Wahl chain.
\item[(ii)] If $[[b_1,\dots,b_l]]$ is a $k$-Wahl chain, then $[[b_1+1,b_2,\dots,b_\ell,2]]$ and $[[2,b_1,\dots, b_{\ell-1},b_{\ell}+1]]$ are also $k$-Wahl chains.
\end{itemize}
\end{definition}

\begin{theorem}\label{thm:k-wahl-chain-Markov}
Let $m_{k,t}$ be a $k$-GM number labeled with an irreducible fraction $t \in (0,1]$, and let $u_{k,t}^+$ be its characteristic number. Then the HJ-continued fraction of $m_{k,t}/u_{k,t}^+$ is a $k$-Wahl chain.
\end{theorem}

 $0$-Wahl chains are defined by Urzúa and Zúñiga \cite{UZ}, and they showed that continued fractions obtained from Markov numbers are $0$-Wahl chains. Theorem \ref{thm:k-wahl-chain-Markov} is a generalization of their result.

\subsection{Hirzebruch-Jung continued fractions and toric surface singularities}
We recall a relation between HJ-continued fractions and algebraic geometry.
 Let $G$ be a finite cyclic group generated by the matrix $\begin{pmatrix} 
 \varepsilon^a & 0 \\ 0 & \varepsilon^b 
\end{pmatrix}$, where $a, b$, and $r$ are  positive integers, and $\varepsilon$ is a primitive $r$-th root of unity. We abbreviate this matrix to $\displaystyle \frac{1}{r}(a,b)$.
Since the group $G$ acts on $\CC^2$ by
$
(x,y) \mapsto (\varepsilon^a x, \varepsilon^b y)
$, then we have the quotient space $\CC^2/G$, which is called a \emph{two-dimensional cyclic quotient singularity}.
This is a classical research object in algebraic geometry. The following facts are well known (see \cite{CLS}*{Chapter 10}):
\begin{itemize}\setlength{\leftskip}{-15pt}
  \item If $G$ is a subgroup of $\slmc{2}$, then $\CC^2/G$ is a $A_{r-1}$-type singularity,
  \item an affine toric singular surface is isomorphic to a two-dimensional cyclic quotient singularity, and
  \item a minimal resolution of a two-dimensional cyclic quotient singularity is constructed by the HJ-continued fractions.
\end{itemize}

\begin{definition}
  Let $X$ be a normal variety and denote by $X_{\rm{sing}}$ the set of singular points of $X$. Let $Y$ be a variety. A birational morphism $f:Y \to X$ is a \emph{resolution of singularities} of $X$ if $Y$ is smooth and $f$ induces an isomorphism 
  $$
  Y \backslash f^{-1}(X_{\rm{sing}}) \cong X \backslash X_{\rm{sing}}
  $$
  as varieties.
\end{definition}

 The subset $E$ of $Y$ is called the \emph{exceptional set} if $f(E)=X_{{\rm sing}}$ holds. In geometry of singularities, the properties of singularities appear in the exceptional set of a resolution of singularities.
For a two-dimensional cyclic quotient singularity, the exceptional set is a union of curves $E_1,\dots, E_s$. The self-intersection number of each exceptional curve is given by the HJ-continued fraction $[[b_1,\dots,b_s]]$ (that is, the number of self-intersections of $E_i$ is $-b_i$). In addition, in toric geometry, we can specifically construct a resolution using the HJ-continued fraction.

\subsection{Cyclic quotient singularities and Markov numbers}

\begin{definition}
For a $k$-GM triple $(m_r,m_t,m_s)$ with $m_t > m_r ,m_s$, we define a \emph{$k$-GM group} $G_{m_t}$ as a cyclic group generated by $\displaystyle \frac{1}{m_t}(m_r,m_s)$. 
The quotient space $\CC^2/G_{m_t}$ is called a \emph{$k$-GM quotient singularity}.
\end{definition}

\begin{proposition}\label{Markovquotient}
Let $(m_r,m_t,m_s)$ be a $k$-GM triple and let $u^+_{t}$ and $v^+_{t}$ be characteristic numbers of $m_t$. Then we have
$$
G_{m_t}=\left\langle \frac{1}{m_t}(1,u^+_{t}) \right\rangle=\left\langle \frac{1}{m_t}(v^+_{t},1) \right\rangle.
$$

\end{proposition}

\begin{proof}
It is easy to check by the definition of characteristic numbers. 
\end{proof}

%またトーリック幾何の一般論として$G=\frac{1}{r}(1,a)$の作用による不変式環$\CC[x,y]^G$は$\displaystyle \frac{r}{r-a}$のHJ連分数から構成できる.$k=0$のとき, Proposition\ref{Markovquotient}が意味しているのは$\frac{c}{u_c}=[[a_1,\dots, a_s]]$ならば, $\frac{c}{c-u_c}=[[a_s,\dots, a_1]]$ということである.つまりこの場合, 不変式環の情報と商特異点の特異点解消の例外集合の情報がちょうど逆順に対応しているという特殊な性質が得られる.
%というような説明を書こうか悩み中。書くとしてもhilbert basisを定義する最後に書くべきかも

In order to characterize $k$-GM quotient singularities, we recall the following classes of singularities.

\begin{definition}[\cite{KSB}*{Definition 3.7}] 
A normal surface singularity is of \emph{class $T$} if it is a two-dimensional quotient singularity and admits a $\QQ$-Gorenstein one parameter smoothing.
\end{definition}

\begin{proposition}[\cite{KSB}*{Proposition 3.10}] 
The quotient singularity of class $T$ is either a rational double point or a finite cyclic singularity of type $\displaystyle \frac{1}{dm^2}(1,adm-1)$ with relatively prime integers $d,a > 0 $, where $m > 1$. 
\end{proposition}

The singularity of class $T$ is an important object in the deformation theory of the quotient singularity. Especially when $d=1$, the finite cyclic singularities of type $\displaystyle \frac{1}{m^2}(1,am-1)$ have $\QQ$-Gorenstein smoothings whose Milnor number is $0$. It is called \emph{a Wahl singularity}. 
Originally, ``a Wahl chain" indicates the HJ-continued fraction of $\displaystyle \frac{m^2}{am-1}$. It corresponds to the case $k=2$ in Definition \ref{def:kwahlchainstwo} (that is a $2$-Wahl chain).  The relation between a Wahl singularity and the Markov equation is studied by Hacking and Prokhorov (\cite{hp10}) and  Perling (\cite{Per22}). They considered the minimal resolution of singularities for weighted projective planes $\PP(a^2,b^2,c^2)$, where $(a,b,c)$ is a Markov triple. This essentially means that they are examining the quotient singularities determined by the integer solutions of the $2$-GM equation and their characteristic numbers. We show that the $2$-GM quotient singularity is a Wahl singularity.

\begin{lemma}[\cite{gyo-maru}*{Lemma 4.7}]\label{lem:characteristic equation} 
  Let $m_{k,t}$  be a $k$-GM number labeled with an irreducible fraction $t$, and let $u^+_{k,t}$ be its characteristic number. Then $u^+_{k,t}$ is a solution to $x^2 + kx + 1 \equiv 0 \ {\rm mod} \  m_{k,t}$.
\end{lemma}

\begin{proposition}\label{prop:2-wahl-chain}
The HJ-continued fraction of $\displaystyle \frac{m_{2,t}}{u^+_{2,t}}$ is a $2$-Wahl chain. Namely, the $2$-GM quotient singularity is a Wahl singularity.
\end{proposition}

\begin{proof}
By Lemma \ref{lem:characteristic equation} and Proposition \ref{squre-markov}, we have the formula
\[
(u^+_{{2,t}})^2+2u^+_{{2,t}}+1 \equiv 0 \mod m_{0,t}^2.
\] 
From this, it follows that $u^+_{{2,t}}+1 $ is divisible by $m_{0,t}$. Since there exists a positive integer $a$ which satisfies $u^+_{{2,t}}+1=a \cdot m_{0,t}$, the fraction  $\displaystyle \frac{m_{2,t}}{u^+_{{2,t}}}=\frac{(m_{0,t})^2}{a\cdot m_{0,t}-1}$ is a $2$-Wahl chain.

\end{proof}

We define \emph{the $k$-Wahl chain tree} as follows:
\begin{itemize}\setlength{\leftskip}{-15pt}
  \item[(1)] The root vertex is $[[k+2]]$,
  \item[(2)] every vertex $[[b_1,\dots,b_l]]$ has two children $[[b_1+1,b_2,\dots,b_\ell,2]]$ and $[[2,b_1,\dots, b_{\ell-1},b_{\ell}+1]]$.
\end{itemize}
There is a one to one correspondence between the $2$-Wahl chain tree and the Farey tree (see \cite{UZ}).

\subsection{Proof of Theorem \ref{thm:k-wahl-chain-Markov}}

The following lemma gives a relation between regular continued fractions and HJ-continued fractions.

\begin{lemma}
For a variable $x$ and positive integers $a_1,a_2$, the following equality holds: \[[a_1,a_2,x]=[[a_1+1,(2)^{a_2-1},x+1]],\]
 where $(2)^{\ell}$ denotes a string of $\ell\ 2'$s. 
\end{lemma}

\begin{proof} We will prove by using induction on $a_2$.
  When $a_2=1$, we have
  \[ [a_1,1,x]=a_1+\frac{x}{x+1}=a_1+1-\frac{1}{x+1}=[[a_1+1,x+1]].\] 
  We assume that $[a_1,k,x]=[[a_1+1,(2)^{k-1},x+1]]$, then it follows that 
  \[
  [a_1,k+1,x]=\left[a_1, k , x/(x+1)\right]=[[a_1+1,(2)^{k-1},1+x/(x+1)]].
  \]
  Since $\displaystyle 1+\frac{x}{x+1}=2-\frac{1}{x+1}$, we conclude that
  \[[a_1, k+1, x]=[[a_1+1,(2)^{k-1},2,x+1]]=[[a_1+1,(2)^{k},x+1]]. \] 
\end{proof}

This lemma leads to the following corollary.

\begin{corollary}\label{cor:posi-to-negative-cf}
  Let $a_i$ be a positive integer for all $i=1,\dots,\ell$. Then we have
\[
   [a_1,\dots,a_{\ell}] = \left\{
\begin{array}{ll}
[[a_1+1,(2)^{a_2-1},a_3+2, (2)^{a_4-1},\dots, a_{\ell-1}+2, (2)^{a_{\ell}-1}]] & \text{if $\ell$ is even,}\\
\lbrack\lbrack a_1+1,(2)^{a_2-1},a_3+2, (2)^{a_4-1},\dots, (2)^{a_{\ell-1}-1}, a_{\ell} + 1\rbrack\rbrack & \text{if $\ell$ is odd.}
\end{array}
\right.
\]

\end{corollary}

We will prove Theorem \ref{thm:k-wahl-chain-Markov}. %This theorem contains Proposition \ref{prop:2-wahl-chain}.
\begin{proof}[Proof of Theorem \ref{thm:k-wahl-chain-Markov}]
Let $m_{k,t}$ be a $k$-GM number labeled with an irreducible fraction $t \in (0,1]$ and  $u_{k,t}^+$ its characteristic number.
By Lemma \ref{lem:semi-palindrome} and Theorem \ref{thm:u_t-denominator}, we have
\[
\frac{m_{k,t}}{u_{k,t}^+}=[a_1,\dots,a_{\frac{\ell}{2}},a_{\frac{\ell}{2}}+k,a_{\frac{\ell}{2}-1},\dots,a_1] \text{ or }[a_1,\dots,a_{\frac{\ell}{2}},a_{\frac{\ell}{2}}-k,a_{\frac{\ell}{2}-1},\dots,a_1] .
\]
We assume $\frac{\ell}{2}$ is even. Then $\displaystyle \frac{m_{k,t}}{u_{k,t}^+}=[a_1,\dots,a_{\frac{\ell}{2}},a_{\frac{\ell}{2}}+k,a_{\frac{\ell}{2}-1},\dots,a_1]$. Applying Corollary \ref{cor:posi-to-negative-cf} to this continued fraction, we have
\[
\frac{m_{k,t}}{u_{k,t}^+}=[[a_1+1,(2)^{a_2-1},a_3+2,\dots,a_{\frac{\ell}{2}-1}+2, (2)^{a_{\frac{\ell}{2}}-1}, a_{\frac{\ell}{2}}+k+2, (2)^{a_{\frac{\ell}{2}-1}-1},\dots, a_{2}+2, (2)^{a_1-1}]].
\]
By definition, $[[b_1+2, b_2,\dots, b_s, (2)^{b_1}]]$ is a $k$-Wahl chain if and only if $[[2, b_2,\dots, b_s ]]$ is a $k$-Wahl chain. We reduce the continued fractions as follows:
  \begin{align*}
&\ [[a_1+1,(2)^{a_2-1},a_3+2,\dots,a_{\frac{\ell}{2}-1}+2, (2)^{a_{\frac{\ell}{2}}-1}, a_{\frac{\ell}{2}}+k+2, (2)^{a_{\frac{\ell}{2}-1}-1},\dots, a_{2}+2, (2)^{a_1-1}]]\\
\to& \  [[2,(2)^{a_2-1},a_3+2,\dots,a_{\frac{\ell}{2}-1}+2, (2)^{a_{\frac{\ell}{2}}-1}, a_{\frac{\ell}{2}}+k+2, (2)^{a_{\frac{\ell}{2}-1}-1},\dots, a_{2}+2]]\\
\to& \  [[a_3+2,\dots,a_{\frac{\ell}{2}-1}+2, (2)^{a_{\frac{\ell}{2}}-1}, a_{\frac{\ell}{2}}+k+2, (2)^{a_{\frac{\ell}{2}-1}-1},\dots,(2)^{a_3-1},2]]\\
\to& \  \dots \to [[(2)^{a_{\frac{\ell}{2}}}, a_{\frac{\ell}{2}}+ k + 2]]\\
\to& \ [[k+2]].
\end{align*}

  \normalsize
Therefore, $\displaystyle \frac{m_{k,t}}{u_{k,t}^+}$ is a $k$-Wahl chain. We can apply the same argument to the case $\frac{\ell}{2}$ is odd. 
%Namely,  $\displaystyle \frac{m_t}{u_{k,t}^+}=[a_1,\dots,a_{\ell/2},a_{\ell/2}-k,a_{\ell/2-1},\dots,a_1]$.

\end{proof}

\begin{remark}
The converse of Theorem \ref{thm:k-wahl-chain-Markov} does not hold. Indeed, the HJ-continued fraction $\displaystyle \frac{10}{3}=[[422]]$ is a $0$-Wahl chain, but $10$ is not a $0$-GM number.
\end{remark}

By Proposition \ref{Markovquotient}, $\displaystyle \frac{m_{k,t}}{v_{k,t}^+}$ is a $k$-Wahl chain. However,  $\displaystyle \frac{m_{k,t}}{u_{k,t}^-}$ and  $\displaystyle \frac{m_{k,t}}{v_{k,t}^-}$ are not $k$-Wahl chains. Instead, they are Wahl chains starting with $[[(2)^{k+1}]]$.

\begin{definition}\label{def:dualkwahlchains}
Let $k\in \mathbb Z_{\geq 0}$. \emph{Dual $k$-Wahl chains} are defined as follows.
\begin{itemize}\setlength{\leftskip}{-15pt}
\item [(i)] $[[(2)^{k+1}]]$ is a dual $k$-Wahl chain.
\item[(ii)] If $[[b_1,\dots,b_l]]$ is a dual $k$-Wahl chain, then $[[b_1+1,b_2,\dots,b_\ell,2]]$ and $[[2,b_1,\dots, b_{\ell-1},b_{\ell}+1]]$ are also dual $k$-Wahl chains.
\end{itemize}
\end{definition}

\begin{proposition}\label{prop:dual-k-wahl-chain-Markov}
Let $m_{k,t}$ be a $k$-GM number labeled with an irreducible fraction $t \in (0,1]$, and let $u_{k,t}^-$ be its characteristic number. Then the HJ-continued fraction of $m_{k,t}/u_{k,t}^-$ is a dual $k$-Wahl chain.
\end{proposition}

\begin{proof}
    By the proof of Theorem \ref{thm:uv-continued-fraction}, we have 
    \[
\frac{m_{k,t}}{u_{k,t}^-}=[1,a_1-1,a_2,\dots,a_{\frac{\ell}{2}},a_{\frac{\ell}{2}}+k,a_{\frac{\ell}{2}-1},\dots,a_1] \text{ or }[1, a_1-1, a_2, \dots,a_{\frac{\ell}{2}},a_{\frac{\ell}{2}}-k,a_{\frac{\ell}{2}-1},\dots,a_1] .
\]
Note that the length of this regular continued fraction is odd. Applying Corollary \ref{cor:posi-to-negative-cf}, we have
\[
\frac{m_{k,t}}{u_{k,t}^-}=[[2,(2)^{a_1-2},a_2+2, (2)^{a_3-1},\dots,(2)^{a_{\frac{\ell}{2}-1}}, a_{\frac{\ell}{2}}+2, (2)^{a_{\frac{\ell}{2}+k}-1}, a_{\frac{\ell}{2}-1}+1,\dots, (2)^{a_2-1},a_1+1 ]].
\]
By the similar argument to the proof of Theorem \ref{thm:k-wahl-chain-Markov}, we conclude that this HJ-continued fraction is a dual $k$-Wahl chain.
\end{proof}

\subsection{Characterization of HJ-continued fractions obtained from $k$-GM numbers}
 We will propose a generalization of Propoition \ref{UZ}. 
For an irreducible fraction $t \in (0, 1)$ and a non-negative integer $k$, we assume that $F^+(k,t)=[a_1,\dots,a_{\ell}]$. We will set $F^-(k,t)=[[b_1,\dots,b_s]]$ the HJ-continued fraction of $[a_1,\dots,a_{\ell}]$, that is, $[a_1,\dots,a_{\ell}]=[[b_1,\dots,b_s]]$.

\begin{proposition}\label{prop:fibonacchi-branch}
For positive integers $a, a_1,\dots, a_s$, we set $F^-\left(k,1/a\right)=[[a_1,\dots,a_s]]$. Then we have
\[
 F^-(k,1/(a+2)) =[[2k+3,a_1,\dots,a_{s-1},a_s+1,(2)^{2k+1}]].
\]

\end{proposition}

\begin{proof}
%First, for irreducible fractions $\displaystyle t=\frac{1}{1}$ and $\displaystyle \frac{1}{2}$, it is easily seen that $\displaystyle F^+\left(k, \frac{1}{1}\right)=[k+1,1]$, $\displaystyle F^+\left(k, \frac{1}{2}\right)=[2k+2,k+2]$.
%↑関係式を証明するのでいらないけれど特別な場合の確認。
We assume $\displaystyle F^+\left(k,1/a\right)=[b_1,\dots,b_{\ell}]$. By Proposition \ref{prop:presnake-relation}, we have $\displaystyle F^+\left(k,1/(a+2)\right)=[2k+2,1,b_1-1,b_2,\dots,b_{l-1},b_l-1,2k+2]$.
This continued fraction can be transformed into a HJ-continued fraction by applying Corollary \ref{cor:posi-to-negative-cf}. Thus we have $\displaystyle F^-\left(k,1/(a+2)\right)=[[2k+3,a_1,\dots,a_{s-1},a_s+1,(2)^{2k+1}]]$.

\end{proof}

%Especially, we have $\displaystyle F^-\left(0,1/a\right)=[[(3)^{a-1},2]]$. The operation of moving to a ``deep" node of the Fibonacci branch of the Markov tree corresponds to the operation of adding a $3$ to the HJ-continued fraction (i.e., adding an exceptional $(-3)$-curve in the resolution of a quotient singularity).\\
Next, we consider $F^-(k,t)$ for general irreducible fraction $t$. For a Farey triple $(r,t,s)$ with $\displaystyle r\neq \frac{0}{1}, s\neq\frac{1}{1}$, we set 
$F^+(k,r)=[a_1,\dots,a_{\ell}]$, and $F^+(k,s)=[b_1,\dots,b_{m}]$. By Proposition \ref{prop:presnake-relation}, we have $F^+(k,t)=[a_{\ell},\dots, a_1, 3k+2, 1 ,b_{m}-1, b_{m-1}, \dots, b_1]$.
By Lemma \ref{lem:semi-palindrome}, we have
\[
 F^+(k,r) = \left\{
\begin{array}{ll}
 [a_1,\dots,a_{\frac{\ell}{2}},a_{\frac{\ell}{2}}+k,a_{\frac{\ell}{2}-1}, \dots ,a_1] & \text{if $\frac{\ell}{2}$ is even,}\\
\lbrack a_1,\dots,a_{\frac{\ell}{2}},a_{\frac{\ell}{2}}-k,a_{\frac{\ell}{2}-1}, \dots ,a_1 \rbrack & \text{if $\frac{\ell}{2}$ is odd.}
\end{array}
\right. 
\]

We apply Proposition \ref{prop:F-G-relation} to $F^+(k,r)$, then 
$$
 G^+(k,r) = \left\{
\begin{array}{ll}
 [a_1,\dots,a_{\frac{\ell}{2}-1},a_{\frac{\ell}{2}}+k,a_{\frac{\ell}{2}},a_{\frac{\ell}{2}-1}, \dots ,a_1] & \text{if $\frac{\ell}{2}$ is even,}\\
\lbrack a_1,\dots,a_{\frac{\ell}{2}-1},a_{\frac{\ell}{2}}-k,a_{\frac{\ell}{2}},a_{\frac{\ell}{2}-1}, \dots ,a_1 \rbrack & \text{if $\frac{\ell}{2}$ is odd.}
\end{array}
\right. 
$$
%F^+(t)=[a_{1},\dots, a_l+k, a_l, \dots, a_1, 3k+2, 1 ,b_{1}-1, b_m+k,b_m, \dots, b_1]
In addition, we will denote by $G^{-}(k,r)$ the HJ-continued fraction determined by $G^{+}(k,r)$.
Corollary \ref{cor:posi-to-negative-cf} leads to the following Theorem. 

\begin{theorem} \label{thm:smallk-GM}
Under the above assumptions, the following holds:
\[
F^-(k,t)=[[ G^{-}(k,r), 3k+4, G^{-}(k,s) ]].
\]
\end{theorem}

For example, if $\dfrac{\ell}{2}$ is even and $\dfrac{m}{2}$ is odd, then we have
  \begin{align*}
   F^-(k,r)&=[[a_1+1, (2)^{a_2-1},\dots, (2)^{a_{\frac{\ell}{2}}-1}, a_{\frac{\ell}{2}}+k+2, \dots, (2)^{a_1-1}]],\\
F^-(k,s)&=[[b_1+1, (2)^{b_2-1},\dots, b_{\frac{m}{2}}+2, (2)^{b_{\frac{m}{2}}-k-1}, b_{\frac{m}{2}-1}+2, \dots, (2)^{b_1-1}]],\\
F^-(k,t)&=[[a_1+1, (2)^{a_2-1},\dots, (2)^{a_{\frac{\ell}{2}}+k-1}, a_{\frac{\ell}{2}}+2, \dots, (2)^{a_1-1}, 3k+4, \\
\quad & \qquad  b_1+1, (2)^{b_2-1},\dots, (2)^{b_{\frac{m}{2}-1}-1}, b_{\frac{m}{2}}-k+2, (2)^{b_{\frac{m}{2}}-1}, b_{\frac{m}{2}-1}+2, \dots, (2)^{b_1-1}]].
\end{align*}

By the above theorem and Theorem \ref{thm:uv-continued-fraction}, we have 
\[
\frac{m_{k,t}}{u_{k,t}^+}=\left[\left[ \frac{m_{k,r}}{v_{k,r}^-}, 3k+4, \frac{m_{k,s}}{v_{k,s}^-} \right]\right],
\]
where $(m_{k,r},m_{k,t},m_{k,s})$ is a $k$-GM triple and $u^+_{k,t}, v^-_{k,r}, v^-_{k,s}$ are these characteristic numbers.
Since $G^-(k,r)=F^-(k,r)$ holds if $k=0$, this theorem is a generalization of Proposition \ref{UZ}.

\subsection{Hilbert basis and HJ-continued fractions}
In this section, we recall some definition and notation of toric geometry, and we explain how the HJ-continued fraction induces resolution of singularities. For details of a toric variety and proofs of propositions and theorems, see \cite{CLS}.
We construct an affine toric variety determined by a polyhedral cone. For simplicity, we only deal with $2$-dimensional cones and toric surfaces.

Let $N$ be $\ZZ^2$ and $N_{\RR} = N \otimes_{\ZZ}\RR$, that is, $N_{\RR} \cong \RR^{2}$. 
Let $\mathbf{e}_1=(1,0),\mathbf{e}_2=(0,1)$ be the canonical basis of the vector space $N_{\RR}$. 
For some $v_1, v_2 \in N$, we define a \emph{rational strongly convex polyhedral cone} $\sigma$ as $\sigma =\RR_{\geq0}v_1 + \RR_{\geq0}v_2$, where $\RR_{\geq0}$ is the set of all non negative elements in $\RR$. We write $\sigma={\rm Cone}(v_1,v_2)$. The \emph{dimension} of a cone $\sigma$ is defined as the dimension of the vector space over $\RR$ generated by $(v_1, v_2)$. 

The dual of a lattice $N$ is defined as $M = N^{\vee} = \mathrm{Hom}_{\ZZ}(N, \ZZ)$, and it is denoted by $\langle n, m \rangle = m(n)$ for $n \in N$ and $m \in M$. For $M_{\RR}= M \otimes_{\ZZ} \RR$ , we will also denote by $\langle \  ,\   \rangle : M_{\RR} \times N_{\RR} \to \RR$ the natural pairing.
The dual of $\sigma$ is given by
$$
\sigma^{\vee} = \{u \in M_{\RR} \mid \langle u, v \rangle \geq 0 \text{ for all } v \in \sigma \}.
$$ 
We introduce a semigroup $S_{\sigma}$ and an affine toric variety $U_{\sigma}$ associated with the cone $\sigma$ as follows:
\begin{align*}
S_{\sigma} &= \sigma^{\vee} \cap M = \{u \in M \mid \langle u, v \rangle \geq 0 \text{ for all } v \in \sigma \}, \\
U_{\sigma} &= \mathrm{Spec}\left(\CC[S_{\sigma}]\right),
\end{align*}
where $\CC[S_{\sigma}]$ is the group ring generated by the semigroup $S_{\sigma}$, and $\mathrm{Spec}\left(\CC[S_{\sigma}]\right)$ denotes an affine variety with coordinate ring $\CC[S_{\sigma}]$.

\begin{remark}
  Let $r$, $a$ be positive integers that are relatively prime. We set $G=\frac{1}{r}(1,a)$ and $\sigma={\rm Cone}(\mathbf{e}_2,r\mathbf{e}_1-a\mathbf{e}_2)$. Let $N^{\prime}$ denote the sublattice generated by the ray generators of $\sigma$. Then we have $G\cong N/N^{\prime}$ and $\CC[S_{\sigma}] \cong \CC[x,y]^G$. Namely, $U_{\sigma}$ is isomorphic to $\CC^2/G$.
\end{remark}

\begin{definition}
  For a  rational strongly convex polyhedral cone $\sigma$, we define a {\it face} $\tau \subset \sigma$ as 
\[
\tau=\sigma \cap u^\perp=\{ v \in \sigma | \langle u,v \rangle =0 \}
\]
for some $u \in \sigma^{\vee}$.
\end{definition}

\begin{definition}
  A set $\Sigma$ of rational strongly convex polyhedral cones  is called a {\it fan} if it satisfies the following:
  \begin{itemize}\setlength{\leftskip}{-15pt}
    \item Each face of a cone in $\Sigma$ is also a cone in $\Sigma$.
    \item The intersection of two cones in $\Sigma$ is a face of each cone.
  \end{itemize}
\end{definition}

A toric variety $X_{\Sigma}$ is defined by naturally gluing affine toric varieties corresponding to each cone in the fan $\Sigma$.

\begin{definition}
   A rational strongly convex polyhedral cone $\sigma$ is {\it smooth} if its
minimal set of generators is a part of an integral basis of $N$. A fan is {\it  smooth} if every cone in the fan is smooth.
\end{definition}
%上下二つの定義と命題はまとめて一つにしようとも考えています。
\begin{proposition}[\cite{CLS}*{Theorem 1.3.12}]
 A cone $\sigma$ is smooth if and only if $U_{\sigma}$ is a smooth surface.
\end{proposition}

Let us explain a relation between the resolution of toric surfaces and the HJ-continued fractions. 

\begin{theorem}[\cite{CLS}*{Theorems 10.2.3, 10.2.5}]\label{thm:Hilbert basis}
  Let $\sigma$ be ${\rm Cone}(\mathbf{e}_2,r\mathbf{e}_1-a\mathbf{e}_2)$, and let $\dfrac{r}{a}=[[a_1,\dots, a_s]]$. Let $u_0, u_1,\dots, u_{s+1}$ be vectors which satisfy 
  $$
  u_{i-1}+u_{i+1} = a_iu_i, \qquad \text{for} \  1\leq i \leq s,
  $$
   where $u_0=\mathbf{e}_2$ and $u_{s+1}=r\mathbf{e}_1-a\mathbf{e}_2.$
Then the cones $\sigma_i={\rm Cone}(u_{i-1},u_i)$ have the following properties:
\begin{itemize}\setlength{\leftskip}{-15pt}
  \item[(i)] Each $\sigma_i$ is a smooth cone and $\sigma_1 \cup \dots \cup \sigma_{s+1} = \sigma$.
  \item[(ii)] For each $i$, $\sigma_{i} \cap \sigma_{i+1}={\rm Cone}(u_i)$.
  \item[(iii)] Let $\Sigma$  be a fan consisting of the $\sigma_i$'s and their faces. Then the toric morphism $\phi: X_{\Sigma} \to U_{\sigma}$ is a resolution of singularities.
  \item[(iv)]  Let $E_i$ be an exceptional curve corresponding to a one-dimensional cone ${\rm Cone}(u_i)$ for $1 \leq i \leq s$. Then its self-intersection number is $-b_i$.
\end{itemize}
  
\end{theorem}

This theorem means that calculating the HJ-continued fraction will induce a resolution of singularity.

\begin{definition}
Let  $\mathrm{Hlb}_{N}(\sigma)$ be as follows:
\[
  \mathrm{Hlb}_{N}(\sigma)=\left\{ n \in \sigma \cap (N \backslash \{ 0 \}) \left|
\begin{array}{c}
 \text{$n$ can not be expressed as}\\  
 \text{a sum of two other vectors}\\ 
 \text{belonging to $\sigma \cap (N \backslash \{ 0 \})$}
\end{array}
\right.\right\}.
\]
The set $\mathrm{Hlb}_{N}(\sigma)$ is called the \emph{Hilbert basis} of $\sigma$ with reference to the lattice $N$.
\end{definition}

\begin{remark}\label{rem:hilbertbasis}
  In the above setting, the following are well known.
  \begin{itemize}\setlength{\leftskip}{-15pt}
    \item The set $\mathrm{Hlb}_{N}(\sigma)$ is $\{ u_0, \dots, u_{s+1} \}$. 
    \item The dual HJ-continued fraction $r/(r-a)=[[b_1,\dots b_t]]$ gives vectors $v_0, \dots,v_{t+1}$ in $M$ similarly to Theorem \ref{thm:Hilbert basis}. Moreover, $\mathrm{Hlb}_{M}(\sigma^{\vee})=\{ v_0, \dots,v_{t+1} \}$.
  \end{itemize}
\end{remark}

This remark indicates that the generators of the invariant ring $\CC[x,y]^G$ are obtained by the HJ-continued fraction of $\displaystyle \frac{r}{r-a}$. In other words, for an irreducible fraction $t\in(0,1]$, we have a $k$-GM number $m_t$ and its characteristic numbers $u_t^+$ and $u_t^-$. Then the HJ-continued fraction of $\displaystyle \frac{m_t}{u_t^+}$ gives a minimal resolution of the $k$-GM quotient singularity $\CC^2/G$, and the HJ-continued fraction of $\displaystyle \frac{m_t}{u_t^-}$ gives the basis of the invariant ring $\CC[x,y]^G$.
Especially, if $k=0$ and $\displaystyle \frac{m_t}{u_t^+}=[[b_1,\dots, b_s]]$, then we have $\displaystyle \frac{m_t}{u_t^-}=[[b_s,\dots,b_1]]$.

%具体的なHJ連分数の記述は$\frac{m_t}{u_t^+}=[[a_1+1,(2)^{a_2-1}, a_3+2, (2)^{a_4-1}, \dots, a_{\ell-1}+2, (2)^{a_{\ell}-1} ]]$のとき[[(2)^{a_1-1},a_2+2,(2)^{a_3-1},\dots, (2)^{a_{\ell}-1},a_{\ell}+1]]という形でかけるが余りきれいではないため記述していない。

Let $f$ and $g$ be operations that give the inverse order of the HJ-continued fraction and the regular continued fraction, respectively. Let us denote by $h$ the operation that gives the dual HJ-continued fraction.
By Proposition  \ref{prop:F-G-relation}, Proposition \ref{prop:tto1/t}, Proposition \ref{Markovquotient} and Remark \ref{rem:hilbertbasis}, we have the following relations.

\[
\scalebox{1}{
\xymatrix@C+4pc@R+1pc{
F^+(t)=\frac{m_t}{u^+_t}\ar@{<->}[r]^-{f}\ar@{<->}[d]_-{g}\ar@{<->}[rd]& F^+\left(\frac{1}{t}\right)=\frac{m_t}{v^+_t}\ar@{<->}[d]^-{g}\\
G^+(t)=\frac{m_t}{v^-_t}\ar@{<->}[r]_-{f} \ar@{<->}[ur]^(0.52){h}& 
G^+\left(\frac{1}{t}\right)=\frac{m_t}{u^-_t}
}
}
\]

\subsection*{Acknowledgements}
The authors would like to express their gratitude to Tsukasa Ishibashi for providing insightful discussions, for checking the contents of Section 5.4 and for providing appropriate references. They also thank Shunsuke Kano for sharing valuable discussions. Special thanks to Toshiki Matsusaka for inviting the first author to Kyudai Daisugaku Seminar, providing a platform for lively discussions, and offering helpful comments. Gratitude is extended to Kazuhiro Hikami for providing the initial comments that inspired this research. Thanks are also due to Osamu Iyama for providing valuable advice on the direction of the research.

\subsection*{Funding}
The first author is supported by JSPS KAKENHI Grant Number JP22KJ0731, and the second author is partially supported by JSPS KAKENHI Grant Number JP23KJ1938, JP23K12971.

\subsection*{Data Availability}
Our manuscript has no associated data.

\section*{Declarations}
\subsection*{Conflict of interest}
On behalf of all authors, the corresponding author states that there is no Conflict of interest.

\bibliography{myrefs}
\end{document}